%% file: ThesisHartglass.tex
\documentclass{ucbthesis}
\usepackage{indentfirst}

\include{Thesis2013Preamble}

\begin{document}


\title{Free Probability, Planar Algebras, and the Multicolored Guionnet-Jones-Shlyakhtenko Construction}
\author{Michael Aaron Hartglass}
\degreesemester{Spring}
\degreeyear{2013}
\degree{Doctor of Philosophy}
\chair{Professor Vaughan F.R. Jones}
\othermembers{Professor Marc Rieffel \\
  Professor Ori Ganor}
\numberofmembers{3}
\prevdegrees{B.A. Mathematics with Honors (University of Rochester) 2007
  }
\field{Mathematics}
\campus{Berkeley}




\maketitle
\copyrightpage

\include{abstract2013}

\begin{frontmatter}

\begin{dedication}
\null\vfil
\begin{center}
This dissertation is dedicated to Douglas Knowles.
\end{center}

\vfil\null
\end{dedication}

\tableofcontents

\begin{acknowledgements}
Vaughan F.R. Jones:  Thank you for your positive attitude, guidance, and constant encouragement.  At many times, it seemed like you believed in me much more than I believed in myself.  In addition, thank you for funding me for countless trips to Nashville as well as other places.  These trips greatly enhanced my mathematical career. Finally, thank you for not being upset after losing so many things in my forays into Bodega Bay!

Thank you to the subfactor and operator algebra community for providing a nurturing environment, as well as many outstanding times together.  In particular, I would like to thank Dave Penneys, James Tener, Arnaud Brothier, Scott Morrison, Noah Snyder, Emily Peters, Dietmar Bisch, Dimitri Shlyakhtenko, Dan Voiculescu, Marc Rieffel, Andre Kornell, Adam Merberg, Weihua Liu, Jamie Mingo, Ken Dykema, and Yasuyuki Kawahigashi.

Thank you to all of my past and present colleagues in the Berkeley math department who became some of the closest friends I have made, and who have shared many good times with me:  Ivan Ventura, Morgan Brown, Kelty Allen, Ian herbert, Cynthia Vinzant, Brad Frohele, Rob Bayer, Kiril Datchev, Doug Rizzolo, Taryn Flock, Ralph Morrison, and Jeff Galkowski.

I am very thankful for the various musical groups I have been involved in during my time in the bay area:  The Math Department Orchestra, West Coundy Winds, Oakland Civic Orchestra, and especially the UC Berkeley Symphony Orchestra.  You have all taught me that there is far more to life than mathematics.  A special thanks to Don Howe for helping improve my bass trombone playing to levels I thought I could not attain.

I feel very blessed to be supported by such a great family.  To my parents, Becky and Bill, thank you for always stressing education and hard work, as well as encouraging me to broaden my interests to places other than math.  I could not ask for more supportive parents.  To my brother, David, thank you for all of the fun times growing up and beyond.  It has been a real treat to see my younger grow into a successful adult.

Finally, an immense thank you to Sarah Knowles.  Sarah, I could not have asked for a better partner.  Thank you for being so supportive of me during the last three years, especially during the times that were not the easiest for you.  I see nothing but better days ahead.

\end{acknowledgements}

\end{frontmatter}

\pagestyle{headings}

\include{Thesis2013Introduction}
\include{Thesis2013Chapter-Graph}
\include{Thesis2013Chapter-Universal}
\include{Thesis2013Chapter-FC}


\bibliographystyle{amsalpha}
\bibliography{ThesisHartglass}
\end{document}

%% file: Thesis2013Preamble.tex
\usepackage{amsmath, amsthm, amssymb}
\usepackage{enumerate}
\usepackage{ifpdf}
\ifpdf
\usepackage[pdftex]{graphicx}
\else
\usepackage[dvips]{graphicx}
\fi
\usepackage{tikz}
 	 \usetikzlibrary{arrows,backgrounds}
\usepackage[all]{xy}

\input xy
\xyoption{all}

\usepackage[pdftex,plainpages=false,hypertexnames=false,pdfpagelabels]{hyperref}
\newcommand{\arXiv}[1]{\href{http://arxiv.org/abs/#1}{\texttt{arXiv:\nolinkurl{#1}}}}
\newcommand{\doi}[1]{\href{http://dx.doi.org/#1}{{\texttt{DOI:#1}}}}

\newcommand{\googlebooks}[1]{(preview at \href{http://books.google.com/books?id=#1}{google books})}

\usepackage{xcolor}
\definecolor{dark-red}{rgb}{0.7,0.25,0.25}
\definecolor{dark-blue}{rgb}{0.15,0.15,0.55}
\definecolor{medium-blue}{rgb}{0,0,.8}
\definecolor{DarkGreen}{RGB}{0,150,0}
\hypersetup{
   colorlinks, linkcolor={purple},
   citecolor={medium-blue}, urlcolor={medium-blue}
}

\theoremstyle{plain}
\newtheorem{thm}{Theorem}[section]
\newtheorem*{thm*}{Theorem}
\newtheorem*{cor*}{Corollary}
\newtheorem{cor}[thm]{Corollary}
\newtheorem{lem}[thm]{Lemma}

\newtheorem{quest}[thm]{Question}
\newtheorem*{quest*}{Question}

\theoremstyle{definition}
\newtheorem{defn}[thm]{Definition}
\newtheorem{assumption}[thm]{Assumption}
\newtheorem{nota}[thm]{Notation}

\newtheorem{note}[thm]{Note}

\newtheorem{ex}[thm]{Example}
\newtheorem{rem}[thm]{Remark}

\DeclareMathOperator{\coev}{coev}

\DeclareMathOperator{\End}{End}
\DeclareMathOperator{\Epi}{Epi}
\DeclareMathOperator{\fdim}{fdim}
\DeclareMathOperator{\Hom}{Hom}
\DeclareMathOperator{\Gr}{Gr}
\DeclareMathOperator{\ev}{ev}
\DeclareMathOperator{\id}{id}

\DeclareMathOperator{\OP}{op}

\DeclareMathOperator{\spann}{span}

\DeclareMathOperator{\Tr}{Tr}
\DeclareMathOperator{\tr}{tr}

\newcommand{\be}{\begin{enumerate}}
\newcommand{\ee}{\end{enumerate}}

\newcommand{\N}{\mathbb{N}}
\newcommand{\Z}{\mathbb{Z}}

\newcommand{\F}{\mathbb{F}}
\newcommand{\R}{\mathbb{R}}
\newcommand{\C}{\mathbb{C}}

\renewcommand{\P}{\mathbb{P}}
\newcommand{\I}{\infty}
\newcommand{\set}[2]{\left\{#1 \middle| #2\right\}}

\newcommand{\cA}{\mathcal{A}}
\newcommand{\cB}{\mathcal{B}}
\newcommand{\cC}{\mathcal{C}}

\newcommand{\cF}{\mathcal{F}}
\newcommand{\cG}{\mathcal{G}}
\newcommand{\cH}{\mathcal{H}}

\newcommand{\cL}{\mathcal{L}}
\newcommand{\cM}{\mathcal{M}}
\newcommand{\cN}{\mathcal{N}}

\newcommand{\cP}{\mathcal{P}}
\newcommand{\cQ}{\mathcal{Q}}

\newcommand{\cS}{\mathcal{S}}

\newcommand{\cV}{\mathcal{V}}
\newcommand{\cW}{\mathcal{W}}

\newcommand{\PA}{{\mathcal{P}\hspace{-.1cm}\mathcal{A}}}
\DeclareMathOperator{\Pro}{{Pro}}
\newcommand{\CF}{{\mathcal{CF}}}
\DeclareMathOperator{\Bim}{{Bim}}
\DeclareMathOperator{\Rep}{{Rep}}
\newcommand{\op}{^{\OP}}

\newcommand{\noshow}[1]{}
\newcommand{\MR}[1]{}
\newcommand{\bluecup}{\textcolor{\cupcolor}{\cup}}

\tikzstyle{shaded}=[fill=red!10!blue!20!gray!30!white]
\tikzstyle{unshaded}=[fill=white]
\tikzstyle{empty box}=[circle, draw, thick, fill=white, opaque, inner sep=2mm]
\tikzstyle{annular}=[scale=.7, inner sep=1mm, baseline]
\tikzstyle{rectangular}=[scale=.75, inner sep=1mm, baseline=-.1cm]

\newcommand{\cupcolor}{DarkGreen}
\newcommand{\alphacolor}{blue}
\newcommand{\betacolor}{red}

\newcommand{\Mbox}[2]{
\begin{tikzpicture}[baseline=-.1cm]
	\draw (0,0)--(0,.8);
	\draw[thick, unshaded] (-.4, -.4) -- (-.4, .4) -- (.4, .4) -- (.4,-.4) -- (-.4, -.4);
	\node at (.2,.6) {{\scriptsize{$#1$}}};
	\node at (0,0) {$#2$};
\end{tikzpicture}
}

\newcommand{\nbox}[3]{
\begin{tikzpicture}[baseline=-.1cm]
	\draw (0,-.8)--(0,.8);
	\draw[thick, unshaded] (-.4, -.4) -- (-.4, .4) -- (.4, .4) -- (.4,-.4) -- (-.4, -.4);
	\node at (.2,.6) {{\scriptsize{$#1$}}};
	\node at (0,0) {$#2$};
	\node at (.2,-.6) {{\scriptsize{$#3$}}};
\end{tikzpicture}
}

\newcommand{\ColorNBox}[3]{
\begin{tikzpicture}[baseline=-.1cm]
	\draw[thick, #3] (0,-.8)--(0,0);
	\draw (0,0)--(0,.8);
	\draw[thick, unshaded] (-.4, -.4) -- (-.4, .4) -- (.4, .4) -- (.4,-.4) -- (-.4, -.4);
	\node at (.2,.6) {{\scriptsize{$#1$}}};
	\node at (0,0) {$#2$};
\end{tikzpicture}
}

\newcommand{\widenbox}[3]{
\begin{tikzpicture}[baseline=-.1cm]
	\draw (0,-.8)--(0,.8);
	\draw[thick, unshaded] (-.65, -.4) -- (-.65, .4) -- (.65, .4) -- (.65,-.4) -- (-.65, -.4);
	\node at (.25,.6) {{\scriptsize{$#1$}}};
	\node at (0,0) {$#2$};
	\node at (.25,-.6) {{\scriptsize{$#3$}}};
\end{tikzpicture}
}

\newcommand{\rotateccw}[5]{
\begin{tikzpicture}[baseline=-.1cm]
	\clip (-1.1,-.9)--(-1.1,.9)--(1.1,.9)--(1.1,-.9);	
	\draw (0,.4) arc (180:0:.4cm)--(.8,-.8);
	\draw (0,-.4) arc (0:-180:.4cm)--(-.8,.8);
	\draw[thick, unshaded] (-.4, -.4) -- (-.4, .4) -- (.4, .4) -- (.4,-.4) -- (-.4, -.4);
	\node at (0,0) {$#5$};
	\node at (-1,.6) {{\scriptsize{$#1$}}};
	\node at (-.2,.6) {{\scriptsize{$#2$}}};
	\node at (.2,-.6) {{\scriptsize{$#3$}}};
	\node at (1,-.6) {{\scriptsize{$#4$}}};
\end{tikzpicture}
}

\newcommand{\rotatecw}[5]{
\begin{tikzpicture}[baseline=-.1cm]
	\clip (-1.1,-.9)--(-1.1,.9)--(1.1,.9)--(1.1,-.9);	
	\draw (0,.4) arc (0:180:.4cm)--(-.8,-.8);
	\draw (0,-.4) arc (-180:0:.4cm)--(.8,.8);
	\draw[thick, unshaded] (-.4, -.4) -- (-.4, .4) -- (.4, .4) -- (.4,-.4) -- (-.4, -.4);
	\node at (0,0) {$#5$};
	\node at (-1,-.6) {{\scriptsize{$#1$}}};
	\node at (-.2,-.6) {{\scriptsize{$#2$}}};
	\node at (.2,.6) {{\scriptsize{$#3$}}};
	\node at (1,.6) {{\scriptsize{$#4$}}};
\end{tikzpicture}
}

\newcommand{\trace}[4]{
\begin{tikzpicture}[baseline = -.1cm]
	\draw (0,.4) arc (180:0:.4cm)--(.8,-.4) arc (0:-180:.4cm);
	\filldraw[unshaded,thick] (-.4,.4)--(.4,.4)--(.4,-.4)--(-.4,-.4)--(-.4,.4);
	\node at (1,0) {{\scriptsize{$#1$}}};
	\node at (-.2,.6) {{\scriptsize{$#2$}}};
	\node at (0,0) {$#3$};
	\node at (-.2,-.6) {{\scriptsize{$#4$}}};
\end{tikzpicture}
}

\newcommand{\traceop}[4]{
\begin{tikzpicture}[baseline = -.1cm]
	\draw (0,.4) arc (0:180:.4cm)--(-.8,-.4) arc (-180:0:.4cm);
	\filldraw[unshaded,thick] (-.4,.4)--(.4,.4)--(.4,-.4)--(-.4,-.4)--(-.4,.4);
	\node at (-1,0) {{\scriptsize{$#1$}}};
	\node at (.2,.6) {{\scriptsize{$#2$}}};
	\node at (0,0) {$#3$};
	\node at (.2,-.6) {{\scriptsize{$#4$}}};
\end{tikzpicture}
}

\newcommand{\PAMultiply}[5]{
\begin{tikzpicture}[baseline=.5cm]
	\draw (0,2)--(0,-.8);
	\filldraw[unshaded,thick] (-.4,.4)--(.4,.4)--(.4,-.4)--(-.4,-.4)--(-.4,.4);
	\draw[thick, unshaded] (-.4, .8) -- (-.4, 1.6) -- (.4, 1.6) -- (.4,.8) -- (-.4, .8);
	\node at (.2,1.8) {{\scriptsize{$#1$}}};
	\node at (0,1.2) {$#2$};
	\node at (.2,.6) {{\scriptsize{$#3$}}};
	\node at (0,0) {$#4$};
	\node at (.2,-.6) {{\scriptsize{$#5$}}};
\end{tikzpicture}
}

\newcommand{\ColorMultiply}[3]{
\begin{tikzpicture}[baseline=.5cm]
	\draw (0,2)--(0,1);
	\draw [thick, #3] (0,1)--(0,-.8);
	\filldraw[unshaded,thick] (-.4,.4)--(.4,.4)--(.4,-.4)--(-.4,-.4)--(-.4,.4);
	\draw[thick, unshaded] (-.4, .8) -- (-.4, 1.6) -- (.4, 1.6) -- (.4,.8) -- (-.4, .8);
	\node at (0,1.2) {$#1$};
	\node at (0,0) {$#2$};
\end{tikzpicture}
}

\newcommand{\PAMultiplyThree}[7]{
\begin{tikzpicture}[baseline=-.1cm]
	\draw (0,1.8)--(0,-2);
	\filldraw[unshaded,thick] (-.4,.4)--(.4,.4)--(.4,-.4)--(-.4,-.4)--(-.4,.4);
	\draw[thick, unshaded] (-.4, .8) -- (-.4, 1.6) -- (.4, 1.6) -- (.4,.8) -- (-.4, .8);
	\draw[thick, unshaded] (-.4, -.8) -- (-.4, -1.6) -- (.4, -1.6) -- (.4,-.8) -- (-.4, -.8);
	\node at (.2,1.8) {{\scriptsize{$#1$}}};
	\node at (0,1.2) {$#2$};
	\node at (.2,.6) {{\scriptsize{$#3$}}};
	\node at (0,0) {$#4$};	
	\node at (.2,-.6) {{\scriptsize{$#5$}}};
	\node at (0,-1.2) {$#6$};
	\node at (.2,-1.8) {{\scriptsize{$#7$}}};
\end{tikzpicture}
}

\newcommand{\ColorMultiplyThree}[5]{
\begin{tikzpicture}[baseline=-.1cm]
	\draw (0,1.8)--(0,1);
	\draw[thick, #4] (0,1)--(0,-1);
	\draw[thick, #5] (0,-1)--(0,-2);
	\filldraw[unshaded,thick] (-.4,.4)--(.4,.4)--(.4,-.4)--(-.4,-.4)--(-.4,.4);
	\draw[thick, unshaded] (-.4, .8) -- (-.4, 1.6) -- (.4, 1.6) -- (.4,.8) -- (-.4, .8);
	\draw[thick, unshaded] (-.4, -.8) -- (-.4, -1.6) -- (.4, -1.6) -- (.4,-.8) -- (-.4, -.8);
	\node at (0,1.2) {$#1$};
	\node at (0,0) {$#2$};	
	\node at (0,-1.2) {$#3$};
\end{tikzpicture}
}

\newcommand{\MInnerProduct}[3]{
\begin{tikzpicture}[baseline=-.1cm]
	\draw [thick, #2] (0,.6)--(0, -.4);
	\draw [thick, #2] (0, -.4) .. controls ++(270:.4cm) and ++(270:.4cm) .. (1.2,-.4);
	\draw[thick, #2] (1.2,-.4) -- (1.2,.6);
	\draw[thick, unshaded] (1.6, -.4) -- (1.6, .4) -- (.8, .4) -- (.8, -.4) -- (1.6, -.4);
	\draw[thick, unshaded] (-.4, -.4) -- (-.4, .4) -- (.4, .4) -- (.4, -.4) -- (-.4, -.4);
	\node at (1.2, 0) {$#3$};
	\node at (0, 0) {$#1$};
\end{tikzpicture}
}

\newcommand{\MInnerProductOp}[3]{
\begin{tikzpicture}[baseline=-.1cm]
	\draw[thick, #2] (0,-.6)--(0, .4);
	\draw[thick, #2] (1.2,.4) -- (1.2,-.6);
	\draw[thick, #2] (0, .4) .. controls ++(90:.4cm) and ++(90:.4cm) .. (1.2,.4);
	\draw[thick, unshaded] (1.6, -.4) -- (1.6, .4) -- (.8, .4) -- (.8, -.4) -- (1.6, -.4);
	\draw[thick, unshaded] (-.4, -.4) -- (-.4, .4) -- (.4, .4) -- (.4, -.4) -- (-.4, -.4);
	\node at (1.2, 0) {$#3$};
	\node at (0, 0) {$#1$};
\end{tikzpicture}
}

%% file: abstract2013.tex
\begin{abstract}

This dissertation consists of three papers I have written or helped write during my time at UC Berkeley.  The three papers all fall under the common theme of exploring connections between free probability and planar algebras.

In Chapter \ref{chap:graph}, an amalgamated free product algebra, $\cN(\Gamma)$, is constructed for any connected, weighted, loopless graph $\Gamma$, and its isomorphism class can be nicely  read off based on the weighting data on the graph.  This construction was heavily influenced by an algebra appearing in \cite{MR2807103}, and it is used, along with some standard embedding arguments, to show that the factors of Guionnet, Jones, and Shlyakhtenko appearing in \cite{MR2732052} are isomorphic to $L(\F_{\I})$ when the planar algebra is infinite depth.

Chapter \ref{chap:universal} is the paper ``Rigid $C^{*}-$tensor categories of bimodules over interpolated free group factors" \cite{MR123525} which was co-authored with Arnaud Brothier and David Penneys.  In this paper, we establish an unshaded planar algebra structure (with multiple colors of strings) which can be used to model a countably generated rigid $C^{*}-$tensor category, $\cC$.  We use this to construct a category of bifinite bimodules over a $II_{1}$ factor $M_{0}$, and we show that this category is equivalent to $\cC$.  Finally, we use the work in Chapter \ref{chap:graph} to show that the factor $M_{0}$ is isomorphic to $L(\F_{\I})$.

Chapter \ref{chap:FC} is a note regarding multishaded planar algebras.  The problem studied consists of placing the ``Fuss-Catalan" potential on a specific kind of subfactor planar algebra $\cQ$, which is the standard invariant for and inclusion $N \subset M$ with an intermediate subfactor $P$.  This potential is best understood by augmenting $\cQ$, forming a bigger multishaded planar algebra $\cP$.  The isomorphism classes of the algebras $M_{\alpha}$ associated to $\cP$ will be computed explicitly.  While the isomorphism class of the von Neumann algebras $N_{k}^{\pm}$ associated to $\cQ$ are still not yet known, they will be shown to be contained in a free group factors and contain a free group factors. This potential is shown to yield a nice free product expression for the law of $\cup$, an element which plays a critical role in understanding algebras that arise from this construction.  Many of the ideas in this chapter influenced the work in the Chapter \ref{chap:universal}.

\end{abstract}

%% file: Thesis2013Introduction.tex
\chapter[Introduction]{Introduction} \label{chap:intro}

\section{Subfactors and planar algebras}

In \cite{MR696688}, Jones initiated the study of subfactor theory.  The objects of interest are inclusions $M_{0} \subset M_{1}$ of finite index $II_{1}$ factors $M_{0}$ and $M_{1}$.   From this, he iterated the basic construction to produce a tower of finite index $II_{1}$ factors
$$
M_{0} \subset M_{1} \subset \cdots \subset M_{n} \subset \cdots,
$$
and studied subfactors by examining their standard invariant. The standard invariant consists of two sequences of finite dimensional vector spaces, $P_{n, +} = M_{0}' \cap M_{n}$ and $P_{n, -} = M_{1}' \cap M_{n}$ \cite{MR696688}.  By 1995, the standard invariant became axiomatized by Ocneanu's \underline{paragroups} \cite{MR996454}, and Popa's \underline{$\lambda-$lattices} \cite{MR1334479}.   Around the same time, Jones discovered that operations and objects in subfactor theory (such as the Jones projections, traces, and conditional expectations) possess a very natural planar structure.  He used these observations to invent subfactor \underline{planar algebras}. Before talking about subfactor planar algebras, we will give a definition of a shaded planar tangle:
\begin{defn}\label{defn:tangle}

A \underline{shaded planar tangle}, $T$ consists of the following data:
\begin{enumerate}[(1)]
\item
a rectangle $D_0(T)\subset \R^2$,
\item
Finitely many disjoint rectangles $\{D_{i}: i \in I\}$ in the interior of $D_0(T)$ ($I$ may be empty),
\item
Finitely many disjoint smooth arcs in $D_0(T)$ called the strings of $T$ which do not meet the interior of any $D_{i}$.
The boundary points of a string of $T$ (if it has any) lie in the boundaries of the $D_i$, and they meet these boundaries transversally (if they meet the boundaries at all).
\item A choice of shading of the regions of $T$.  Each region can be either shaded or unshaded, but an unshaded region can only border a shaded region or vice versa.  Notice that this means that the number of strings intersecting any box must be even.
\end{enumerate}
The boundaries of the strings divide the boundaries of the rectangles into \underline{intervals}. For each rectangle, there is a distinguished interval denoted $\star$. The intervals of $D_i(T)$ are divided by the \underline{marked points} of $D_i(T)$, i.e., the points at which the strings meet the boundary of $D_i(T)$. Starting at $\star$, the marked points on each rectangle are number clockwise.

If we have two planar tangles $S,T$ satisfying the following \underline{boundary conditions}:
\begin{itemize}
\item
some internal rectangle $D_i(S)$ agrees with $D_0(T)$,
\item
the marked points of $D_i(S)$ agree with the marked points of $D_0(T)$,
\item
the distinguished interval of $D_i(S)$ agrees with the distinguished interval of $T$, and
\item
the label from $\cL$ of each marked point of $D_i(S)$ agrees with the label of each marked point from $D_0(T)$,
\end{itemize}
then we may compose $S$ and $T$ to get the planar tangle $S\circ_{D_{i}} T$ by taking $S$ union the interior of $D_0(T)$, removing the boundary of $D_i(S)$, and smoothing the strings.
\end{defn}

An example of such a tangle is:
$$
\begin{tikzpicture}[baseline = .5cm]
	\draw[thick] (-2,-.8)--(-2,2)--(.8,2)--(.8,-.8)--(-2,-.8);
	\draw (.8,1.4) arc (-90:-180:.6cm);
	\draw (-.2,.4) arc (0:90:.6cm);
    \draw (.2, .4) arc (0:90:1cm);
	\draw (-.4,-.3) arc (270:90:.3cm);
	\draw (-1,2)--(-1,1)--(-2,1);
	\draw (-1.4,2)--(-1.4,1.4)--(-2,1.4);
	\draw (0,-.8)--(0,0)--(.8,0);
	\draw (-1.4,0) circle (.3cm);
	\filldraw[unshaded,thick] (-1.6,.8)--(-1.6,1.6)--(-.8,1.6)--(-.8,.8)--(-1.6,.8);
	\filldraw[unshaded,thick] (-.4,.4)--(.4,.4)--(.4,-.4)--(-.4,-.4)--(-.4,.4);
	\node at (-.5,-.5) {$\star$};
	\node at (-1.6,.6) {$\star$};
	\node at (-2.2,-.6) {$\star$};
    \node at (-1.8, 1.75) {\scriptsize{$+$}};
    \node at (-1.2, 1.75) {\scriptsize{$-$}};
    \node at (0, 1.5) {\scriptsize{$+$}};
    \node at (.6, 1.75) {\scriptsize{$-$}};
    \node at ( -.1, .8) {\scriptsize{$-$}};
    \node at ( .6, -.6) {\scriptsize{$-$}};
    \node at ( -.9, .3) {\scriptsize{$+$}};
    \node at ( -1.4, 0) {\scriptsize{$-$}};
    \node at ( -.55, 0) {\scriptsize{$-$}};
    \node at ( -1.8, 1.2) {\scriptsize{$-$}};
\end{tikzpicture}
$$
where a $+$ indicates an unshaded region and a $-$ indicates a shaded region.  Loosely speaking, a shaded planar algebra consists of a family of vector spaces having which exhibit a natural action by planar diagrams.  In most definitions of planar algebras (such as \cite{math.QA/9909027}), the rectangles are taken to be smoothly embedded subdisks in the plane, and the action of planar tangles is invariant under an orientation preserving diffeomorphism.  Therefore, one should picture the rectangles as having their corners smoothed out.  In future drawings, the $\star$ on a box will be placed at the bottom-left of a box unless otherwise indicated.

We now define a subfactor planar algebra.
\begin{defn} \label{defn:SPA}
\begin{itemize}
A shaded planar algebra $\cP$ comes equipped with the following data:
\item
a collection of vector spaces $(P_{2n, +})_{n \in \N}$ and $(P_{2n, -})_{n \in \N}$
\item
an action of planar tangles by multilinear maps, i.e., for each planar $(2n, \epsilon)$ tangle $T$, whose rectangles $D_i(T)$ are $(2n_{i}, \epsilon_{i})$-rectangles, there is a multilinear map
$$
Z_T\colon \prod_{i\in I} P_{2n_{i}, \epsilon_{i}}\to P_{2n, \epsilon}
$$
satisfying the following axioms:
\begin{enumerate}
\item[\underline{\text{Isotopy}:}]
If $\theta$ is an orientation preserving diffeomorphism of $\R^2$, then $Z_{\theta(T)}=Z_T$.  That is, let $T^{0}$ be the interior of $T$, and let $f \in \prod_{D \subset T^{0}}P_{2n_{D}, \epsilon_{D}}$.  Then
$$
Z_{\theta(T)}(f_{\theta}) = Z_{T}(f)
$$
where $f_{\theta}(\theta(D)) = f(D)$.

\item[\underline{\text{Naturality}:}]
For $S,T$ composable tangles, $Z(S\circ_D T) = Z(S)\circ_D Z(T)$, where the composition on the right hand side is the composition of multilinear maps.
\end{enumerate}
\end{itemize}
Moreover, $\cP$ is called a \underline{subfactor planar algebra} if $\cP$ is
\begin{itemize}
\item
\underline{evaluable}, i.e., $\dim(P_{2n, \pm})<\infty$ for all $n$ and $P_{0, \pm} \cong \C$ via the map that sends the empty diagram to $1\in\C$. Hence, by naturality, there is a scalar $\delta \in \C$ such that any labeled diagram containing a closed loop is equal to the same diagram without the closed loop multiplied by $\delta$.

\item $\cP$ is \underline{unital} \cite{JonesPAnotes}:  Let $S$ be a shaded planar $2n, \epsilon$ tangle with no input disks.  Then, there is an element $Z(S) \in P_{2n,\epsilon}$ so that the following holds:

Let $T$ be a tangle with a nonempty set of internal disks such that $S$ can be glued into the internal disk $D^{S}$ of $T$.  Then
    $$
    Z(T \circ S) = Z(T) \circ Z_{S}.
    $$
    Here $(Z(T) \circ Z_{S})(f) = \tilde{f}$ where
    $$
    \tilde{f}(D) = \begin{cases} f(D) \text{ if } D \neq D^{S} \\ Z(S) \text{ if } D = D^{S} \end{cases}
    $$

\item \underline{involutive}, i.e.  there is a map $*\colon P_{2n, \pm}\to P_{2n, \pm}$ with $*\circ *=\id$ which is compatible with the reflection of tangles, i.e., if $T$ is a planar tangle acting on $x_{1}, ..., x_{n}$, and $\varphi$ is an orientation \underline{reversing} diffeomorphism of $\R^{2}$, then
$$
Z(T)(f)^*=Z(\varphi(T))(\overline{f})
$$
where $\overline{f}(\varphi(D)) = f(D)^{*}$
\item
\underline{spherical}, i.e., for all $\alpha\in\Lambda$ and all $x\in P_{2n, \pm}$, we have
$$
\tr(x)
=
\trace{n}{n}{x}{n}
=
\traceop{n}{n}{x}{n}\,.
$$
\item
\underline{positive}, i.e., the map $\langle\cdot,\cdot\rangle \colon P_{2n, \pm} \times P_{2n, \pm} \to P_{0, \pm}\cong\C$ given by
$$
\langle x,y\rangle =
\begin{tikzpicture}[baseline = -.1cm]
	\draw (0,0)--(1.2,0);	
	\filldraw[unshaded,thick] (-.4,.4)--(.4,.4)--(.4,-.4)--(-.4,-.4)--(-.4,.4);
	\filldraw[unshaded,thick] (.8,.4)--(1.6,.4)--(1.6,-.4)--(.8,-.4)--(.8,.4);
	\node at (0,0) {$x$};
	\node at (1.2,0) {$y^*$};
	\node at (.6,.2) {{\scriptsize{$2n$}}};
\end{tikzpicture}
$$
is a positive definite inner product. Hence $\delta >0$.  It further follows from the work of Jones \cite{MR696688} that
$$
\delta \in \{ 2\cos(\pi/n) : n = 3, 4, 5, ...\} \cup [2, \I)
$$
\end{itemize}
\end{defn}
Jones showed that the relative commutants $P_{2n, +} = M_{0}'\cap M_{n}$ and $P_{2n, -} = M_{1}'\cap M_{n+1}$ admit the structure of a subfactor planar algebra.  In Chapters \ref{chap:universal} and \ref{chap:FC}, we will define notions of planar algebras similar to the one presented above.
\begin{rem}
The spaces $P_{2n, \pm}$ as above are traditionally referred to as $P_{n, \pm}$ in the literature.  We adopt this new numbering convention as we will be considering unshaded planar algebras in Chapter \ref{chap:universal} which can admit actions by planar tangles with \emph{odd} numbers of strings.
\end{rem}

The positive definiteness assures that each $P_{2n, \pm}$ is a finite dimensional $C^{*}$ algebra where the multiplication is given by
$$
x \cdot y = \PAMultiply{n}{y}{n}{x}{n}
$$
and non-normalized trace $\Tr$ is given by
$$
\Tr(x) = \trace{n}{n}{x}{n}.
$$
Of critical importance is the study of minimal projections in these algebras via a \underline{principal graph}, $\Gamma$ and \underline{dual principal graph} $\Gamma'$.  To begin, we consider the algebras $P_{2n, +}$.  Let $p \in P_{2j, +}$ and $q \in P_{2k, +}$.  We say that $p$ and $q$ are equivalent if there is a $u \in P_{j+k, +}$ satisfying
$$
\PAMultiply{j}{u}{k}{u^{*}}{j} = p \, \, \, \text{ and } \, \, \, \PAMultiply{k}{u^{*}}{j}{u}{k} = q.
$$
Notice that for equivalence of $p$ and $q$ to make sense, we must have $k + j$ even.  We then form $\Gamma$ as follows:
\begin{itemize}
\item The vertices $v$ of $\Gamma$ are in a one to one correspondence with equivalence classes $[p_{v}]$ of minimal projections in the algebras $P_{2n, +}$.

\item There are $n$ edges connecting $v$ and $w$ if the following condition is met:  Suppose $p \in P_{2j, +}$ is a minimal projection equivalent to $p_{v}$.  Then the element
$$
i(p) = \begin{tikzpicture}[baseline = 0cm]
    \draw(0, .8)--(0, -.8);
    \draw(.6, .8)--(.6, -.8);
    \node at (.2, .6) {$\scriptsize{j}$};
    \node at (.2, -.6) {$\scriptsize{j}$};
    \filldraw[thick, unshaded](.4, .4)--(.4, -.4)--(-.4, -.4)--(-.4, .4)--(.4, .4);
    \node at (0, 0) {$p$};
\end{tikzpicture}
$$
is a projection $P_{2(j+1), +}$.  The number $n$ will be the maximum $k$ such that there exist orthogonal projections $q_{1}, ..., q_{k} \in P_{2(j+1), +}$ which are equivalent to $p_{w}$ and satisfy $\sum_{i=1}^{k}q_{i} \leq i(p)$.  This number is independent of the representative of $p_{v}$.  Also, if there are $n$ edges connecting $v$ to $w$, then there are also $n$ edges connecting $w$ to $v$ \cite{math.QA/9909027}.  See Chapter \ref{chap:FC} for an explanation of this fact for the multishaded case.
\end{itemize}
It follows from this construction that $\Gamma$ is a connected, unoriented, bipartite graph.  If we weight each vertex $v$ by $\Tr(p_{v})$, then it follows that the weights satisfy the Perron-Frobenius condition, i.e.
$$
\delta \Tr(p_{v}) = \sum_{w \sim v} n_{v, w}\Tr(p_{w}).
$$
We write $w \sim v$ if $w$ is connected to $v$ and we let $n_{v, w}$ be the number of edges connecting $v$ and $w$.  We will mark the vertex corresponding to the empty diagram with a *.

By looking at the algebras $P_{2n, -}$, one can similarly construct the dual principal graph of $\cP$.  It is straightforward to check that a $180^{\circ}$ rotation is an anti $*-$isomorphism between $P_{2n, +}$ and $P_{2n, -}$ for $n$ odd.  It follows that the odd vertices of $\Gamma$ and $\Gamma'$ are in bijective correspondence.

 We say that $\cP$ is \underline{finite depth} is $\Gamma$ is finite and \underline{infinite depth} if $\Gamma$ is infinite.  Note that if $\Gamma$ is finite, then there are only finitely many odd vertices of $\Gamma'$.  Since the algebras $P_{2n, -}$ are finite dimensional for each $n$, it follows from the construction of the edges that $\Gamma'$ must be finite as well.
\section{Free probability}

Around the same time as the beginning of subfactor theory, Voiculescu developed an analogue of probability called \underline{free probability}(see the book \cite{MR1217253}).  At the heart of free probability is the concept of free independence, which we will now define
\begin{defn}
Let $M$ be a von Neumann algebra with faithful normal state $\phi$, and $(M_{\iota})_{\iota\in I}$ a family of unital von Neumann subalgebras of $M$.  Then we say that the algebras $M_{\iota}$ are freely independent (or \underline{free}) if the following condition holds:
$$
\phi(x_{\iota_{1}}\cdots x_{\iota_{n}}) = 0 \text{ whenever } \iota_{1} \neq \iota_{2} \neq \cdots \neq \iota_{n} \text{ and } \phi(x_{\iota_{j}}) = 0 \text{ for all } j.
$$
We say that elements $(x_{\iota})_{\iota \in I}$ are free if the corresponding von Neumann algebras they generate are free.
\end{defn}
In the rest of the thesis, $\phi$ will always be assumed to be tracial.  Given any collection $(M_{\iota})_{\iota \in I}$ of von Neumann algebras with faithful states $\phi_{\iota}$, one can form a von Neumann algebra $M$ with faithful state $\phi$ which is generated by isomorphic images of the $M_{\iota}$ such that $\phi|_{M_{\iota}} = \phi_{\iota}$ and the family $(M_{\iota})_{\iota \in I}$ is free.  $M$ is unique up to spatial isomorphism, and the state $\phi$ is uniquely determined by the states $\phi_{\iota}$.  See \cite{MR1217253} for more details.

Voiculescu noticed strong parallels between commutative probability and free probability theory.  By realizing that the semicircular law should play the same role in free probability as the gaussian law plays in commutative probability, he established a central limit theorem for additive free convolution which mimics the central limit theorem in probability \cite{MR1217253}.  Building on the proof of Wigner's semicircular law \cite{MR0077805}, Voiculescu proved that a family of independent self adjoint Gaussian random matrices converges in joint distribution to a free family of semicircular elements \cite{MR1094052}.  Therefore, one can think of free probability as a random matrix limit of commutative probability.

\section{Connection between subfactors and free probability}

In the 1990's, evidence started to surface that the combinatorial structure of free probability is closely related to non-crossing diagrams.  One of the key observations is that if $x$ is a self adjoint random variable having the semicircular law of radius 2 with respect to some state $\phi$, then $\phi(x^{2n}) = C_{n}$, the $n^{th}$ Catalan number.  The Catalan numbers count many things including the numbers of noncrossing pair partitions of an ordered set of $2n$ elements.

It was shown by Speicher \cite{MR1268597} that additive free convolution of noncommutative random variables is easily described in terms of free cumulants which are naturally defined according to noncrossing partitions.  A combinatorial proof of the free central limit theorem boils down to counting noncrossing pair partitions in the large $N$ limit (as opposed to all pair partitions for the commutative central theorem) \cite{MR1030725}.  An explicit use of planar diagrams in free probability and random matrix theory with regards to second-order freeness was considered by Mingo, Speicher, and others \cite{MR2216446,MR2294222, MR2302524}.   Furthermore, Voiculescu's theorem on the asymptotic freeness of independent gaussian random matrices can be recast as a combinatorial proof involving counting certain noncrossing partitions \cite{MR2266879}.

These observations helped lead Guionnet, Jones, and Shlyakhtenko (GJS) to find connections between subfactors, planar algebras, free probability, and random matrix theory \cite{MR2732052}.  In \cite{MR2732052}, they constructed graded algebras $\Gr_{n}(\cP^{\pm})$ associated to a subfactor planar algebra $\cP$ and endowed them with a trace $\tr$ emanating from all Temperley-Lieb (planar) diagrams.  They were able to model $\tr$ by a random matrices associated to the principal graph, $\Gamma$ of $\cP$ and were able to construct a Jones tower of factors $M_{n}^{+} = \Gr(\cP_{2n, +})''$ whose standard invariant is $\cP$, thus recovering Popa's reconstruction theorem \cite{MR1334479}.  It was also shown that in the case where $\cP$ is finite depth (has a finite principal graph), the factors $M_{n}^{+}$ are interpolated free group factors, whose parameter depends on data in the planar algebra.

In Chapter \ref{chap:graph}, it will be shown that for infinite depth, the factors $M_{n}^{+}$ are all isomorphic to $L(\F_{\I})$.  In Chapter \ref{chap:universal}, we examine the GJS construction with regards to unshaded planar algebras, and obtain a universality result regarding countably generated rigid $C^{*}-$tensor categories and bimodules over $L(\F_{\I})$.  In the final chapter, we will examine the situation when the trace involving all Temperley-Lieb diagrams is replaced by the trace involving all Fuss-Catalan diagrams.

\subsection{Planarity and freeness}

The algebras $\Gr_{0}(\cP)$ in the following chapters are presented as variations of the following idea:  We set 
$$
\Gr_{0}(\cP) = \bigoplus_{n=0}^{\infty} P_{n}
$$
with multiplication given by
$$
x \cdot y =
\Mbox{}{x}
\,
\Mbox{}{y}
$$
and trace, $\tr$ given by 
$$
\tr(x) = \begin{tikzpicture}[baseline=.5cm]
	\draw (0,0)--(0,.8);
	\filldraw[unshaded,thick] (-.4,.4)--(.4,.4)--(.4,-.4)--(-.4,-.4)--(-.4,.4);
	\draw[thick, unshaded] (-.7, .8) -- (-.7, 1.6) -- (.7, 1.6) -- (.7,.8) -- (-.7, .8);
	\node at (0,0) {$x$};
	\node at (0,1.2) {$\Sigma TL$};
\end{tikzpicture}
$$
where $\sum TL$ is the sum of all planar string diagrams with no strings emanating out of the top of the box or the sides.  Assuming for the moment that $\tr$ is positive definite, we have the following lemma:
\begin{lem}\label{lem:freeplanar} 
Suppose $x_{1}, ..., x_{n}$ are elements in $\cP$ with the following property:  For every tangle, $T$,  with at least one string connecting two distinct internal disks $D_{1}$ and $D_{2}$, we have $Z(T)(f) = 0$ whenever $f(D_{1}) = x_{i_{1}}^{\epsilon_{1}}$, $f(D_{2}) = x_{i_{2}}^{\epsilon_{2}}$ with $i_{1} \neq i_{2}$ and $\epsilon_{j}$ is either nothing or a *.  Then $\{x_{1}, ..., x_{n}\}$ are free in $M_{0} = \Gr_{0}(\cP)''$.
\end{lem}
Loosely speaking, this lemma says that whenever an $x_{i}$ is connected by at least one string to $x_{j}$ for $i \neq j$, then $x_{i}$ and $x_{j}$ are free.

\begin{proof}
Suppose $p_{j}$ is a (non-commutative) polynomial in $x_{i_{j}}$ and $x_{i_{j}}^{*}$ with $\tr(p_{j}) = 0$, and consider the element $p_{1}\cdots p_{m}$ for some $m$.  Suppose further that $i_{k} \neq i_{k+1}$ for$1 \leq k \leq m$.  We need to show that $\tr(p_{1}\cdots p_{m}) = 0$.  By definition, we have:
$$
\tr(x) = \begin{tikzpicture}[baseline=.5cm]
	\draw (0,0)--(0,.8);
    \draw (1.4, 0)--(1.4, .8);
	\filldraw[unshaded,thick] (-.4,.4)--(.4,.4)--(.4,-.4)--(-.4,-.4)--(-.4,.4);
	\draw[thick, unshaded] (-.7, .8) -- (-.7, 1.6) -- (2.1, 1.6) -- (2.1,.8) -- (-.7, .8);
	\node at (0,0) {$p_{1}$};
	\node at (.7,1.2) {$\Sigma TL$};
    \node at (.7, 0) {$\cdots$};
    \filldraw[unshaded, thick]  (1,.4)--(1.8,.4)--(1.8,-.4)--(1,-.4)--(1,.4);
    \node at (1.4, 0) {$p_{m}$};
\end{tikzpicture}\, . 
$$
By assumption on the variables $x_{i}$, any element of $TL$ which connects a summand of $p_{i}$ to a summand of $p_{i+1}$ gives a zero contribution.  Therefore, by planarity, every contributing $TL$ element must pair some summand of a $p_{j}$ with itself.  Summing over all of these  partitions, and using $\tr(p_{j}) = 0$ for all $j$ gives the result.
\end{proof} 

The idea in the sections that follow is to find elements $x_{i}$ satisfying this property which generate $M_{0}$ in an appropriate sense.  In practice, we will pass to an infinite amplification of $M_{0}$ as it will be more natural to find such elements there. 

%% file: Thesis2013Chapter-Graph.tex
\chapter[GJS in infinite depth]{Free Product von Neumann Algebras Associated to Graphs and the Guionnet, Jones, Shlyakhtenko Construction in Infinite Depth} \label{chap:graph}

\section{Introduction}

In \cite{MR696688}, Jones initiated the study of modern subfactor theory.  Given a finite index $II_{1}$ subfactor $M_{0} \subset M_{1}$, one computes its standard invariant: two towers $(M_{0}' \cap M_{j}: j \geq 0)$  and $(M_{1}' \cap M_{j}: j \geq 1)$ of finite dimensional von Neumann algebras \cite{MR696688}. The standard invariant has been axiomatized by Ocneanu's paragroups \cite{MR996454}, Popa's $\lambda-$lattices \cite{MR1334479}, and Jones' subfactor planar algebras \cite{math/1007.1158}. Popa showed that given a standard invariant $\cP$, we can reconstruct a $II_{1}$ subfactor $M_{0} \subset M_{1}$ whose standard invariant is $\cP$ \cite{MR1334479}. Guionnet, Jones, and Shlyakhtenko \cite{MR2732052} give a planar-algebraic proof of the above result.  Moreover, if $\cP$ is finite depth with loop paramater $\delta > 1$, they showed that $M_{k}$, the $k^{th}$ factor in the Jones tower, is isomorphic to $L(\F(1 + 2\delta^{-2k}(\delta - 1)I))$ where $I$ is the global index of $\cP$ \cite{MR2807103}.   Kodiyalam and Sunder also obtained this formula when $\cP$ is depth 2 \cite{MR2574313, MR2557729}.  In this paper, we prove the following theorem:

\begin{thm*}
If $\cP$ is infinite depth, then every factor in the construction of \cite{MR2732052} is isomorphic to $L(\F_{\I})$.
\end{thm*}

 Using this theorem, we recover a diagrammatic proof of a result of Popa and Shlyakhtenko for $\cP$ infinite depth \cite{MR2051399}:

 \begin{cor*}
 Every infinite depth subfactor planar algebra is the standard invariant of $\cN \subset \cM$ where $\cN, \cM \cong L(\F_{\I})$.
 \end{cor*}

\subsection{Outline of the proof}

To prove the above theorem, we will bootstrap the proof from \cite{MR2807103} of the finite-depth case to the infinite-depth case using standard embedding tricks.

\paragraph {The GJS Construction:}

      We will recall the construction of Guionnet, Jones and Shlyakhtenko. For more details, see \cite{MR2732052} and \cite{MR2807103}, and Chapters \ref{chap:universal} and \ref{chap:FC}.  The placement of the $\star$ below is how the construction first appeared in \cite{MR2732052} and \cite{MR2807103}.  The author has since found it more convenient to place the $\star$ at the bottom of the box.  Let $\cP = (P_{n, \pm})_{n \geq 0}$ be a subfactor planar algebra with loop parameter $\delta > 1$.  Set $\Gr_{k}(\cP^{+})= \bigoplus_{n \geq 0} P_{k,n}^{+}$ where $P^{+}_{k,n} = \cP_{2k+2n, +}$ and an element of $P^{+}_{k,n}$ is represented as
       $$
    \begin{tikzpicture}
       \draw (-.5, -.5) -- (-.5, .5) -- (.5, .5) -- (.5, -.5) -- (-.5, -.5);
       \draw (-1, -1) -- (-1, 1) -- (1, 1) -- (1, -1) -- (-1, -1);
       \draw[ultra thick] (0, 1) -- (0, .5);
       \draw[ultra thick] (1, 0) -- (.5, 0);
       \draw[ultra thick] (-1, 0) -- (-.5, 0);
       \node at (0, 0) {$x$};
       \node at (-.75, .25) {$k$};
       \node at (.75, .25) {$k$};
       \node at (.30, .75) {$2n$};
       \node at (-.5, .6) {$\star$};
       \node at (-1, 1.1) {$\star$};
       \end{tikzpicture}
       $$
       where the $\star$ is always in an unshaded region and a thick string with a $j$ next to it denotes $j$ parallel strings.  If $x \in P_{k,n}^{+}$ and $y \in P_{k,m}^{+}$ then define a multiplication $\wedge_{k}$ by
       $$
        x \wedge_{k} y =
       \begin{tikzpicture} [baseline = 0 cm]
       \draw (-1.5, -.5) -- (-1.5, .5) -- (-.5, .5) -- (-.5, -.5) -- (-1.5, -.5);
       \draw[ultra thick] (-1, 1) -- (-1, .5);
       \draw[ultra thick] (0, 0) -- (-.5, 0);
       \draw[ultra thick] (-2, 0) -- (-1.5, 0);
       \node at (-1, 0) {$x$};
       \node at (-1.75, .25) {$k$};
       \node at (-.7, .75) {$2n$};
       \node at (-1.5, .6) {$\star$};
       \draw (.5, -.5) -- (.5, .5) -- (1.5, .5) -- (1.5, -.5) -- (.5, -.5);
       \draw[ultra thick] (1, 1) -- (1, .5);
       \draw[ultra thick] (2, 0) -- (1.5, 0);
       \draw[ultra thick] (0, 0) -- (.5, 0);
       \node at (1, 0) {$y$};
       \node at (1.75, .25) {$k$};
       \node at (1.30, .75) {$2m$};
       \node at (.5, .6) {$\star$};
       \draw (-2, -1) -- (2,-1)--(2,1)--(-2, 1)--(-2, -1);
       \node at (-2, 1.1) {$\star$};
       \end{tikzpicture}
       $$
        which is an element in $P_{k, m + n}^{+}$.  One can endow $Gr_{k}(P^{+})$ with the following trace:  if $x \in P^{+}_{k, n}$ then
        $$
        \tr(x) = \delta^{-k} \cdot \begin{tikzpicture}[baseline=.5cm]
	\draw (.5,.5) --(-.5,.5) -- (-.5,-.5) -- (.5,-.5) -- (.5, .5) ;
    \draw (.7,2)--(-.7,2)--(-.7,1)--(.7, 1)--(.7,2);
\draw[ultra thick] (-.5, 0) arc (90:270: .5cm) -- (.5, -1) arc(-90:90: .5cm);
\draw[ultra thick] (0, .5) -- (0, 1);
\node at (0, 0) {$x$};
\node at (-.5, .6) {$\star$};
\node at (-.7, 2.1) {$\star$};
\node at (0, 1.5) {$\sum TL$};
\node at (0, -.75) {$k$}; \end{tikzpicture}
$$
 where $\sum TL$ denotes the sum of all Temperely-Lieb diagrams, i.e. all planar pairings of the $2n$ strings on top of $x$.  This trace is positive definite, and one can form the von Neumann algebra $M_{k}$ which is the strong closure of $Gr_{k}(P^{+})$ acting on $L^{2}(Gr_{k}(P^{+}))$ by left multiplication (under $\wedge_{k}$).  It is shown that $M_{k}$ is a $II_{1}$ factor.  Moreover one can view $x \in M_{k}$ as an element in $M_{k+1}$ as follows:
 $$
    \begin{tikzpicture} [baseline = 0 cm]
       \draw (-.5, -.5) -- (-.5, .5) -- (.5, .5) -- (.5, -.5) -- (-.5, -.5);
       \draw[ultra thick] (0, 1) -- (0, .5);
       \draw[ultra thick] (1, 0) -- (.5, 0);
       \draw[ultra thick] (-1, 0) -- (-.5, 0);
       \node at (0, 0) {$x$};
       \node at (-.75, .25) {$k$};
       \node at (.75, .25) {$k$};
       \node at (-.5, .6) {$\star$};
       \node at (-1, 1.1) {$\star$};
       \draw[thick] (-1, -1) -- (-1, 1) -- (1, 1) -- (1, -1) -- (-1, -1);
       \draw (-1, -.65) -- (1, -.65);
    \end{tikzpicture} \, .
    $$
   With this identification, $M_{k}$ is a von Neumann subalgebra of $M_{k+1}$ and $M_{0} \subset M_{1} \subset \cdots \subset M_{k} \subset \cdots$ is a Jones tower of $II_{1}$ factors with standard invariant $\cP$.

   To identify the isomorphism type of the $M_{k}$, we look at the semi-finite algebra  $$V_{+} = \bigoplus_{k + l + m \textrm{ even }} P^{+}_{k,l,m}$$ where $P^{+}_{k,l,m} = P_{k+l+m, +}$ and is spanned by boxes of the form
   $$
    \begin{tikzpicture} [baseline = 0 cm]
       \draw (-.5, -.5) -- (-.5, .5) -- (.5, .5) -- (.5, -.5) -- (-.5, -.5);
       \draw[ultra thick] (0, 1) -- (0, .5);
       \draw[ultra thick] (1, 0) -- (.5, 0);
       \draw[ultra thick] (-1, 0) -- (-.5, 0);
       \node at (0, 0) {$x$};
       \node at (-.75, .25) {$k$};
       \node at (.75, .25) {$l$};
       \node at (.25, .75) {$m$};
       \node at (-.5, -.7) {$\star$};
       \end{tikzpicture} \, .
       $$
       The element $x$ above is identified with the following element of $P_{k+2p, l+2q, m}^{+}$:
       \begin{equation}
       \label{r}
  \delta^{-(p+q)/2}  \begin{tikzpicture}[baseline = -.5cm]
       \draw (-.5, -.5) -- (-.5, .5) -- (.5, .5) -- (.5, -.5) -- (-.5, -.5);
       \draw[ultra thick] (0, 1) -- (0, .5);
       \draw[ultra thick] (1, 0) -- (.5, 0);
       \draw[ultra thick] (-1, 0) -- (-.5, 0);
       \draw[thick] (-1,1)--(-1,-2)--(1, -2)--(1,1)--(-1,1);
       \draw (-1, -.5) arc(90:-90 : .2cm);
       \draw (-1, -1.5) arc(90:-90 : .2cm);
       \draw (1, -0.5) arc(90:270 : .2cm);
       \draw (1, -1.5) arc(90:270 : .2cm);
       \node at (0, 0) {$x$};
       \node at (-.75, .25) {$k$};
       \node at (.75, .25) {$l$};
       \node at (.25, .75) {$m$};
       \node at (-1, -2.2) {$\star$};
       \node at (-.85, -1.1) {\vdots};
       \node at (.85, -1.1) {\vdots};
       \node at (-.7, -1.1) {$p$};
       \node at (.7, -1.1) {$q$};
       \node at (-.5, -.7) {$\star$};
       \end{tikzpicture}
       \end{equation}
        where there are $p$ cups on the left and $q$ cups on the right.  In the future chapters, it is convenient not to make such an identification, but this approach to the semifinite algebra was the original one that appeared in \cite{MR2807103}.  Under these identifications, $V_{+}$ completes to a semifinite von Neumann algebra, $\cM_{+}$ where the multiplication is given by
        $$
        \left(\begin{tikzpicture} [baseline = 0cm]
       \draw (-.5, -.5) -- (-.5, .5) -- (.5, .5) -- (.5, -.5) -- (-.5, -.5);
       \draw[ultra thick] (0, 1) -- (0, .5);
       \draw[ultra thick] (1, 0) -- (.5, 0);
       \draw[ultra thick] (-1, 0) -- (-.5, 0);
       \node at (0, 0) {$x$};
       \node at (-.75, .25) {$k$};
       \node at (.75, .25) {$l$};
       \node at (.25, .75) {$m$};
       \node at (-.5, -.7) {$\star$};
       \end{tikzpicture}\right) \cdot \left(\begin{tikzpicture}[baseline = 0cm]
       \draw (-.5, -.5) -- (-.5, .5) -- (.5, .5) -- (.5, -.5) -- (-.5, -.5);
       \draw[ultra thick] (0, 1) -- (0, .5);
       \draw[ultra thick] (1, 0) -- (.5, 0);
       \draw[ultra thick] (-1, 0) -- (-.5, 0);
       \node at (0, 0) {$y$};
       \node at (-.75, .25) {$k'$};
       \node at (.75, .25) {$l'$};
       \node at (.25, .75) {$m'$};
       \node at (-.5, -.7) {$\star$};
       \end{tikzpicture}\right) = \delta_{l, k'} \begin{tikzpicture}[baseline = 0cm]
       \draw (-1.5, -.5) -- (-1.5, .5) -- (-.5, .5) -- (-.5, -.5) -- (-1.5, -.5);
       \draw[ultra thick] (-1, 1) -- (-1, .5);
       \draw[ultra thick] (0, 0) -- (-.5, 0);
       \draw[ultra thick] (-2, 0) -- (-1.5, 0);
       \node at (-1, 0) {$x$};
       \node at (-1.75, .25) {$k$};
       \node at (-.7, .75) {$n$};
       \node at (-1.5, -.7) {$\star$};
       \draw (.5, -.5) -- (.5, .5) -- (1.5, .5) -- (1.5, -.5) -- (.5, -.5);
       \draw[ultra thick] (1, 1) -- (1, .5);
       \draw[ultra thick] (2, 0) -- (1.5, 0);
       \draw[ultra thick] (0, 0) -- (.5, 0);
       \node at (1, 0) {$y$};
       \node at (1.75, .25) {$l'$};
       \node at (1.30, .75) {$m'$};
       \node at (0, .3) {$l$};
       \node at (.5, -.7) {$\star$};
       \draw (-2, -1) -- (2,-1)--(2,1)--(-2, 1)--(-2, -1);
       \node at (-2, -1.2) {$\star$};
       \end{tikzpicture}
       $$
         where we have assumed that we have added enough cups as in diagram \eqref{r} so that $l$ and $k'$ are either the same or differ by 1.  The trace on $\cM_{+}$ is given by
         $$
         \Tr(x) = \begin{tikzpicture}[baseline=.5cm]
	\draw (.5,.5) --(-.5,.5) -- (-.5,-.5) -- (.5,-.5) -- (.5, .5) ;
    \draw (.7,2)--(-.7,2)--(-.7,1)--(.7, 1)--(.7,2);
\draw[ultra thick] (-.5, 0) arc (90:270: .5cm) -- (.5, -1) arc(-90:90: .5cm);
\draw[ultra thick] (0, .5) -- (0, 1);
\node at (0, 0) {$x$};
\node at (-.5, -.7) {$\star$};
\node at (-.7, 2.1) {$\star$};
\node at (0, 1.5) {$\sum TL$};
\node at (0, -.75) {$k$}; \end{tikzpicture}
$$
 provided that the number of strings on the left and right of $x$ have the same parity and is zero otherwise.  It is easy to check that the identification in diagram \eqref{r} respects both the trace and multiplication.

The algebras $M_{2k}$ above are a compression of $\cM_{+}$ by the projection $p_{2k}^{+}$ where for general $n$,
$$
p_{n}^{+} = \begin{tikzpicture}[baseline = 0cm]
       \draw[thick] (-.5, -.5) -- (-.5, .5) -- (.5, .5) -- (.5, -.5) -- (-.5, -.5);
       \draw[ultra thick] (-.5, 0) -- (.5, 0);
              \node at (0, .2) {$n$};
       \node at (-.5, -.7) {$\star$};
       \end{tikzpicture}
       $$
        Similarly, we can consider a semifinite von Neumann algebra $\cM_{-}$ generated by the $P_{n, -}$'s (where the $\star$ is in a \emph{shaded} region and on the bottom of the box), and if we define projections $p_{n}^{-}$, then $M_{2k+1}$ is the compression of $\cM_{-}$ by $p_{2k+1}^{-}$.

        A diagrammatic argument shows that $\cM_{+}$ is generated by
        $$
        \cA_{+} = \displaystyle \left(\bigcup_{k, \ell} P^{+}_{k, \ell, 0}\right)^{''} \textrm{ and } X = \textrm{s}-\lim_{k \rightarrow \infty} \begin{tikzpicture}[baseline = 0cm]
       \draw[thick] (-.7, -.7) -- (-.7, .7) -- (.7, .7) -- (.7, -.7) -- (-.7, -.7);
       \draw (-.7, .4)--(0, .4) arc(-90:0 : .3cm);
       \draw[ultra thick] (-.7, 0) -- (.7, 0);
       \node at (0, -.3) {$2k$};
       \node at (-.7, -.9) {$\star$};
       \node at (-.7, -.9) {$\star$};
       \end{tikzpicture}
         +
        \begin{tikzpicture} [baseline = 0cm]
       \draw[thick] (-.7, -.7) -- (-.7, .7) -- (.7, .7) -- (.7, -.7) -- (-.7, -.7);
       \draw (.7, .4)--(0, .4) arc(-90:-180 : .3cm);
       \draw[ultra thick] (-.7, 0) -- (.7, 0);
       \node at (0, -.3) {$2k$};
       \node at (-.7, -.9) {$\star$};
       \end{tikzpicture}
       $$
where the limit above is in the strong operator topology. This element is an $\cA_{+}$-valued semicircular element in the sense of \cite{MR1704661} and is used in the calculation of the isomorphism class of the algebras $M_{k}$

   \paragraph{The finite depth case:}  Let $\Gamma$ denote the principal graph of $\cP$ with edge set $E(\Gamma)$ and initial vertex *.  Let $\ell^{\infty}(\Gamma)$ as the von Neumann algebra of bounded functions on the vertices of $\Gamma$ and endow $\ell^{\infty}(\Gamma)$ with a trace $\tr$ such that $\tr(p_{v}) = \mu_{v}$, where $p_{v}$ is the delta function at $v$ and $\mu_{v}$ is the entry corresponding to a fixed Perron-Frobenius eigenvector for $\Gamma$ with $\mu_* = 1$.  From \cite{MR2807103}, $M_{0} = p_{*}\cN(\Gamma)p_{*}$ where $\cN(\Gamma)$ is the von Neumann algebra generated by $(\ell^{\I}(\Gamma), \tr)$ and $\ell^{\infty}(\Gamma)$-valued semicircular elements $\{X_{e}: e \in E(\Gamma)\}$ which are compressions of $X$ by partial isometries in $\cA_{+}$ and are free with amalgamation over $\ell^{\infty}(\Gamma)$.  Each $X_{e}$ is supported under $p_{v} + p_{w}$, where $e$ connects $v$ and $w$, and we have $X_{e} = p_{v}X_{e}p_{w} + p_{w}X_{e}p_{v}$.  Assuming that $\mu_v \geq \mu_w$, the scalar-valued distribution of $X_{e}^{2}$ in $(p_v + p_w)\cN(\Gamma)(p_v + p_w)$ is free-Poisson with an atom of size $\frac{\mu_v - \mu_w}{\mu_v + \mu_w}$ at 0.  Therefore,
  $$
  vN(\ell^{\infty}(\Gamma), X_{e}) = L(\Z) \otimes M_{2}(\C) \oplus \C \oplus \ell^{\I}(\Gamma \setminus \{v, w\})
  $$
   with $p_w = (1\otimes e_{1,1})\oplus 0\oplus 0$ and $p_v = (1\otimes e_{2,2})\oplus 1\oplus 0$, where $\{e_{i,j}: 1 \leq i,j \leq 2\}$ is a system of matrix units for $M_{2}(\C)$.  If $\Gamma$ is finite, Dykema's formulas for computing certain amalgamated free products \cite{MR1201693, MR2765550} show that $\cN(\Gamma)$ is an interpolated free group factor and the compression formula gives the result for $M_{0}$.  Since $M_{2n}$ is a $\delta^{2n}-$amplification of $M_{0}$, the result holds for $M_{2n}$.  The factor $M_{1}$ is a compression of $\cN(\Gamma^{*})$ with $\Gamma^{*}$ the dual principal graph of $\cP$.  Applying the same analysis to $\Gamma^{*}$ gives the formula for the $M_{2n+1}$'s.

   \paragraph{The infinite depth case:}

    We similarly define $\cN(\Gamma)$ for an arbitrary connected, loopless (not necessarily bipartite) graph $\Gamma$.  If $\Gamma$ is finite, we show that $\cN(\Gamma) \cong L(\F_{t}) \oplus A$ where $A$ is finite-dimensional and abelian ($A$ can possibly be $\{0\}$).  Furthermore, if $p_{\Gamma}$ is the identity of $L(\F_{t})$ and $\Gamma'$ is a finite graph containing $\Gamma$, then the inclusion $p_{\Gamma}\cN(\Gamma)p_{\Gamma} \rightarrow p_{\Gamma}\cN(\Gamma')p_{\Gamma}$ is a standard embedding of interpolated free group factors (see Definition \ref{defn:Dyk} and Remark \ref{rem:Dyk} below).  Therefore, if $\cP$ is infinite depth with principal graph $\Gamma$, we write $\Gamma$ as an increasing union of finite graphs $\Gamma_{k}$ where $\Gamma_{k}$ is $\Gamma$ truncated at depth $k$.  Since standard embeddings are preserved by cut-downs, the inclusion $p_{*}\cN(\Gamma_{k})p_{*} \rightarrow p_{*}\cN(\Gamma_{k+1})p_{*}$ is a standard embedding.  As $M_{0}$ is the inductive limit of the $p_{*}\cN(\Gamma_{k})p_{*}$'s, it is an interpolated free group factor where the parameter is the limit of the parameters for the $p_{*}\cN(\Gamma_{k})p_{*}$'s, which is $\I$.  Since the factors $M_{2k}$ are amplifications of $M_{0}$, $M_{2k} \cong L(\F_{\infty})$.  Applying the same analysis to $\Gamma^{*}$ (the dual principal graph of $\cP$) shows that $M_{2k+1} \cong L(\F_{\infty})$.

 \paragraph{Organization:}  Section \ref{sec:preliminaries} covers some preliminary material on interpolated free group factors, free dimension, and standard embeddings.  Section \ref{sec:VNGraph} introduces $\cN(\Gamma)$ and establishes both its structure and how it includes into $\cN(\Gamma')$ for $\Gamma$ a subgraph of $\Gamma'$.  Section \ref{sec:GJSInfinite} provides the proof that the factors $M_{k}$ above are all isomorphic to $L(\F_{\I})$.

\paragraph{Acknowledgements:}  The author would like to thank Arnaud Brothier, Vaughan Jones, David Penneys, and Dimitri Shlyakhtenko for their conversations and encouragement.  The author was supported by NSF Grant DMS-0856316 and DOD-DARPA grants HR0011-11-1-0001 and HR0011-12-1-0009.

\section{Preliminaries}\label{sec:preliminaries}

Dykema \cite{MR1256179} and R\u{a}dulescu \cite{MR1258909} independently developed interpolated free group factors $L(\F_{t})$ for $1 < t \leq \I$.  These coincide with the usual free group factors when $t \in \N \cup \{\infty\}$ and they satisfy
$$
L(\F_{t})*L(\F_{s}) = L(\F_{s+t}) \textrm{ and } L(\F_{t})_{\gamma} = L(\F(1 + \gamma^{-2}(t-1))),
$$
where $M_{\gamma}$ is the $\gamma-$amplification of the $II_{1}$ factor $M$. It is known that either the interpolated free group factors are all isomorphic or they are pairwise non-isomorphic \cite{MR1256179, MR1258909}.

\begin{nota} \label{nota:vNA} Throughout this paper, we will be concerned with finite von Neumann algebras $(\cM, tr)$ which can be written in the form
$$
 \cM = \overset{p_{0}}{\underset{\gamma_{0}}{\cM_{0}}} \oplus \bigoplus_{j \in J} \overset{p_{j}}{\underset{\gamma_{j}}{L(\F_{t_{j}})}} \oplus  \bigoplus_{k \in K} \overset{q_{k}}{\underset{\alpha_{k}}{M_{n_{k}}(\C)}}
 $$
 where $\cM_{0}$ is a diffuse hyperfinite von Neumann algebra, $L(\F_{t_{j}})$ is an interpolated free group factor with parameter $t_{j}$, $M_{n_{k}}(\C)$ is the algebra of $n_{k} \times n_{k}$ matrices over the scalars, and the sets $J$ and $K$ are at most finite and countably infinite respectively.  We use $p_{j}$ to denote the projection in $L(\F_{t_{j}})$ corresponding to the identity of $L(\F_{t_{j}})$ and $q_{k}$ to denote a minimal projection in $M_{n_{k}}(\C)$. The projections $p_{j}$ and $q_{k}$ have traces $\gamma_{j}$ and $\alpha_{k}$ respectively.  Let $p_{0}$ be the identity in $\cM_{0}$ with trace $\gamma_{0}$.   We write $\overset{p,q}{M_{2}}(\C)$ to mean $M_{2}(\C)$ with a choice of minimal orthogonal projections $p$ and $q$. \end{nota}

If the interpolated free group factors turn out to be non-isomorphic, it is desirable to be able to calculate which interpolated free group factors appear in amalgamated free products.  To help facilitate this calculation, Dykema defined the notion of free dimension.  In general, one has
$$
\fdim(\cM_{1} \underset{D}* \cM_{2}) = \fdim(\cM_{1}) + \fdim(\cM_{2}) - \fdim(D)
$$
whenever $\cM_{1}$ and $\cM_{2}$ are of the form of Notation \ref{nota:vNA} and $D$ is finite dimensional \cite{MR1201693, MR1363079,MR2765550,1110.5597}.  In general, for the algebra $\cM$ in Notation \ref{nota:vNA}, we have $$\fdim(\cM) = 1 + \sum_{j \in J}\gamma_{j}^{2}(t_{j} - 1) - \sum_{k \in K}\alpha_{k}^{2}.$$
Notice that this includes the special case $\fdim(L(\F_{t})) = t$.

Of course if the interpolated free group factors are isomorphic, then the free dimension is not well defined; however, the only purpose of the free dimension is to determine the parameter of interpolated free group factors which show up in amalgamated free products.  Therefore all the lemmas below will remain valid if all references to free dimension are removed.

Of critical importance will be the notion of a \emph{standard embedding} of interpolated free group factors \cite{MR1201693}.  This is a generalization of a mapping $\F_{n} \rightarrow \F_{m}$ for $m > n$ sending the $n$ generators of $\F_{n}$ onto $n$ of the $m$ generators for $\F_{m}$.

\begin{defn} \label{defn:Dyk}

Let $1 < r < s$ and $\phi : L(\F_{r}) \rightarrow L(\F_{s})$ be a von Neumann algebra homomorphism.  We say that $\phi$ is a standard embedding if there exist nonempty sets $S \subset S'$, a family of projections $\{p_{s'}: s' \in S'\}$ with $p_{s'} \in R$ (the hyperfinite $II_{1}$ factor), a free family $\{X^{s'}: s' \in S'\}$ of semicircular elements which are also free from $R$, and isomorphisms
 $$
 \alpha: L(\F_{r}) \rightarrow (R \cup \{p_{s}X^{s}p_{s}\}_{s \in S})'' \textrm{ and } \beta : L(\F_{s}) \rightarrow (R \cup \{p_{s'}X^{s'}p_{s'}\}_{s' \in S'})''
 $$
 such that $\phi = \beta^{-1} \circ \iota \circ \alpha$ where $\iota: (R \cup \{p_{s}X^{s}p_{s}\}_{s \in S})'' \rightarrow (R \cup \{p_{s'}X^{s'}p_{s'}\}_{s' \in S'})''$ is the canonical inclusion.  We will write $A \overset{s.e.}{\hookrightarrow} B$ to mean that the inclusion of $A$ into $B$ is a standard embedding.
\end{defn}

\begin{rem} \label{rem:Dyk} Dykema in \cite{MR1201693} and \cite{MR1363079} shows the following useful properties of standard embeddings which we will use extensively in this paper.

\begin{itemize}

\item[(1)] If $A$ is an interpolated free group factor, the canonical inclusion $A \rightarrow A * \cM$ is a standard embedding whenever $\cM$ is of the form in Notation \ref{nota:vNA}.
\item[(2)] A composite of standard embeddings is a standard embedding.
\item[(3)] If $A_{n} = L(\F_{s_{n}})$ with $s_{n} < s_{n+1}$ for all $n$ and $\phi_{n}: A_{n} \overset{s.e.}{\hookrightarrow} A_{n+1}$, then the inductive limit of the $A_{n}$ with respect to $\phi_{n}$ is $L(\F_{s})$ where $s = \displaystyle \lim_{n \rightarrow \infty}s_{n}$.
\item[(4)] If $t > s$  then $\phi: L(\F_{s}) \overset{s.e.}{\hookrightarrow} L(\F_{t})$ if and only if for any nonzero projection $p \in L(\F_{s})$, $\phi|_{pL(\F_{s})p}: pL(\F_{s})p \overset{s.e.}{\hookrightarrow} \phi(p)L(\F_{t})\phi(p)$.

\end{itemize}

\end{rem}

Our work will rely heavily on the following two lemmas.
\begin{lem}  [\cite{1110.5597}] \label{lem:DR1}

Let $\cN = (\overset{p}{M_{n}(\C)} \oplus B) \underset{D}{*} C$ and $\cM = (M_{n}(\C) \otimes A \oplus B) \underset{D}{*} C$ where $A$, $B$ and $C$ are finite von Neumann algebras and $D$ is a finite dimensional abelian von Neumann algebra.  Let $E$ be the trace-preserving conditional expectation of $\cM$ onto $D$.  Assume $p$ lies under a minimal projection in $D$ and $E|_{M_{n}(\C) \otimes A} = E|_{M_{n}(\C)} \otimes \tr_{A}$.  Then $p\cM p = p\cN p * A$ and the central support of $p$ in $\cM$ is the same as that in $\cN$.

\end{lem}

\begin{lem}  [\cite{1110.5597}] \label{lem:DR2}

Let $\cN = (\overset{p}{\underset{\gamma}{M_{m}(\C)}} \oplus \overset{q}{\underset{\gamma}{M_{n-m}(\C)}} \oplus B) \underset{D}{*} C$ and $\cM = (\underset{\gamma}{M_{n}(\C)} \oplus B) \underset{D}{*} C$ with $B$, $C$, $D$ as in Lemma \ref{lem:DR1}.  Assume $p$ and $q$ sit under minimal projections in $D$ and $p$ is equivalent to $q$ in $\cN$.  Then $p\cM p = p\cN p * L(\Z)$ and the central support of $p$ in $\cM$ is the same as that in $\cN$.

\end{lem}

 Note that if $A$, $B$ and $C$ are in the form in Notation \ref{nota:vNA}, and if $\cN$ is an interpolated free group factor, then the proofs of the above lemmas in \cite{1110.5597} show that $p\cN p \rightarrow p\cM p$ of Lemmas \ref{lem:DR1} and \ref{lem:DR2} are standard embeddings.  This implies $\cN \overset{s.e.}{\hookrightarrow} \cM$ by Remark \ref{rem:Dyk}.

 Also, we will use Dykema's results for free products of finite-dimensional, abelian, von Neumann algebras.  More precisely, we will use the following theorem, which appears in \cite{MR1201693}

\begin{thm} \label{thm:Dyk}
We have the following formulas for free products:
\begin{itemize}
\item Assume
$$
1 \geq \alpha_{1} \geq \beta_{1} \geq \beta_{2} \geq \alpha_{2} \geq 0 \text{ with } \alpha_{1} + \alpha_{2} = 1 = \beta_{1} + \beta_{2}.
$$
Then
$$
\left(\overset{p_{1}}{\underset{\alpha_{1}}{\C}} \oplus \overset{p_{2}}{\underset{\alpha_{2}}{\C}}\right) * \left(\overset{q_{1}}{\underset{\beta_{1}}{\C}} \oplus \overset{q_{2}}{\underset{\beta_{2}}{\C}}\right) = \left( \underset{2\alpha_{2}}{(M_{2}(\C)\otimes L(\Z))} \oplus \overset{p_{1}\wedge q_{1}}{\underset{\alpha_{1} - \beta_{1}}{\C}} \oplus  \overset{p_{1}\wedge q_{2}}{\underset{\alpha_{1} - \beta_{2}}{\C}}\right)
$$

\item Assume
$$
1 \geq \alpha_{1} \geq \cdots \geq \alpha_{n} \geq 0 \text{ and } 1 \geq \beta_{1} \geq \cdots \geq \beta_{m} \geq 0
$$
with $n, m \geq 2$, $n + m \geq 5$, and $\sum_{i=1}^{n} \alpha_{i} = 1 = \sum_{j=1}^{m}\beta_{j}$.  Set $B = \{ (i, j): \alpha_{i} + \beta_{j} > 1\}$.  Then
$$
\left(\overset{p_{1}}{\underset{\alpha_{1}}{\C}} \oplus \cdots \oplus \overset{p_{n}}{\underset{\alpha_{n}}{\C}}\right) * \left(\overset{q_{1}}{\underset{\alpha_{1}}{\C}} \oplus \cdots \oplus \overset{q_{m}}{\underset{\beta_{m}}{\C}}\right) = L(\F_{t}) \oplus\left(\bigoplus_{(i, j) \in B}\overset{p_{i}\wedge q_{j}}{\underset{\alpha_{i} + \beta_{j} - 1}{\C}}\right)
$$
where $t$ is calculated using free dimension.

\item Assume that $A$ is either a diffuse hyperfinite finite von Neumann algebra or an interpolated free group factor.  Let $\alpha_{i}$ and $\beta_{j}$ be as above, except that $i + j$ need not exceed 4 and $\sum_{i=1}^{n} \alpha_{i} < 1$.  Then
$$
\left(A \oplus \overset{p_{1}}{\underset{\alpha_{1}}{\C}} \oplus \cdots \oplus \overset{p_{n}}{\underset{\alpha_{n}}{\C}}\right) * \left(\overset{q_{1}}{\underset{\alpha_{1}}{\C}} \oplus \cdots \oplus \overset{q_{m}}{\underset{\beta_{m}}{\C}}\right) = L(\F_{t}) \oplus\left(\bigoplus_{(i, j) \in B}\overset{p_{i}\wedge q_{j}}{\underset{\alpha_{i} + \beta_{j} - 1}{\C}}\right)
$$
where $t$ is calculated using free dimension.
\end{itemize}
\end{thm}


\section{A von Neumann algebra associated to a finite connected graph}\label{sec:VNGraph}

Let $\Gamma$ be a connected, loopless, finite graph with edge set $E(\Gamma)$ and vertex set $V(\Gamma)$.  Assume further that each vertex $v \in V(\Gamma)$ is weighted by a real constant $\gamma_{v} > 0$ with $\sum_{v \in \Gamma} \gamma_{v} = 1$ (the weighting does \emph{not} have to be the Perron-Frobenius weighting.  Consider the abelian von Neumann algebra $(\ell^{\infty}(\Gamma), \tr)$ where $\tr$ is defined as follows: Let $p_{v}$ be the indicator function on the vertex $v$.  Then $\tr(p_{v})$ will be $\gamma_{v}$.  We construct a finite von Neumann algebra associated to $\Gamma$ (also see \cite{MR2772347}).

\begin{defn}\label{defn:vNGraph}
Let $\Gamma$ be as above, $e$ be an edge in $\Gamma$ connecting the vertices $v$ and $w$, and assume $\gamma_{v} \geq \gamma_{w}$.  Define
$$
\cA_{e} = \underset{2\gamma_{w}}{M_{2}(\C) \otimes L(\Z)} \oplus \underset{\gamma_{v} - \gamma_{w}}{\overset{p^{e}_{v}}{\C}} \oplus \ell^{\infty}(\Gamma \setminus \{v, w\})
$$
 where the trace on $M_{2}(\C) \otimes L(\Z)$ is $\tr_{M_{2}(\C)} \otimes \tr_{L(\Z)}$.  $\cA_{e}$ includes $\ell^{\infty}(\Gamma)$ by setting \begin{align*}p_{w} &= 1 \otimes e_{1, 1} \oplus 0 \oplus 0 \textrm{ and } \\ p_{v} &= 1 \otimes e_{2, 2} \oplus 1 \oplus 0.\end{align*}   Therefore, the trace preserving conditional expectation $E_{e}: \cA_{e} \rightarrow \ell^{\infty}(\Gamma)$ has the property $E_{e}|_{M_{2} \otimes L(Z)} = E_{e}|_{M_{2}} \otimes \tr|_{L(\Z)}$. We define $\cN(\Gamma)$, the von Neumann algebra associated to $\Gamma$, by $$
 \cN(\Gamma) = \underset{{\ell^{\infty}(\Gamma)}}{*} (\cA_{e}, E_{e})_{e \in E(\Gamma)}.
 $$
\end{defn}

\begin{rem} If $\Gamma$ is an infinite graph with a weighting that is not $\ell^{1}$, then we can still define $\cN(\Gamma)$ as in \ref{defn:vNGraph} although it will be a semifinite algebra.  Given $e \in E(\Gamma)$ connecting vertices $v$ and $w$, the compressed algebra $(p_{v} + p_{w})\cA_{e}(p_{v} + p_{w})$ is still finite, and if $E_{e}: \cA_{e} \rightarrow \ell^{\infty}(\Gamma)$ is the (tracial-weight) preserving conditional expectation, then $E_{e}$ is clearly normal on $(p_{v} + p_{w})\cA_{e}(p_{v} + p_{w})$ and is the identity on $(1 - p_{v} - p_{w})\cA_{e}(1 - p_{v} - p_{w})$.  Therefore one can take the algebraic free product $Q$ of $(\cA_{e})_{e\in E(\Gamma)}$ with amalgamation over $\ell^{\I}(\Gamma)$ and represent it on $L^{2}(Q, Tr \circ \underset{\ell^{\infty}(\Gamma)}{*}E_{e})$ to obtain $\cN(\Gamma)$.

 \end{rem}

 \begin{defn} \label{defn:H2} Let $v, w \in V(\Gamma)$ We write $v \sim w$ if $v$ and $w$ are connected by at least 1 edge in $\Gamma$ and denote $n_{v, w}$ be the number of edges joining $v$ and $w$.  We set $\alpha^{\Gamma}_{v} = \sum_{w\sim v} n_{v, w}\gamma_{w}$, and define $B(\Gamma) = \{ v \in V(\Gamma) : \gamma_{v} > \alpha^{\Gamma}_{v}\}$.  \end{defn}

 For the rest of this section, we assume $\Gamma$ is finite. We show that $\cN(\Gamma)$ is the direct sum of an interpolated free group factor and a finite dimensional abelian algebra.  More precisely, we prove the following theorem, which has a direct analogy with Theorem \ref{thm:Dyk}:

\begin{thm} \label{thm:H1}

Let $\Gamma$ and $\Gamma'$ be connected, finite, loopless, and weighted graphs with at least 2 edges.  Then $\cN(\Gamma) \cong \overset{p^{\Gamma}}{L(\F_{t_{\Gamma}})} \oplus \underset{{v \in B(\Gamma)}}{\bigoplus} \overset{r_{v}^{\Gamma}}{\underset{\gamma_{v} - \alpha^{\Gamma}_{v}}{\C}}$ where $r_{v}^{\Gamma} \leq p_{v}$ and $t_{\Gamma}$ is such that this algebra has the appropriate free dimension.  Furthermore, if $\Gamma$ is a subgraph of $\Gamma'$, then $p^{\Gamma}\cN(\Gamma)p^{\Gamma} \overset{s.e.}{\hookrightarrow} p^{\Gamma}\cN(\Gamma')p^{\Gamma}$.

\end{thm}

 Notice that since we are assuming that all vertices have positive weight, it follows that $p_{v}p^{\Gamma}\neq 0$ for all $v \in \Gamma$.  We will prove Theorem \ref{thm:H1} in a series of lemmas.

\begin{lem} \label{lem:H1}

Let $\Gamma$ be a finite, connected, weighted, loopless graph with 2 edges.  Then $\cN(\Gamma)$ is of the form in Theorem \ref{thm:H1}.

\end{lem}

\begin{proof}

 Set $D = \ell^{\infty}(\Gamma)$.  There are two overlying cases to consider.  One where $\Gamma$ has 2 vertices and the other where $\Gamma$ has 3 vertices.

   \emph{\underline{Case 1}}: Assume that $\Gamma$ has 2 vertices $v, w$ and 2 edges $e_{1}$ and $e_{2}$ connecting $v$ and $w$ and without loss of generality assume $\gamma_{v} \geq \gamma_{w}$.  We obtain the desired formula for $\cN(\Gamma)$ by examining the following sequence of inclusions:

 \begin{align*}
 \cN_{0} &= \left(\overset{p_{w}}{\underset{\gamma_{w}}{\C }}\oplus \overset{p_{v}^{I}}{\underset{\gamma_{w}}{\C}} \oplus \underset{\gamma_{v} - \gamma_{w}}{\C}\right) \underset{D}* \left(\overset{p_{w}}{\underset{\gamma_{w}}{\C }}\oplus \overset{q_{v}^{I}}{\underset{\gamma_{w}}{\C}} \oplus \underset{\gamma_{v} - \gamma_{w}}{\C}\right)\\
 \cap&\\
     \cN_{1} &= \left(\overset{p_{w}, p_{v}^{I}}{\underset{\gamma_{w}}{M_{2}(\C)}} \oplus \underset{\gamma_{v} - \gamma_{w}}{\C}\right) \underset{D}{*} \left(\overset{p_{w}}{\underset{\gamma_{w}}{\C }}\oplus \overset{q_{v}^{I}}{\underset{\gamma_{w}}{\C}} \oplus \underset{\gamma_{v} - \gamma_{w}}{\C}\right)\\
      \cap&\\
       \cN_{2} &= \left(\overset{p_{w}, p_{v}^{I}}{\underset{\gamma_{w}}{M_{2}(\C)}} \oplus \underset{\gamma_{v} - \gamma_{w}}{\C}\right) \underset{D}{*} \left(\overset{p_{w}, q_{v}^{I}}{\underset{\gamma_{w}}{M_{2}(\C)}} \oplus \underset{\gamma_{v} - \gamma_{w}}{\C}\right)\\
        \cap&\\
        \cN_{3} &= \left(\underset{2\gamma_{w}}{M_{2}(\C) \otimes L(\Z)} \oplus \underset{\gamma_{v} - \gamma_{w}}{\C}\right) \underset{D}{*} \left(\overset{p_{w}}{\underset{\gamma_{w}}{M_{2}(\C)}} \oplus \underset{\gamma_{v} - \gamma_{w}}{\C}\right)\\
         \cap&\\
          \cN(\Gamma) &= \left(\underset{2\gamma_{w}}{M_{2}(\C) \otimes L(\Z)} \oplus \underset{\gamma_{v} - \gamma_{w}}{\C}\right) \underset{D}{*} \left(\underset{2\gamma_{w}}{M_{2}(\C) \otimes L(\Z)} \oplus \underset{\gamma_{v} - \gamma_{w}}{\C}\right),
          \end{align*}
        where $p_{v}$ decomposes as $(1 \otimes e_{2, 2}) \oplus 1$ in $\cA_{e_{1}}$ and $\cA_{e_{2}}$ with $p_{v}^{I} = 1 \otimes e_{2,2}$ in $\cA_{e_{1}}$ and $q_{v}^{I} = 1 \otimes e_{2,2}$ in $\cA_{e_{2}}$.   From Lemma \ref{lem:DR1} and \cite{MR1201693}, we see that
        $$
        p_{v}\cN_{0}p_{v} = \overset{p_{v}^{I}}{\underset{\frac{\gamma_{w}}{\gamma{v}}}{\C}} \oplus \underset{\frac{\gamma_{v} - \gamma_{w}}{\gamma{v}}}{\C}*\overset{q_{v}^{I}}{\underset{\frac{\gamma_{w}}{\gamma{v}}}{\C}} \oplus \underset{\frac{\gamma_{v} - \gamma_{w}}{\gamma{v}}}{\C} = \begin{cases} \underset{2\frac{\gamma_{v} - \gamma_{w}}{\gamma_{v}}}{M_{2}(\C) \otimes L(\Z)} \oplus \overset{p_{v}^{I} \wedge q_{v}^{I}}{\underset{\frac{2\gamma_{w} - \gamma_{v}}{\gamma_{v}}}{\C}} & \textrm{if } 2\gamma_{w} \geq\gamma_{v}\\ \underset{\frac{2\gamma_{w}}{\gamma_{v}}}{M_{2}(\C) \otimes L(\Z)} \oplus \overset{(p_{v} - p_{v}^{I})\wedge(p_{v} - q_{v}^{I})}{\underset{\frac{\gamma_{v} - 2\gamma_{w}}{\gamma_{v}}}{\C}} & \textrm{if }\gamma_{v} > 2\gamma_{w} \end{cases}
        $$
        where in the first algebra, the identity element copy of $\C$ is $p_{v}^{I} \wedge q_{v}^{I}$ and and in the second algebra, the identity of the copy of $\C$ is orthogonal to both $p_{v}^{I}$ and $q_{v}^{I}$.

     \emph{\underline{Case 1a}}:  Assume $2\gamma_{w} \geq \gamma_{v}$.  As $p_{v}\wedge q_{v}$ is minimal and central in $\cN_{0}$, we see that
      $$
      \cN_{1} = \underset{3(\gamma_{v} - \gamma_{w})}{M_{3}(\C)\otimes L(\Z)} \oplus \underset{2\gamma_{w} - \gamma_{v}}{\overset{p_{v}^{I} \wedge q_{v}^{I}}{M_{2}(\C)}}.
      $$
       By \cite{MR1201693}, the projections $p_{v}^{I}$ and $q_{v}^{I}$ are equivalent in $\cN_{0}$, so it follows that $p_{w}$ is equivalent to $q_{v}^{I}$ in $\cN_{1}$.  Therefore by Lemma \ref{lem:DR2},
        $$
        p_{w}\cN_{2}p_{w} = p_{w}\cN_{1}p_{w} * L(\Z) = (\underset{\frac{\gamma_{v} - \gamma_{w}}{\gamma_{w}}}{L(\Z)} \oplus \underset{\frac{2\gamma_{w} - \gamma_{v}}{\gamma_{w}}}{\C}) * L(\Z),
        $$
         which is an interpolated free group factor $L(\F_{t})$ by Theorem \ref{thm:Dyk}.  As the central support of $p_{w}$ in $\cN_{2}$ is 1, it follows that $\cN_{2}$ is also an interpolated free group factor $L(\F_{t_{1}})$.  To finish up this case, we write
         \begin{align*}
         \cN_{2} \subset \cN_{3} &= \left(\underset{2\gamma_{w}}{M_{2}(\C) \otimes L(\Z)} \oplus \underset{\gamma_{v} - \gamma_{w}}{\C}\right) \underset{D}{*} \left(\overset{p_{w}}{\underset{\gamma_{w}}{M_{2}(\C)}} \oplus \underset{\gamma_{v} - \gamma_{w}}{\C}\right)  \textrm{ and }\\
      \cN_{3} \subset \cN(\Gamma) &= \left(\underset{2\gamma_{w}}{M_{2}(\C) \otimes L(\Z)} \oplus \underset{\gamma_{v} - \gamma_{w}}{\C}\right) \underset{D}{*} \left(\underset{2\gamma_{w}}{M_{2}(\C) \otimes L(\Z)} \oplus \underset{\gamma_{v} - \gamma_{w}}{\C}\right),
      \end{align*}
       and use Lemma \ref{lem:DR1} twice, as well as the amplification formula for interpolated free group factors to obtain that $\cN(\Gamma)$ is an interpolated free group factor.

      \emph{\underline{Case 1b}}: The case $\gamma_{v} > 2\gamma_{w}$ for $\cN_{0}$ is treated exactly the same as the first with only the caveat that the central support of $p_{w}$ in $\cN_{1}$ is a projection of trace $3\gamma_{w}$, so $\cN_{1}$, and thus $\cN_{2}$, $\cN_{3}$, and $\cN(\Gamma)$, have a minimal central projection of trace $\gamma_{v} - 2\gamma_{w}$.\\

 \emph{\underline{Case 2}}: Next we consider the case where $\Gamma$ has 3 vertices $v_{1}$, $v_{2}$, and $v_{3}$ with weights $\gamma_{1}$, $\gamma_{2}$, and $\gamma_{3}$ respectively, where $v_{2}$ is connected to $v_{1}$ by $e_{1}$ and to $v_{3}$ by $e_{2}$.  There are two sub-cases to consider.  The first is when $\gamma_{2} \geq \gamma_{1} \geq \gamma_{3}$, and the second is when $\gamma_{1} > \gamma_{2}$ and $\gamma_{1} \geq \gamma_{3}$.

 \emph{\underline{Case 2a}}: We examine the following sequence of inclusions:
  \begin{align*}
  \cN_{0} &= \left(\overset{p_{v_{1}}}{\underset{\gamma_{1}}{\C}} \oplus \overset{p^{I}_{2}}{\underset{\gamma_{1}}{\C}} \oplus \overset{p^{II}_{2}}{\underset{\gamma_{2} - \gamma_{1}}{\C}} \oplus \overset{p_{v_{3}}}{\underset{\gamma_{3}}{\C}}\right) \underset{D}{*} \left(\overset{p_{v_{1}}}{\underset{\gamma_{1}}{\C}} \oplus \overset{q^{I}_{2}}{\underset{\gamma_{2} - \gamma_{3}}{\C}} \oplus \overset{q^{II}_{2}}{\underset{\gamma_{3}}{\C}} \oplus \overset{p_{v_{3}}}{\underset{\gamma_{3}}{\C}}\right)\\
  \cap& \\
  \cN_{1} &= \left(\overset{p_{v_{1}}, p_{2}^{I}}{\underset{\gamma_{1}}{M_{2}(\C)}} \oplus \overset{p_{2}^{II}}{\underset{\gamma_{2} - \gamma_{1}}{\C}} \oplus \overset{p_{v_{3}}}{\underset{\gamma_{3}}{\C}}\right) \underset{D}{*} \left(\overset{p_{v_{1}}}{\underset{\gamma_{1}}{\C}} \oplus \overset{q_{2}^{I}}{\underset{\gamma_{2} - \gamma_{3}}{\C}} \oplus \overset{q_{2}^{II},p_{v_{3}}}{\underset{\gamma_{3}}{M_{2}(\C)}}\right)\\
  \cap &\\
  \cN_{2} &= \left(\underset{2\gamma_{1}}{M_{2}(\C)\otimes L(\Z)} \oplus \overset{p_{2}^{II}}{\underset{\gamma_{2} - \gamma_{1}}{\C}} \oplus \overset{p_{v_{3}}}{\underset{\gamma_{3}}{\C}}\right) \underset{D}{*} \left(\overset{p_{v_{1}}}{\underset{\gamma_{1}}{\C}} \oplus \overset{q_{2}^{I}}{\underset{\gamma_{2} - \gamma_{3}}{\C}} \oplus \overset{q_{2}^{II}, p_{v_{3}}}{\underset{\gamma_{3}}{M_{2}(\C)}}\right)\\
  \cap &\\
   \cN(\Gamma) &= \left(\underset{\gamma_{1}}{M_{2}(\C)\otimes L(\Z)} \oplus \overset{p_{2}^{II}}{\underset{\gamma_{2} - \gamma_{1}}{\C}} \oplus \overset{p_{v_{3}}}{\underset{\gamma_{3}}{\C}}\right) \underset{D}{*} \left(\overset{p_{v_{1}}}{\underset{\gamma_{1}}{\C}} \oplus \overset{q_{2}^{I}}{\underset{\gamma_{2} - \gamma_{3}}{\C}} \oplus \underset{2\gamma_{3}}{M_{2}(\C)\otimes L(\Z)}\right),
  \end{align*}
  where $p_{v_{2}}$ decomposes as $1 \otimes e_{22} \oplus 1 \oplus 0$ in $\cA_{e_{1}}$ and $0 \oplus 1 \oplus 1 \otimes e_{1, 1}$ in $\cA_{e_{2}}$.  We set $p_{2}^{I}$ and $p_{2}^{II}$ as the summands of $p_{v_{2}}$ supported in the diffuse and atomic parts of $\cA_{e_{1}}$ respectively and $q_{2}^{I}$ and $q_{2}^{II}$ as the summands of $p_{v_{2}}$ supported in the atomic and diffuse parts of $\cA_{e_{2}}$ respectively.  As above,
    $$
    p_{v_{2}}\cN_{0}p_{v_{2}} = \overset{p^{I}_{2}}{\underset{\frac{\gamma_{1}}{\gamma_{2}}}{\C}} \oplus \overset{p^{II}_{2}}{\underset{\frac{\gamma_{2} - \gamma_{1}}{\gamma_{2}}}{\C}} * \overset{q^{I}_{2}}{\underset{\frac{\gamma_{2} - \gamma_{3}}{\gamma_{2}}}{\C}} \oplus \overset{q^{II}_{2}}{\underset{\frac{\gamma_{3}}{\gamma_{2}}}{\C}} = \begin{cases} \underset{2\frac{\gamma_{2} - \gamma_{1}}{\gamma_{2}}}{M_{2}(\C) \otimes L(\Z)} \oplus \overset{p_{2}^{I} \wedge q_{2}^{I}}{\underset{\frac{\gamma_{1} - \gamma_{3}}{\gamma_{2}}}{\C}} \oplus \overset{p_{2}^{I} \wedge q_{2}^{II}}{\underset{\frac{\gamma_{1} - \gamma_{2} +  \gamma_{3}}{\gamma_{2}}}{\C}} & \textrm{ if } \gamma_{2} \leq \gamma_{1} + \gamma_{3} \\ \underset{2\frac{\gamma_{3}}{\gamma_{2}}}{M_{2}(\C) \otimes L(\Z)} \oplus \overset{p_{2}^{I} \wedge q_{2}^{I}}{\underset{\frac{\gamma_{1} - \gamma_{3}}{\gamma_{2}}}{\C}} \oplus \overset{p_{2}^{II} \wedge q_{2}^{I}}{\underset{\frac{\gamma_{2} - \gamma_{1} -  \gamma_{3}}{\gamma_{2}}}{\C}} & \textrm{ if } \gamma_{2} > \gamma_{1} + \gamma_{3} \end{cases}.
    $$

 \emph{\underline{Case 2a(i)}}: Assume $\gamma_{2} \leq \gamma_{1} + \gamma_{3}$.  Since the two new matrix units in $\cN_{1}$ introduce equivalences between $p_{v_{1}}$ and $p_{2}^{I}$ and between $q_{2}^{II}$ and $p_{v_{3}}$ respectively, we see that $\cN_{1}$ has the same number of summands as $p_{v_{2}}\cN_{0}p_{v_{2}}$, but with suitable amplifications.  Explicitly, we find that
   $$
   \cN_{1} = \underset{4(\gamma_{2} - \gamma_{1})}{M_{4}(\C)\otimes L(\Z)} \oplus \overset{p_{2}^{I}\wedge q_{2}^{I}}{\underset{\gamma_{1} - \gamma_{3}}{M_{2}(\C)}} \oplus \overset{p_{2}^{I} \wedge q_{2}^{II}}{\underset{\gamma_{1} + \gamma_{3} - \gamma_{2}}{M_{3}(\C)}}
   $$
    where the central support of $p_{v_{1}}$ is the identity.  By applying Lemma \ref{lem:DR1} and applying the same reasoning as case 1,
    we see that $\cN_{2}$ is an interpolated free group factor.  Applying Lemma \ref{lem:DR1} again shows that $\cN(\Gamma)$ is an interpolated free group factor.

 \emph{\underline{Case 2a(ii)}}: Assume $\gamma_{2} > \gamma_{1} + \gamma_{3}$.  This case is treated in the same way as above except that in $\cN_{1}$, $q_{2}^{I} \wedge p_{2}^{II}$ with trace $\gamma_{2} - \gamma_{3} - \gamma_{1}$ is minimal and central, so it is minimal and central in $\cN(\Gamma)$.

 \emph{\underline{Case 2b}}: Now let $\gamma_{1}$ be the largest weight. First assume $\gamma_{3} \geq \gamma_{2}$.  We consider the algebra
 $$
 \cN_{1} = \left(\overset{p_{1}^{I}}{\underset{\gamma_{1} - \gamma_{2}}{\C}}  \oplus \overset{p_{v_{2}}}{\underset{\gamma_{2}}{M_{2}(\C)}} \oplus \overset{p_{v_{3}}}{\underset{\gamma_{3}}{\C}}\right) \underset{D}{*} \left(\overset{p_{v_{1}}}{\underset{\gamma_{1}}{\C}} \oplus \overset{p_{v_{2}}}{\underset{\gamma_{2}}{M_{2}(\C)}} \oplus \overset{p_{3}^{I}}{\underset{\gamma_{3} - \gamma_{2}}{\C}}\right),
 $$
 where the projections orthogonal to $p_{v_{2}}$ in each copy of $M_{2}(\C)$ sit under $p_{i}$ and $p_{i}^{I} \leq p_{v_{i}}$ for $i = 1$ or 3.  It follows that $\cN_{1} = \overset{p_{1}^{I}}{\underset{\gamma_{1} - \gamma_{2}}{\C}} \oplus \overset{p_{v_{2}}}{\underset{\gamma_{2}}{M_{3}(\C)}} \oplus \overset{p_{3}^{I}}{\underset{\gamma_{3} - \gamma_{2}}{\C}}$, so tensoring each copy of $M_{2}(\C)$ with $L(\Z)$ and using the standard arguments as above show that
 $$
 \cN(\Gamma)  = \overset{p_{1}^{I}}{\underset{\gamma_{1} - \gamma_{2}}{\C}} \oplus \underset{3\gamma_{2}}{L(\F_{t})} \oplus \overset{p_{3}^{I}}{\underset{\gamma_{3} - \gamma_{2}}{\C}}.
 $$
 Finally, if $\gamma_{2} > \gamma_{3}$ then we consider
 $$
 \cN_{1} = \left(\overset{p_{1}^{I}}{\underset{\gamma_{1} - \gamma_{2}}{\C}}  \oplus \overset{p_{v_{2}}}{\underset{\gamma_{2}}{M_{2}(\C)}} \oplus \overset{p_{v_{3}}}{\underset{\gamma_{3}}{\C}}\right) \underset{D}{*} \left(\overset{p_{v_{1}}}{\underset{\gamma_{1}}{\C}} \oplus \overset{p_{2}^{I}}{\underset{\gamma_{2} - \gamma_{3}}{\C}} \oplus \overset{p_{v_{3}}}{\underset{\gamma_{3}}{M_{2}(\C)}}\right) = \overset{p_{1}^{I}}{\underset{\gamma_{1} - \gamma_{2}}{\C}} \oplus \underset{\gamma_{3}}{M_{3}(\C)} \oplus \underset{\gamma_{2} - \gamma_{3}}{M_{2}(\C)},
 $$
 where the central support of $p_{v_{2}}$ is $1 - p_{1}^{I}$.  Therefore, tensoring each copy of $M_{2}(\C)$ with $L(\Z)$ gives $\cN(\Gamma) = \overset{p_{1}^{I}}{\underset{\gamma_{1} - \gamma_{2}}{\C}} \oplus \underset{2\gamma_{2} + \gamma_{3}}{L(\F_{t})}$ as desired.
\end{proof}

We now inductively assume that for some $\Gamma$, $\cN(\Gamma)$ has the form as described in Theorem \ref{thm:H1}.\\

\begin{lem} \label{lem:H2}

  Suppose $\Gamma'$ is a graph obtained from $\Gamma$ by adding an edge $e$ connecting two vertices $v$ and $w$ of $\Gamma$ (so that in particular $\Gamma$ and $\Gamma'$ have the same underlying set of vertices with the same weighting).  Assume that
  $$
  \cN(\Gamma) = \overset{p^{\Gamma}}{L(\F_{t_{\Gamma}})} \oplus \underset{{v \in B(\Gamma)}}{\bigoplus} \overset{r_{v}^{\Gamma}}{\underset{\gamma_{v} - \alpha^{\Gamma}_{v}}{\C}}
  $$
  as in Theorem \ref{thm:H1}.  Then
  $$
  \cN(\Gamma') = \overset{p^{\Gamma'}}{L(\F_{t_{\Gamma'}})} \oplus \underset{{v \in B(\Gamma')}}{\bigoplus} \overset{r_{v}^{\Gamma'}}{\underset{\gamma_{v} - \alpha^{\Gamma'}_{v}}{\C}}
  $$
  where $p^{\Gamma} \leq p^{\Gamma'}$, $r_{v}^{\Gamma'} \leq r_{v}^{\Gamma}$, and $p^{\Gamma}\cN(\Gamma)p^{\Gamma} \overset{s.e.}{\hookrightarrow} p_{\Gamma}\cN(\Gamma')p_{\Gamma}$.

\end{lem}

\begin{proof}

  We use the convention that if the term $\overset{p}{\underset{\alpha}{\C}}$ appears where $\alpha \leq 0$ then this term is identically zero.  All parts of the proof below are valid if this modification is made.

  Set $D = \ell^{\infty}(\Gamma') = \ell^{\infty}(\Gamma)$ and let the new edge $e$ connect $v$ to $w$ with $\gamma_{v} \geq \gamma_{w}$.  We examine the following sequence of inclusions:

\begin{align*}
\cN(\Gamma) \subset \cN_{1} &= \cN(\Gamma) \underset{D}{*} \left(\overset{p_{w}}{\underset{\gamma_{w}}{\C}} \oplus \left(\bigoplus_{k=1}^{n} \overset{p_{v, k}}{\underset{\gamma_{w}/n}{\C}} \oplus \overset{p_{v}^{I}}{\underset{\gamma_{v} - \gamma_{w}}{\C}}\right) \oplus \ell^{\I}(\Gamma \setminus \{v, w\})\right)\\
\cap &\\
\cN_{2} &= \cN(\Gamma) \underset{D}{*} \left(\bigoplus_{k=1}^{n} \overset{p_{w, k}}{\underset{\gamma_{w}/n}{\C}} \oplus  \left(\bigoplus_{k=1}^{n} \overset{p_{v, k}}{\underset{\gamma_{w}/n}{\C}} \oplus \overset{p_{v}^{I}}{\underset{\gamma_{v} - \gamma_{w}}{\C}}\right) \oplus \ell^{\I}(\Gamma \setminus \{v, w\})\right)\\
\cap &\\
\cN_{3} &= \cN(\Gamma) \underset{D}{*} \left(\bigoplus_{k=1}^{n}\overset{p_{w, k}, p_{v, k}}{\underset{\gamma_{w}/n}{M_{2}(\C)}} \oplus \overset{p_{v}^{I}}{\underset{\gamma_{v} - \gamma_{w}}{\C}} \oplus \ell^{\I}(\Gamma \setminus \{v, w\})\right)\\
\cap &\\
\cN(\Gamma') &= \cN(\Gamma) \underset{D}{*} \left(L(\Z) \otimes M_{2}(\C) \oplus \overset{p_{v}^{I}}{\underset{\gamma_{v} - \gamma_{w}}{\C}} \oplus \ell^{\I}(\Gamma \setminus \{v, w\})\right).
\end{align*}
The projections $p_{w, k}$ are an orthogonal family with trace $\gamma_{w}/n$ in $\cA_{e}$ whose sum is $p_{w}$.  The projection $p_{v}$ decomposes as $\sum_{k=1}^{n}p_{v, k} + p_{v}^{I}$ with  $p_{v}^{I}$ supported in the atomic part of $\cA_{e}$ and the $p_{v, k}$ are an orthogonal family of projections with trace $\gamma_{w}/n$ supported in the diffuse part of $\cA_{e}$. The positive integer $n$ is chosen large enough such that $\frac{1}{n} + \frac{\gamma_{w} - \alpha^{\Gamma}_{w}}{\gamma_{w}} < 1$ and $\frac{\gamma_{w}}{n\gamma_{v}} + \frac{\gamma_{v} - \alpha^{\Gamma}_{v}}{\gamma_{v}} < 1$.  From the induction hypothesis,
$$
p_{v}\cN(\Gamma)p_{v} = \overset{p^{\Gamma}_{v}}{L(\F_{t_{v}})} \oplus \overset{r_{v}^{\Gamma}}{\underset{\frac{\gamma_{v} - \alpha_{v}}{\gamma_{v}}}{\C}}, \textrm{ and }p_{w}\cN(\Gamma)p_{w} = \overset{p^{\Gamma}_{w}}{L(\F_{t_{w}})} \oplus \overset{r_{w}^{\Gamma}}{\underset{\frac{\gamma_{w} - \alpha_{w}}{\gamma_{w}}}{\C}},
$$
with $p^{\Gamma}_{u} = p^{\Gamma}p_{u}$ for any vertex $u$.   From Lemma \ref{lem:DR1},
$$
p_{v}\cN_{1}p_{v} = \left(\overset{p^{\Gamma}_{v}}{L(\F_{t_{v}})} \oplus \overset{r_{v}^{\Gamma}}{\underset{\frac{\gamma_{v} - \alpha^{\Gamma}_{v}}{\gamma_{v}}}{\C}}\right) * \left(\bigoplus_{k=1}^{n} \overset{p_{v, k}}{\underset{\frac{\gamma_{w}}{n\gamma_{v}}}{\C}} \oplus \overset{p_{v}^{I}}{\underset{\frac{\gamma_{v} - \gamma_{w}}{\gamma_{v}}}{\C}}\right) = L(\F_{t_{v,1}}) \oplus \overset{p_{v}^{I}\wedge r_{v}^{\Gamma}}{\underset{\frac{\gamma_{v} - \alpha^{\Gamma'}_{v}}{\gamma_{v}}}{\C}}.
$$
 Lemma \ref{lem:DR1} applied to the inclusion
 $$
 \left(\overset{p^{\Gamma}_{v}}{\C} \oplus \overset{r_{v}^{\Gamma}}{\underset{\frac{\gamma_{v} - \alpha_{v}}{\gamma_{v}}}{\C}}\right) * \left(\bigoplus_{k=1}^{n} \overset{p_{v, k}}{\underset{\frac{\gamma_{w}}{n\gamma_{v}}}{\C}} \oplus \overset{p_{v}^{I}}{\underset{\frac{\gamma_{v} - \gamma_{w}}{\gamma_{v}}}{\C}}\right) \rightarrow \left(\overset{p^{\Gamma}_{v}}{L(\F_{t_{v}})} \oplus \overset{r_{v}^{\Gamma}}{\underset{\frac{\gamma_{v} - \alpha_{v}}{\gamma_{v}}}{\C}}\right) * \left(\bigoplus_{k=1}^{n} \overset{p_{v, k}}{\underset{\frac{\gamma_{w}}{n\gamma_{v}}}{\C}} \oplus \overset{p_{v}^{I}}{\underset{\frac{\gamma_{v} - \gamma_{w}}{\gamma_{v}}}{\C}}\right),
 $$
  shows that the inclusion $\displaystyle L(\F_{t_{v}}) = p^{\Gamma}_{v}\cN(\Gamma)p^{\Gamma}_{v} \rightarrow p^{\Gamma}_{v}\cN_{1}p^{\Gamma}_{v}$
  is equivalent to the canonical inclusion
  $$
   L(\F_{t_{v}}) \rightarrow L(\F_{t_{v}}) * p^{\Gamma}_{v}\left[\left(\overset{p'_{v}}{\C} \oplus \underset{\frac{\gamma_{v} - \alpha_{v}}{\gamma_{v}}}{\C}\right) * \left( \bigoplus_{k=1}^{n} \overset{p_{v, k}}{\underset{\frac{\gamma_{w}}{n\gamma_{v}}}{\C}} \oplus \overset{p_{v}^{I}}{\underset{\frac{\gamma_{v} - \gamma_{w}}{\gamma_{v}}}{\C}}\right)\right]p^{\Gamma}_{v}
   $$
so $p^{\Gamma}_{v}\cN(\Gamma)p^{\Gamma}_{v}\overset{s.e.}{\hookrightarrow} p^{\Gamma}_{v}\cN_{1}p^{\Gamma}_{v}$.  From Remark \ref{rem:Dyk}, $p^{\Gamma}\cN(\Gamma)p^{\Gamma} \overset{s.e.}{\hookrightarrow} p^{\Gamma}\cN_{1}p^{\Gamma}$ as well.

  By Lemma \ref{lem:DR1} we have
$$
p_{w}\cN_{2}p_{w} = L(\Z/n\Z) * p_{w}\cN_{1}p_{w} = L(\Z/n\Z) * \left(\overset{p_{w}^{\cN_{1}}}{L(\F_{t_{w}})} \oplus \overset{r_{w}^{\cN_{1}}}{\underset{\gamma_{w}}{\C}}\right),
$$
 where $p_{w}^{\cN_{1}} = p_{w}p^{\cN_{1}}$ with $p^{\cN_{1}}$ the central support of $p_{\Gamma}$ in $\cN_{1}$ (note $p_{w}^{\cN_{1}} \geq p_{w}^{\Gamma}$ so $r_{w}^{\cN_{1}} \leq r_{w}^{\Gamma}$ which implies $\delta_{w} \leq \gamma_{w} - \alpha^{\Gamma}_{w}$).  From these observations, it follows that $p_{w}\cN_{2}p_{w}$ is an interpolated free group factor (since $n$ was chosen such that $\frac{\gamma_{w}}{n\gamma_{v}} + \frac{\gamma_{v} - \alpha_{v}}{\gamma_{v}} < 1$) and the arguments used in the inclusion $\cN(\Gamma) \rightarrow \cN_{1}$ imply $p_{w}^{\cN_{1}}\cN_{1}p_{w}^{\cN_{1}} \overset{s.e.}{\hookrightarrow} p_{w}^{\cN_{1}}\cN_{2}p_{w}^{\cN_{1}}$. Therefore $p_{w}^{\Gamma}\cN_{1}p_{w}^{\Gamma} \overset{s.e.}{\hookrightarrow} p_{w}^{\Gamma}\cN_{2}p_{w}^{\Gamma}$ so $p^{\Gamma}\cN_{1}p^{\Gamma} \overset{s.e.}{\hookrightarrow} p^{\Gamma}\cN_{2}p^{\Gamma}$.  Also, observe that since the projections $p_{v, k}$ and $p_{w, k}$ lie in the interpolated free group factor summand of $\cN_{2}$, they are equivalent in $\cN_{2}$. We now define algebras $\cN_{2, j}$ for $j = 0,...,n$ so that

$$
\cN_{2} = \cN_{2, 0} \subset \cN_{2, 1} \subset \cN_{2, 2} \subset \cdots \subset \cN_{2, n} = \cN_{3} \textrm{ where }
$$ $$
\cN_{2, j} = \left(\bigoplus_{k=j+1}^{n}  \overset{p_{w, k}}{\underset{\gamma_{w}/n}{\C}} \oplus \bigoplus_{k=1}^{j} \overset{p_{w, k}, p_{v, k}}{\underset{\gamma_{w}/n}{M_{2}(\C)}} \oplus \bigoplus_{k=j+1}^{n}\overset{p_{v, k}}{\underset{\gamma_{w}/n}{\C}} \oplus \overset{p_{v}^{I}}{\underset{\gamma_{v} - \gamma_{w}}{\C}} \oplus \ell^{\I}(\Gamma \setminus \{v, w\})\right) \underset{D}{*} \cN(\Gamma).
$$

Let $p^{\cN_{2}}$ be the central support of $p_\Gamma$ in $\cN_{2}$.  Applying Lemma \ref{lem:DR2} to the inclusion
$$
p_{w, j+1}\cN_{2, j}p_{w, j+1} \rightarrow p_{w, j+1}\cN_{2,j+1}p_{w, j+1} = p_{w, j+1}\cN_{2, j}p_{w, j+1} * L(\Z)
$$
shows that this inclusion is a standard embedding, so it follows from Remark \ref{rem:Dyk} that $p^{\cN_{2}}\cN_{2, j}p^{\cN_{2}} \overset{s.e.}{\hookrightarrow} p^{\cN_{2}}\cN_{2, j+1}p^{\cN_{2}}$, implying $p^{\Gamma}\cN_{2, j}p^{\Gamma} \overset{s.e.}{\hookrightarrow} p^{\Gamma}\cN_{2, j+1}p^{\Gamma}$ for all $j$.  Inductively,
$$
\cN_{3} = \left(\overset{p^{\cN_{2}}}{L(\F_{t_{3}})} \oplus \overset{p_{v}^{I}\wedge r_{v}^{\Gamma}}{\underset{\gamma_{v} - \alpha_{v}^{\Gamma'}}{\C}} \bigoplus_{u \in L(\Gamma)\setminus\{v, w\}} \underset{\gamma_{u} - \alpha^{\Gamma'}_{u}}{\overset{r^{\Gamma}_{u}}{\C}}\right)
$$
and $p^{\Gamma}\cN_{2}p^{\Gamma} \overset{s.e.}{\hookrightarrow} p^{\Gamma}\cN_{3}p^{\Gamma}$.  To finish, we look at the sequence of algebras
$$
\cN_{3} = \cN_{3, 0} \subset \cN_{3, 1} \subset ... \subset \cN_{3, n} = \cN(\Gamma') \textrm{ where }
$$ $$\cN_{3, j} = \left(\bigoplus_{k=1}^{j} \overset{p_{w, k} + q_{w, k}}{M_{2}(\C) \otimes L(\Z)} \oplus \bigoplus_{k=j+1}^{n} \overset{p_{w, k}, p_{v, k}}{M_{2}(\C)} \oplus \overset{p_{v}^{I}}{\underset{\gamma_{v} - \gamma_{w}}{\C}} \bigoplus \ell^{\infty}(\Gamma \setminus \{v, w\})\right) \underset{D}{*} \cN(\Gamma).
$$
Lemma \ref{lem:DR1} implies that the inclusion
$$
p_{w, j+1}\cN_{3, j}p_{w, j+1} \rightarrow p_{w, j+1}\cN_{3, j+1}p_{w, j+1} = p_{w, j+1}\cN_{3, j}p_{w, j+1} * L(\Z)
$$
is a standard embedding, so by Remark \ref{rem:Dyk}, $p^{\cN_{2}}\cN_{3, j}p^{\cN_{2}} \overset{s.e.}{\hookrightarrow} p^{\cN_{2}}\cN_{3, j+1}p^{\cN_{2}}$ and thus $p^{\Gamma}\cN_{3, j}p^{\Gamma} \overset{s.e.}{\hookrightarrow} p^{\Gamma}\cN_{3, j+1}p^{\Gamma}$.  Therefore the inclusion $p^{\Gamma}\cN_{3}p^{\Gamma} \rightarrow p^{\Gamma}\cN(\Gamma')p^{\Gamma}$ is standard since it is a composite of standard embeddings.  This implies $\cN(\Gamma')$ has the desired formula and $p^{\Gamma}\cN(\Gamma)p^{\Gamma} \overset{s.e.}{\hookrightarrow} p^{\Gamma}\cN(\Gamma')p^{\Gamma}$.\end{proof}
We again assume that $\cN(\Gamma)$ is in the form of Theorem \ref{thm:H1}.

\begin{lem} \label{lem:H3}

 Let $\Gamma'$ be a weighted graph obtained from $\Gamma$ by adding a vertex $v$ and an edge $e$ connecting $v$ to $w \in V(\Gamma)$ woth weighting $\gamma_{v}$, and assume $\cN(\Gamma) = \overset{p^{\Gamma}}{L(\F_{t_{\Gamma}})} \oplus \underset{{v \in B(\Gamma)}}{\bigoplus} \overset{r_{v}^{\Gamma}}{\underset{\gamma_{v} - \alpha^{\Gamma}_{v}}{\C}}$ with notation as in Theorem \ref{thm:H1}.  Then
 $$
 \cN(\Gamma') = \overset{p^{\Gamma'}}{L(\F_{t_{\Gamma'}})} \oplus \underset{{u \in B(\Gamma')}}{\bigoplus} \overset{r_{u}^{\Gamma'}}{\underset{\gamma_{u} - \alpha^{\Gamma'}_{u}}{\C}}
 $$ where $p^{\Gamma} \leq p^{\Gamma'}$, $r_{u}^{\Gamma'} \leq r_{u}^{\Gamma}$ for all $u$, and $p_{\Gamma}\cN(\Gamma)p_{\Gamma} \overset{s.e.}{\hookrightarrow} p_{\Gamma}\cN(\Gamma')p_{\Gamma}$.
\end{lem}

Notice that the natural inclusion $\cN(\Gamma) \rightarrow \cN(\Gamma')$ is not unital, but the compressed inclusion $p_{\Gamma}\cN(\Gamma)p_{\Gamma} \rightarrow p_{\Gamma}\cN(\Gamma')p_{\Gamma}$ is.

\begin{proof}

  Just as in the proof of Lemma \ref{lem:H2}, if the term $\overset{p}{\underset{\alpha}{\C}}$ appears where $\alpha \leq 0$ then this term is identically zero.

   Set $D = \ell^{\infty}(\Gamma')$.  We rescale all of the weights on $\Gamma$ such that all of the weights on $\Gamma'$ sum to 1.  We have 2 cases: when $\gamma_{v} > \gamma_{w}$ and when $\gamma_{w} \geq \gamma_{v}$.

  \emph{\underline{Case 1, $\gamma_{v} > \gamma_{w}$}}:  We look at the following sequence of inclusions:

   \begin{align*}\cN(\Gamma) \oplus \overset{p_{v}}{\underset{\gamma_{v}}{\C}} \subset \cN_{1} &= \left(\cN(\Gamma) \oplus \overset{p_{v}}{\underset{\gamma_{v}}{\C}}\right) \underset{D}{*} \left(\ell^{\infty}(\Gamma' \setminus \{v, w\}) \oplus \bigoplus_{k=1}^{n} \overset{p_{w, k}}{\underset{\gamma_{w}/n}{\C}} \oplus \left(\bigoplus_{k=1}^{n} \overset{p_{v, k}}{\underset{\gamma_{w}/n}{\C}} \oplus \overset{p_{v}^{I}}{\underset{\gamma_{v} - \gamma_{w}}{\C}}\right)\right)\\
   \cap &\\
  \cN_{2} &= \left(\cN(\Gamma) \oplus \overset{p_{v}}{\underset{\gamma_{v}}{\C}}\right) \underset{D}{*} \left(\ell^{\infty}(\Gamma' \setminus \{v, w\}) \oplus \bigoplus_{k=1}^{n}\overset{p_{w, k}, p_{v, k}}{\underset{\gamma_{w}/n}{M_{2}(\C)}} \oplus \overset{p_{v}^{I}}{\underset{\gamma_{v} - \gamma_{w}}{\C}}\right)\\
  \cap &\\
    \cN(\Gamma') &= \left(\cN(\Gamma) \oplus \overset{p_{v}}{\underset{\gamma_{v}}{\C}}\right) \underset{D}{*} \left(\ell^{\infty}(\Gamma' \setminus \{v, w\}) \oplus \underset{2\gamma_{w}}{L(\Z) \otimes M_{2}(\C)} \oplus \overset{p_{v}^{I}}{\underset{\gamma_{v} - \gamma_{w}}{\C}}\right).
    \end{align*}
    The projections $p_{w, k}$ are an orthogonal family with trace $\gamma_{w}/n$ in $\cA_{e}$ whose sum is $p_{w}$.  In $\cA_{e}$, $p_{v}$ decomposes as $\sum_{k=1}^{n}p_{v, k} + p_{v}^{I}$ with  $p_{v}^{I}$ supported in the atomic part of $\cA_{e}$, and the $p_{v, k}$ are an orthogonal family of projections with trace $\gamma_{w}/n$ supported in the diffuse part of $\cA_{e}$.
   By the inductive hypothesis,
    $$
    p_{w}\cN(\Gamma)p_{w} = \overset{p^{\Gamma}_{w}}{\underset{\frac{\alpha^{\Gamma}_{w}}{\gamma_{w}}}{L(\F_{t_{w}})}} \oplus \overset{r_{w}^{\Gamma}}{\underset{\frac{\gamma_{w} - \alpha^{\Gamma}_{w}}{\gamma_{w}}}{\C}},
    $$
     with $p_{w}^{\Gamma} = p_{w}p^{\Gamma}$.  We choose $n$ large enough such that $\frac{1}{n} + \frac{\gamma_{w} - \alpha_{w}}{\gamma_{w}} < 1$, i.e., so that $p_{w}\cN(\Gamma)p_{w} * L(\Z / n\Z)$ is an interpolated free group factor.
  From Lemma \ref{lem:DR1},
  $$
  p_{w}\cN_{1}p_{w} = p_{w}\cN(\Gamma)p_{w} * \left(\bigoplus_{k=1}^{n} \overset{p_{w, k}}{\underset{1/n}{\C}}\right) = \left(\overset{p'_{w}}{\underset{\frac{\alpha_{w}}{\gamma_{w}}}{L(\F_{t_{w}})}} \oplus \underset{\frac{\gamma_{w} - \alpha_{w}}{\gamma_{w}}}{\C}\right) * \left(\bigoplus_{k=1}^{n} \overset{p_{w, k}}{\underset{1/n}{\C}}\right),
  $$
   so it is an interpolated free group factor, and applying Lemma $\ref{lem:DR1}$ again, we see that
   $$
   p^{\Gamma}_{w}\cN_{1}p^{\Gamma}_{w} = p^{\Gamma}_{w}\cN(\Gamma)p^{\Gamma}_{w} * p^{\Gamma}_{w}\left[\left(\overset{p^{\Gamma}_{w}}{\underset{\frac{\alpha_{w}}{\gamma_{w}}}{\C}} \oplus \overset{r_{w}^{\Gamma}}{\underset{\frac{\gamma_{w} - \alpha_{w}}{\gamma_{w}}}{\C}}\right) * \left(\bigoplus_{k=1}^{n} \overset{p_{w, k}}{\underset{1/n}{\C}}\right)\right]p^{\Gamma}_{w}
   $$
   with the inclusion $p^{\Gamma}_{w}\cN(\Gamma)p^{\Gamma}_{w} \rightarrow p'_{w}\cN_{1}p'_{w}$ the canonical one.  Therefore $\displaystyle p^{\Gamma}_{w}\cN(\Gamma)p^{\Gamma}_{w} \overset{s.e.}{\hookrightarrow} p^{\Gamma}_{w}\cN_{1}p^{\Gamma}_{w}$, so it follows that $p^{\Gamma}\cN(\Gamma)p^{\Gamma} \overset{s.e.}{\hookrightarrow} p^{\Gamma}\cN_{1}p^{\Gamma}$ as well.  It is clear that $p_{v}^{I}$ will be a minimal central projection in $\cN_{2}$, and since the projections $p_{v, k}$ lie under the minimal projection $p_{v} \in \cN(\Gamma) \oplus \overset{p_{v}}{\underset{\gamma_{v}}{\C}}$, it follows that
   $$
   \cN_{2} = L(\F_{t_{2}}) \oplus \overset{p_{v}^{I}}{\underset{\gamma_{v} - \gamma_{w}}{\C}} \oplus \bigoplus_{u \in B(\Gamma)\setminus\{w\}} \underset{\gamma_{u} - \alpha^{\Gamma}_{u}}{\overset{r^{\Gamma}_{u}}{\C}},
   $$
   where $L(\F_{t_{2}})$ is an amplification of $p^{\Gamma}\cN_{1}p^{\Gamma}$.  Hence  $p^{\Gamma}\cN_{1}p^{\Gamma} = p^{\Gamma}\cN_{2}p^{\Gamma}$.  As a final step, we tensor each copy of $M_{2}(\C)$ with $L(\Z)$ to obtain $\cN(\Gamma')$ and apply Lemma \ref{lem:DR1} and Remark \ref{rem:Dyk} $n$ times as in the proof of Lemma \ref{lem:H2} to conclude that
   $$
   \cN(\Gamma') = L(\F_{t_{3}}) \oplus \overset{p_{v}^{I}}{\underset{\gamma_{v} - \gamma_{w}}{\C}} \oplus \bigoplus_{u \in B(\Gamma)\setminus\{w\}} \underset{\gamma_{u} - \alpha^{\Gamma'}_{u}}{\overset{r^{\Gamma}_{u}}{\C}}
   $$
   and $p^{\Gamma}\cN_{2}p^{\Gamma} \overset{s.e.}{\hookrightarrow} p^{\Gamma}\cN(\Gamma')p^{\Gamma}$.  Therefore $p^{\Gamma}\cN(\Gamma)p^{\Gamma} \overset{s.e.}{\hookrightarrow} p^{\Gamma}\cN(\Gamma')p^{\Gamma}$ and $\cN(\Gamma')$ has the desired form.

       \emph{\underline{Case 2, $\gamma_{w} \geq \gamma_{v}$}}:  We look at a sequence of inclusions similar to those in the previous case:
    \begin{align*}
     \cN(\Gamma) \oplus \overset{p_{v}}{\underset{\gamma_{v}}{\C}} \subset \cN_{1} &= \left(\cN(\Gamma) \oplus \overset{p_{v}}{\underset{\gamma_{v}}{\C}}\right) \underset{D}{*} \left(\ell^{\infty}(\Gamma' \setminus \{v, w\}) \oplus \left(\bigoplus_{k=1}^{n} \overset{p_{w, k}}{\underset{\gamma_{v}/n}{\C}} \oplus \overset{p_{w}^{I}}{\underset{\gamma_{w} - \gamma_{v}}{\C}}\right) \oplus \bigoplus_{k=1}^{n} \overset{p_{v, k}}{\underset{\gamma_{v}/n}{\C}}\right)\\
    \cap &\\
    \cN_{2} &= \left(\cN(\Gamma) \oplus \overset{p_{v}}{\underset{\gamma_{v}}{\C}}\right) \underset{D}{*} \left(\ell^{\infty}(\Gamma' \setminus \{v, w\})\oplus \bigoplus_{k=1}^{n} \overset{p_{w, k}, p_{v, k}}{\underset{\gamma_{v}/n}{M_{2}(\C)}} \oplus \overset{p_{w}^{I}}{\underset{\gamma_{w} - \gamma_{v}}{\C}}\right)\\
    \cap &\\
    \cN(\Gamma') &= \left(\cN(\Gamma) \oplus \overset{p_{v}}{\underset{\gamma_{v}}{\C}}\right) \underset{D}{*} \left(\ell^{\infty}(\Gamma' \setminus \{v, w\})\oplus \underset{2\gamma_{v}}{L(\Z)\otimes M_{2}(\C)} \oplus \overset{p_{w}^{I}}{\underset{\gamma_{w} - \gamma_{v}}{\C}}\right).
    \end{align*}
  The projections $p_{v, k}$ are an orthogonal family with trace $\gamma_{v}/n$ in $\cA_{e}$ whose sum is $p_{v}$.  In $\cA_{e}$, $p_{w}$ decomposes as $\sum_{k=1}^{n}p_{w, k} + p_{w}^{I}$ where $p_{w}^{I}$ is supported in the atomic part of $\cA_{e}$, and the $p_{w, k}$ are an orthogonal family of projections with trace $\gamma_{v}/n$ supported in the diffuse part of $\cA_{e}$.  We choose $n$ large enough so that $\frac{\gamma_{w} - \alpha_{w}}{\gamma_{w}} + \frac{\gamma_{v}}{n\gamma_{w}} < 1$.   Observe by the condition on $n$ that $p_{w}\cN_{1}p_{w} = \overset{p^{\cN_{1}}_{w}}{L(\F_{t'_{1}})} \oplus \underset{\gamma_{w} - \alpha_{w} - \gamma_{v}}{\C}$ where the copy of $\C$ is orthogonal to each $p_{w, k}$.  Therefore as in the proof of Lemma \ref{lem:H2} $p_{\Gamma}\cN(\Gamma)p_{\Gamma} \overset{s.e.}{\hookrightarrow} p_{\Gamma}\cN_{1}p_\Gamma$.  We next look at
   $$
   \cN_{1} \subset \cN_{2} = \left(\cN(\Gamma) \oplus \overset{p_{v}}{\underset{\gamma_{v}}{\C}}\right) \underset{D}{*} \left(\ell^{\infty}(\Gamma' \setminus \{v, w\})\oplus \bigoplus_{k=1}^{n} \overset{p_{w, k}, p_{v, k}}{\underset{\gamma_{v}/n}{M_{2}(\C)}} \oplus \overset{p_{v}^{I}}{\underset{\gamma_{w} - \gamma_{v}}{\C}}\right).
   $$
    Since the $p_{v,k}$ lie under the minimal central projection $p_{v} \in \cN(\Gamma) \oplus \overset{p_{v}}{\underset{\gamma_{v}}{\C}}$, the arguments above imply
    $$
    \cN_{2} = L(\F_{t_{2}}) \oplus \underset{\gamma_{w} - \gamma_{v} - \alpha_{w}}{\C} \oplus \bigoplus_{u \in L(\Gamma)\setminus\{w\}} \underset{\gamma_{u} - \alpha_{u}}{\overset{r^{\Gamma}_{u}}{\C}}
    $$
    and $p^{\Gamma}\cN_{1}p^{\Gamma} = p^{\Gamma}\cN_{2}p^{\Gamma}$.  To finish, we tensor each copy of $M_{2}(\C)$ with $L(\Z)$ and apply Lemma \ref{lem:DR1} and Remark \ref{rem:Dyk} $n$ times as in the end of the proof of Lemma \ref{lem:H2} to obtain
    $$
    \cN(\Gamma') = L(\F_{t_{3}}) \oplus \underset{\gamma_{w} - \gamma_{v} - \alpha_{w}}{\C} \oplus \bigoplus_{u \in L(\Gamma)\setminus\{w\}} \underset{\gamma_{u} - \alpha_{u}}{\overset{r^{\Gamma}_{u}}{\C}}
    $$
     with the inclusion $p^{\Gamma}\cN(\Gamma)p^{\Gamma} \rightarrow p^{\Gamma}\cN(\Gamma')p^{\Gamma}$ standard. \end{proof}

\begin{proof}[Proof of Theorem \ref{thm:H1}]

Note that if $\Gamma'$ and $\Gamma$ are connected, loopless, finite graphs, then $\Gamma'$ can be constructed form $\Gamma$ by considering the steps in Lemmas $\ref{lem:H2}$ and $\ref{lem:H3}$.  Therefore, we can deduce Theorem $\ref{thm:H1}$ by observing that the composite of standard embeddings is a standard embedding and that standard embeddings are preserved by cut-downs by projections. \end{proof}


\section{The GJS construction in infinite depth} \label{sec:GJSInfinite}

    Recall that the vertices on a principal graph for $M_{0} \subset M_{1}$ represent isomophism classes of irreducible $M_{0}-M_{0}$ and $M_{0}-M_{1}$ subbimodules  of tensor products of $X = _{M_{0}}L^{2}(M_{1})_{M_{1}}$ and its dual, $X^{*} = _{M_{1}}L^{2}(M_{1})_{M_{0}}$.  Assume $\Gamma$ is the principal graph for an infinite-depth subfactor.  If $*$ is the depth-0 vertex of $\Gamma$, then the factor $M_{0}$ as in the introduction is isomorphic to $p_{*}\cN(\Gamma)p_{*}$.    $\cN(\Gamma)$ is now a semifinite algebra where the weighting $\gamma$ on $\ell^{\I}(\Gamma)$ corresponds to the bimodule dimension obtained by identifying each vertex with an irreducible bimodule as above.  Under this identification, $\gamma_* = 1$ and $\delta\cdot\gamma_{v} = \sum_{w \sim v}n_{v, w}\gamma_{w}$ where $\delta = [M_{1}:M_{0}]^{1/2}$.  To circumvent the difficulty of dealing with a semifinite algebra, we realize that $M_{0}$ is an inductive limit of the algebras $p_{*}\cN(\Gamma_{k})p_{*}$ where $\Gamma_{k}$ is $\Gamma$ truncated at depth $k$.  To aid our computation of the isomorphism class of $M_{0}$, we have the following lemma, whose proof is a routine calculation and is identical to that in $\cite{MR2807103}$.

\begin{lem} \label{lem:GJS2}

The free dimension of $\cN(\Gamma_{k})$ is
$$
1 + \frac{1}{\Tr(F_{k})^{2}}\left( -\sum_{v \in \Gamma_{k}} \gamma_{v}^{2} + \sum_{v \in \Gamma_{k}}\sum_{w\sim v} n_{v, w}\gamma_{v}\gamma_{w}\right)
$$  where $w \sim v$ means $w$ is connected to $v$ in $\Gamma_{k}$, $F_{k} = \sum_{u \in \Gamma_{k}} p_{u}$, and $\Tr$ is the trace on the semifinite algebra $\cN(\Gamma)$.

\end{lem}

\begin{thm} \label{thm:H2}

Let $\cP$ be an infinite depth subfactor planar algebra.  Then the factor $M_{0}$ in the construction of \cite{MR2732052} is isomorphic to $L(\F_{\I})$.

\end{thm}

\begin{proof}

For a given $k$, we write
$$
\cN(\Gamma_{k}) = \overset{p_{k}}{\underset{\underset{w \not\in B(\Gamma_{k})}{\sum\gamma_{w}} + \underset{v \in B(\Gamma_{k})}{\sum \alpha^{\Gamma_{k}}_{v}}}{L(\F_{t_{k}})}} \oplus \bigoplus_{v \in B(\Gamma_{k})} \overset{p^{\Gamma_{k}}_{v}}{\underset{\gamma_{v} - \alpha^{\Gamma_{k}}_{v}}{\C}}.
$$
The free dimension of this algebra is
$$
1 + (t_{k}-1)\left(\frac{\sum_{w \not\in B(\Gamma_{k})}\gamma_{w} + \sum_{v \in B(\Gamma_{k})}\gamma_{v}}{\Tr(F_{k})}\right)^{2} - \frac{\sum_{v\in B(\Gamma_{k})}(\gamma_{v} - \alpha^{\Gamma_{k}}_{v})^{2}}{\Tr(F_{k})^{2}},
$$
so by Lemma \ref{lem:GJS2}, we have the equation
$$
(t_{k}-1)\left(\sum_{w \not\in B(\Gamma_{k})}\gamma_{w} + \sum_{v \in B(\Gamma_{k})}\alpha^{\Gamma_{k}}_{v}\right)^{2} =  \sum_{u \in \Gamma_{k}}\sum_{w\sim u} n_{u, w}\gamma_{u}\gamma_{w} -\sum_{u \in \Gamma_{k}}\gamma_{u}^{2} + \sum_{v\in B(\Gamma_{k})}(\gamma_{v} - \alpha^{\Gamma_{k}}_{v})^{2}.
$$
Observe that in $\Gamma_{k}$, the vertices up to depth $k-1$ are connected to all of their neighbors in $\Gamma$, so by the Perron-Frobenius condition and the fact that $\delta > 1$, none of these vertices are in $B(\Gamma_{k})$.  If we let $B'(\Gamma_{k})$ be the vertices $v$ at depth $k$ with $\gamma_{v} \leq \sum_{w\sim v}n_{v, w}\gamma_{w}$, then the right hand side of the above equality becomes

\begin{align*}
(\delta - 1)\sum_{v \in \Gamma_{k-2}}\gamma_{v}^{2} &+ \sum_{v \in \Gamma_{k-1}\setminus\Gamma_{k-2}}\gamma_{v} \left(-\gamma_{v} + \sum_{\substack{ w\not\in B(\Gamma_{k}) \\ w \sim{v}}} n_{v, w}\gamma_{w} + \sum_{\substack{w \in B(\Gamma_{k}) \\ w \sim v}} \alpha^{\Gamma_{k}}_{w} \right) \\  &+ \sum_{v \in B'(\Gamma_{k})} \gamma_{v}\left(-\gamma_{v} + \sum_{w \sim v}n_{v,w}\gamma_{w}\right)
\end{align*}
where we have used $\alpha^{\Gamma_{k}}_{v}  = \sum_{w\sim v}n_{v, w}\gamma_{w}$.  This quantity majorizes $(\delta - 1)\sum_{v \in \Gamma_{k-2}}\gamma_{v}^{2}$. Since the bimodule dimensions of any irreducible sumbimodule of $(X \otimes _{M_{1}} X^{*})^{\otimes_{M_{0}}^{n}}$ and $(X \otimes _{M_{1}} X^{*})^{\otimes_{M_{0}}^{n}} \otimes_{M_{0}} X$ are bounded below by 1, $\gamma_{v} \geq 1$ for all $v \in \Gamma$ so we conclude that
$$
(t_{k}-1)\left(\sum_{w \not\in L(\Gamma_{k})}\gamma_{w} + \sum_{v \in L(\Gamma_{k})}\alpha^{\Gamma_{k}}_{v}\right)^{2} \rightarrow \infty
$$
as $k \rightarrow \infty$.  From the amplification formula, $p_{*}\cN(\Gamma_{k})p_{*} = L(\F_{t'_{k}})$ where
$$
t'_{k} = 1 + (t_{k}-1)\left(\sum_{w \not\in L(\Gamma_{k})}\gamma_{w} + \sum_{v \in L(\Gamma_{k})}\alpha^{\Gamma_{k}}_{v}\right)^{2}.
$$  Hence $p_{\Gamma_{k}}\cN(\Gamma_{k})p_{\Gamma_{k}} \overset{s.e.}{\hookrightarrow} p_{\Gamma_{k}}\cN(\Gamma_{k+1})p_{\Gamma_{k}}$ so by Remark \ref{rem:Dyk}, $p_{*}\cN(\Gamma_{k})p_{*} \overset{s.e.}{\hookrightarrow} p_{*}\cN(\Gamma_{k+1})p_{*}$.  As $p_{*}\cN(\Gamma)p_{*}$ is the inductive limit of the $p_{*}\cN(\Gamma_{k})p_{*}$, it follows that $p_{*}\cN(\Gamma)p_{*} = L(\F_{t})$ where $t = \lim t'_{k} = \infty$. \end{proof}

\begin{cor} The factors $M_{k}$ are isomorphic to $L(\F_{\I})$.

\end{cor}

\begin{proof}  If $k$ is even, then $M_{k}$ is an amplification of $M_{0}$ so it follows for $M_{k}$.  If $k$ is odd, then $M_{k}$ are cut-downs of $\cN(\Gamma')$ with $\Gamma'$ the dual principal graph of $\cP$.  Applying the same analysis as in Theorem \ref{thm:H2} shows that $M_{k} \cong L(F_{\infty})$. \end{proof} 

%% file: Thesis2013Chapter-Universal.tex
\chapter{Universality of $L(\F_{\I})$ for Rigid $C^{*}-$Tensor Categories}\label{chap:universal}

\section{Introduction}

Jones initiated the modern theory of subfactors in his breakthrough paper \cite{MR696688} in which he classified the possible values for the index of a $II_1$ subfactor to the range $\set{4\cos^2(\pi/n)}{n\geq 3}\cup [4,\infty]$, and he found a subfactor of the hyperfinite $II_1$-factor $R$ for each allowed index.

A finite index subfactor $N\subset M$ is studied by analyzing its standard invariant, i.e., two rigid $C^*$-tensor categories of $N-N$ and $M-M$ bimodules and the module categories of $N-M$ and $M-N$ bimodules which arise from the Jones tower. The standard invariant has been axiomatized in three similar ways, each emphasizing slightly different structure: Ocneanu's paragroups \cite{MR996454,MR1642584}, Popa's $\lambda$-lattices \cite{MR1334479}, and Jones' planar algebras \cite{math.QA/9909027}.

In \cite{MR1198815,MR1334479,MR1887878}, Popa starts with a $\lambda$-lattice $A_{\bullet,\bullet}=(A_{i,j})$ and constructs a $II_1$-subfactor whose standard invariant is $A_{\bullet,\bullet}$. Hence for each subfactor planar algebra $P_\bullet$, there is some subfactor whose planar algebra is $P_\bullet$. However, the following question remains unanswered:

\begin{quest}\label{quest:WhichP}
For which subfactor planar algebras $\cP$ is there a subfactor of $R$ whose planar algebra is $\cP$?
\end{quest}

Using his reconstruction theorems, Popa gave a positive answer to Question \ref{quest:WhichP} for (strongly) amenable subfactor planar algebras \cite{MR1278111}.

In \cite{MR2051399}, Popa and Shlyakhtenko were able to identify the factors in certain cases of Popa's reconstruction theorems. Using this, they gave a positive answer to Question \ref{quest:WhichP} for $L(\F_\infty)$, i.e., every subfactor planar algebra arises as the standard invariant of some subfactor $N\subset M$ such that $N,M$ are both isomorphic to $L(\F_\infty)$. This theorem was reproduced by Hartglass \cite{1208.2933} using the reconstruction results of Guionnet-Jones-Shlyakhtenko-Walker (GJSW) \cite{MR2732052,MR2645882,MR2807103} which produce subfactors of interpolated free group factors.

It is natural to extend these questions to rigid $C^*$-tensor categories, i.e.,
\begin{quest}\label{quest:WhichC}
For which rigid $C^*$-tensor categories $\cC$ is there a category $\cC_{bim}$ of bifnite bimodules over $R$ such that $\cC_{bim}$ is equivalent to $\cC$?
\end{quest}

As in the subfactor case, Hayashi and Yamagami gave a positive result for amenable rigid $C^*$-tensor categories \cite{MR1749868} (amenability for $C^*$-tensor categories was first studied by Hiai and Izumi \cite{MR1644299}). Moreover, given a rigid $C^*$-tensor category $\cC$, Yamagami constructed a category of bifite bimodules $\cC_{bim}$ over an amalgamated free product $II_1$-factor such that $\cC_{bim}$ is equivalent to $\cC$ \cite{MR1960417}. However, one can show these factors have property $\Gamma$, so they are not interpolated free group factors (we briefly sketch this in Appendix \ref{sec:Appendix}).

In this paper, we give a result analogous to Popa and Shlyakhtenko's results for $L(\F_\infty)$ for countably generated rigid $C^*$-tensor categories, which answers part of Question 9 in \cite[Section 6]{MR2681261}. Recall that a rigid $C^*$-tensor category $\cC$ is generated by a set of objects $\cS$ if for every $Y\in\cC$, there are $X_1,\dots, X_n\in \cS$ such that
$$
\cC(X_1\otimes\cdots \otimes X_n,Y)\neq (0),
$$
i.e., $Y$ is (isomorphic to) a sub-object of $X_1\otimes\cdots \otimes X_n$.

\begin{thm}\label{thm:Main}
Every countably generated rigid $C^*$-tensor category can be realized as a category of bifinite bimodules over $L(\F_\infty)$.
\end{thm}

\begin{rem}
Note that when $\cC$ is finitely generated, we can prove Theorem \ref{thm:Main} using \cite{MR2051399} by adapting the technique in \cite[Theorem 4.1]{1112.4088}. We provide a sketch of the proof in Appendix \ref{sec:Appendix}, where we also point out some difficulties of using the results of \cite{MR2051399} when $\cC$ is not finitely generated  (see also \cite[Section 4]{MR1960417}).

Hence we choose to use planar algebra technology to prove Theorem \ref{thm:Main} since it offers the following advantages. First, the same construction works for both the finitely and infinitely generated cases. Second, planar diagrams arise naturally in the study of tensor categories, and a reader familiar with the diagrams may benefit from a planar algebraic approach.
Third, we get an elegant description of the bimodules over $L(\F_\I)$ directly from the planar algebra (see Section \ref{sec:MoreBimodules}).
\end{rem}

There are three steps to the proof of Theorem \ref{thm:Main}.
\begin{enumerate}[(1)]
\item
Given a countably generated $C^*$-tensor category $\cC$, we get a factor planar algebra $\cP$ such that the $C^*$-tensor category $\Pro$ of projections of $\cP$ is equivalent to $\cC$.

A \underline{factor planar algebra} (called a fantastic planar algebra in \cite{1208.3637}) is an unshaded, spherical, evaluable $C^*$-planar algebra.

This step is well known to experts; we give most of the details in Section \ref{sec:TCandPA}.

\item
Given a factor planar algebra $\cP$, we construct a $II_1$-factor $M$ and two rigid $C^*$-tensor categories of bifinite bimodules over $M$:
\begin{itemize}
\item
$\Bim$, built entirely from $\cP$ and obviously equivalent to $\Pro$, and
\item
$\CF$, formed using Connes' fusion and linear operators.
\end{itemize}
These categories are defined in Definitions \ref{defn:Bim} and \ref{defn:CF}. We then show $\Bim\simeq \CF$ in Theorem \ref{thm:BimCF}.

We use results of GJSW to accomplish this step in Section \ref{sec:GJS}. Along the way, we adapt Brothier's treatment \cite{1202.1298} of GJSW results \cite{MR2732052,MR2645882} for unshaded planar algebras.

\item
We show $M\cong L(\F_\I)$.

This last step is similar to results of GJS \cite{MR2807103} and Chapter \ref{chap:graph}.
\end{enumerate}

One can use similar analysis as in the proof of Theorem \ref{thm:Main} to prove the following theorem:

\begin{thm}\label{thm:Finite}
Suppose that in addition, $\cC$ has finitely many isomorphism classes of simple objects, i.e., $\cC$ is a unitary fusion category. Picking an object $X\in\cC$ which generates $\cC$, then $\cC$ can be realized as a category of bifinite bimodules over $L(\F_t)$ with
$$
t=1+\dim(\cC)(\dim(X\oplus\overline{X})-1)=1+\dim(\cC)(2\dim(X)-1),
$$
where $\dim(\cC)$ is the Frobenius-Perron dimension of $\cC$.
\end{thm}

\begin{rem}

Note that a version of Theorem \ref{thm:Finite} can be obtained by using \cite[Theorem 4.1]{1112.4088} together with \cite{MR2807103}.  If $Z_{1},\dots, Z_{n}$ are representatives for the simple objects in $\cC$ and $Y = \bigoplus_{k=1}^{n} Z_{k}$, then one obtains the factor $L(\F_s)$ for
$$
s = 1+\dim(\cC)(\dim(Y)-1),
$$
which will be a different parameter than what we obtained in Theorem \ref{thm:Finite}.
\end{rem}

On the other end of the spectrum, one should also note that there has been interesting work on rigid $C^*$-tensor categories of bimodules over $II_1$-factors by Vaes, Falgui\`{e}res, and Raum \cite{0811.1764,1112.4088}. Given a rigid $C^*$-tensor category $\cC$ which is either $\Rep(G)$ for $G$ a compact quantum group \cite{0811.1764} or a unitary fusion category \cite{1112.4088}, they construct a $II_1$-factor $M$ whose category $\Bim(M)$ of bifinite bimodules is \underline{exactly} $\cC$ (up to equivalence). Their results can be interpreted as rigidity results in contrast to the universality of $L(\F_\infty)$ to rigid $C^*$-tensor categories.

\paragraph{Acknowledgements.} We would like to thank Vaughan Jones, Scott Morrison, Noah Snyder, and Stefaan Vaes for many helpful conversations. The majority of this work was completed at the 2012 NCGOA on Conformal field theory and von Neumann algebras at Vanderbilt University and the 2012 Subfactors in Maui conference. The authors would like to thank Dietmar Bisch, Vaughan Jones, James Tener, and the other organizers for those opportunities. The authors were supported by DOD-DARPA grants HR0011-11-1-0001 and HR0011-12-1-0009. Michael Hartglass and David Penneys were also supported by NSF Grant DMS-0856316.  Arnaud Brothier was also supported by ERC Starting Grant VNALG-200749.

\section{Tensor categories and planar algebras}\label{sec:TCandPA}

We briefly recall how to go back and forth between rigid $C^*$-tensor categories and factor planar algebras.
The contents of this subsection are well known to experts.
Our treatment follows \cite{MR2559686,JonesPAnotes,1207.1923,MR2811311,MR1960417, MR2681261}.

\begin{nota}
Categories will be denoted with the letter $\cC$ unless otherwise specified. We write $X\in\cC$ to mean $X$ is an object in $\cC$, and we write $\cC(X\to Y)$ or $\cC(X,Y)$ for the set of morphisms from $X$ to $Y$ in $\cC$.
\end{nota}

\subsection{Factor planar algebras and principal graphs}

We briefly recall the definition of an unshaded factor planar algebra where the strings are labelled.
For more details, see \cite{JonesPAnotes}.

\begin{defn}
Given a set $\cL$, the \underline{planar operad with string labels $\cL$}, denoted $\P_{\cL}$, is the set of all planar tangles whose strings are labelled by elements of $\cL$, e.g.,
$$
\begin{tikzpicture}[baseline = .5cm]
	\draw[thick] (-2,-.8)--(-2,2)--(.8,2)--(.8,-.8)--(-2,-.8);
	\draw (.8,1.4) arc (-90:-180:.6cm);
	\draw (0,.4) arc (0:90:.8cm);
	\draw (-.4,-.2) arc (270:90:.2cm);
	\draw (-1,2)--(-1,1)--(-2,1);
	\draw (-1.4,2)--(-1.4,1.4)--(-2,1.4);
	\draw (0,-.8)--(0,0)--(.8,0);
	\draw (-1.4,0) circle (.3cm);
	\filldraw[unshaded,thick] (-1.6,.8)--(-1.6,1.6)--(-.8,1.6)--(-.8,.8)--(-1.6,.8);
	\filldraw[unshaded,thick] (-.4,.4)--(.4,.4)--(.4,-.4)--(-.4,-.4)--(-.4,.4);
	\node at (-1.6,-.4) {{\scriptsize{$a$}}};
	\node at (-1.6,1.8) {{\scriptsize{$b$}}};
	\node at (-1.8,1.6) {{\scriptsize{$b$}}};
	\node at (-.8,1.8) {{\scriptsize{$c$}}};
	\node at (-1.8,.8) {{\scriptsize{$c$}}};
	\node at (.4,1.4) {{\scriptsize{$a$}}};
	\node at (0,1) {{\scriptsize{$d$}}};
	\node at (-.6,.4) {{\scriptsize{$d$}}};
	\node at (.6,.2) {{\scriptsize{$b$}}};
	\node at (.2,-.6) {{\scriptsize{$c$}}};
	\node at (-.5,-.5) {$\star$};
	\node at (-1.6,.6) {$\star$};
	\node at (-2.2,-.6) {$\star$};
\end{tikzpicture}
\text{ where }a,b,c,d\in\cL.
$$
A \underline{planar tangle} $T$ consists of the following data:
\begin{enumerate}[(1)]
\item
a rectangle $D_0(T)\subset \R^2$,
\item
Finitely many disjoint rectangles $D_1(T),\dots ,D_s(T)$ in the interior of $D_0(T)$ ($s$ may be zero),
\item
Finitely many disjoint smooth arcs in $D_0(T)$ called the strings of $T$ which do not meet the interior of any $D_1(T),\dots, D_s(T)$.
The boundary points of a string of $T$ (if it has any) lie in the boundaries of the $D_i(T)$, and they meet these boundaries transversally (if they meet the boundaries at all).
\end{enumerate}
The boundaries of the strings divide the boundaries of the rectangles into \underline{intervals}. For each rectangle, there is a distinguished interval denoted $\star$. The intervals of $D_i(T)$ are divided by the \underline{marked points} of $D_i(T)$, i.e., the points at which the strings meet the boundary of $D_i(T)$. Starting at $\star$, the marked points on each rectangle are number clockwise.

Each string is labelled by an element from $\cL$, which induces a labeling on the marked points of the $D_i(T)$. Reading clockwise around the boundary of $D_i(T)$ starting at $\star$, we get a word $\alpha_i\in \Lambda$, the set of all finite words on $\cL$. We call $D_i(T)$ an $\alpha_i$-rectangle, and we call such a planar tangle $T$ a \underline{planar $\alpha_0$-tangle}.

If we have two planar tangles $S,T$ satisfying the following \underline{boundary conditions}:
\begin{itemize}
\item
some internal rectangle $D_i(S)$ agrees with $D_0(T)$,
\item
the marked points of $D_i(S)$ agree with the marked points of $D_0(T)$,
\item
the distinguished interval of $D_i(S)$ agrees with the distinguished interval of $T$, and
\item
the label from $\cL$ of each marked point of $D_i(S)$ agrees with the label of each marked point from $D_0(T)$,
\end{itemize}
then we may compose $S$ and $T$ to get the planar tangle $S\circ_i T$ by taking $S$ union the interior of $D_0(T)$, removing the boundary of $D_i(S)$, and smoothing the strings.
\end{defn}

\begin{rem}
When we draw a planar tangle, we will often suppress the external rectangle, which is assumed to be large. If we omit the $\star$, it is always assumed the $\star$ is in the \underline{lower left corner}. Finally, if $a_1,\dots,a_k\in\cL$ and  $\alpha=a_1\dots a_k\in\Lambda$, we draw one string labelled $\alpha$ rather than $k$ parallel strings labelled $a_1,\dots ,a_k$ where we always read the strands from \underline{left to right} and \underline{top to bottom}.

For each word $\alpha=a_1\cdots,a_n\in\Lambda$, we write $\overline{\alpha}=a_n\cdots a_1$ for the word in the reverse order.
\end{rem}

\begin{defn} \label{defn:planar}
A \underline{planar algebra $\cP$ with string labels $\cL$} is
\begin{itemize}
\item
a collection of vector spaces $(P_\alpha)_{\alpha\in\Lambda}$ (recall $\Lambda$ is the set of finite words on $\cL$)
\item
an action of planar tangles by multilinear maps, i.e., for each planar $\alpha$-tangle $T$, whose rectangles $D(T)$ are $\alpha_D$-rectangles, there is a multilinear map
$$
Z_T\colon \prod_{D \subset T_{0}} P_{\alpha_D}\to P_{\alpha}
$$
satisfying the following axioms:
\begin{enumerate}
\item[\underline{\text{Isotopy}:}]
If $\theta$ is an orientation preserving diffeomorphism of $\R^2$, then $Z_{\theta(T)}=Z_T$.  This means that 
$$
Z_{\theta(T)}(f_{\theta}) = Z_{T}(f)
$$
where $f_{\theta}(\theta(D)) = f(D)$.
\item[\underline{\text{Naturality}:}]
For $S,T$ composable tangles, $Z(S\circ_D T) = Z(S)\circ_D Z(T)$, where the composition on the right hand side is the composition of multilinear maps.
\end{enumerate}
\end{itemize}
Moreover, $\cP$ is called a \underline{factor planar algebra} if $\cP$ is
\begin{itemize}
\item
\underline{evaluable}, i.e., $\dim(P_\alpha)<\infty$ for all $\alpha\in\Lambda$ and $P_\emptyset\cong \C$ via the map that sends the empty diagram to $1\in\C$. Hence, by naturally, to each $a\in\cL$, there is a scalar $\delta_a\in\C$ such that any labelled diagram containing a closed loop labelled $a$ is equal to the same diagram without the closed loop multiplied by $\delta_a$. We use the notation $\delta_{\alpha}=\delta_{a_1}\cdots \delta_{a_n}$ if $\alpha = a_1\cdots a_n$.
\item
\underline{involutive}, i.e., for each $\alpha\in\Lambda$, there is a map $*\colon P_\alpha\to P_{\overline{\alpha}}$ with $*\circ *=\id$ which is compatible with the reflection of tangles, i.e., if $T$ is a planar tangle labelled by $\alpha_1,\dots,\alpha_s$, then
$$
T(\alpha_1,\dots, \alpha_n)^*=T^*(\alpha_1^*,\dots,\alpha_n^*)
$$
where $T^*$ is the reflection of $T$.
\item \underline{unital} \cite{JonesPAnotes}:  Let $S$ be a shaded planar $\alpha_{0}$ tangle with no input disks.  Then, there is an element $Z(S) \in P_{\alpha_{0}}$ so that the following holds:

Let $T$ be a tangle with a nonempty set of internal disks such that $S$ can be glued into the internal disk $D^{S}$ of $T$.  Then
    $$
    Z(T \circ S) = Z(T) \circ Z_{S}.
    $$
    Here $(Z(T) \circ Z_{S})(f) = \tilde{f}$ where
    $$
    \tilde{f}(D) = \begin{cases} f(D) \text{ if } D \neq D^{S} \\ Z(S) \text{ if } D = D^{S} \end{cases}
    $$
    
\item \underline{spherical}, i.e., for all $\alpha\in\Lambda$ and all $x\in P_{\alpha\overline{\alpha}}$, we have
$$
\tr(x)
=
\trace{\overline{\alpha}}{\alpha}{x}{\alpha}
=
\traceop{\overline{\alpha}}{\alpha}{x}{\alpha}\,.
$$
\item
\underline{positive}, i.e., for every $\alpha\in\Lambda$, the map $\langle\cdot,\cdot\rangle \colon P_\alpha\times P_\alpha\to P_\emptyset\cong\C$ given by
$$
\langle x,y\rangle =
\begin{tikzpicture}[baseline = -.1cm]
	\draw (0,0)--(1.2,0);	
	\filldraw[unshaded,thick] (-.4,.4)--(.4,.4)--(.4,-.4)--(-.4,-.4)--(-.4,.4);
	\filldraw[unshaded,thick] (.8,.4)--(1.6,.4)--(1.6,-.4)--(.8,-.4)--(.8,.4);
	\node at (0,0) {$x$};
	\node at (1.2,0) {$y^*$};
	\node at (.6,.2) {{\scriptsize{$\alpha$}}};
\end{tikzpicture}
$$
is a positive definite inner product. Hence for all $a\in\cL$, $\delta_a>0$.
\end{itemize}
\end{defn}

\begin{nota}
For each $\alpha,\beta\in\Lambda$, we write $P_{\alpha\to\beta}$ to denote the box space $P_{\alpha\overline{\beta}}$ of elements of the form
$$
\nbox{\alpha}{x}{\beta}\,.
$$
\end{nota}

\begin{rem}\label{rem:multiply}
Note that for each $\alpha\in\Lambda$, the multiplication tangle
$$
\PAMultiply{\alpha}{}{\alpha}{}{\alpha}
$$
makes $P_{\alpha\to\alpha}$ into an associative algebra. If $\cP$ is a factor planar algebra, then the multiplication tangle makes $P_{\alpha\to\alpha}$ a finite dimensional $C^*$-algebra.
\end{rem}

\begin{defn}
Suppose $\cP$ is a factor planar algebra. A \underline{projection} in $\cP$ is an element $p\in P_{\alpha\to\alpha}$ satisfying $p=p^2=p^*$ where the multiplication is as in Remark \ref{rem:multiply}. A projection $p\in P_{\alpha\to\alpha}$ is called \underline{simple} if it is a minimal projection in $P_{\alpha\to\alpha}$. Since $P_{\alpha\to \alpha}$ is a finite dimensional $C^*$-algebra, every projection is the sum of finitely many simple projections. This property is called \underline{semi-simpicity}.

Given a projection $p\in P_{\alpha\to\alpha}$, the dual projection $\overline{p}\in P_{\overline{\alpha}\to\overline{\alpha}}$ is obtained by
$$
\overline{p}
=
\rotateccw{\overline{\alpha}}{\alpha}{\alpha}{\overline{\alpha}}{p}
=
\rotatecw{\overline{\alpha}}{\alpha}{\alpha}{\overline{\alpha}}{p}
\,.
$$

Projections $p\in P_{\alpha\to\alpha}$ and $q\in P_{\beta\to\beta}$ are \underline{isomorphic} or \underline{equivalent}, denoted $p\simeq q$, if there is a $u\in P_{\alpha\to\beta}$ such that
$$
u^*u
=
\PAMultiply{\alpha}{u}{\beta}{u^*}{\alpha}
=
p
\text{ and }
uu^*
=
\PAMultiply{\beta}{u^*}{\alpha}{u}{\beta}
=
q.
$$

Given a projection $p\in P_{\alpha\to\alpha}$ and a $\beta\in \Lambda$, we can form the projection
$$
p\otimes \id_{\beta} =
\nbox{\alpha}{p}{\alpha}\,
\begin{tikzpicture}[baseline=-.1cm]
	\draw (0,-.8)--(0,.8);
	\node at (.2,0) {{\scriptsize{$\beta$}}};
\end{tikzpicture}
\in P_{\alpha\beta\to\alpha\beta}.
$$
The \underline{principal graph  of $\cP$ with respect to $\beta\in\Lambda$}, denoted $\Gamma_\beta$, is the graph whose vertices are the isomorphism classes of simple projections in $\cP$, and if $p\in P_{\alpha\to\alpha}$ and $q\in P_{\alpha\beta\to\alpha\beta}$ are simple projections, then the vertices $[p]$ and $[q]$ are connected by $\dim(qP_{\alpha\beta\to\alpha\beta}(p\otimes\id_{\beta}))$ edges.

The \underline{principal graph of $\cP$}, denoted $\Gamma$, is the push out of the $\Gamma_b$ for $b\in\cL$ over the isomorphism classes of simple projections, i.e., the vertices are the same as before, and the edge set is the union of the edge sets of the $\Gamma_b$ for $b\in \cL$.

Since $\cP$ is factor and $\cL$ is countable, $\Gamma$ has countably many vertices, although it may not be locally finite.
However, $\Gamma_b$ is always locally finite for $b\in \cL$.

Given a vertex $[p]$ of $\Gamma$, the number $\tr(p)$ is independent of the choice of representative of $[p]$. The vector $(\tr(p))_{[p]\in V(\Gamma)}$ defines a \underline{Frobenius-Perron weight vector} on the vertices of $\Gamma$ satisfying the following equation for each $b\in\cL$:
$$
\delta_b \tr(p) = \sum_{[q]\in V(\Gamma_b)} n_{[p],[q]}^b \tr(q)
$$
where $n_{[p],[q]}^b$ is the number of edges connecting $[p]$ and $[q]$ in $\Gamma_b$.
\end{defn}

\subsection{Rigid $C^*$-tensor categories and fusion graphs}
We briefly recall the definition of a rigid $C^*$-tensor category.

\begin{defn}
A \underline{rigid $C^*$-tensor category} is a pivotal, spherical, positive/unitary, rigid, semisimple, linear ( Vect-enriched) monoidal category such that $\End(1)\cong\C$.
\end{defn}

We now unravel this definition and state many properties that follow. The interested reader should see \cite{MR1960417, MR2681261} for more details. As we go through the properties, we will also go through the well-known graphical calculus used for strict tensor categories. We will immediately see that we get a factor planar algebra from a rigid $C^*$-tensor category.

We start with an abelian category $\cC$ together with
\begin{itemize}
\item
a bifunctor $\otimes\colon \cC\times \cC\to \cC$ which is associative up to a natural isomorphism (the pentagon axiom is satisfied), and
\item
a unit object $1\in\cC$ which is a left and right identity for $\otimes$ up to natural isomorphism (the triangle identity is satisfied).
\end{itemize}

\begin{rem}
Recall that a tensor category is called strict if the above natural isomorphisms are identities, i.e., for each $X,Y,Z\in\cC$, we have
\begin{align*}
(X\otimes Y)\otimes Y &= X(\otimes (Y\otimes Z)\text{ and}\\
1\otimes X &= X= X\otimes 1.
\end{align*}
Since each tensor category is equivalent to a strict tensor category by Theorem 7.2.1 in \cite{MR1712872}, our tensor categories will be assumed to be strict unless otherwise stated. Note that even with all the properties we want, we can still restrict our attention to strict categories.
\end{rem}

First, since $\cC$ is  Vect-enriched, for each $X,Y\in\cC$, $\cC(X\to Y)$ is a finite dimensional complex vector space. The morphisms in $\cC$ are drawn as boxes with strings emanating from the top and bottom. The strings are labelled by the objects, and the diagram is read from \underline{top to bottom}. For example,
$$
f
=
\nbox{X}{f}{Y}
\in\cC(X\to Y),
$$
and the identity morphism $\id_X$ is denoted by the horizontal strand labelled $X$. We compose morphisms by vertical concatenation
$$
\widenbox{X}{f\circ g}{Z}
=
\PAMultiply{X}{g}{Y}{f}{Z}
\in \cC(X\to Z),
$$
and we tensor morphisms by horizontal concatenation:
$$
\widenbox{\hspace{.8cm}X_1\otimes Y_1}{f_1\otimes f_2}{\hspace{.8cm}X_2\otimes Y_2}
=
\nbox{X_1}{f_1}{Y_1}
\,
\nbox{X_2}{f_2}{Y_2}
\in\cC(X_1\otimes X_2\to Y_1\otimes Y_2).
$$

Since $\cC$ is rigid, for each $X\in\cC$, there is a \emph{dual} or \emph{conjugate} $\overline{X}\in\cC$, and there is a natural isomorphism $\overline{\overline{X}}\cong X$. Along with the dual object, we have an \emph{evaluation map} $\ev_X\colon \overline{X}\otimes X\to 1$ and a \emph{coevaluation} map $\coev_X\colon 1\to X\otimes \overline{X}$ such that the diagram
$$
\xymatrix{
&X\otimes \overline{X}\otimes X \ar[dr]^{1\otimes \ev_X }\\
X\ar[ur]^{\coev_{X}\otimes 1}\ar[dr]_{1\otimes \coev_{\overline{X}}}\ar[rr]^{\id_X} && X\\
& X\otimes \overline{X}\otimes X \ar[ur]_{\ev_{\overline{X}}\otimes 1}
}
$$
commutes. The evaluation is denoted by a cap, and we a draw a cup for the coevaluation:
$$
\ev_X=
\begin{tikzpicture}[baseline=-.1cm]
	\draw[dotted] (0,-.4)--(0,0);
	\draw (-.4,.4) arc (-180:0:.4cm);
	\node at (.6,.4) {{\scriptsize{$X$}}};
	\node at (-.6,.4) {{\scriptsize{$\overline{X}$}}};
	\node at (0,-.6) {$1$};
\end{tikzpicture}
\text{ and }
\coev_X=
\begin{tikzpicture}[baseline=.1cm]
	\draw[dotted] (0,.4)--(0,0);
	\draw (-.4,-.4) arc (180:0:.4cm);
	\node at (-.6,-.4) {{\scriptsize{$X$}}};
	\node at (.6,-.4) {{\scriptsize{$\overline{X}$}}};
	\node at (0,.6) {$1$};
\end{tikzpicture}\,.
$$
The diagram above commuting is sometimes referred to as the \emph{zig-zag relation}, since it is the straightening of the kinked string:
$$
\begin{tikzpicture}[baseline=-.1cm]
	\draw (0,-.8)--(0,.8);
	\node at (.2,0) {{\scriptsize{$X$}}};
\end{tikzpicture}
=
\begin{tikzpicture}[baseline=-.1cm]
	\draw (0,.4) arc (180:0:.4cm)--(.8,-1.2);
	\draw (0,-.4)--(0,.4);
	\draw[dotted] (.4,.8)--(.4,1.2);
	\draw[dotted] (-.4,-.8)--(-.4,-1.2);
	\draw (0,-.4) arc (0:-180:.4cm)--(-.8,1.2);
	\node at (.2,0) {{\scriptsize{$\overline{X}$}}};
	\node at (1,-.6) {{\scriptsize{$X$}}};
	\node at (-.6,.6) {{\scriptsize{$X$}}};
\end{tikzpicture}
=
\begin{tikzpicture}[baseline=-.1cm]
	\draw (0,.4) arc (0:180:.4cm)--(-.8,-1.2);
	\draw (0,-.4)--(0,.4);
	\draw[dotted] (-.4,.8)--(-.4,1.2);
	\draw[dotted] (.4,-.8)--(.4,-1.2);
	\draw (0,-.4) arc (-180:0:.4cm)--(.8,1.2);
	\node at (-.2,0) {{\scriptsize{$\overline{X}$}}};
	\node at (-1,-.6) {{\scriptsize{$X$}}};
	\node at (.6,.6) {{\scriptsize{$X$}}};
\end{tikzpicture}\,.
$$
In general, we don't draw a string connected to the trivial object $1\in\cC$. For each $X,Y\in\cC$, $\overline{X\otimes Y}$ is naturally isomorphic to $\overline{Y}\otimes \overline{X}$, and the diagram
$$
\xymatrix{
\overline{X\otimes Y} \otimes (X\otimes Y)\ar[d]  \ar[rr]^(.6){\ev_{X\otimes Y}}&& 1\\
\overline{Y}\otimes \overline{X}\otimes X\otimes Y\ar[rr]^(.6){1\otimes \ev_X\otimes 1} && \overline{Y}\otimes Y\ar[u]^{\ev_Y}\\
}
$$
commutes, and similarly for the $\coev$'s. This diagram just means that we can write one cap labelled $X\otimes Y$ and its dual instead of two separate caps labelled $X$ and $Y$ and their duals:
$$
\begin{tikzpicture}[baseline=-.1cm]
	\draw (-.4,.4) arc (-180:0:.4cm);
	\draw (-.8,.4) arc (-180:0:.8cm);
	\node at (.6,.4) {{\scriptsize{$X$}}};
	\node at (-.6,.4) {{\scriptsize{$\overline{X}$}}};
	\node at (1,.4) {{\scriptsize{$Y$}}};
	\node at (-1,.4) {{\scriptsize{$\overline{Y}$}}};
\end{tikzpicture}
=
\begin{tikzpicture}[baseline=-.1cm]
	\draw (-.4,.4) arc (-180:0:.4cm);
	\node at (.9,.4) {{\scriptsize{$X\otimes Y$}}};
	\node at (-.9,.4) {{\scriptsize{$\overline{X\otimes Y}$}}};
\end{tikzpicture},
$$
and similarly for the cups.

The \emph{pivotality} axiom in $\cC$ requires that for all $f\in \cC(X\to Y)$,
$$
(\ev_{Y}\otimes \id_{\overline{X}})\circ(\id_{\overline{Y}}\otimes f\otimes \id_{\overline{X}})\circ(\id_{\overline{Y}} \otimes \coev_X)
=(\id_X\otimes \ev_{Y})\circ(\id_{\overline{Y}}\otimes f\otimes \id_{\overline{X}})\circ(\coev_{\overline{X}}\otimes\id_{\overline{Y}}).
$$
The equation above has an elegant representation in diagrams:
$$
\rotateccw{\overline{Y}}{X}{Y}{\overline{X}}{f}
=
\rotatecw{\overline{X}}{Y}{X}{\overline{Y}}{f}\,.
$$
For $f\in \cC(X\to Y)$, the above diagram defines a dual map $\overline{f}\in \cC(\overline{Y}\to \overline{X})$, and $\overline{\overline{f}}=f$.

The evaluations and coevaluations together with pivotality allow us to define a \emph{left} and \emph{right trace} on $\End_\cC(X)$:
\begin{align*}
\tr_L(f)&=\ev_{X}\circ (\id_{\overline{X}}\otimes f)\circ \coev_{\overline{X}}
=
\traceop{\overline{X}}{X}{f}{X}
\in \End(1)\cong \C\text{ and}\\
\tr_R(f)&=\ev_{\overline{X}}\circ (f\otimes \id_{\overline{X}})\circ \coev_{X}
=
\trace{\overline{X}}{X}{f}{X}
\in \End(1)\cong \C.
\end{align*}
Similarly, for each $X\in\cC$, there are numbers $d^L_{X}$ and $d^R_{X}$ which are the left and right traces of the identity morphism respectively, and $d_{X}^L=d_{\overline{X}}^R$ and $d_{X}^R=d_{\overline{X}}^L$

\emph{Sphericality} means that these two traces are equal, and we denote the common number by $\tr(f)$. The sphericality allows us to perform isotopy on closed diagrams as if they were drawn on a sphere. Hence $\dim(X):=d_X^L=d_X^R$ and $\dim(X)=\dim(\overline{X})$ for all $X\in \cC$.

The \emph{positivity} or \emph{unitarity} of $\cC$ means there is a contravariant functor $*\colon \cC\to \cC$ which is the identity on all objects, and on morphisms, it is anti-linear, involutive ($*\circ *=\id_{\cC}$), monoidal ($(f\circ g)^*=g^*\circ f^*$ for composable $f,g$), and positive ($f^*\circ f=0$ implies $f=0$). We require $*$ to be compatible with the duality ($\overline{f}^*=\overline{f^*}$) and with the evaluations and coevaluations (for all $X\in\cC$, $\coev_X=\ev_{\overline{X}}^*$). On diagrams, we perform $*$ by reflecting the diagram, keeping the labels on the strings, and placing a $*$ on all morphisms.

For all $X,Y\in\cC$, we now have that $\cC(X\to Y)$ is a Banach space with positive definite inner product
$$
\langle f,g\rangle = \tr(g^*f)=
\begin{tikzpicture}[baseline=.5cm]
	\draw (0,1.6) arc (180:0:.4cm) -- (.8,-.4) arc (0:-180:.4cm);
	\draw (0,.4)--(0,.8);
	\filldraw[unshaded,thick] (-.4,.4)--(.4,.4)--(.4,-.4)--(-.4,-.4)--(-.4,.4);
	\draw[thick, unshaded] (-.4, .8) -- (-.4, 1.6) -- (.4, 1.6) -- (.4,.8) -- (-.4, .8);
	\node at (-.2,1.8) {{\scriptsize{$X$}}};
	\node at (0,1.2) {$f$};
	\node at (-.2,.6) {{\scriptsize{$Y$}}};
	\node at (0,0) {$g^*$};
	\node at (-.2,-.6) {{\scriptsize{$X$}}};
	\node at (1,.6) {{\scriptsize{$\overline{X}$}}};
\end{tikzpicture}.
$$
The inner product makes $\End_{\cC}(X)$ a finite dimensional $C^*$-algebra, so in particular, all projections are sums of finitely many simple projections, and $\cC$ is \emph{semi-simple}, i.e., every exact sequence in $\cC$ splits. This also means that any object in $\cC$ can be written as a finite direct sum of simple objects. Recall that $X\in\cC$ is \emph{simple} if $\dim(\End_{\cC}(X))=1$. Thus if $X,Y$ are non-isomorphic simple objects, $\cC(X,Y)=(0)$. This means that for each simple $X,Y,Z\in \cC$, there are non-negative integers $N_{X,Y}^Z $ such that $X\otimes Y = \bigoplus_{Z\in\cC} N_{X,Y}^Z Z$, i.e.,
$$
N_{X,Y}^Z = \dim(\cC(X\otimes Y\to Z)).
$$
Moreover, we have \emph{Frobenius reciprocity}, i.e., for each $X,Y,Z\in\cC$, there are natural isomorphisms
$$
\cC(X\otimes Y\to Z)\cong \cC(X\to Z\otimes \overline{Y}) \cong \cC(Y\to \overline{X}\otimes Z)
$$
which are implemented by the evaluation and coevaluation maps:
$$
\begin{tikzpicture}[baseline=-.1cm]
	\draw (0,0)--(0,-.8);
	\draw (-.2,.8)--(-.2,0);
	\draw (.2,.8)--(.2,0);
	\draw[thick, unshaded] (-.4, -.4) -- (-.4, .4) -- (.4, .4) -- (.4, -.4) -- (-.4, -.4);
	\node at (0, 0) {$f$};
	\node at (-.4,.6) {\scriptsize{$X$}};
	\node at (.4,.6) {\scriptsize{$Y$}};
	\node at (.2,-.6) {\scriptsize{$Z$}};
\end{tikzpicture}
\leftrightarrow
\begin{tikzpicture}[baseline=-.1cm]
	\draw (0,0)--(0,-.8);
	\draw (-.2,.8)--(-.2,0);
	\draw (.2,.4) arc (180:0:.2cm) -- (.6,-.8);
	\draw[thick, unshaded] (-.4, -.4) -- (-.4, .4) -- (.4, .4) -- (.4, -.4) -- (-.4, -.4);
	\node at (0, 0) {$f$};
	\node at (-.4,.6) {\scriptsize{$X$}};
	\node at (.4,.8) {\scriptsize{$Y$}};
	\node at (.8,-.6) {\scriptsize{$\overline{Y}$}};
	\node at (.2,-.6) {\scriptsize{$Z$}};
\end{tikzpicture}
\leftrightarrow
\begin{tikzpicture}[baseline=-.1cm]
	\draw (0,0)--(0,-.8);
	\draw (.2,.8)--(.2,0);
	\draw (-.2,.4) arc (0:180:.2cm) -- (-.6,-.8);
	\draw[thick, unshaded] (-.4, -.4) -- (-.4, .4) -- (.4, .4) -- (.4, -.4) -- (-.4, -.4);
	\node at (0, 0) {$f$};
	\node at (-.4,.8) {\scriptsize{$X$}};
	\node at (-.8,-.6) {\scriptsize{$\overline{X}$}};
	\node at (.4,.8) {\scriptsize{$Y$}};
	\node at (.2,-.6) {\scriptsize{$Z$}};
\end{tikzpicture}.
$$
Hence, for all simple $X,Y,Z\in\cC$, we have $N_{X,Y}^Z=N_{Z,\overline{Y}}^X=N_{\overline{X},Z}^{Y}$.

\begin{defn}
An object $X\in \cC$ has a \underline{self-duality} if there is an invertible $\varphi\in \cC(X,\overline{X})$, which must satisfy certain compatibility axioms. We would like this $\varphi$ to allow us to define evaluation and coevaluation maps $X\otimes X\to 1$ and $1\to X\otimes X$, i.e, they are adjoint to each other, satisfy the zig-zag relation, and give a positive scalar for $\dim(X)$ when composed in the natural way. We define these maps by
$$
\begin{tikzpicture}[baseline=-.1cm]
	\draw (0,.6) -- (0,-.2) arc (-180:0:.4cm) -- (.8,.6);
	\filldraw[unshaded,thick] (-.25,.25)--(.25,.25)--(.25,-.25)--(-.25,-.25)--(-.25,.25);
	\node at (-.2,.5) {{\scriptsize{$X$}}};
	\node at (0,0) {$\varphi$};
	\node at (-.15,-.5) {{\scriptsize{$\overline{X}$}}};
	\node at (1,0) {{\scriptsize{$X$}}};
\end{tikzpicture}
\text{ and }
\begin{tikzpicture}[baseline=.1cm]
	\draw (0,-.6) -- (0,.2) arc (180:0:.4cm) -- (.8,-.6);
	\filldraw[unshaded,thick] (-.25,.25)--(.25,.25)--(.25,-.25)--(-.25,-.25)--(-.25,.25);
	\node at (-.2,-.5) {{\scriptsize{$X$}}};
	\node at (0,0) {$\varphi^*$};
	\node at (-.15,.5) {{\scriptsize{$\overline{X}$}}};
	\node at (1,0) {{\scriptsize{$X$}}};
\end{tikzpicture}
$$
respectively. Since $X$ is naturally isomorphic to $\overline{\overline{X}}$, $\overline{\varphi}$ is naturally in $\cC(X,\overline{X})$. Therefore, the compatibility requirements are that $\varphi$ must satisfy $\varphi\overline{\varphi^*}=\id_X$ and $\tr(\varphi\varphi^*)=\dim(X)$.
However, to be able to draw these diagrams naively by just a cup and a cap without the label $\varphi$, we must have that each of these maps is preserved by rotation:
$$
\begin{tikzpicture}[baseline=-.1cm]
	\draw (0,.6) -- (0,-.2) arc (-180:0:.4cm) -- (.8,.6) arc (180:0:.3cm) -- (1.4,-.4) .. controls ++(270:.6cm) and ++(270:.6cm) .. (-.6,-.4)--(-.6,.6);
	\filldraw[unshaded,thick] (-.25,.25)--(.25,.25)--(.25,-.25)--(-.25,-.25)--(-.25,.25);
	\node at (-.2,.5) {{\scriptsize{$X$}}};
	\node at (0,0) {$\varphi$};
	\node at (-.15,-.5) {{\scriptsize{$\overline{X}$}}};
	\node at (1,0) {{\scriptsize{$X$}}};
	\node at (1.6,0) {{\scriptsize{$\overline{X}$}}};
	\node at (-.8,0) {{\scriptsize{$X$}}};
\end{tikzpicture}
=
\begin{tikzpicture}[baseline=-.1cm]
	\draw (0,.6) -- (0,-.2) arc (0:-180:.4cm) -- (-.8,.6);
	\filldraw[unshaded,thick] (-.25,.25)--(.25,.25)--(.25,-.25)--(-.25,-.25)--(-.25,.25);
	\node at (.2,.5) {{\scriptsize{$X$}}};
	\node at (0,0) {$\varphi$};
	\node at (.15,-.5) {{\scriptsize{$\overline{X}$}}};
	\node at (-1,0) {{\scriptsize{$X$}}};
\end{tikzpicture}
=
\begin{tikzpicture}[baseline=-.1cm]
	\draw (0,.6) -- (0,-.2) arc (-180:0:.4cm) -- (.8,.6);
	\filldraw[unshaded,thick] (-.25,.25)--(.25,.25)--(.25,-.25)--(-.25,-.25)--(-.25,.25);
	\node at (-.2,.5) {{\scriptsize{$X$}}};
	\node at (0,0) {$\overline{\varphi}$};
	\node at (-.15,-.5) {{\scriptsize{$\overline{X}$}}};
	\node at (1,0) {{\scriptsize{$X$}}};
\end{tikzpicture}
=
\begin{tikzpicture}[baseline=-.1cm]
	\draw (0,.6) -- (0,-.2) arc (-180:0:.4cm) -- (.8,.6);
	\filldraw[unshaded,thick] (-.25,.25)--(.25,.25)--(.25,-.25)--(-.25,-.25)--(-.25,.25);
	\node at (-.2,.5) {{\scriptsize{$X$}}};
	\node at (0,0) {$\varphi$};
	\node at (-.15,-.5) {{\scriptsize{$\overline{X}$}}};
	\node at (1,0) {{\scriptsize{$X$}}};
\end{tikzpicture}
$$
i.e., the \underline{Frobenius-Schur indicator} \cite{MR2381536} of the evaluation must be equal to $+1$. This tells us that $\varphi=\overline{\varphi}$, and $\varphi$ is unitary ($\varphi^*\varphi=\id_X$ and $\varphi\varphi^*=\id_{\overline{X}}$). A self-duality satisfying this extra axiom is called a \underline{symmetric self-duality}.
\end{defn}

\begin{assumption}\label{assume:Countable}
We assume that our rigid $C^*$-tensor category $\cC$ is countably generated, i.e., there is a countable set $\cS$ of objects in $\cC$ such that for each $Y\in\cC$, there are $X_1,\dots, X_n\in \cS$ such that
$$
\cC(X_1\otimes\cdots \otimes X_n,Y)\neq (0),
$$
i.e., $Y$ is (isomorphic to) a sub-object of $X_1\otimes\cdots \otimes X_n$.

To perform the calculations needed to prove Theorem \ref{thm:Main}, we want the planar algebra $\cP$ associated to $\cC$ in Definition \ref{defn:PAfromTC} to be non-oriented, have a non-oriented fusion graph, and have all loop parameters greater than 1.

Hence given a countable generating set $\cS$, we work with the generating set $\cL=\set{X\oplus \overline{X}}{X\in\cS}$. Note that the objects in $\cL$ are \underline{not} simple, but they are symmetrically self-dual and have dimension greater than 1.
\end{assumption}

\begin{defn}
The \underline{fusion graph of $\cC$ with respect to $Y\in\cC$}, denoted $\cF_\cC(Y)$, is the oriented graph whose vertices are the isomorphism classes of simple objects of $\cC$, and between simple objects $X,Z\in\cC$, there are $N_{X,Y}^Z=\dim(\cC(X\otimes Y\to Z))$ oriented edges pointing from $X$ to $Z$. Note that if $Y$ is self-dual, then by semi-simplicity, we have $N_{X,Y}^Z=N_{Z,Y}^X$, and we may ignore the orientation of the edges.

The \underline{fusion graph of $\cC$ with respect to $\cL$} (with $\cL$ as in Assumption \ref{assume:Countable}), denoted $\cF_\cC(\cL)$, is the push out of the $\cF_\cC(Y)$ over the isomorphism classes of simple objects $Y\in\cL$, i.e., the vertices are the same as before, and the edge set is the union of the edge sets of the $\cF_\cC(Y)$ for $Y\in\cL$. If $e$ is an edge in $\cF_\cC(\cL)$ which comes from an edge in $\cF_\cC(Y)$, then we color $e$ by $Y$.

Since $\cL$ is countable, $\cF_\cC(\cL)$ has countably many vertices, although it may not be locally finite.
However, $\cF_\cC(X)$ is always locally finite for $X\in \cC$.

Given a vertex $[X]$ of $\cF_\cC(\cL)$, the number $\dim(X)$ is independent of the choice of representative of $[X]$.
Again, we get a \underline{Frobenius-Perron weight vector} on the vertices of $\cF_\cC(\cL)$, given by $(\dim(X))_{[X]\in V(\cF_\cC(\cL))}$, which satisfies the following equation for each $Y\in\cL$:
$$
\dim(X) \dim(Y) = \sum_{[Z]\in V(\cF_\cC(Y))}N_{X,Y}^Z \dim(Z).
$$
\end{defn}

For convenience, we will identify words on $\cL$ with their products, i.e., the word $\alpha= X_1X_2\dots X_n$ is identified with $X_1\otimes X_2\otimes \cdots \otimes X_n$.

\begin{defn}\label{defn:PAfromTC}
To get a factor planar algebra $\PA(\cC)_\bullet$, for each word $\alpha$ on $\cL$, let $\PA(\cC)_{\alpha} = \cC(\alpha\to 1)$, whose elements are represented diagrammatically as
$$
\Mbox{\alpha}{f}
$$
Frobenius reciprocity allows us to identify $\PA(\cC)_\alpha$ with $\cC(\beta\to\gamma)$ where $\alpha=\beta\overline{\gamma}$:
$$
\Mbox{\alpha}{f}
=
\nbox{\beta}{f}{\gamma}\,.
$$

We may now interpret any planar tangle in $\mathbb{P}_\cL$ labelled by morphisms of $\cC$ as one morphism in $\cC$ in the usual way. First, isotope the tangle so that each string travels transversally to each horizontal line, except at finitely many critical points. Then isotope the tangle so that each labelled rectangle and each critical point occurs at a different vertical height, and read the diagram from bottom to top to see what the morphism is. The zig-zag relation, Frobenius-reciprocity, pivotality, and symmetric self-dualities of the objects in $\cL$ ensure that the answer is well-defined.
\end{defn}

\subsection{From planar algebras to tensor categories}

Given a factor planar algebra $\cP$, we obtain its $C^*$-tensor category $\Pro({\cP})$ of projections as described in \cite{MR2559686}. We briefly recall the construction here.

\begin{defn}
Let $\Pro(\cP)$ (abbreviated $\Pro$ when $\cP$ is understood) be the rigid $C^*$-tensor category given as follows.
\begin{enumerate}
\item[\text{\underline{Objects:}}]
The objects of $\Pro$ are formal finite direct sums of \underline{projections} in $\cP$, i.e., all $p\in P_{\alpha\to\alpha}$ satisfying $p=p^2=p^*$ for all words $\alpha$ on $\cL$. The trivial object is the empty diagram.

\item[\text{\underline{Tensor:}}]
We tensor objects in $\Pro$ by horizontal concatenation; e.g., if $p\in P_{\alpha\to\alpha}$ and $q\in P_{\beta\to\beta}$, then $p\otimes q\in P_{\alpha\beta\to\alpha\beta}$ is given by
$$
\widenbox{\alpha\beta}{p\otimes q}{\alpha\beta}
=
\nbox{\alpha}{p}{\alpha}
\,
\nbox{\beta}{q}{\beta}
\in P_{\alpha\beta\to\alpha\beta}.
$$
Note that the simple objects in $\Pro$ are the simple projections in $\cP$.

We extend the tensor product to direct sums of projections linearly.

\item[\text{\underline{Morphisms:}}]
The morphisms in $\Pro$ are matrices of intertwiners between the projections. If $p\in P_{\alpha\to\alpha}$ and $q\in P_{\beta\to\beta}$, then elements in $\Pro(p,q)=qP_{\alpha\to \beta}p$ are all $x\in P_{\alpha\to\beta}$ such that $x=qxp$, i.e.,
$$
\nbox{\alpha}{x}{\beta}
=
\PAMultiplyThree{\alpha}{p}{\alpha}{x}{\beta}{q}{\beta}\,.
$$
We compose morphisms by vertical concatenation of elements in the planar algebra. If we have $x\in\Pro(p\to q)$ and $y\in\Pro(q\to r)$, then the composite $xy$ is given by
$$
\nbox{}{xy}{}
=
\PAMultiply{}{x}{}{y}{}
\in\Pro(p\to r).
$$
Composition of matrices of morphisms occurs in the usual way.

\item[\text{\underline{Tensoring:}}]
We tensor morphisms by horizontal concatenation. If $x\in\Pro(p_1\to q_1)$ and $y\in\Pro(p_2\to q_2)$, then the tensor product $x\otimes y$ is given by
$$
\widenbox{}{x\otimes y}{}
=
\nbox{}{x}{}
\,
\nbox{}{y}{}\,.
$$
The tensor product of matrices of intertwiners is the tensor product of matrices followed by tensoring of morphisms.

\item[\text{\underline{Duality:}}]
The \emph{duality} operation on objects and morphisms is rotation by $\pi$
$$
\nbox{}{\overline{p}}{}
=
\rotateccw{}{}{}{}{p}
=
\rotatecw{}{}{}{}{p}.
$$
The evaluation and coevaluation maps are given by the caps and cups between the projections in the obvious way.

\item[\text{\underline{Adjoint:}}]
The \emph{adjoint} operation in $\Pro$ is
the identity on objects. The adjoint of a $1$-morphism
is the same as the adjoint operation in the planar algebra $\cP$.
If $x\in\Pro(p\to q)$ where $p\in P_{\alpha\to \alpha}$ and $q\in P_{\beta\to \beta}$, then consider $x\in P_{\alpha\to \beta}$, take the adjoint, which is an element in $P_{\beta\to \alpha}$, and consider the result $x^*$ as an element in $\Pro(q\to p)$.

For matrices of intertwiners, the adjoint is the $*$-transpose.
\end{enumerate}
\end{defn}

\begin{ex}
We copy the example from \cite{MR2559686} as it is highly instructional. If $p,q\in P_{\alpha\to\alpha}$ are orthogonal, then if we define the matrix
$$
u=
\begin{pmatrix}
p & q
\end{pmatrix}\in \Pro(\id_\alpha\oplus\id_\alpha\to \id_\alpha),
$$
we get an isomorphism $p\oplus q = u^*u \simeq uu^* = p+q$.
\end{ex}

\begin{rem}
Note that $\Pro$ is strict. For any projection $p\in P_{\alpha\to\alpha}$, $p\otimes 1_{\Pro}=1_{\Pro}\otimes p = p$ since $1_{\Pro}$ is the empty diagram. For all projections $p,q,r\in \cP$,
$$
(p\otimes q)\otimes r
=
\nbox{}{p}{}
\,
\nbox{}{q}{}
\,
\nbox{}{r}{}
=
p\otimes (q\otimes r).
$$
\end{rem}

The following theorem is well-known to experts, and one can easily work it out from the definitions. See part (ii) of the remark on page 10 of \cite{1207.1923} for more details.

\begin{thm}\label{thm:PAandTC}\mbox{}
\begin{enumerate}[(1)]
\item
Let $\cC$ be a strict rigid $C^*$-tensor category. Then $\Pro(\PA(\cC)_\bullet)$ is equivalent to $\cC$.
\item
Let $\cP$ be a factor planar algebra. Then $\PA(\Pro(\cP))_\bullet$=$\cP$.
\end{enumerate}
\end{thm}

\begin{cor}
Suppose that
\begin{itemize}
\item
$\cC=\Pro(\cP)$ and $\cP$ has a countable set of string labels $\cL$, or
\item
$\cP=\PA(\cC)_\bullet$ and $\cC$ has countable generating set $\cL$ of symmetrically self-dual objects.
\end{itemize}
Then we may identify the fusion graph $\Gamma$ of $\cP$ with the fusion graph $\cF_\cC(\cL)$ of $\cC$.
\end{cor}

\section{GJS results for factor planar algebras}\label{sec:GJS}

Given a subfactor planar algebra $\cP$, GJSW constructed a subfactor $N\subset M$ whose planar algebra is $\cP$ \cite{MR2732052,MR2645882}. Moreover, they identified the factors as interpolated free group factors \cite{MR2807103}.

Suppose we have a factor planar algebra $\cP$ with a countable set of string labels $\cL$ such that for each $c\in\cL$, $\delta_c> 1$. (One can assume $\cP$ is the factor planar algebra associated to a rigid $C^*$-tensor category $\cC$ with generating set $\cL$ as in Assumption \ref{assume:Countable}.) We mimic the construction of GJSW to obtain a factor $M_{0}$ and rigid $C^*$-tensor categories $\Bim$ and $\CF$ of bifinite bimodules over $M_0$ such that $\Pro$ is equivalent to $\Bim$ and $\CF$.

\begin{rem}
Recall that when we suppress the $\star$ of an input rectangle, it is assumed that $\star$ is in the \underline{lower-left} corner.
Recall that if a string is labelled by the word $\alpha\in \Lambda$, it is read either \underline{top to bottom} or \underline{left to right}.
\end{rem}

\subsection{The graded algebras and their orthogonalized pictures} \label{sec:graded}

To start, we set $\Gr_{0}(P) = \bigoplus_{\alpha \in \Lambda} P_{\alpha}$ where $\Lambda$ denotes the set of all finite sequences of colorings for strings and endow $\Gr_{0}(\cP)$  with a multiplication $\wedge$ which satisfies
$$
x \wedge y =
\Mbox{\alpha}{x}
\,
\Mbox{\beta}{y}
$$
where $x \in P_{\alpha}$ and $y \in P_{\beta}$. We endow $\Gr_{0}(P)$ with the following trace:
\begin{equation}\label{r}
\tr(x) =
\begin{tikzpicture}[baseline=.5cm]
	\draw (0,0)--(0,.8);
	\filldraw[unshaded,thick] (-.4,.4)--(.4,.4)--(.4,-.4)--(-.4,-.4)--(-.4,.4);
	\draw[thick, unshaded] (-.7, .8) -- (-.7, 1.6) -- (.7, 1.6) -- (.7,.8) -- (-.7, .8);
	\node at (0,0) {$x$};
	\node at (0,1.2) {$\Sigma CTL$};
	\node at (.2,.6) {{\scriptsize{$\alpha$}}};
\end{tikzpicture}
\end{equation}
where $x \in P_{\alpha}$ and $\sum CTL$ denotes the sum of all colored Temperely-Lieb diagrams, i.e. all planar ways of pairing the colors on top of $x$ in a way which respects the word $\alpha$.

\begin{lem} \label{lem:PositiveBounded}
The inner product on $\Gr_{0}(\cP)$ given by $\langle x, y \rangle = \tr(y^{*}x)$ is positive definite.  Furthermore, left and right multiplication by elements in $\Gr_{0}(\cP)$ is bounded with respect to this inner product
\end{lem}

The proof of the above lemma will closely follow the orthogonalization approach in \cite{MR2645882}.  To begin, we define a new algebraic structure $\star$ on $\Gr_{0}(\cP)$ defined as follows.  Suppose $x \in P_{\alpha}$ and $y \in P_{\beta}$.  Then by letting $|\alpha|$ denote the length of $\alpha$, we have
$$
x \star y =
\sum_{
\substack{
\gamma \text{ s.t.}\\
\alpha=\alpha'\gamma\\
\beta=\overline{\gamma}\beta'
}}
\begin{tikzpicture}[baseline = .6cm]
	\draw (-.2,0)--(-.2,1);	
	\draw (1.4,0)--(1.4,1);
	\draw (.2,.4) arc (180:0:.4);	
	\filldraw[unshaded,thick] (-.4,.4)--(.4,.4)--(.4,-.4)--(-.4,-.4)--(-.4,.4);
	\filldraw[unshaded,thick] (.8,.4)--(1.6,.4)--(1.6,-.4)--(.8,-.4)--(.8,.4);
	\node at (0,0) {$x$};
	\node at (1.2,0) {$y$};
	\node at (-.4,.6) {{\scriptsize{$\alpha'$}}};
	\node at (1.6,.6) {{\scriptsize{$\beta'$}}};
	\node at (.6,1) {{\scriptsize{$\gamma$}}};
\end{tikzpicture}
$$
where it is understood that if a string connects two different colors, then that term in the sum is zero. We let $F_{0}(\cP)$ be the vector space $\Gr_{0}(\cP)$ endowed with the multiplication $\star$.  Given $x$ in $F_0(\cP)$, let $x_{\emptyset}$ denote the component of $x$ in $P_{\emptyset} \cong \C$.  We define a trace $\tr_{F}$ on $F_{0}(\cP)$ by $\tr_{F}(x) = x_\emptyset$.  Since $\cP$ is a $C^{*}$-planar algebra, the sesquilinear form
$$
\langle x, y \rangle
= \tr_{F_0(\cP)}(x\star y^{*})
=
\begin{tikzpicture}[baseline = -.1cm]
	\draw (0,0)--(1.2,0);	
	\filldraw[unshaded,thick] (-.4,.4)--(.4,.4)--(.4,-.4)--(-.4,-.4)--(-.4,.4);
	\filldraw[unshaded,thick] (.8,.4)--(1.6,.4)--(1.6,-.4)--(.8,-.4)--(.8,.4);
	\node at (0,0) {$x$};
	\node at (1.2,0) {$y^*$};
\end{tikzpicture}
$$
is a positive definite inner product.

Set $\Epi(CTL)$ to be the set of colored Temperely-Lieb boxes with strings at the top and bottom where any string touching the top of the box must be through.  One can argue exactly as in section 5 of \cite{MR2645882} that the map $\Phi: \Gr_{0}(\cP)\rightarrow F_{0}(\cP)$ given by
$$
\Phi(x) = \sum_{E \in \Epi(CTL)}
\begin{tikzpicture}[baseline=.7cm]
	\draw (0,0)--(0,2);
	\filldraw[unshaded,thick] (-.4,.4)--(.4,.4)--(.4,-.4)--(-.4,-.4)--(-.4,.4);
	\draw[thick, unshaded] (-.4, .8) -- (-.4, 1.6) -- (.4, 1.6) -- (.4,.8) -- (-.4, .8);
	\node at (0,0) {$x$};
	\node at (0,1.2) {$E$};
\end{tikzpicture}
$$
is a bijection with the property that $\Phi(x \wedge y) = \Phi(x) \star \Phi(y)$, $\Phi(x^{*}) = \Phi(x)^{*}$ and $\tr(x) = \tr_{F}(\Phi(x))$.  Hence $\star$ is an associative multiplication, $F_{0}(\cP)$ and $\Gr_{0}(\cP)$ are isomorphic as $*$-algebras, and the inner product on $\Gr_{0}(\cP)$ is positive definite.

We now prove that left multiplication by $x \in F_{0}(\cP)$ is bounded (this will closely follow arguments in \cite{1202.1298}).   We may assume $x \in P_{\alpha}$ for a fixed word $\alpha$.  For fixed words $\beta$ and $\gamma$ such that $\alpha = \overline{\beta}\gamma$, the element
$$
\begin{tikzpicture}[baseline = -.1cm]
	\draw (-.8,0)--(2,0);	
	\filldraw[unshaded,thick] (-.4,.4)--(.4,.4)--(.4,-.4)--(-.4,-.4)--(-.4,.4);
	\filldraw[unshaded,thick] (.8,.4)--(1.6,.4)--(1.6,-.4)--(.8,-.4)--(.8,.4);
	\node at (0,0) {$x^*$};
	\node at (1.2,0) {$x$};
	\node at (-.6,.2) {{\scriptsize{$\gamma$}}};
	\node at (.6,.2) {{\scriptsize{$\beta$}}};
	\node at (1.8,.2) {{\scriptsize{$\gamma$}}};
\end{tikzpicture}
$$
is positive in the finite dimensional $C^{*}$ algebra $P_{\overline{\gamma}\to\overline{\gamma}}$, since for any $y \in P_{\overline{\gamma}\to\overline{\gamma}}$,
$$
\langle x^*x\cdot y, y \rangle_{P_{\overline{\gamma}\to\overline{\gamma}}}
=
\begin{tikzpicture}[baseline = .4cm]
	\draw (-.4,0)--(1.6,0) arc (-90:90:.5cm) -- (-.4,1) arc (90:270:.5cm);	
	\filldraw[unshaded,thick] (-.4,1.4)--(.4,1.4)--(.4,.6)--(-.4,.6)--(-.4,1.4);
	\filldraw[unshaded,thick] (.8,1.4)--(1.6,1.4)--(1.6,.6)--(.8,.6)--(.8,1.4);
	\filldraw[unshaded,thick] (-.4,.4)--(.4,.4)--(.4,-.4)--(-.4,-.4)--(-.4,.4);
	\filldraw[unshaded,thick] (.8,.4)--(1.6,.4)--(1.6,-.4)--(.8,-.4)--(.8,.4);
	\node at (0,0) {$x^*$};
	\node at (1.2,0) {$x$};
	\node at (0,1) {$y^*$};
	\node at (1.2,1) {$y$};
	\node at (-.6,.2) {{\scriptsize{$\gamma$}}};
	\node at (.6,.2) {{\scriptsize{$\beta$}}};
	\node at (.6,1.2) {{\scriptsize{$\overline{\gamma}$}}};
	\node at (1.8,.2) {{\scriptsize{$\gamma$}}};
	\node at (0,-.55) {$\star$};
	\node at (1.2,-.55) {$\star$};
	\node at (0,1.55) {$\star$};
	\node at (1.2,1.55) {$\star$};
\end{tikzpicture}
=
\left\|
\begin{tikzpicture}[baseline = .3cm]
	\draw (-.2,0)--(-.2,1);	
	\draw (1.4,0)--(1.4,1);
	\draw (.2,.4) arc (180:0:.4);	
	\filldraw[unshaded,thick] (-.4,.4)--(.4,.4)--(.4,-.4)--(-.4,-.4)--(-.4,.4);
	\filldraw[unshaded,thick] (.8,.4)--(1.6,.4)--(1.6,-.4)--(.8,-.4)--(.8,.4);
	\node at (0,0) {$x$};
	\node at (1.2,0) {$y$};
	\node at (-.4,.6) {{\scriptsize{$\overline{\beta}$}}};
	\node at (1.6,.6) {{\scriptsize{$\gamma$}}};
	\node at (.6,1) {{\scriptsize{$\gamma$}}};
\end{tikzpicture}
\right\|^2_{L^2(F_0(\cP))}\geq 0.
$$
Given $x$ and $w$ with $x \in P_{\alpha}$ with $\alpha=\overline{\beta}\gamma$, $x \star w$ is a sum of terms of the form
$$
\begin{tikzpicture}[baseline = .3cm]
	\draw (-.2,0)--(-.2,1);	
	\draw (1.4,0)--(1.4,1);
	\draw (.2,.4) arc (180:0:.4);	
	\filldraw[unshaded,thick] (-.4,.4)--(.4,.4)--(.4,-.4)--(-.4,-.4)--(-.4,.4);
	\filldraw[unshaded,thick] (.8,.4)--(1.6,.4)--(1.6,-.4)--(.8,-.4)--(.8,.4);
	\node at (0,0) {$x$};
	\node at (1.2,0) {$w$};
	\node at (-.4,.6) {{\scriptsize{$\overline{\beta}$}}};
	\node at (.6,1) {{\scriptsize{$\gamma$}}};
\end{tikzpicture}\,,
$$
and we see that the 2-norm of the above diagram is
\begin{equation} \label{eqn:TwoNorm}
\begin{tikzpicture}[baseline = .4cm]
	\draw (-.4,0)--(1.6,0) arc (-90:90:.5cm) -- (-.4,1) arc (90:270:.5cm);	
	\filldraw[unshaded,thick] (-.4,1.4)--(.4,1.4)--(.4,.6)--(-.4,.6)--(-.4,1.4);
	\filldraw[unshaded,thick] (.8,1.4)--(1.6,1.4)--(1.6,.6)--(.8,.6)--(.8,1.4);
	\filldraw[unshaded,thick] (-.4,.4)--(.4,.4)--(.4,-.4)--(-.4,-.4)--(-.4,.4);
	\filldraw[unshaded,thick] (.8,.4)--(1.6,.4)--(1.6,-.4)--(.8,-.4)--(.8,.4);
	\node at (0,0) {$x^*$};
	\node at (1.2,0) {$x$};
	\node at (0,1) {$w^*$};
	\node at (1.2,1) {$w$};
	\node at (-.6,.2) {{\scriptsize{$\gamma$}}};
	\node at (.6,.2) {{\scriptsize{$\beta$}}};
	\node at (1.8,.2) {{\scriptsize{$\gamma$}}};
	\node at (0,-.55) {$\star$};
	\node at (1.2,-.55) {$\star$};
	\node at (0,1.55) {$\star$};
	\node at (1.2,1.55) {$\star$};
\end{tikzpicture}
=
\tr(x^*xww^*)
\leq \|x^*x\|_{P_{\overline{\gamma}\to\overline{\gamma}}} \| ww^*\|_{L^2(F_0(\cP))}
\end{equation}
where $\|\cdot\|_{P_{\overline{\gamma}\to\overline{\gamma}}}$ is the operator norm in the $C^*$-algebra $P_{\overline{\gamma}\to\overline{\gamma}}$.
Hence using Equation \eqref{eqn:TwoNorm} repeatedly, we have
$$
\|x \star w\|_{L^{2}(F_{0}(P))} \leq \left(\sum_{\alpha=\overline{\beta}\gamma} \|x\|_{P_{\overline{\gamma}\to\overline{\gamma}}}\right) \cdot \|w\|_{L^{2}(F_{0}(\cP))},
$$
and thus left multiplication is bounded on $F_0(\cP)$. The boundedness of right multiplication is similar.

Since the multiplication is bounded, we can represent $F_{0}(\cP)$ on $L^{2}(F_{0}(\cP))$ acting by left multiplication.  We denote $M_{0} = F_{0}(\cP)''$.  We also use $M_{0}$ to denote $\Gr_{0}(\cP)''$ acting on $L^{2}(\Gr_{0}(\cP))$, but it will be clear from context which picture we are using.  Of course, from the discussion above, both von Neumann algebras are isomorphic.

Given $\alpha \in \Lambda$, we draw a \textcolor{\alphacolor}{\alphacolor} string for a string labelled $\alpha$. We will provide the $\alpha$ label only when it is possible to confuse $\alpha$ and $\overline{\alpha}$.

We define the graded algebra $\Gr_{\alpha}(\cP) = \bigoplus_{\beta \in \Lambda} P_{\overline{\alpha} \beta \alpha}$ with multiplication $\wedge_{\alpha}$ by
$$
x \wedge_{\alpha} y =
\begin{tikzpicture}[baseline = -.1cm]
	\draw [thick, \alphacolor] (-.8,0)--(2,0);	
	\draw (0,0)--(0,.8);
	\draw (1.2,0)--(1.2,.8);
	\filldraw[unshaded,thick] (-.4,.4)--(.4,.4)--(.4,-.4)--(-.4,-.4)--(-.4,.4);
	\filldraw[unshaded,thick] (.8,.4)--(1.6,.4)--(1.6,-.4)--(.8,-.4)--(.8,.4);
	\node at (0,0) {$x$};
	\node at (1.2,0) {$y$};
	\node at (.2,.6) {{\scriptsize{$\beta$}}};
	\node at (1.4,.6) {{\scriptsize{$\gamma$}}};
\end{tikzpicture}
$$
for $x\in P_{\overline{\alpha}\beta\alpha}$ and $x\in P_{\overline{\alpha}\gamma\alpha}$, and trace
$$
\tr(x) = \frac{1}{\delta^{\alpha}}
\begin{tikzpicture}[baseline=.3cm]
	\draw (0,0)--(0,.8);
	\draw [thick, \alphacolor] (.4,0) arc (90:-90:.4cm) -- (-.4,-.8) arc (270:90:.4cm);
	\filldraw[unshaded,thick] (-.4,.4)--(.4,.4)--(.4,-.4)--(-.4,-.4)--(-.4,.4);
	\draw[thick, unshaded] (-.7, .8) -- (-.7, 1.6) -- (.7, 1.6) -- (.7,.8) -- (-.7, .8);
	\node at (0,0) {$x$};
	\node at (0,1.2) {$\Sigma CTL$};
	\node at (.2,.6) {{\scriptsize{$\beta$}}};
	\node at (.6,.2) {{\scriptsize{$\alpha$}}};
	\node at (-.6,.2) {{\scriptsize{$\alpha$}}};
	\node at (0,-.6) {{\scriptsize{$\overline{\alpha}$}}};
\end{tikzpicture}.
$$
\begin{note}\label{note:ReverseMultiplication}
Be warned that the multiplication $\wedge_{\alpha}$ in the GJSW diagrams when restricted to $P_{\overline{\alpha} \to \overline{\alpha}}$ is in the opposite order with the multiplication in the introduction!
\end{note}
The inner product
$$
\langle \cdot, \cdot \rangle_{F_\alpha(\cP)}=
\begin{tikzpicture}[baseline = -.1cm]
	\draw [thick, \alphacolor] (-.4,0)--(1.6,0) arc (90:-90:.4cm) -- (-.4,-.8) arc (270:90:.4cm);	
	\draw (0, .4) .. controls ++(90:.4cm) and ++(90:.4cm) .. (1.2,.4);
	\filldraw[unshaded,thick] (-.4,.4)--(.4,.4)--(.4,-.4)--(-.4,-.4)--(-.4,.4);
	\filldraw[unshaded,thick] (.8,.4)--(1.6,.4)--(1.6,-.4)--(.8,-.4)--(.8,.4);
	\node at (0,0) {};
	\node at (1.2,0) {};
	\node at (-.6,.2) {{\scriptsize{$\alpha$}}};
\end{tikzpicture}
$$
makes $P_{\overline{\alpha}\beta\alpha}\perp P_{\overline{\alpha}\gamma\alpha}$ for $\beta\neq \gamma$. We get a trace on $F_\alpha(\cP)$ by $\tr_{F_\alpha(\cP)}(x)=\langle x, 1_\alpha\rangle_{F_\alpha(\cP)}$ where $1_\alpha$ the the horizontal strand labelled $\alpha$.

The multiplication $\star_\alpha$ given by
$$
x \star_\alpha y =
\sum_{
\substack{
\kappa \text{ s.t.}\\
\beta=\beta'\kappa\\
\gamma=\overline{\kappa}\gamma'
}}
\begin{tikzpicture}[baseline = .6cm]
	\draw [thick, \alphacolor] (-.8,0)--(2,0);
	\draw (-.2,0)--(-.2,1);	
	\draw (1.4,0)--(1.4,1);
	\draw (.2,.4) arc (180:0:.4);	
	\filldraw[unshaded,thick] (-.4,.4)--(.4,.4)--(.4,-.4)--(-.4,-.4)--(-.4,.4);
	\filldraw[unshaded,thick] (.8,.4)--(1.6,.4)--(1.6,-.4)--(.8,-.4)--(.8,.4);
	\node at (0,0) {$x$};
	\node at (1.2,0) {$y$};
	\node at (-.4,.6) {{\scriptsize{$\beta'$}}};
	\node at (1.6,.6) {{\scriptsize{$\gamma'$}}};
	\node at (.6,1) {{\scriptsize{$\kappa$}}};
\end{tikzpicture}
$$
for $x\in P_{\overline{\alpha}\beta\alpha}$ and $x\in P_{\overline{\alpha}\gamma\alpha}$ makes $F_{\alpha}(P)$ isomorphic as a $*$-algebra to $\Gr_{\alpha}(\cP)$, preserving the inner product.  The same techniques as above with heavier notation show that the inner product on $F_{\alpha}(\cP)$ (hence $\Gr_{\alpha}(\cP)$) is positive definite and that left and right multiplication in $F_{\alpha}(\cP)$ (hence $\Gr_{\alpha}(\cP)$) is bounded.  We can therefore form the von Neumann algebra $M_{\alpha} = (F_{\alpha}(\cP))''$ acting by left multiplication on $L^{2}(F_{\alpha})(\cP)$ (or $(\Gr_{\alpha}(\cP))''$ acting by left multiplication on $L^{2}(\Gr_{\alpha})(\cP)$).  Again, it will be clear from context which picture we are considering.

\subsection{Factorality of $M_{\alpha}$} \label{sec:factor}

In this section, we aim to prove the following theorem:  \begin{thm} \label{thm:Factor}

The algebra $M_{\alpha}$ is a $II_{1}$ factor. We have an embedding $\iota_{\alpha}: M_0 \hookrightarrow M_{\alpha}$ which is the extension of the map $F_0(\cP)\to F_\alpha(\cP)$ given by
$$
\iota_{\alpha}
\left(
\Mbox{}{x}
\right)
=
\begin{tikzpicture}[baseline = 0cm]
	\draw (0,0)--(0,.8);
	\draw [thick, \alphacolor] (-.8,-.6)--(.8,-.6);
	\filldraw[unshaded,thick] (-.4,.4)--(.4,.4)--(.4,-.4)--(-.4,-.4)--(-.4,.4);
	\node at (0,0) {$x$};
\end{tikzpicture},
$$
and $\iota_{\alpha}(M_{0})' \cap M_{\alpha} = P_{\overline{\alpha}\to\overline{\alpha}}\op\cong P_{\alpha\to \alpha}$.
\end{thm}

Throughout this subsection, we use the orthogonal picture $M_{\alpha} = (F_{\alpha}(\cP))''$.  Pick a specific color, $c\in\cL$, which we will denote by the color \textcolor{\cupcolor}{green}, and recall $\delta_c>1$.
Let $A$ be the abelian von Neumann subalgebra of $M_{\alpha}$ generated by the cup element
$$
\bluecup_{\alpha}
 =
 \begin{tikzpicture} [baseline = 0cm]
 	\draw[thick, \alphacolor]  (-.4, -.2) -- (.4, -.2);
	\draw[\cupcolor, thick] (-.25, .4) arc(-180:0: .25cm);
 	\draw[thick] (-.4, -.4) -- (-.4, .4) -- (.4, .4) -- (.4, -.4) -- (-.4, -.4);
\end{tikzpicture}\, .
$$
We will obtain the factorality of $M_{\alpha}$ be first examining $L^{2}(M_{\alpha})$ as an $A-A$ bimodule.

To begin, assuming $|\gamma| \geq 2$, we set
$$
V_{\gamma} = \left\{x \in P_{\overline{\alpha} \gamma \alpha} :
 \begin{tikzpicture} [baseline = 0cm]
 	\draw [thick, \alphacolor] (-.9, 0) -- (.9,0);
	\draw (.3,.4)--(.3,.8);
	\draw[thick, \cupcolor] (-.3, .4) arc (180:0: .2cm);
 	\draw[thick, unshaded] (-.5, -.4) -- (-.5, .4) -- (.5, .4) -- (.5, -.4) -- (-.5, -.4);
	\node at (-.1, .75)  {{\scriptsize{$c$}}};
	\node at (0,0) {$x$};
\end{tikzpicture}
=
 \begin{tikzpicture} [baseline = 0cm]
 	\draw [thick, \alphacolor] (-.9, 0) -- (.9,0);
	\draw (-.3,.4)--(-.3,.8);
	\draw[thick, \cupcolor] (-.1, .4) arc (180:0: .2cm);
 	\draw[thick, unshaded] (-.5, -.4) -- (-.5, .4) -- (.5, .4) -- (.5, -.4) -- (-.5, -.4);
	\node at (.1, .75)  {{\scriptsize{$c$}}};
	\node at (0,0) {$x$};
\end{tikzpicture}
= 0 \right\},
$$
$V = \oplus_{\gamma \in \Lambda} V_{\gamma}$ and $\cV$ the $A-A$ bimodule generated by $V$.  For $b\in\cL$, let $\cV_{b}$ be the $A-A$ bimodule generated by elements of the form
$$
\begin{tikzpicture}[baseline = 0cm]
	\draw (0,0)--(0,.8);
	\draw [thick, \alphacolor] (-.8,0)--(.8,0);
	\filldraw[unshaded,thick] (-.4,.4)--(.4,.4)--(.4,-.4)--(-.4,-.4)--(-.4,.4);
	\node at (0,0) {$x$};
	\node at (.2,.6) {{\scriptsize{$b$}}};
\end{tikzpicture},
$$
and let $W$ be the $A-A$ bimodule generated by $P_{\alpha\to\alpha}$.
We claim that we have the following decomposition:

\begin{lem} \label{lem:decomposition}
 As $A-A$ bimodules, $L^{2}(M_{\alpha}) = \displaystyle W\oplus \cV\oplus  \bigoplus_{b \in \cL} \cV_{b}$.
\end{lem}

\begin{proof}
The proof is exactly the same as \cite{1202.1298} Proposition 2.1 except that we induct on the length of a word in $\Lambda$ as opposed to the number of strands of a single color.
\end{proof}

Notice that we can decompose $V$ further as $V = V_{cc} \oplus V_{co} \oplus V_{oc} \oplus V_{oo}$.  Here, $V_{cc}$ consists of elements in $V$ whose leftmost and rightmost strings are colored $c$, $V_{co}$ consists of elements of $V$ whose leftmost string is colored $c$ and rightmost string is colored differently than $c$, $V_{oc}$ consists of elements of $V$ whose leftmost string is colored differently than $c$ and whose rightmost string is colored $c$, and $V_{oo}$ consists of the elements in $V$ whose leftmost and rightmost strings are colored other than $c$.  We set $\cV_{cc}$ the $A-A$ bimodule generated by $V_{cc}$ and we define $\cV_{co}$, $\cV_{oc}$, and $\cV_{oo}$ analogously.

Let $(V_{cc})_{n}$ be the subspace of $V_{cc}$ spanned by boxes with a word of length $n$ on top and let $\{\zeta_{n, i}\}$ be an orthonormal basis for $(V_{cc})_{n}$.  It straightforward to see that the set
$$
\left\{
\zeta_{n, i}^{l, r} = \delta_{c}^{-(l + r)/2}\,
\begin{tikzpicture} [baseline = -.1 cm]
	\draw (0,.4) -- (0, .7);
	\draw [thick, \cupcolor] (-1.8, .7) arc(-180:0: .15cm);
	\draw [thick, \cupcolor] (-.8, .7) arc(-180:0: .15cm);
	\draw [thick, \cupcolor] (.4, .7) arc(-180:0: .15cm);
	\draw [thick, \cupcolor] (1.4, .7) arc(-180:0: .15cm);
	\draw [thick, \alphacolor] (-2, 0) -- (2, 0);
	\filldraw[unshaded,thick] (-.4,.4)--(.4,.4)--(.4,-.4)--(-.4,-.4)--(-.4,.4);
	\node at (0, 0) {$\zeta_{n,i}$};
	\node at (-1.1, .6) {{\scriptsize{$\cdots$}}};
	\node at (1.1, .6)  {{\scriptsize{$\cdots$}}};
	\node at (1.1, .4)  {{\scriptsize{$r$}}};
	\node at (-1.1, .4)  {{\scriptsize{$l$}}};
	\draw [thick] (-2, .7) -- (2, .7) -- (2, -.7) -- (-2, -.7) -- (-2, .7);
\end{tikzpicture}\,
\colon
n, \, l, \, r \in \N
\right\}
 $$
is an orthonormal basis for $\cV_{cc}$.  There are similar orthonormal bases for $\cV_{co}$, $\cV_{oc}$ and $\cV_{oo}$. Let $B$ denote $V_{cc}$, $V_{co}$, $V_{oc}$, $V_{oo}$ or $P_{\overline{\alpha} b \alpha}$ for $b \neq c$ and let $\cB$ denote $\cV_{cc}$, $\cV_{co}$, $\cV_{oc}$, $\cV_{oo}$, or $\cV_{b}$.   Let $\pi,\rho$ denote the left, right representation of $M_{0}$ on $L^{2}(M_{0})$ respectively. We have the following lemma whose proof is straightforward:

\begin{lem} \label{lem:unitary}
Let $\phi: \cB \rightarrow \ell^{2}(\N) \otimes B \otimes \ell^{2}(\N)$ be defined on the orthonormal basis of $\cB$ by
$$
\phi\left(\delta_{c}^{-(l + r)/2} \,
\begin{tikzpicture} [baseline = -.1 cm]
	\draw (0,.4) -- (0, .7);
	\draw [thick, \cupcolor] (-1.8, .7) arc(-180:0: .15cm);
	\draw [thick, \cupcolor] (-.8, .7) arc(-180:0: .15cm);
	\draw [thick, \cupcolor] (.4, .7) arc(-180:0: .15cm);
	\draw [thick, \cupcolor] (1.4, .7) arc(-180:0: .15cm);
	\draw [thick, \alphacolor] (-2, 0) -- (2, 0);
	\filldraw[unshaded,thick] (-.4,.4)--(.4,.4)--(.4,-.4)--(-.4,-.4)--(-.4,.4);
	\node at (0, 0) {$v$};
	\node at (-1.1, .6) {{\scriptsize{$\cdots$}}};
	\node at (1.1, .6)  {{\scriptsize{$\cdots$}}};
	\node at (1.1, .4)  {{\scriptsize{$r$}}};
	\node at (-1.1, .4)  {{\scriptsize{$l$}}};
	\draw [thick] (-2, .7) -- (2, .7) -- (2, -.7) -- (-2, -.7) -- (-2, .7);
\end{tikzpicture}\,
\right)
= \xi_{l} \otimes v \otimes \xi_{r}
$$
with $v \in B_{n}$, where $\set{\xi_i}{i\in\N}$ is the usual orthonormal basis of $\ell^2(\N)$.  Then $\phi$ extends to a unitary operator and we have the following representations of $\bluecup_{\alpha}$:

\begin{itemize}
\item
If $\cB = \cV_{cc}$, $\cV_{co}$, $\cV_{oc}$ or $\cV_{oo}$ then
$$
\phi \cdot \pi\left(\frac{\bluecup_{\alpha} - 1}{\delta_{c}^{1/2}}\right)\cdot \phi^{*} = x \otimes 1 \otimes 1 \text{ where }
$$
$x$ is $s + s^{*}$ if the top leftmost color in for boxes in $B$ is $c$ and $s + s^{*} - \delta_{c}^{-1/2}e_{\xi_{0}}$ if the this color differs from $c$.  Similarly,
$$
\phi \cdot \rho\left(\frac{\bluecup_{\alpha} - 1}{\delta_{c}^{1/2}}\right)\cdot \phi^{*} = 1 \otimes 1 \otimes x \text{ where }
$$
$x$ is $s + s^{*}$ if the top rightmost color in for boxes in $B$ is $c$ and $s + s^{*} - \delta_{c}^{-1/2}e_{\xi_{0}}$ if the this color differs from $c$.
\item
If $\cB = \cV_{b}$ and $b \neq c$ then
\begin{align*}
\phi\cdot \pi\left(\frac{\bluecup_{\alpha} - 1}{\delta_{c}^{1/2}}\right)\cdot \phi^{*}
&= (s + s^{*} - \delta_{c}^{-1/2}e_{\xi_{0}}) \otimes 1 \otimes 1 \textrm{ and }\\
\phi\cdot \rho\left(\frac{\bluecup_{\alpha} - 1}{\delta_{c}^{1/2}}\right)\cdot \phi^{*}
&= 1 \otimes 1 \otimes (s + s^{*} - \delta_{c}^{-1/2}e_{\xi_{0}}).
\end{align*}
\item
(See \cite{1202.1298}) If $\cB = \cV_{c}$ then there is a unitary (which is \underline{not} $\phi$) $v: \cV_{c} \rightarrow \ell^{2}(\cN) \otimes L^{2}(P_{\overline{\alpha}c\alpha}) \otimes \ell^{2}(\N)$ such that
\begin{align*}
v \cdot \pi\left(\frac{\bluecup_{\alpha} - 1}{\delta_{c}^{1/2}}\right)\cdot v^{*} &= (s+s^{*}) \otimes 1 \otimes 1 \textrm{ and }\\
v\cdot \rho\left(\frac{\bluecup_{\alpha} - 1}{\delta_{c}^{1/2}}\right)\cdot v^{*} &= 1 \otimes 1 \otimes (s + s^{*}).
\end{align*}
\end{itemize}
Here, $s$ is the unilateral shift operator on $\ell^{2}(\N)$ and $e_{\xi_{0}}$ is the orthogonal projection in $\cB(\ell^{2}(\N))$ onto the one-dimensional space spanned by $\xi_{0}$.
\end{lem}

We now show that the operators $s+s^{*}$ and $s + s^{*} - \delta_{c}^{-1/2}e_{\xi_{0}}$ are unitary equivalent in $\cB(\ell^{2}(\N))$.  We begin with the following lemma
\begin{lem} \label{lem:essential}
The spectra of $s+s^{*}$ and $y = s + s^{*} - \delta_{c}^{-1/2}e_{\xi_{0}}$ are the same.
\end{lem}

\begin{proof}
Since the operators differ by a finite rank operator, by the Weyl-von Neumann Theorem (see \cite{MR1335452} p.523), they have the same essential spectrum.  The operator $s+s^{*}$ has essential spectrum $[-2, 2]$ and since the complement of the essential spectrum in the spectrum is an isolated set of eigenvalues, we just need to show that $y$ has no eigenvalues outside of $[-2, 2]$ since $y$ is self adjoint and must have real spectrum.

To this end, let $\lambda > 2$ and $\xi = \sum_{j} x_{j}\xi_{j}$.  If $\xi$ is an eigenvector of $y$ with eigenvalue $\lambda$ then we have the equations
$$
x_{1} = \left(\lambda + \delta_{c}^{-1/2}\right)x_{0} \textrm{ and } \lambda x_{n+1} = x_{n} + x_{n+2}
$$
for $n \geq 0$. The characteristic equation for this linear recurrence is $x^{2} - \lambda x + 1 = 0$ which has roots
$$
l = \frac{\lambda + \sqrt{\lambda^{2} - 4}}{2} \textrm{ and } r = \frac{\lambda - \sqrt{\lambda^{2} - 4}}{2}.
$$
This implies the existence of constants $C$ and $D$ with $x_{n} = C\cdot l^{n} + D \cdot r^{n}$ for all $n$.  Since the sequence $x_{n}$ is $\ell^{2}$, $C = 0$ meaning $x_{n} = x_{0}\cdot r^{n}$ for all $n$.  However, this means $x_{1} = (\lambda + 1/\sqrt{\delta_{c}})x_{0}$ and $x_{1} = rx_{0}$ must both be satisfied.  This can only happen if $x_{0} = 0$ since $\lambda + 1/\sqrt{\delta_{c}} > 1$ and $r < 1$.  This implies $\xi = 0$.\\

Now let $\lambda < -2$.  With $l$ and $r$ as above we must have $x_{n} = C \cdot l^{n} + D \cdot r^{n}$.  This time we must have $D = 0$ and we obtain the equation $x_{n} = x_{0}l^{n}$.  This gives us the two equations
$$
x_{1} = \left(\lambda + \delta_{c}^{-1/2}\right)x_{0} \text{ and }  x_{1} = l x_{0}
$$
which implies $(\lambda - \frac{1}{2}(\lambda + \sqrt{\lambda^{2} - 4}) + \delta_{c}^{-1/2})x_{0} = 0$.  This forces $x_{0}$ to be 0 since by the choice of $\lambda$, $\lambda - \frac{1}{2}(\lambda + \sqrt{\lambda^{2} - 4}) < -1$ and $|\delta_{c}^{-1/2}| < 1$.  Therefore $\xi=0$ and $y$ has no eigenvectors.
\end{proof}

For any self-adjoint operator $a \in B(\cH)$, we set $\sigma(a)$ to be the spectrum of $a$.  Given $\xi \in \cH$, a Radon measure $\mu_{\xi}$ on the real line is induced by the formula $\mu_{\xi}(f) = \langle f(a)\xi,  \xi \rangle$ for any bounded continuous $f$.  We set $\cH_{ac}$ the Hilbert space of vectors $\xi$ where $\mu_{\xi}$ is absolutely continuous with respect to the Lebesgue measure, $\cH_{sc}$ the Hilbert space of vectors $\xi$ where $\mu_{\xi}$ is singular with respect to the Lebesgue measure, and $\cH_{pp}$ the Hilbert space of vectors $\xi$ where $\mu_{\xi}$ has purely atomic measure.  We define the absolutely continuous spectrum of $a$ as the spectrum of $a$ on $\cH_{ac}$ and denote it as  $\sigma_{ac}(a)$.  We define $\sigma_{cc}(a)$ and $\sigma_{pp}(a)$ in a similar manner.

We say $a$ has uniform multiplicity 1 if there is a measure $\mu$ on the spectrum of $a$ and a unitary $w: \cH \rightarrow L^{2}(\sigma(a), \mu)$ such that for any $f \in L^{2}(\sigma(a), \mu)$, $waw^{*}(f)(x) = xf(x)$.  In this case $\sigma_{ac}(a)$, $\sigma_{cc}(a)$ and $\sigma_{pp}(a)$ form a partition of $\sigma(a)$.  See \cite{MR1335452} for more details.

\begin{lem}

$y$ is unitary equivalent to $s + s^{*}$.

\end{lem}

\begin{proof}

We first show that $y$ has uniform multiplicity $1$.  To this end, consider the following sequence of polynomials:
\begin{align*}
&p_{0}(x) = 1 \\
&p_{1}(x) = x + \delta_{c}^{-1/2}\\
&p_{n+2}(x) = xp_{n+1} - p_{n}(x) \textrm{ for } n \geq 0
\end{align*}

By induction, it is straightforward to check that $p_{n}(y)(\xi_{0}) = \xi_{n}$ and hence the map $w: \cH \rightarrow L^{2}(\sigma(a), \mu_{\xi_{0}})$ given by $w(\xi_{n}) = p_{n}$ is a unitary satisfying $wyw^{*}(f)(x) = xf(x)$ for all $f \in L^{2}(\sigma(a), \mu_{\xi_{0}})$.  By the Kato-Rosenblum theorem (see \cite{MR1335452} p.540), if $a$ and $b$ are self adjoint operators and $b$ is trace class then the absolutely continuous parts of $b$ and $a+b$ are unitary equivalent.  This implies that the absolutely continuous parts of $y$ and $s + s^{*}$ are unitary equivalent.  The absolutely continuous spectrum of $s+s^{*}$ is $[-2, 2] = \sigma(s + s^{*})$ so the same must hold for $y$.  As was discussed in the previous lemma, $y$ has no spectral values outside $[-2, 2]$ and from above, the spectrum of $y$ is partitioned into $\sigma_{ac}$, $\sigma_{cc}$ and $\sigma_{pp}$.  Therefore $\sigma_{pp} = \sigma_{cc} = \emptyset$ and $\sigma_{ac} = \sigma$ showing the unitary equivalence.
\end{proof}

Therefore by conjugating by a unitary, we may assume that $L^{2}(M_{\alpha}) \cong L^{2}(A) \oplus \ell^{2}(\N) \otimes \cH \otimes \ell^{2}(\N)$ for some Hilbert space $\cH$ where on the second summand, $\pi\left(\frac{\cup - 1}{\sqrt{\delta_{c}}}\right)$ acts as $(s + s^{*}) \otimes 1 \otimes 1$ and $\rho\left(\frac{\cup - 1}{\sqrt{\delta_{c}}}\right)$ acts as $1 \otimes 1 \otimes (s+s^{*})$. The lemmas below follow \cite{MR2645882} but we supply proofs for the readers' convenience.

\begin{lem} \label{lem:shift}
$A' \cap M_{\alpha} = AP_{\overline{\alpha}\to\overline{\alpha}}\op$.
\end{lem}

\begin{proof}
First, note that $AP_{\overline{\alpha}\to\overline{\alpha}}\op\subseteq A'\cap M_\alpha$ in the obvious way
$$
\begin{tikzpicture} [baseline = -.1cm]
	\draw [thick, \cupcolor] (-1.4, .8) arc(180:360: .3cm);
	\draw [thick, \alphacolor] (-1.4, 0)--(.8, 0);
	\filldraw[thick, unshaded] (-.4, -.4) -- (-.4, .4) -- (.4, .4) -- (.4, -.4) -- (-.4, -.4);
	\node at (0,0) {$x$};
\end{tikzpicture}
=
\begin{tikzpicture} [baseline = -.1cm]
	\draw [thick, \cupcolor] (-.3, .8) arc(180:360: .3cm);
	\draw [thick, \alphacolor] (-.8, 0)--(.8, 0);
	\filldraw[thick, unshaded] (-.4, -.4) -- (-.4, .4) -- (.4, .4) -- (.4, -.4) -- (-.4, -.4);
	\node at (0,0) {$x$};
\end{tikzpicture}
=
\begin{tikzpicture} [baseline = -.1cm]
	\draw [thick, \cupcolor] (.8, .8) arc(180:360: .3cm);
	\draw [thick, \alphacolor] (-.8, 0)--(.8, 0);
	\filldraw[thick, unshaded] (-.4, -.4) -- (-.4, .4) -- (.4, .4) -- (.4, -.4) -- (-.4, -.4);
	\node at (0,0) {$x$};
\end{tikzpicture}
$$
where we need the ``op" since multiplication in the GJSW picture happens in the opposite order (see Note \ref{note:ReverseMultiplication}).

If $A' \cap M_{\alpha}$ were larger than $AP_{\overline{\alpha}\to\overline{\alpha}}\op$ then by looking at the orthogonal compliment of $AP_{\overline{\alpha}\to\overline{\alpha}}\op$ in $L^{2}(M_{\alpha})$, there is a nonzero vector $\xi \in \ell^{2}(\N) \otimes \ell^{2}(\N)$ with $((s+s^{*}) \otimes 1)\xi = (1 \otimes s+s^{*})\xi$.  Viewing $\ell^{2}(\N) \otimes \ell^{2}(\N)$ as the Hilbert Schmidt operators on $\ell^{2}(\N)$ this means $\xi(s+s^{*}) = (s+s^{*})\xi$ which is impossible since $s+s^{*}$ has no eigenvalues.
\end{proof}

We will realize $M_{0}$ as a unital subalgebra of $M_{\alpha}$ via the map $\iota_{\alpha}$.

\begin{lem} \label{lem:RelativeCommutant}
$M_{0}' \cap M_{\alpha} = P_{\overline{\alpha}\to\overline{\alpha}}\op\cong P_{\alpha\to \alpha}$ as an algebra.
\end{lem}

\begin{proof}

If $x \in M_{0}' \cap M_{\alpha}$ then by Lemma \ref{lem:shift}, $x \in AP_{\overline{\alpha}\to\overline{\alpha}}\op$, so we write $x$ as an $\ell^{2}$ sum
$$
x = \sum_{n=0}^{\infty}\frac{1}{\sqrt{\delta_{c}^{n}}}\cdot x_{n}
\star
\,
\begin{tikzpicture} [baseline = -.1cm]
	\draw [thick, \cupcolor] (-.8, .4) arc(180:360: .2cm);
	\draw [thick, \cupcolor] (.8, .4)  arc(0:-180: .2cm);
	\node at (0, .3) {{\scriptsize{$\cdots$}}};
	\node at (0,.1) {{\scriptsize{$n$}}};
	\draw [thick, \alphacolor] (-1, -.2)--(1, -.2);
	\draw [thick] (-1, -.4) -- (-1, .4) -- (1, .4) -- (1, -.4) -- (-1, -.4);
\end{tikzpicture}
\, ,
$$
where $x_{n} \in P_{\overline{\alpha}\to\overline{\alpha}}\op$ for all $n$.
Consider the following elements of $M_{0}$:
$$
z =
\begin{tikzpicture} [baseline = -.1cm]
	\draw [thick, \cupcolor] (-.3, .4) arc(180:360: .3cm);
	\draw [thick, \cupcolor] (-.15, .4) arc(180:360: .15cm);
	\draw [thick, \alphacolor] (-.4, -.1)--(.4, -.1);
	\draw[thick] (-.4, -.4) -- (-.4, .4) -- (.4, .4) -- (.4, -.4) -- (-.4, -.4);
\end{tikzpicture} \,,
l_{n}
=  \frac{1}{\sqrt{\delta_{c}^{n}}}\,
\begin{tikzpicture} [baseline = -.1cm]
	\draw [thick, \cupcolor] (-.8, .4) arc(180:360: .2cm);
	\draw [thick, \cupcolor] (.8, .4)  arc(0:-180: .2cm);
	\draw [thick, \cupcolor] (-1.8, .4) arc(180:360: .4cm);
	\draw [thick, \cupcolor](-1.6, .4) arc(180:360: .2cm);
	\draw [thick, \alphacolor] (-2, -.2)--(1, -.2);
	\node at (0, .3) {{\scriptsize{$\cdots$}}};
	\node at (0,.1) {{\scriptsize{$n$}}};
	\draw [thick] (-2, -.4) -- (-2, .4) -- (1, .4) -- (1, -.4) -- (-2, -.4);
\end{tikzpicture}\,,
\text{ and }
r_{n} = \frac{1}{\sqrt{\delta_{c}^{n}}} \,
\begin{tikzpicture} [baseline = -.1cm]
	\draw [thick, \cupcolor] (.4, .4) arc(-180:0: .2cm);
	\draw [thick, \cupcolor] (-.8, .4)  arc(-180:0: .2cm);
	\draw [thick, \cupcolor] (1, .4) arc(180:360: .4cm);
	\draw [thick, \cupcolor](1.2, .4) arc(180:360: .2cm);
	\draw [thick, \alphacolor] (2, -.2)--(-1, -.2);
	\node at (0, .3) {{\scriptsize{$\cdots$}}};
	\node at (0,.1) {{\scriptsize{$n$}}};
	\draw [thick] (2, -.4) -- (2, .4) -- (-1, .4) -- (-1, -.4) -- (2, -.4);
\end{tikzpicture}\,.
$$
By a direct diagrammatic computation, $z\star x - x\star z$ is the orthogonal sum
$$
\sum_{n=1}^{\infty} (x_{n} + \frac{1}{\sqrt{\delta_{c}}}\cdot x_{n+1})\star(l_{n} - r_{n}).
$$
For this to be zero, we must have $x_{n+1} = -\sqrt{\delta_{c}}\cdot x_{n}$ for $n \geq 1$.  Since the sum for $x$ must be $\ell^{2}$, this implies $x_{n} = 0$ for $n \geq 1$, i.e., $x \in P_{\overline{\alpha}\to\overline{\alpha}}\op$.

Obviously the rotation by 180 degrees gives the isomorphism $P_{\overline{\alpha}\to\overline{\alpha}}\op\cong P_{\alpha\to \alpha}$.
\end{proof}

\begin{proof} [Proof of Theorem \ref{thm:Factor}]
If $x$ were in the center of $M_{\alpha}$ then we know that $x \in P_{\overline{\alpha}\to\overline{\alpha}}\op$.  The element $x$ must commute with
$$
\begin{tikzpicture}[baseline = -.1cm]
	\draw [thick, \alphacolor] (-.4, .1) arc(-90:0: .3cm);
	\draw [thick, \alphacolor] (.4, .1) arc(-90:-180: .3cm);
	\draw[thick] (-.4, -.4) -- (-.4, .4) -- (.4, .4) -- (.4, -.4) -- (-.4, -.4);
\end{tikzpicture}\,,
$$
giving the equation
$$
\begin{tikzpicture} [baseline = -.1cm]
	\draw[thick, \alphacolor]  (-1,0) -- (-.4, 0) arc(-90:0: .3cm) -- (-.1,.4);
	\draw[thick, \alphacolor] (1, 0) -- (.4, 0) arc(-90:-180: .3cm)--(.1,.4);
	\draw[thick] (-1, -.4) -- (-1, .4) -- (1, .4) -- (1, -.4) -- (-1, -.4);
	\draw[thick, unshaded] (-.8, -.2) -- (-.8, .2) -- (-.4, .2) -- (-.4, -.2) -- (-.8, -.2);
	\node at (-.6, 0) {$x$};
\end{tikzpicture}
=
\begin{tikzpicture} [baseline = -.1cm]
	\draw[thick, \alphacolor]  (-1,0) -- (-.4, 0) arc(-90:0: .3cm) -- (-.1,.4);
	\draw[thick, \alphacolor] (1, 0) -- (.4, 0) arc(-90:-180: .3cm)--(.1,.4);
	\draw[thick] (-1, -.4) -- (-1, .4) -- (1, .4) -- (1, -.4) -- (-1, -.4);
	\draw[thick, unshaded] (.8, -.2) -- (.8, .2) -- (.4, .2) -- (.4, -.2) -- (.8, -.2);
	\node at (.6, 0) {$x$};
\end{tikzpicture}  \, .
$$
Joining the leftmost strings to the top implies that $x$ is a scalar multiple of the identity.
\end{proof}

\subsection{Jones' Towers Associated to $M_{0}$}\label{sec:towers}

If $\alpha$ and $\beta$ are in $\Lambda$ such that $\alpha = \beta\gamma $ with $\gamma \in \Lambda$ then we have a unital, trace-preserving inclusion $\iota: M_{\beta} \rightarrow M_{\alpha}$ given by
$$
\iota(x) =
\begin{tikzpicture}[baseline = -.1cm]
	\draw (0,0)--(0,.8);
	\draw (-.8,0)--(.8,0);
	\draw (-.8,-.6)--(.8,-.6);
	\filldraw[unshaded,thick] (-.4,.4)--(.4,.4)--(.4,-.4)--(-.4,-.4)--(-.4,.4);
	\node at (0,-.75) {{\scriptsize{$\gamma$}}};
	\node at (-.6,.2) {{\scriptsize{$\beta$}}};
	\node at (.6,.2) {{\scriptsize{$\beta$}}};
	\node at (0,0) {$x$};
\end{tikzpicture}\,.
$$
When we write $M_{\beta} \subset M_{\alpha}$ we mean that $M_{\beta}$ is included into $M_{\alpha}$ in the manner described above.  We have the following theorem:

\begin{thm} \label{thm:tower}
The following is a Jones' tower of factors:
$$
M_{0} \subset M_{\alpha} \subset M_{\alpha\overline{\alpha}} \subset \cdots \subset M_{(\alpha\overline{\alpha})^n\alpha} \subset M_{(\alpha\overline{\alpha})^{n+1}} \subset \cdots .
$$
Moreover $[M_{\alpha}:M_{0}] = \delta_{\alpha}^{2}$.
\end{thm}

\begin{proof}
This proof closely follows Section 4 in \cite{MR2645882}.  We will show that $M_{\alpha\overline{\alpha}}$ is the basic construction of $M_{0} \subset M_{\alpha}$.  The proof for higher steps in the tower is the same but with heavier notation.  To begin, set
$$
e_{0} =  \frac{1}{\delta_{\alpha}}
\begin{tikzpicture} [baseline = -.1cm]
	\draw [thick, \alphacolor] (-.4, .25) arc(90:-90: .25cm);
	\draw [thick, \alphacolor] (.4, .25) arc(90:270: .25cm);
	\draw[thick] (-.4, -.4) -- (-.4, .4) -- (.4, .4) -- (.4, -.4) -- (-.4, -.4);
\end{tikzpicture}
\in
P^{op}_{\alpha\overline{\alpha}\to \alpha\overline{\alpha}}\subset
M_{\alpha\overline{\alpha}}
$$
then $e_{0}$ is a projection.  It is a straightforward diagrammatic computation to show that if $x \in M_{\alpha}$ then $e_{0}xe_{0} = E_{M_{0}}(x)$ with $E_{M_{0}}$ the trace preserving conditional expectation.

 We now claim that $(M_{\alpha}, e_{0})''$ is a $II_{1}$ factor.  Indeed, if $y$ were in the center of $(M_{\alpha}, e_{0})''$ it would have to commute with $M_{0}$, implying that $y \in P_{\alpha\overline{\alpha}\to\alpha\overline{\alpha}}$.  The element $y$ also has to commute with $M_{\alpha}$, in particular it has to commute with the element
$$
\begin{tikzpicture}[baseline = -.1cm]
	\draw [thick, \alphacolor] (-.4, -.1) -- (.4, -.1);
	\draw [thick, \alphacolor] (-.4, .1) arc(-90:0: .3cm);
	\draw [thick, \alphacolor] (.4, .1) arc(-90:-180: .3cm);
	\node at (0, -.25) {\scriptsize{$\overline{\alpha}$}};
	\draw[thick] (-.4, -.4) -- (-.4, .4) -- (.4, .4) -- (.4, -.4) -- (-.4, -.4);
\end{tikzpicture}\,.
$$
This implies $y$ must be of the form
$$
\begin{tikzpicture} [baseline = 0cm]
	\draw [thick, \alphacolor] (-1, .3) -- (1, .3);
	\node at (0, -.4) {$y$};
	\draw [thick, \alphacolor] (-1, -.4) -- (-.4, -.4);
	\draw [thick, \alphacolor] (.4, -.4) -- (1, -.4);
	\node at (.7, -.6) {\scriptsize{$\overline{\alpha}$}};
	\node at (-.7, -.6) {\scriptsize{$\overline{\alpha}$}};
	\draw[thick] (-.4, -.8) -- (-.4, 0) -- (.4, 0) -- (.4, -.8) -- (-.4, -.8);
\end{tikzpicture}\,,
$$
but commuting with $e_{0}$ forces this diagram to be a scalar multiple of the identity.

One then observes that if $z \in (M_{\alpha}, e_{0})''$, then $ze_{0} = (\delta^{\alpha})^{2}E_{M_{\alpha}}(ze_{0})e_{0}$.  This is done by realizing that as a von Neumann algebra, $(M_{\alpha}, e_{0})''$ is generated by $M_{\alpha}$ and $M_{\alpha}e_{0}M_{\alpha}$ and hence one can assume $z$ is in either of these spaces.  The equality then becomes a straightforward diagrammatic check.  From this, one can deduce that the map $x \mapsto \delta^{u}xe_{0}$ from $M_{\alpha}$ to $(M_{\alpha}, e_{0})''$ is a surjective isometry intertwining $E_{M_{0}}$ on $L^{2}(M_{\alpha})$ and left multiplication by $e_{0}$.  From this, we deduce that $(M_{\alpha}, e_{0})''$ is the basic construction of $M_{0} \subset M_{\alpha}$ and evaluating $[M_{\alpha}:M_{0}]$ is a matter of calculating the trace of $e_{0}$.

The same arguments applied to $M_{\alpha} \subset M_{\alpha\overline{\alpha}}$ implies that $[M_{\alpha\overline{\alpha}}: M_{\alpha}] = (\delta^{\alpha})^{2}$ hence $[M_{\alpha\overline{\alpha}}: (M_{\alpha}, e_{0})''] = 1$ i.e. $M_{\alpha\overline{\alpha}} = (M_{\alpha}, e_{0})''$.
\end{proof}

As an aside, since $M_{0}' \cap M_{\beta} = P_{\beta\to\beta}$ for any $\beta$, it follows that the sequence of vector spaces $P_{0}, \, P_{\alpha\to\alpha}, \, P_{\alpha\overline{\alpha}\to\alpha\overline{\alpha}},\cdots$ forms a subfactor planar algebra.

\subsection{Some bifinite bimodules over $M$} \label{sec:MoreBimodules}

\begin{nota}
We now use the notation $M$ instead of $M_0$ as it makes the rest of this section easier on the eyes.
\end{nota}

In this subsection and the next, we define a category $\cC_{bim}$ of bifinite bimodules over $M$.  To begin, we use $\cP$ to construct another von Neumann algebra $\cM_{\alpha, \beta}$ which contains the factors $M_{\alpha}$ and $M_{\beta}$ as corners/cut-downs.  Let $\Gr_{\alpha, \beta}(P) = \bigoplus_{\gamma \in \Lambda} P_{\overline{\alpha}\gamma\beta}$.  If $x \in P_{\overline{\alpha}\gamma\beta}$, we view $x$ as
$$
\begin{tikzpicture}[baseline=-.1cm]
	\draw [thick, \alphacolor]  (-.8,0)--(.4,0);
	\draw [thick, \betacolor] (.4, 0)--(.8, 0);
	\draw (0,.4)--(0,.8);
	\draw[thick, unshaded] (-.4, -.4) -- (-.4, .4) -- (.4, .4) -- (.4, -.4) -- (-.4, -.4);
	\node at (0, 0) {$x$};
	\node at (.2,.6) {\scriptsize{$\gamma$}};
\end{tikzpicture}\,.
$$
where we now draw \textcolor{\betacolor}{\betacolor} strings for strings labelled $\beta$.
There is a sesquilinear form $\langle \cdot, \cdot \rangle$ on $Gr_{\alpha, \beta}(P)$ given by
$$
\langle x, y \rangle =
\begin{tikzpicture} [baseline = .3cm]
	\draw [thick, \alphacolor] (-.4, 0) arc(90:270: .4cm) -- (1.6, -.8) arc(-90:90: .4cm);
	\draw [thick, \betacolor]( .4, 0)--(.8, 0);
	\draw (0,.4)--(0,.8);
	\draw (1.2,.4)--(1.2,.8);
	\draw[thick, unshaded] (-.4,.8) -- (-.4, 1.6) -- (1.6,1.6) -- (1.6,.8) -- (-.4,.8);
	\draw[thick, unshaded] (-.4, -.4) -- (-.4, .4) -- (.4, .4) -- (.4, -.4) -- (-.4, -.4);
	\draw[thick, unshaded] (.8, -.4) -- (.8, .4) -- (1.6, .4) -- (1.6, -.4) -- (.8, -.4);
	\node at (0, 0) {$x$};
	\node at (1.2, 0) {$y^{*}$};
	\node at (.6, 1.2) {$\sum CTL$};
\end{tikzpicture} \, .
$$
As a vector space, $\Gr_{\alpha, \alpha}(\cP) = \Gr_{\alpha}(\cP)$.  For $\alpha \neq \beta$, we form the $*$-algebra $\cG_{\alpha, \beta}$ which is generated by the vector spaces $\Gr_{\alpha}(\cP)$, $\Gr_{\beta}(\cP)$, $\Gr_{\alpha, \beta}(\cP)$ and $\Gr_{\beta, \alpha}(\cP)$, under the multiplication
$$
\begin{tikzpicture}[baseline=-.1cm]
	\draw (-.8,0)--(.8,0);
	\draw (0,.4)--(0,.8);
	\draw[thick, unshaded] (-.4, -.4) -- (-.4, .4) -- (.4, .4) -- (.4, -.4) -- (-.4, -.4);
	\node at (0, 0) {$x$};
	\node at (-.6, .2) {\scriptsize{$\kappa$}};
	\node at (.2,.6) {\scriptsize{$\gamma$}};
	\node at (.6, .2) {\scriptsize{$\theta$}};
\end{tikzpicture}
\wedge
\begin{tikzpicture}[baseline=-.1cm]
	\draw (-.8,0)--(.8,0);
	\draw (0,.4)--(0,.8);
	\draw[thick, unshaded] (-.4, -.4) -- (-.4, .4) -- (.4, .4) -- (.4, -.4) -- (-.4, -.4);
	\node at (0, 0) {$x$};
	\node at (-.6, .2) {\scriptsize{$\omega$}};
	\node at (.2,.6) {\scriptsize{$\gamma'$}};
	\node at (.6, .2) {\scriptsize{$\chi$}};
\end{tikzpicture}
= \delta_{\omega,\theta}
\begin{tikzpicture} [baseline = -.1cm]
	\draw (0,.4)--(0,.8);
	\draw (1.2,.4)--(1.2,.8);
	\draw (-.8,0)--(2,0);
	\draw[thick, unshaded] (-.4, -.4) -- (-.4, .4) -- (.4, .4) -- (.4, -.4) -- (-.4, -.4);
	\draw[thick, unshaded] (.8, -.4) -- (.8, .4) -- (1.6, .4) -- (1.6, -.4) -- (.8, -.4);
	\node at (0, 0) {$x$};
	\node at (1.2, 0) {$y^{*}$};
	\node at (-.6,.2) {\scriptsize{$\kappa$}};
	\node at (.2,.6) {\scriptsize{$\gamma$}};
	\node at (.6,.2) {\scriptsize{$\theta$}};
	\node at (1.4,.6) {\scriptsize{$\gamma'$}};
	\node at (1.8,.2) {\scriptsize{$\chi$}};
\end{tikzpicture}\,,
 $$
 where $\kappa,\theta,\omega,\chi\in\{\alpha,\beta\}$. There is also a (non-normalized) trace on $\cG_{\alpha, \beta}$ given by
$$
\tr(x) =
\begin{tikzpicture}[baseline=.3cm]
	\draw (0,0)--(0,.8);
	\draw (.4,0) arc (90:-90:.4cm) -- (-.4,-.8) arc (270:90:.4cm);
	\filldraw[unshaded,thick] (-.4,.4)--(.4,.4)--(.4,-.4)--(-.4,-.4)--(-.4,.4);
	\draw[thick, unshaded] (-.7, .8) -- (-.7, 1.6) -- (.7, 1.6) -- (.7,.8) -- (-.7, .8);
	\node at (0,0) {$x$};
	\node at (0,1.2) {$\Sigma CTL$};
\end{tikzpicture}
$$
if $x \in \Gr_{\alpha}(\cP)$ or $\Gr_{\beta}(\cP)$, and is zero otherwise.  Just as in the case for the algebras $\Gr_{\alpha}(\cP)$ one can show (by orthogonalizing) that the trace is positive definite and that multiplication is bounded.  Therefore, one can form the von Neumann algebra $\cM_{\alpha, \beta} = \cG_{\alpha, \beta}''$ acting on $L^{2}(\cG_{\alpha, \beta})$ by left and right multiplication.  Set
$$
p_{\alpha} =
\begin{tikzpicture} [baseline = -.1cm]
	\draw [thick, \alphacolor] (-.4, 0) -- (.4, 0);
	\node at (0, .2) {\scriptsize{$\alpha$}};
	\draw[thick] (-.4,.4)--(.4,.4)--(.4,-.4)--(-.4,-.4)--(-.4,.4);
\end{tikzpicture}
\text{ and }
p_{\beta} =
\begin{tikzpicture} [baseline = -.1cm]
	\draw [thick, \betacolor] (-.4, 0) -- (.4, 0);
	\node at (0, .2) {\scriptsize{$\beta$}};
	\draw[thick] (-.4,.4)--(.4,.4)--(.4,-.4)--(-.4,-.4)--(-.4,.4);
\end{tikzpicture} \, .
$$
We see that $p_{\gamma}\cM_{\alpha, \beta}p_{\gamma} = M_{\gamma}$ (with the non-normalized trace) for $\gamma = \{\alpha,\beta\}$ so $L^{2}(\cM_{\alpha, \beta})$ is naturally an $M_{\alpha} \oplus M_{\beta}$ bimodule.  Hence we may consider $L^{2}(\cM_{\alpha, \beta})$ as an $M-M$ bimodule via the embedding $x \mapsto \iota_{\alpha}(x) \oplus \iota_{\beta}(x) \in M_{\alpha} \oplus M_{\beta}$.  Under this identification, we define $\cH_{\alpha, \beta}$ to be the the $M-M$ bimodule
$$
\cH_{\alpha, \beta} = p_{\alpha}\wedge L^{2}(\cM_{\alpha, \beta})\wedge p_{\beta}.
$$

As mentioned above, we can give $\cG_{\alpha,\beta}$ an orthogonalized inner product and multiplication exactly as in Section \ref{sec:graded}.  We use the notation $\cF_{\alpha,\beta}$ to denote the vector space $\cG_{\alpha,\beta}$ with the orthogonalized inner product and multiplication $\star$.  The vector space $p_{\alpha} \star L^{2}(\cF_{\alpha,\beta})\star p_{\beta}$ is naturally an $M-M$ bimodule which is isomorphic to $p_{\alpha}\wedge L^{2}(\cM_{\alpha, \beta})\wedge p_{\beta}$.  Our first lemma is proven by making use of the orthogonal picture:

\begin{lem} \label{lem:central}
The vector space of $M-M$ central vectors of $p_{\alpha} \star L^{2}(\cF_{\alpha,\beta}) \star p_{\beta}$ is the vector space $P_{\overline{\alpha} \to \overline{\beta}}$ (rotating the GJS diagrams 90 degrees clockwise).
\end{lem}

\begin{proof}
As in Section \ref{sec:factor}, we let $A$ be the von Neumann subalgebra of $M$ generated by $\bluecup$ and let $W$ be the $A-A$ bimodule generated by $P_{\overline{\alpha} \to \overline{\beta}}$. With the same approach as \ref{sec:factor}, we see that as an $A-A$ bimodule, we have
$$
\cH_{\alpha, \beta} \cong W \bigoplus \left(\ell^{2}(\N) \otimes \cH \otimes \ell^{2}(\N)\right)
$$
for $\cH$ an auxiliary Hilbert space.  The operator $\bluecup$ acts on the left of the second factor by $(s + s^{*}) \otimes 1 \otimes 1$  and on the right by $1 \otimes 1 \otimes (s + s^{*})$.  Therefore the proof of Lemma \ref{lem:shift} applies here and we see that the $A-A$ central vectors of $L^{2}(\cF_{\alpha, \beta})$ are exactly $W$.  To finish the proof, one repeats the argument in Lemma \ref{lem:RelativeCommutant}.
\end{proof}

\begin{cor} \label{cor:central}
The vector space of $M-M$ central vectors of $p_{\alpha} \wedge L^{2}(\cF_{\alpha,\beta}) \wedge p_{\beta}$ is the vector space $P_{\overline{\alpha} \to \overline{\beta}}$ (rotating the GJS diagrams 90 degrees clockwise).
\end{cor}

For the rest of this section, we use the non-orthogonalized picture for $\cH_{\alpha, \beta}$.  We first write down some straightforward isomorphisms between these bimodules:

 \begin{lem} \label{lem:rotate}
 Let $\alpha, \beta, \gamma \in \Lambda$. Then as $M-M$ bimodules, $ \cH_{\alpha\beta,\gamma}\cong \cH_{\alpha, \overline{\beta}\gamma}$.
\end{lem}
 \begin{proof}
 The map
 $$
\begin{tikzpicture}[baseline=-.1cm]
	\draw (0,0)--(.8,0);
	\draw[thick, \alphacolor] (-.8,.2)--(0,.2);
	\draw[thick, \betacolor] (-.8,-.2)--(0,-.2);
	\draw (0,.4)--(0,.8);
	\draw[thick, unshaded] (-.4, -.4) -- (-.4, .4) -- (.4, .4) -- (.4, -.4) -- (-.4, -.4);
	\node at (0, 0) {$x$};
	\node at (.6, .2) {\scriptsize{$\gamma$}};
\end{tikzpicture}
 \longmapsto
\begin{tikzpicture}[baseline=-.1cm]
	\draw (0,0)--(.8,0);
	\draw[thick, \alphacolor] (-.8,.2)--(0,.2);
	\draw[thick, \betacolor] (-.4,-.2) arc(90:270:.2cm)--(.8,-.6);
	\draw (0,.4)--(0,.8);
	\draw[thick, unshaded] (-.4, -.4) -- (-.4, .4) -- (.4, .4) -- (.4, -.4) -- (-.4, -.4);
	\node at (0, 0) {$x$};
	\node at (.6, .2) {\scriptsize{$\gamma$}};
\end{tikzpicture}
$$
is a unitary operator intertwining the left and right $M$-actions.
\end{proof}

\begin{defn}
By Lemma \ref{lem:rotate}, we now define the bimodule $H_\alpha:=\cH_{\alpha,\emptyset}\cong \cH_{\emptyset,\overline{\alpha}}$. We draw elements of $H_\alpha$ with strings emanating from the bottom instead of the sides
$$
\ColorNBox{\gamma}{x}{\alphacolor}\,,
$$
and the left and right $M$-actions are given by the obvious diagrams. For $\xi\in p_{\alpha}\wedge \cM_{\alpha, \emptyset}\wedge p_{\emptyset}\cong p_{\emptyset} \wedge \cM_{\emptyset,\overline{\alpha}}\wedge p_{\overline{\alpha}}$ and $x,y\in M$,
$$
x\xi y =
\Mbox{}{x}
\,
\ColorNBox{}{\xi}{\alphacolor}
\,
\Mbox{}{y}
\,.
$$
The inner product on $H_\alpha$ is given by
\begin{equation}\label{eqn:InnerProduct}
\langle x, y \rangle_{H_\alpha} =
\begin{tikzpicture} [baseline = .3cm]
	\draw (0,.8)--(0, -.4);
	\draw (1.2,-.4) -- (1.2,.8);
	\draw[thick, \alphacolor] (0, -.4) .. controls ++(270:.4cm) and ++(270:.4cm) .. (1.2,-.4);
	\draw[thick, unshaded] (-.4,.8) -- (-.4, 1.6) -- (1.6,1.6) -- (1.6,.8) -- (-.4,.8);
	\draw[thick, unshaded] (-.4, -.4) -- (-.4, .4) -- (.4, .4) -- (.4, -.4) -- (-.4, -.4);
	\draw[thick, unshaded] (.8, -.4) -- (.8, .4) -- (1.6, .4) -- (1.6, -.4) -- (.8, -.4);
	\node at (0, 0) {$\overline{y^*}$};
	\node at (1.2, 0) {$x$};
	\node at (.6, 1.2) {$\sum CTL$};
	\node at (1.4,-.6) {\scriptsize{$\alpha$}};
	\node at (-.2,-.6) {\scriptsize{$\overline{\alpha}$}};
\end{tikzpicture}
\end{equation}
where $\overline{y^*}$ is the rotation of $y^*$ by 180 degrees.
\end{defn}

Our goal will now be to show that $H_\alpha \otimes_M H_\beta\cong H_{\alpha\beta}$. In diagrams:
$$
\ColorNBox{}{}{\alphacolor}
\,
\ColorNBox{}{}{\betacolor}
\leftrightarrow
\begin{tikzpicture}[baseline=-.1cm]
	\draw (0,0)--(0,.8);
	\draw[thick,\alphacolor] (-.2,-.8)--(-.2,0);
	\draw[thick,\betacolor] (.2,-.8)--(.2,0);
	\draw[thick, unshaded] (-.4, -.4) -- (-.4, .4) -- (.4, .4) -- (.4, -.4) -- (-.4, -.4);
	\node at (0, 0) {};
\end{tikzpicture}
$$
We will need the following lemma, whose purpose is to ``close up" the space between the boxes.

\begin{lem} \label{lem:cup}
The $\alpha-$cup element
$\,
\begin{tikzpicture} [baseline=-.1cm]
	\filldraw[unshaded,thick] (-.4,.4)--(.4,.4)--(.4,-.4)--(-.4,-.4)--(-.4,.4);
	\draw[thick, \alphacolor] (-.3, .4) arc(180:360: .3cm);
\end{tikzpicture}
$
is positive and invertible in $M$.
\end{lem}
\begin{proof}
We induct on the size of $\alpha$.  If $|\alpha| = 1$ then it follows from \cite{MR2732052} that the operator is a free-poisson element whose spectrum is supported away from 0 hence the lemma holds for this case.  Suppose that the lemma holds for some color pattern $\beta$ and suppose $\alpha = \beta c$ for a fixed color $c$, denoted by a \textcolor{\cupcolor}{green} string.  It follows that the element
$$
x =
\begin{tikzpicture} [baseline=-.1cm]
	\filldraw[unshaded,thick] (-.4,.4)--(.4,.4)--(.4,-.4)--(-.4,-.4)--(-.4,.4);
	\draw[thick, \betacolor] (-.2, .4) arc (180:360: .2cm);
	\draw[thick, \cupcolor] (-.4,-.15)--(.4,-.15);
\end{tikzpicture}
$$
is positive invertible in $M_{c}$.  Let $z \in \cM_{\emptyset,c}$ be the element
$$
z =
\begin{tikzpicture} [baseline=-.1cm]
	\filldraw[unshaded,thick] (-.4,.4)--(.4,.4)--(.4,-.4)--(-.4,-.4)--(-.4,.4);
	\draw[thick, \cupcolor] (-.4, 0) arc(-90:0: .4cm);
\end{tikzpicture}
$$
Then the $\alpha-$cup element has the form $z^{*}\wedge x\wedge z \in \cM_{\emptyset, c}$ so it immediately follows that it is positive.  Note that $z^{*}\wedge z$ is invertible so there is a positive constant $k_{z}$ so that $\langle z^{*}\wedge z \wedge \xi, \xi \rangle \geq k_{z}\|\xi\|_{2}^{2}$ for all $\xi \in L^{2}(M)$.  Letting $r$ be the positive square root of $x$, there is a strictly positive constant $k_{r}$ so that for all $\eta \in L^{2}(M_{c})$, $\langle r\wedge\eta, r\wedge\eta \rangle \geq k_{r} \|\eta\|_{2}^{2}$.  Therefore, for all $\xi \in L^{2}(M)$,
\begin{align*}
\langle z^{*}\wedge x\wedge z \wedge \xi, \xi \rangle
&= \langle r\wedge z \wedge \xi, r\wedge z \wedge \xi \rangle
\geq k_{r} \langle z\wedge \xi, z\wedge \xi \rangle \\
 &= k_{r} \langle z^{*}\wedge z \wedge\xi, \xi \rangle \geq k_{r}k_{z}\|\xi\|_{2}^{2}
\end{align*}
implying invertibility.
\end{proof}

\begin{defn}
Recall from \cite{MR561983,MR1424954,bimodules} that given a bimodule $\sb{M}K_M$, a vector $\xi\in K$ is called \underline{right $M$-bounded} if the map $L(\xi)^0\colon M\to K$ given by $m\mapsto \xi m$ extends to a bounded linear operator $L(\xi)\colon L^2(M)\to K$. There is a similar definition of a left $M$-bounded vector, and if $K$ is bifinite, then the sets of left and right $M$-bounded vectors agree. For such a bimodule $K$, we denote the set of left/right $M$-bounded vectors by $D(K)$.
\end{defn}

\begin{ex}
For all $\alpha,\beta\in\Lambda$, we have
$D(H_{\alpha\beta})=D(\cH_{\alpha,\overline{\beta}})=p_\alpha \wedge\cM_{\alpha,\overline{\beta}}\wedge p_{\overline{\beta}}$.
\end{ex}

\begin{lem} \label{lem:TensorProduct}
As $M - M$ bimodules, the Connes' fusion $H_\alpha\otimes_M H_\beta\cong H_{\alpha\beta}$.
\end{lem}
\begin{proof}
By  \cite{MR1424954}, $H_\alpha\otimes_M H_\beta$ is the completion of the algebraic tensor product
$
D(H_\alpha) \odot D(H_\beta)
$
(first modding out by vectors of length zero) under the semi-definite inner product given by
$$
\langle \xi_{1} \otimes \eta_{1}, \xi_{2} \otimes \eta_{2} \rangle_{H_\alpha \otimes_M H_\beta}
= \langle  \langle \xi_2|\xi_1 \rangle_M \cdot \eta_{1}, \eta_{2} \rangle_{H_\beta}
$$
where $\langle \xi_2|\xi_1 \rangle_M$ is the unique element in $M$ satisfying
$$
\langle \xi_{1}x, \xi_{2} \rangle_{H_\alpha} = \tr(x\langle \xi_2|\xi_1 \rangle_M) \text{ for all }x\in M.
$$
We define a map $u$ by the linear extension of
\begin{align*}
D(H_\alpha) \odot D(H_\beta)
&\longrightarrow \cH_{\alpha,\overline{\beta}}\cong H_{\alpha\beta}
\text{ by}\\
\ColorNBox{}{\xi}{\alphacolor}
\otimes
\ColorNBox{}{\eta}{\betacolor}
&\longmapsto
\ColorNBox{}{\xi}{\alphacolor}
\,
\ColorNBox{}{\eta}{\betacolor}
\,.
\end{align*}
Obviously the map $u$ intertwines the left and right $M$-actions and is $M$-middle linear. We will show it is isometric with dense range, and thus $u$ has a unique extension to a $M-M$ bilinear unitary giving the desired isomorphism.

First, it is not hard to check directly that the $M$-valued inner product of the $M$-bounded vectors $\xi_{1}, \xi_{2} \in D(H_\alpha)$ is given by
$$
\langle \xi_2|\xi_1\rangle_M
=
\begin{tikzpicture}[baseline=-.1cm]
	\draw (0,.6)--(0, -.4);
	\draw[thick, \alphacolor] (0, -.4) .. controls ++(270:.4cm) and ++(270:.4cm) .. (1.2,-.4);
	\draw (1.2,-.4) -- (1.2,.6);
	\draw[thick, unshaded] (1.6, -.4) -- (1.6, .4) -- (.8, .4) -- (.8, -.4) -- (1.6, -.4);
	\draw[thick, unshaded] (-.4, -.4) -- (-.4, .4) -- (.4, .4) -- (.4, -.4) -- (-.4, -.4);
	\node at (1.2, 0) {$\xi_1$};
	\node at (-.2,-.6) {\scriptsize{$\overline{\alpha}$}};
	\node at (1.4,-.6) {\scriptsize{$\alpha$}};
	\node at (0, 0) {$\overline{\xi_2^*}$};
\end{tikzpicture}
\, .
$$
Hence if
$$
\zeta = \sum_{i=1}^{n} \xi_{i} \otimes \eta_{i} \in
D(H_\alpha) \odot D(H_\beta),
$$
we immediately have
\begin{equation}\label{eqn:InnerProductForTensorProduct}
\|\zeta\|_{2}^{2}
= \sum_{i,j=1}^n \langle  \langle \xi_j|\xi_i \rangle_M \cdot \eta_{i}, \eta_{j} \rangle_{H_\beta}
=
\begin{tikzpicture} [baseline = .3cm]
	\draw (0,.8)--(0, -.4);
	\draw (1.2,-.4) -- (1.2,.8);
	\draw[thick, \alphacolor] (0, -.4) .. controls ++(270:.4cm) and ++(270:.4cm) .. (1.2,-.4);
	\draw (-1.2,.8)-- (-1.2,-.4);
	\draw (2.4,-.4)--(2.4,.8);
	\draw[thick, \betacolor] (-1.2,-.4) .. controls ++(270:1cm) and ++(270:1cm) .. (2.4,-.4);
	\draw[thick, unshaded] (-1.6,.8) -- (-1.6, 1.6) -- (2.8,1.6) -- (2.8,.8) -- (-1.6,.8);
	\draw[thick, unshaded] (-.4, -.4) -- (-.4, .4) -- (.4, .4) -- (.4, -.4) -- (-.4, -.4);
	\draw[thick, unshaded] (-.8, -.4) -- (-.8, .4) -- (-1.6, .4) -- (-1.6, -.4) -- (-.8, -.4);
	\draw[thick, unshaded] (.8, -.4) -- (.8, .4) -- (1.6, .4) -- (1.6, -.4) -- (.8, -.4);
	\draw[thick, unshaded] (2, -.4) -- (2, .4) -- (2.8, .4) -- (2.8, -.4) -- (2, -.4);
	\node at (-1.2, 0) {$\overline{\eta_j^*}$};
	\node at (0, 0) {$\overline{\xi_j^*}$};
	\node at (1.2, 0) {$\xi_i$};
	\node at (2.4, 0) {$\eta_i$};
	\node at (.6, 1.2) {$\sum CTL$};
\end{tikzpicture} \, .
\end{equation}
Since
$$
u(\zeta) =
\sum_{i=1}^n
\ColorNBox{}{\xi_i}{\alphacolor}
\,
\ColorNBox{}{\eta_i}{\betacolor}
\,,
$$
we see that $\|u(\zeta)\|_{2}^{2}$ is the same diagram as in \eqref{eqn:InnerProductForTensorProduct}.  Hence $u$ can be extended to be an isometry on $H_{\alpha}\otimes_M H_\beta$.

For the rest of this proof we will only use the $\wedge$ multiplication, so to clean up notation we will omit the $\wedge$. We must show the image of $u$ is dense in $H_{\alpha\beta}$. We may assume $\delta_\alpha\leq \delta_\beta$, and the proof is similar if $\delta_\beta <\delta_\alpha$.
Set
$$
w =
\begin{tikzpicture} [baseline=-.1cm]
	\filldraw[unshaded,thick] (-.4,.4)--(.4,.4)--(.4,-.4)--(-.4,-.4)--(-.4,.4);
	\draw[thick,\alphacolor] (-.4, 0) arc(-90:0: .4cm);
\end{tikzpicture}\,.
$$
Then $w^{*}w$ is positive and invertible in $M$, so if $v$ is the polar part of $w$, then
$v \in D(H_\alpha)$
with right support $p_{\emptyset}$ and left support $e$ under $p_{\alpha}$.
Choose partial isometries $x_{1}, ..., x_{n}$ in $M_{\alpha}$ such that $x_{1}, ..., x_{n-1}$ have right support $e$ and orthogonal left supports.
Choose the partial isometry $x_{n}$ so that its left support is $1 - \sum_{i=1}^{n-1}x_{i}x_{i}^{*}$ and its right support is under $e$.

Similarly, there is a partial isometry $\tilde{v} \in D(H_\beta)$ whose right support is $f \leq p_{\overline{\beta}}$ and whose left support is $p_{\emptyset}$.
Choose partial isometries $y_{1}, ..., y_{n} \in M_{\overline{\beta}}$ such that $y_{1}, ..., y_{n-1}$ have left support $f$ and orthogonal right supports.
Choose the partial isometry $y_{n}$ so its right support is orthogonal to $\sum_{i=1}^{n-1}y_{i}^{*}y_{i}$ and its left support is the right support of $x_{n}v\tilde{v}$ (we can do this since $\delta_{\alpha} \leq \delta_{\overline{\beta}}$).

It now follows that $z = \sum_{i=1}^{n}x_{i}v\tilde{v}y_{i}$ is a partial isometry in $D(H_{\alpha\beta})$ with full left support $p_{\alpha}$.  Note that every element, $r$ in $D(H_{\alpha\beta})$ is of the form $zx$ for $x \in M_{\overline{\beta}}$ since we simply choose $x = z^{*}r$.  Since $zx = \sum_{i=1}^{n}(x_{i}v)\cdot(\tilde{v}y_{i}x)$ is in the image of $u$, we are finished.
\end{proof}

\begin{lem} \label{lem:bifinite}
As an $M-M$ bimodule, $H_\alpha$ is bifinite. Thus by Lemma \ref{lem:rotate}, so is $\cH_{\alpha,\beta}$.
\end{lem}

\begin{proof}
By the proof of Lemma \ref{lem:TensorProduct}, there is a partial isometry $v \in p_{\alpha}\wedge \cM_{\alpha, \emptyset}\wedge p_{\emptyset}$ with right support $p_{\emptyset}$.  The argument at the end of Lemma \ref{lem:TensorProduct} shows that every element of $p_{\alpha}\wedge\cM_{\alpha, \emptyset}\wedge p_{\emptyset}$ is of the form $x\wedge v$ for $x \in M_{\alpha}$.  Therefore $v$ is cyclic for the left action of $M_{\alpha}$ on $\cH_{\alpha,\emptyset}\cong H_\alpha$, and since $[M_{\alpha}:M] < \infty$, $\dim_{M-}(H_{\alpha})<\infty$. Similarly $\dim_{-M}(H_{\alpha})<\infty$.
\end{proof}

\subsection{The categories $\Bim(\cP)$ and $\CF(\cP)$} \label{sec:CategoriesOfBimodules}

We now define two rigid $C^*$-tensor categories $\Bim(\cP)$ and $\CF(\cP)$ whose objects are bifinite $M-M$ bimodules. We show given a factor planar algebra $\cP$, we have equivalences $\Pro(\cP)\cong \Bim(\cP)\cong \CF(\cP)$.

\begin{defn}
If $p$ is a projection in $P_{\alpha\to \alpha}$, we define $H_p = \cH_{\emptyset, \overline{\alpha}}\wedge p$, which is a bifinite $M-M$ bimodule. Note that elements of $H_p$ are obtained from elements in $H_\alpha$ by putting a $p$ on the bottom. Thus the linear span of elements of the form
$$
\ColorMultiply{x}{p}{\alphacolor}
$$
where $x\in p_{\emptyset}\wedge \cM_{\emptyset, \overline{\alpha}}\wedge p$ forms a dense subset of $H_p$. Recall that elements of $M$ act on the left and right as in the non-orthogonal picture. These actions clearly do not affect the $p$ on the bottom.
\end{defn}

\begin{defn}\label{defn:Bim}
Let $\Bim(\cP)$ (abbreviated $\Bim$) be the strict rigid $C^*$-tensor category defined as follows.
\begin{enumerate}
\item[\text{\underline{Objects:}}] The objects of $\Bim$ are finite direct sums of the bimodules $H_p$ for the projections $p\in P_{\alpha\to \alpha}$ for $\alpha\in \Lambda$. Note that the unit object $1_{\Bim}=H_\emptyset\cong L^2(M)$.
\item[\text{\underline{Tensor:}}]
For $p\in P_{\alpha\to \alpha}$ and $q\in P_{\beta\to \beta}$, we define $H_p\otimes_{\Bim} H_q=H_{p\otimes q}$ where $p\otimes q$ is the tensor product in $\cP$. The tensor product is extended to direct sums linearly.

Note that since the tensor product of projections in $\cP$ is strict, so is the tensor product in $\Bim$.
\item[\text{\underline{Morphisms:}}]
For $p\in P_{\alpha\to \alpha}$ and $q\in P_{\beta\to \beta}$, we define $\Bim(H_p\to H_q)=qP_{\alpha\to \beta}p$ and composition is the usual composition in $\cP$. Morphisms between direct sums are matrices of such maps.
 \item[\text{\underline{Tensoring:}}]
For $x_i\in \Bim(H_{p_i}\to H_{q_i})=q_iP_{\alpha_i\to \beta_i}p_i$ for $i=1,2$, we define $x_1\otimes_{\Bim} x_2$ as the tensor product of morphisms in $\cP$. Similarly for matrices of such maps.
\item[\text{\underline{Duality:}}]
The dual of $H_p$ is $H_{\overline{p}}$. The evaluation map $\ev_{H_p}\colon H_{\overline{p}}\otimes_{\Bim} H_p \to H_\emptyset=1_{\Bim}$ is given by the $\alpha$-cup with projections $\overline{p},p$ on top:
$$
\MInnerProduct{\overline{p}}{\alphacolor}{p}
$$
Of course, the coevalutation map $1_{\Bim}\to H_p\otimes H_{\overline{p}}$ is given by the adjoint of $\ev_{H_{\overline{p}}}$, which is the $\alpha$-cap with projections underneath:
$$
\MInnerProductOp{p}{\alphacolor}{\overline{p}}\,.
$$
One easily checks that the necessary relations hold.

The dual map is extended to direct sums linearly.
\item[\text{\underline{Adjoint:}}]
The adjoint map $*$ is the identity on all objects, and the adjoint of a morphism $x\in \Bim(H_p\to H_q)=qP_{\alpha\to \beta}p$ is the adjoint in the planar algebra $x^*\in pP_{\beta\to \alpha}q=\Bim(H_q\to H_p)$. The adjoint of a matrix of maps is the $*$-transpose of the matrix.
\end{enumerate}
\end{defn}

\begin{thm}
The map $\Pro\to \Bim$ by $p\mapsto H_p$ and the identity on morphisms is an equivalence of categories.
\end{thm}
\begin{proof}
Note that $\Pro(p\to q) = qP_{\alpha\to \beta}p=\Bim(H_p\to H_q)$.
One now checks that the described map is an additive, monoidal, dual-preserving, $*$-preserving, fully faithful, essentially surjective functor.
\end{proof}

\begin{defn}\label{defn:CF}
Let $\CF(\cP)$ (abbreviated $\CF$, which stands for \emph{Connes' fusion}) be the rigid $C^*$-tensor category defined as follows.
\begin{enumerate}
\item[\text{\underline{Objects:}}] The objects of $\CF$ are finite direct sums of the bimodules $H_p$ as in $\Bim$.
\item[\text{\underline{Tensor:}}]
For bimodules $K,L\in \CF$, we define $K\otimes_{\CF} L=K\otimes_M L$, the Connes' fusion of $II_1$-factor bimodules.

The associator $a_\CF$ is defined by restricting to $M$-bounded vectors as in \cite{MR1424954}.
\item[\text{\underline{Morphisms:}}]
$\CF(K\to L)$ is the set of $M-M$ bilinear maps $K\to L$. Composition is the usual composition of linear maps.
\item[\text{\underline{Tensoring:}}]
For $M-M$ bilinear maps $\varphi_i\colon K_{i}\to L_{i}$ for $i=1,2$, we define $\varphi_1\otimes_{M} \varphi_2$ by the Connes' fusion of intertwiners. If we have $M$-bounded vectors $\xi_i\in K_{i}$ for $i=1,2$, then the map $\xi_1\otimes \xi_2\mapsto \varphi_1(\xi_1)\otimes \varphi_2(\xi_2)$ is clearly $M$-middle linear and bounded, so it extends to a unique $M-M$ bilinear map.
\item[\text{\underline{Duality:}}]
The dual of $K$ is the contragredient bimodule $\overline{K}=\set{\overline{\xi}}{\xi\in K}$ where  $\lambda \overline{\xi}+\overline{\eta}=\overline{\overline{\lambda}\xi+\eta}$ for all $\lambda\in\C$ and $\eta,\xi\in K$, and the action is given by $a\overline{\xi}b = \overline{b^*\xi a^*}$. The evaluation map
$$
\ev_{K}\colon \overline{K}\otimes_M K\longrightarrow H_\emptyset=L^2(M)=1_\CF
$$
is the unique extension of the map $\overline{\xi}\otimes \eta\mapsto \langle \xi|\eta\rangle_M$ where $\xi,\eta$ are $M$-bounded vectors in $K$.
The coevaluation map $\coev_{K}$ is the unique map in $\CF(1\to K\otimes_M \overline{K})$ corresponding to $\id_{K}\in\CF(K\to K)$ under the natural isomorphism given by Frobenius reciprocity. For an explicit formula, just pick an orthonormal left $M$-basis $\{\zeta\}\subset K$ (e.g., see \cite{MR561983,1110.3504}), and we have that
$$
\coev_{K}(1_M)=\sum_\zeta \zeta\otimes \overline{\zeta}
$$
is $M$-central and independent of the choice of $\{\zeta\}$. The zig-zag relation is now given by
$$
\sum_\zeta \zeta \langle \zeta|\xi\rangle_M = \xi
$$
for all $M$-bounded $\xi\in K$.

\item[\text{\underline{Adjoint:}}]
The adjoint map $*$ is the identity on all objects and on a morphism $x\in \CF(K\to L)$ is the adjoint linear operator $x^*\in \CF(L\to K)$.
\end{enumerate}
\end{defn}

Note $\CF$ is a rigid $C^*$-tensor category by well known properties of Connes' fusion (e.g., see \cite{MR1424954,1110.3504}).
It is now our task to prove the following theorem:

\begin{thm}\label{thm:BimCF}
Define a map $\Phi\colon \Bim\to \CF$ as follows. First, $\Phi$ is the identity on objects. Second, for a morphism $x\in qP_{\alpha\to \beta}p$, we get an $M-M$ bimodule map $\Phi_x\colon H_p\to H_q$ by
\begin{equation}\label{eqn:BimoduleMaps}
\ColorMultiply{\xi}{p}{\alphacolor}
\longmapsto
\ColorMultiplyThree{\xi}{p}{x}{\alphacolor}{\betacolor}\in H_q,
\end{equation}
Finally, $\Phi$ is applied entry-wise to matrices over such morphisms.

The map $\Phi$ is an equivalence of categories $\Bim\simeq \CF$.
\end{thm}

In the lemmas below, unless otherwise stated, $p,q$ are projections in $P_{\alpha \to \alpha},P_{\beta \to \beta}$ respectively.

\begin{rem}\label{rem:Functor}
Note that composition of the bimodule maps given by Equation \eqref{eqn:BimoduleMaps} corresponds to the usual composition in $\cP$, i.e., if we have $x\in qP_{\alpha\to \beta}p$ and $y\in rP_{\beta\to \gamma}q$, then $\Phi_y\circ \Phi_x =\Phi_{yx}\colon H_p\to H_q$, where
$$
\nbox{\alpha}{yx}{\gamma}=\PAMultiply{\alpha}{x}{\beta}{y}{\gamma}\,.
$$
It is obvious that $\Phi_{p}=\id_{H_p}$, so $\Phi$ is a functor.
\end{rem}

\begin{lem}\label{lem:TensorProduct2}
For $p,q$ projections in $\cP$, the maps
$$
\phi_{p,q}\colon H_p\otimes_M H_q\to H_p\otimes_{\Bim} H_q=H_{p\otimes q}
$$
(where $p\otimes q$ is the tensor product in $\cP$) given by the unique extension of
$$
\ColorMultiply{\xi}{p}{\alphacolor}
\otimes
\ColorMultiply{\eta}{q}{\betacolor}
\longmapsto
\ColorMultiply{\xi}{p}{\alphacolor}
\,
\ColorMultiply{\eta}{q}{\betacolor}
$$
for $\xi\in D(H_p)$ and $\eta\in D(H_q)$, are $M-M$ bilinear isomorphisms which satisfy associativity, i.e., the following diagram commutes:
$$
\xymatrix{
(H_p\otimes_M H_q)\otimes_M  H_r \ar[rr]^{a_\CF}\ar[d]_{\phi_{p,q}\otimes_M\id_{H_r}} && H_p\otimes_M (H_q\otimes_M H_r)\ar[d]^{\id_{H_p}\otimes_M \phi_{q,r}}\\
(H_{p}\otimes_{\Bim} H_q) \otimes_M H_r\ar[d]_{\phi_{p\otimes q,r}} && H_p\otimes_M ( H_q \otimes_{\Bim} H_r)\ar[d]^{\phi_{p,q\otimes r}}\\
(H_p\otimes_{\Bim} H_q)\otimes_{\Bim} H_r \ar[rr]^{=} && H_p\otimes_{\Bim} (H_q\otimes_{\Bim} H_r).
}
$$
\end{lem}
\begin{proof}
As in the proof of Lemma \ref{lem:TensorProduct}, the map for $\xi\in D(H_p),\eta\in D(H_q)$ is an isometry intertwining the left and right $M$ actions.  Since the linear span of elements of the form $x\wedge y$ with $x\in D(\cH_{\alpha, \emptyset})$ and $y\in D(\cH_{\emptyset, \overline{\beta}})$ is equal to $D(\cH_{\alpha, \overline{\beta}}) \cong D(H_{\alpha\beta})$ by the proof of Lemma \ref{lem:TensorProduct}, the above map is also surjective.

Associativity follows from looking at $M$-bounded vectors (see \cite{MR1424954}).
\end{proof}

\begin{lem}\label{lem:FullyFaithful}
As complex vector spaces,
$$
\CF(H_p\to H_q)\cong q(P_{\alpha \to \beta})p=\Bim(H_p\to H_q).
$$
Moreover, the composition of maps $\varphi\in \CF(H_p\to H_q)$ and $\psi\in \CF(H_q\to H_r)$ corresponds to the composition in $\Bim$.
\end{lem}
\begin{proof}
Recall that an element $x\in q(P_{\alpha \to \beta})p=\Bim(H_p\to H_q)$ gives a map $\Phi_x\in \CF(H_p\to H_q)$ as in Equation \eqref{eqn:BimoduleMaps}, and composition of morphisms is exactly multiplication in $\cP$ by Remark \ref{rem:Functor}. Note further that $\Bim(H_p\to H_q)$ and $\CF(H_p\to H_q)$ are finite dimensional, so it remains to show they have the same dimension.

By Frobenius reciprocity, we have a natural isomorphism
$$
\Hom_{M-M}(H_p\to H_q)\cong \Hom_{M-M}(1\to \overline{H_p}\otimes H_q),
$$
and the latter space is naturally identified with the $M-M$ central vectors in $\overline{H_p}\otimes H_q$. Note that
$$
\overline{H_{p}} \otimes_M H_{q}
\cong
H_{\overline{p}} \otimes_{M} H_{q}
\cong
H_{\overline{p}\otimes q}
\cong
p\wedge\cH_{\overline{\alpha},\overline{\beta}}\wedge q.
$$
From Corollary \ref{cor:central}, the set of central vectors in $p\wedge\cH_{\overline{\alpha},\overline{\beta}}\wedge q$ is $q(P_{\alpha \to \beta})p$
$$
p\wedge_{\overline{\alpha}}\cH_{\overline{\alpha},\overline{\beta}}\wedge_{\overline{\beta}} q
\ni
\begin{tikzpicture} [baseline = -.1cm]
	\draw[thick, \betacolor] (0,0)--(2,0);
	\draw[thick, \alphacolor] (-2,0)--(0,0);
	\draw[thick, unshaded] (-.4, -.4) -- (-.4, .4) -- (.4, .4) -- (.4, -.4) -- (-.4, -.4);
	\draw[thick, unshaded] (.8, -.4) -- (.8, .4) -- (1.6, .4) -- (1.6, -.4) -- (.8, -.4);
	\draw[thick, unshaded] (-.8, -.4) -- (-.8, .4) -- (-1.6, .4) -- (-1.6, -.4) -- (-.8, -.4);
	\node at (0, 0) {$x$};
	\node at (1.2, 0) {$q$};
	\node at (-1.2, 0) {$p$};
\end{tikzpicture}
\longleftrightarrow
\begin{tikzpicture}[baseline=-.1cm]
	\draw[thick, \betacolor] (0,0)--(0,-2);
	\draw[thick, \alphacolor] (0,0)--(0,2);
	\filldraw[unshaded,thick] (-.4,.4)--(.4,.4)--(.4,-.4)--(-.4,-.4)--(-.4,.4);
	\draw[thick, unshaded] (-.4, .8) -- (-.4, 1.6) -- (.4, 1.6) -- (.4,.8) -- (-.4, .8);
	\draw[thick, unshaded] (-.4, -.8) -- (-.4, -1.6) -- (.4, -1.6) -- (.4,-.8) -- (-.4, -.8);
	\node at (0,1.2) {$p$};
	\node at (0,0) {$x$};	
	\node at (0,-1.2) {$q$};
\end{tikzpicture}
\in q(P_{\alpha \to \beta})p,
$$
proving $\dim( \CF(H_p\to H_q))=\dim(\Bim(H_p\to H_q))$.
\end{proof}

\begin{lem}\label{lem:StarPreserving}
If $x\in \Bim(H_p\to H_q)=qP_{\alpha\to \beta}p$ and $\Phi_x\in\CF( H_p\to H_q)$ is as in Equation \eqref{eqn:BimoduleMaps}, then $\Phi_x^* = \Phi_{x^*}$.
\end{lem}
\begin{proof}
By Equation \eqref{eqn:InnerProduct}, if $\xi\in D(H_p)$ and $\eta\in D(H_q)$, then
$$
\langle \Phi_x \xi,\eta\rangle_{H_q}
=
\begin{tikzpicture}[baseline=1.7cm]
	\draw (0,2.4)--(0,3.2);
	\draw (1.2,2.4)--(1.2,3.2);
	\draw [thick, red] (0,2.4)-- (0, -.4) .. controls ++(270:.4cm) and ++(270:.4cm) .. (1.2,-.4);
	\draw [thick, blue] (1.2,0) -- (1.2,2.8);
	\draw[thick, unshaded] (-.4,3.2) -- (-.4, 4) -- (1.6, 4) -- (1.6, 3.2) -- (-.4, 3.2);
	\draw[thick, unshaded] (1.6, -.4) -- (1.6, .4) -- (.8, .4) -- (.8, -.4) -- (1.6, -.4);
	\draw[thick, unshaded] (1.6, .8) -- (1.6, 1.6) -- (.8, 1.6) -- (.8, .8) -- (1.6, .8);
	\draw[thick, unshaded] (1.6, 2) -- (1.6, 2.8) -- (.8, 2.8) -- (.8, 2) -- (1.6, 2);
	\draw[thick, unshaded] (-.4, .8) -- (-.4, 1.6) -- (.4, 1.6) -- (.4, .8) -- (-.4, .8);
	\draw[thick, unshaded] (-.4, 2) -- (-.4, 2.8) -- (.4, 2.8) -- (.4, 2) -- (-.4, 2);
	\node at (.6, 3.6) {$\sum CTL$};
	\node at (0, 2.4) {$\overline{\eta^*}$};
	\node at (1.2, 2.4) {$\xi$};
	\node at (0, 1.2) {$\overline{q}$};
	\node at (1.2, 1.2) {$p$};
	\node at (1.2,0) {$x$};
\end{tikzpicture}
=
\begin{tikzpicture}[baseline=1.7cm]
	\draw (0,2.4)--(0,3.2);
	\draw (1.2,2.4)--(1.2,3.2);
	\draw [thick, red] (0,2.4)--(0,0);
	\draw [thick, blue] (0, -.4) .. controls ++(270:.4cm) and ++(270:.4cm) .. (1.2,-.4) -- (1.2,2.4);;
	\draw[thick, unshaded] (-.4,3.2) -- (-.4, 4) -- (1.6, 4) -- (1.6, 3.2) -- (-.4, 3.2);
	\draw[thick, unshaded] (1.6, .8) -- (1.6, 1.6) -- (.8, 1.6) -- (.8, .8) -- (1.6, .8);
	\draw[thick, unshaded] (1.6, 2) -- (1.6, 2.8) -- (.8, 2.8) -- (.8, 2) -- (1.6, 2);
	\draw[thick, unshaded] (-.4, -.4) -- (-.4, .4) -- (.4, .4) -- (.4, -.4) -- (-.4, -.4);
	\draw[thick, unshaded] (-.4, .8) -- (-.4, 1.6) -- (.4, 1.6) -- (.4, .8) -- (-.4, .8);
	\draw[thick, unshaded] (-.4, 2) -- (-.4, 2.8) -- (.4, 2.8) -- (.4, 2) -- (-.4, 2);
	\node at (.6, 3.6) {$\sum CTL$};
	\node at (0, 2.4) {$\overline{\eta^*}$};
	\node at (1.2, 2.4) {$\xi$};
	\node at (0, 1.2) {$\overline{q}$};
	\node at (1.2, 1.2) {$p$};
	\node at (0,0) {$\overline{x}$};
\end{tikzpicture}
=
\langle \xi, \Phi_{x^*} \eta\rangle_{H_p}
$$
where $\overline{x}$ is the 180 degree rotation of $x$.
\end{proof}

For the next lemma, recall that the conjugate Hilbert space of $K$, is the set of formal symbols $\set{\overline{\xi}}{\xi\in H_p}$ such that $\lambda \overline{\xi}+\overline{\eta}=\overline{\overline{\lambda}\xi+\eta}$ for all $\lambda\in\C$ and $\eta,\xi\in H_p$, together with left and right $M$-actions given by $x \overline{\xi} y = \overline{y^* \xi x^*}$ for $x,y\in M$.

\begin{lem} \label{lem:DualPreserving}
For $p\in P_{\alpha\to\alpha}$, define $\psi_p\colon H_{\overline{p}}\to \overline{H_p}$ (where $\overline{p}$ is the dual projection of $p$ in $\cP$) by the unique extension of the map
$$
\ColorMultiply{x}{\overline{p}}{\alphacolor}
\longmapsto
\overline{
\ColorMultiply{x^*}{p}{\alphacolor}
}
$$
for $\xi\in D(H_{\overline{p}})$ (the {\alphacolor} strand on the left is labelled $\overline{\alpha}$). Then $\psi_p$ is an $M-M$ bilinear isomorphism such that the following diagrams commute:
$$
\xymatrix{
H_{\overline{p}}\otimes_M H_{\overline{q}}
\ar[rr]^{\psi_p\otimes_M \psi_q}
\ar[d]_{\phi_{\overline{p},\overline{q}}}
&&
\overline{H_p}\otimes_M \overline{H_q}
\ar[rr]^{\cong}
&&
\overline{H_q\otimes_M H_p}
\ar[d]^{\overline{\phi_{q,p}}}
\\
H_{\overline{p}\otimes \overline{q}}
\ar[rr]^{=}
&&
H_{\overline{q\otimes p}}
\ar[rr]^{\psi_{p\otimes q}}
&&
\overline{H_{q\otimes p}}
}
$$
and
$$
\xymatrix{
H_{\overline{\overline{p}}}
\ar[rr]^{\psi_p}\ar[d]_{=}
&&
\overline{H_{\overline{p}}}
\ar[d]^{\overline{\psi_{p}}}
\\
H_p
\ar[rr]^{\cong}
&&
\overline{\overline{H_p}}.
}
$$
where we just write $\cong$ for the obvious isomorphism.
\end{lem}
\begin{proof}
That $\psi_p$ is a well-defined $M-M$ bilinear isomorphism is trivial. Commutativity of the diagrams follows by looking at $M$-bounded vectors.
\end{proof}

\begin{lem}\label{lem:Rigid}
For $p\in P_{\alpha\to\alpha}$, $\Phi_{\ev_{H_p}^{\Bim}}=\ev_{H_p}^\CF$ and $\Phi_{\coev_{H_p}^{\Bim}}=\coev_{H_p}^\CF$. Hence $\Phi$ preserves the rigid structure, and $\overline{\Phi_x}=\Phi_{\overline{x}}\in \CF(H_{\overline{q}}\to H_{\overline{p}})$ for all $x\in \Bim(H_p\to H_q)$.
\end{lem}
\begin{proof}
In diagrams, $\ev_{H_p}^\CF$ is given for $M$-bounded vectors by
$$
\ev_{H_p}^\CF(\overline{x}\otimes y)=
\ColorMultiply{\overline{x}}{\overline{p}}{\alphacolor}
\otimes
\ColorMultiply{y}{p}{\alphacolor}
\longmapsto
\begin{tikzpicture}[baseline=.5cm]
	\draw [thick, \alphacolor] (0,1)--(0, -.4) .. controls ++(270:.4cm) and ++(270:.4cm) .. (1.2,-.4) -- (1.2,1);
	\draw (0,2)--(0,1);
	\draw (1.2,2)--(1.2,1);
	\draw[thick, unshaded] (1.6, -.4) -- (1.6, .4) -- (.8, .4) -- (.8, -.4) -- (1.6, -.4);
	\draw[thick, unshaded] (1.6, .8) -- (1.6, 1.6) -- (.8, 1.6) -- (.8, .8) -- (1.6, .8);
	\draw[thick, unshaded] (-.4, -.4) -- (-.4, .4) -- (.4, .4) -- (.4, -.4) -- (-.4, -.4);
	\draw[thick, unshaded] (-.4, .8) -- (-.4, 1.6) -- (.4, 1.6) -- (.4, .8) -- (-.4, .8);
	\node at (0, 0) {$\overline{p}$};
	\node at (1.2, 0) {$p$};
	\node at (0, 1.2) {$\overline{x}$};
	\node at (1.2, 1.2) {$y$};
\end{tikzpicture}
=\langle x| y\rangle_M.
$$
Hence using the isomorphism
$$
\xymatrix{
\overline{H_p}\otimes_M H_p
\ar[rr]^{\psi_p\otimes_M \id_{H_p}}
&&
H_{\overline{p}}\otimes_M H_p
\ar[rr]^(.55){\phi_{\overline{p},p}}
&&
H_{\overline{p}\otimes p}
}
$$
from Lemmas \ref{lem:TensorProduct2} and \ref{lem:DualPreserving}, we have that $\ev_{H_p}^\CF=\Phi_{\ev_{H_p}^{\Bim}}$.
Since $\Phi$ is $*$-preserving by Lemma \ref{lem:StarPreserving}, we must have that $\Phi_{\coev_{H_p}^{\Bim}}=\coev_{H_p}^\CF$. Thus
$$
\coev_{H_p}^\CF(1_M)=\MInnerProductOp{p}{\alphacolor}{\overline{p}}\,.
$$
It is now easy to check that $\overline{\Phi_x}=\Phi_{\overline{x}}\in \CF(H_{\overline{q}}\to H_{\overline{p}})$ for all $x\in \Bim(H_p\to H_q)$.
\end{proof}

\begin{proof}[Proof of Theorem \ref{thm:BimCF}]
We must show $\Phi$ is additive, monoidal, dual-preserving, $*$-preserving, fully faithful, and essentially surjective.

Additivity on objects and essentially surjective come for free, since the objects are the same. Additivity on morphisms and fully faithful follows from Lemma \ref{lem:FullyFaithful}. Monoidal follows from Lemma \ref{lem:TensorProduct2},  $*$-preserving follows from Lemmas \ref{lem:FullyFaithful} and \ref{lem:StarPreserving}, and dual-preserving follows from Lemmas \ref{lem:DualPreserving} and \ref{lem:Rigid}. Note that the results of the previous five lemmas extend in the obvious ways to direct sums by looking at matrices over the maps $H_p\to H_q$.
\end{proof}

\section{The isomorphism class of $M$}\label{sec:vNa}

We now determine the isomorphism class of the $II_1$-factor $M$ from Section \ref{sec:GJS}.

\begin{assumption}
Note that if our rigid $C^*$-tensor category $\cC$ is countably generated, then the isomorphism classes of objects of $\cC$ form a countable set. Let $\cS$ consist of a set of representatives for the isomorphism classes of objects in $\cC$, and let $\cL=\set{X\oplus\overline{X}}{X\in\cS}$ as in Assumption \ref{assume:Countable}.
Again, the objects in $\cL$ are not simple, and all have dimension greater than 1.
In particular, $\cF_\cC(\cL)$ is not locally finite, and each vertex has self loops, since objects of the form $2\oplus X\oplus \overline{X}$ are in $\cL$.
\end{assumption}

With this assumption, we prove in Section \ref{sec:vNAGraph} that $M \cong L(\F_{\I})$. In the case that $\cC$ has finitely many isomorphism classes of simple objects, i.e., $\cC$ is a unitary fusion category, we can find a single object $X$ that generates $\cC$. We explain in Remark \ref{rem:Finite} that if we choose $\cL=\{X\oplus\overline{X}\}$, then $M\cong L(\F_t)$ with
$$
t = 1 + \dim(\cC)(2\dim(X) - 1),
$$
similar to the result in \cite{MR2807103}.

To begin, we describe in the next two sections a semifinite algebra associated to $\cP$.  Many of the ideas in these two sections mirror \cite{MR2807103}, except that the semifinite algebra makes all box-spaces orthogonal. One also must be more careful since the fusion graph $\cF_{\cC}(\cL)$ is not bipartite.

\subsection{A semifinte algebra associated to $\cP$}

Set $\displaystyle \cG_{\infty} = \bigoplus_{\alpha, \beta \in \Lambda} \Gr_{\alpha, \beta}(\cP)$.  We endow $\cG_{\infty}$ with the multiplication
$$
\begin{tikzpicture}[baseline=-.1cm]
	\draw (-.8,0)--(.8,0);
	\draw (0,.4)--(0,.8);
	\draw[thick, unshaded] (-.4, -.4) -- (-.4, .4) -- (.4, .4) -- (.4, -.4) -- (-.4, -.4);
	\node at (0, 0) {$x$};
	\node at (-.6, .2) {\scriptsize{$\kappa$}};
	\node at (.2,.6) {\scriptsize{$\gamma$}};
	\node at (.6, .2) {\scriptsize{$\theta$}};
\end{tikzpicture}
\wedge
\begin{tikzpicture}[baseline=-.1cm]
	\draw (-.8,0)--(.8,0);
	\draw (0,.4)--(0,.8);
	\draw[thick, unshaded] (-.4, -.4) -- (-.4, .4) -- (.4, .4) -- (.4, -.4) -- (-.4, -.4);
	\node at (0, 0) {$x$};
	\node at (-.6, .2) {\scriptsize{$\omega$}};
	\node at (.2,.6) {\scriptsize{$\gamma'$}};
	\node at (.6, .2) {\scriptsize{$\chi$}};
\end{tikzpicture}
= \delta_{\omega,\theta}
\begin{tikzpicture} [baseline = -.1cm]
	\draw (0,.4)--(0,.8);
	\draw (1.2,.4)--(1.2,.8);
	\draw (-.8,0)--(2,0);
	\draw[thick, unshaded] (-.4, -.4) -- (-.4, .4) -- (.4, .4) -- (.4, -.4) -- (-.4, -.4);
	\draw[thick, unshaded] (.8, -.4) -- (.8, .4) -- (1.6, .4) -- (1.6, -.4) -- (.8, -.4);
	\node at (0, 0) {$x$};
	\node at (1.2, 0) {$y^{*}$};
	\node at (-.6,.2) {\scriptsize{$\kappa$}};
	\node at (.2,.6) {\scriptsize{$\gamma$}};
	\node at (.6,.2) {\scriptsize{$\theta$}};
	\node at (1.4,.6) {\scriptsize{$\gamma'$}};
	\node at (1.8,.2) {\scriptsize{$\chi$}};
\end{tikzpicture}\,,
 $$
where $\kappa,\theta,\omega,\chi\in\Lambda$.  There is a (semifinite) trace, $\Tr$ on $\cG_{\infty}$ given by
$$
\Tr(x) =
\begin{tikzpicture}[baseline=.3cm]
	\draw (0,0)--(0,.8);
	\draw (.4,0) arc (90:-90:.4cm) -- (-.4,-.8) arc (270:90:.4cm);
	\filldraw[unshaded,thick] (-.4,.4)--(.4,.4)--(.4,-.4)--(-.4,-.4)--(-.4,.4);
	\draw[thick, unshaded] (-.7, .8) -- (-.7, 1.6) -- (.7, 1.6) -- (.7,.8) -- (-.7, .8);
	\node at (0,0) {$x$};
	\node at (0,1.2) {$\Sigma CTL$};
\end{tikzpicture}
$$
for $x \in \Gr_{\alpha}$, and $\Tr(x) = 0$ for $x \in \Gr_{\alpha, \beta}$ with $\alpha \neq \beta$.

\begin{defn}

As in Section \ref{sec:graded}, one argues (by once again orthogonalizing) that $\Tr$ is positive definite on $\cG_{\infty}$ and multiplication is bounded on $L^{2}(\cG_{\infty})$.  Therefore, we form the (semifinite) von Neumann algebra $\cM_{\infty}=\cG_{\infty}''$ acting on $L^{2}(\cG_{\infty},\Tr)$.

We also use the von Neumann subalgebra $\cA_{\infty} \subset \cM_{\infty}$ which is generated by all boxes in $\cG_{\infty}$ with \underline{no} strings on top.
\end{defn}

As in Section \ref{sec:CategoriesOfBimodules}, let
$$p_{\alpha} =
\begin{tikzpicture} [baseline = -.1cm]
	\filldraw[unshaded,thick] (-.4,.4)--(.4,.4)--(.4,-.4)--(-.4,-.4)--(-.4,.4);
	\draw [thick, \alphacolor] (-.4, 0) -- (.4, 0);
\end{tikzpicture}\,.
$$
It is easy to see that $p_{\alpha}\wedge \cM_{\infty} \wedge p_{\alpha} = M_{\alpha}$ so that all factors in the various Jones towers in Section \ref{sec:towers} appear as cut-downs/corners of $\cM_{\infty}$.   We now record some lemmas about the structure of $\cM_{\infty}$ and $\cA_{\infty}$.

\begin{lem} \label{lem:2inftyfactor}
$\cM_{\infty}$ is a $II_{\infty}$ factor.
\end{lem}

\begin{proof}
Note that $p_{\emptyset}\wedge \cM_{\infty}\wedge p_{\emptyset} = M$ is a $II_{1}$ factor and the trace of $1 = 1_{\cM_{\infty}} = \sum_{\alpha \in \Lambda} p_{\alpha}$ (with convergence of the orthogonal sum in the strong operator topology) is infinite, so we only need to show that $\cM_{\infty}$ is a factor.  To do so, we show that the central support of $p_{\emptyset}$ in $\cM_{\infty}$ is 1, which is enough since $p_{\emptyset}\wedge \cM_{\infty}\wedge p_{\emptyset}$ is a factor.  We let
$$
w_{\alpha} =
\begin{tikzpicture} [baseline=-.1cm]
	\filldraw[unshaded,thick] (-.4,.4)--(.4,.4)--(.4,-.4)--(-.4,-.4)--(-.4,.4);
	\draw [thick, \alphacolor] (-.4, 0) arc(-90:0: .4cm);
\end{tikzpicture}\,,
$$
and we let $v_\alpha$ be the polar part of $w_\alpha$.  As was discussed in the proof of Lemma \ref{lem:TensorProduct}, $v_\alpha$ induces an equivalence of projections between $p_{\emptyset}$ and $e \leq p_{\alpha}$.  Let $z$ be the central support of $p_{\emptyset}$ (and $e$).  Since $p_{\alpha}\wedge \cM_{\I}\wedge p_{\alpha}$ is a factor, the central support of $e$ must lie above $p_{\alpha}$, so $z \geq p_{\alpha}$.  Since this holds for all $\alpha \in \Lambda$, we have $z \geq 1$. Hence $z = 1$, and we are finished.
\end{proof}

\begin{cor}
The factor $M_{\alpha}$ is a $\delta_{\alpha}$-amplification of $M$.
\end{cor}

\begin{cor}
The algebras $\cM_{\alpha, \beta}$ are $II_{1}$ factors.
\end{cor}
\begin{proof}
$\cM_{\alpha, \beta}$ is the compression of $\cM_{\I}$ by the finite projection $p_{\alpha} + p_{\beta}$.
\end{proof}

\begin{lem} \label{lem:1infty}
We have a direct sum decomposition
$$\displaystyle \cA_{\infty} = \bigoplus_{v \in V(\cF_\cC(\cL))} \cA_{v}
$$
where the sum is over all vertices $v$ in the fusion graph $\cF_\cC(\cL)$, and each $\cA_{v}$ is a type $I_{\infty}$ factor.  If $p$ is a minimal projection in $P_{\alpha \to \alpha}$ whose equivalence class represents the vertex $v$, then $p \leq 1_{\cA_{v}}$.
\end{lem}

\begin{proof}
For each vertex $v$, choose a minimal projection $p_{v} \in P_{\alpha_{v} \to \alpha_{v}}$ whose equivalence class corresponds to the vertex $v$.  We first see that
$$
p_{v} \wedge \cA_{\infty} \wedge p_{v} = p_{\alpha_v}P_{\alpha_{v} \to \alpha_{v}}p_{\alpha_{v}} \cong \C
$$
so $\cA_{\infty}$ has minimal projections. Letting $p\in P_{\alpha \to \alpha}$ and $q \in P_{\beta \to \beta}$, then the observation that $p\wedge \cA_{\I} \wedge q = qP_{\alpha \to \beta}p$ implies $p$ is equivalent to $q$ in the planar algebra sense if and only if $p$ is equivalent to $q$ in $\cA_{\I}$.  Let $1_{v}$ be the central support of $p_{v}$.  Then by construction, $\cA_{v} = 1_{v}\wedge \cA_{\I}\wedge 1_{v}$ is a type $I$ factor which must be type $I_{\I}$ since there are infinitely many mutually orthogonal projections equivalent to $p_{v}$ in $\cP$.  It is easy to see that by construction, if $v \neq w$ then $1_{v}\wedge 1_{w} = 0$, which implies that
$$
\cA_{\I} = \bigoplus_{v \in V(\cF_\cC(\cL))} \cA_{v} \oplus \cB
$$
for some von Neumann algebra $\cB$.

We claim that $\cB = \{0\}$.  Indeed we know that $p_{\alpha}$ can be written as an orthogonal sum of projections equivalent to some of the $p_{w}$'s.  Since $1 = \sum_{\alpha \in \Lambda} p_{\alpha}$, this implies $\sum_{v} 1_{v} = 1$, and thus $\cB = \{0\}$.
\end{proof}

\begin{rem}
For the rest of this section, all multiplication will be the $\wedge$ multiplication in the GJS picture. Hence for the rest of this section, we just write $xy$ for $x\wedge y$ for convenience.
\end{rem}

We now show that we can obtain $\cM_{\I}$ from a base ``building block" $\cA_{\infty}$ and various free ``corner elements."  Fixing a color $c$, which we again represent by the color \textcolor{\cupcolor}{green}, we define
$$
X_{c} = \sum_{\substack{ \alpha \in \Lambda\\
|\alpha| \in 2\N}}
\,
\begin{tikzpicture}[baseline=-.1cm]
    \draw [thick] (-.4, -.4)--(-.4, .4)--(.4, .4)--(.4, -.4)--(-.4, -.4);
    \draw [thick, \cupcolor] (-.4, .1) arc(-90:0: .3cm);
    \draw [thick, \alphacolor] (-.4, -.2)--(.4, -.2);
\end{tikzpicture}
+
\begin{tikzpicture}[baseline=-.1cm]
    \draw [thick] (-.4, -.4)--(-.4, .4)--(.4, .4)--(.4, -.4)--(-.4, -.4);
    \draw [thick, \cupcolor] (.4, .1) arc(-90:-180: .3cm);
    \draw [thick, \alphacolor] (-.4, -.2)--(.4, -.2);
\end{tikzpicture}
\,.
$$

\begin{rem}

We note that this sum defines a bounded operator.  Indeed, the individual terms in the sum are supported under the mutually orthogonal family of projections $\{p_{c\alpha} + p_{\alpha} : |\alpha| \in 2\N\}$ and each term in the sum has operator norm
$$\left\|\,
\begin{tikzpicture} [baseline = -.1cm]
    \draw [thick] (-.4, -.4)--(-.4, .4)--(.4, .4)--(.4, -.4)--(-.4, -.4);
    \draw [thick, \cupcolor] (-.3, .4) arc(-180:0: .3cm);
    \draw [thick, \alphacolor] (-.4, -.2)--(.4, -.2);
\end{tikzpicture} +
\begin{tikzpicture} [baseline = -.1cm]
    \draw [thick] (-.4, -.4)--(-.4, .4)--(.4, .4)--(.4, -.4)--(-.4, -.4);
    \draw [thick, \cupcolor] (.4, .1) arc(-90:-180: .3cm);
    \draw [thick, \cupcolor] (-.4, .1) arc(-90:0: .3cm);
    \draw [thick, \alphacolor] (-.4, -.2)--(.4, -.2);
    \end{tikzpicture}
\,
\right\|_{\infty}^{1/2} =
\left\|\,
\begin{tikzpicture} [baseline = -.1cm]
    \draw [thick] (-.4, -.4)--(-.4, .4)--(.4, .4)--(.4, -.4)--(-.4, -.4);
    \draw [thick, \cupcolor] (-.3, .4) arc(-180:0: .3cm);
\end{tikzpicture} +
\begin{tikzpicture} [baseline = -.1cm]
    \draw [thick] (-.4, -.4)--(-.4, .4)--(.4, .4)--(.4, -.4)--(-.4, -.4);
    \draw [thick, \cupcolor] (.4, .1) arc(-90:-180: .3cm);
    \draw [thick, \cupcolor] (-.4, .1) arc(-90:0: .3cm);
\end{tikzpicture}
\,
\right\|_{\infty}^{1/2}.
$$
We also note the simple fact that for $|\alpha|$ even, $p_{c\alpha}X_{c}p_{\alpha}$ is the corner diagram
$$
\begin{tikzpicture}[baseline=-.1cm]
    \draw [thick] (-.4, -.4)--(-.4, .4)--(.4, .4)--(.4, -.4)--(-.4, -.4);
    \draw [thick, \cupcolor] (-.4, .1) arc(-90:0: .3cm);
    \draw [thick, \alphacolor] (-.4, -.2)--(.4, -.2);
\end{tikzpicture}\,,
$$
so each term in the sum defining $X_{c}$ appears in the von Neumann algebra $W^{*}(\cA_{\infty}, X_{c})$.  We only sum over $\alpha$ with $|\alpha|$ even as it makes computations involving freeness in Section \ref{sec:free1} much easier.
\end{rem}

The $X_{c}$ elements give us a very nice way of obtaining $\cM_{\infty}$ from $\cA_{\infty}$.

\begin{lem} \label{lem:generateI}
$\cM_{\I} \cong W^{*}(\cA_{\I}, \, \{X_{c} : c \in \cL\})$.
\end{lem}

\begin{proof}
Every element in $\cG_{\infty}$ is a linear combination of elements of the following form:
$$
\begin{tikzpicture}[baseline=-.1cm]
    \draw (-.7, .7) arc(-180:0: .7cm);
    \draw (-.8, -.2)--(.8, -.2);
    \draw [thick, unshaded] (-.4, -.4)--(-.4, .4)--(.4, .4)--(.4, -.4)--(-.4, -.4);
    \node at (0, 0) {$x$};
    \node at (-.8, .6) {\scriptsize{$\alpha$}};
    \node at (.8, .6) {\scriptsize{$\beta$}};
\end{tikzpicture}.
$$
The above diagram is $x \in \cA_{\infty}$ multiplied on the left and right by diagrams of the form
$$
\begin{tikzpicture}[baseline=-.1cm]
    \draw [thick] (-.4, -.4)--(-.4, .4)--(.4, .4)--(.4, -.4)--(-.4, -.4);
    \draw (-.4, .1) arc(-90:0: .3cm);
    \draw (-.4, -.2)--(.4, -.2);
    \node at (0, .2) {\scriptsize{$\gamma$}};
\end{tikzpicture}
$$
and their adjoints.  This diagram is a product of diagrams of the form
$$
\begin{tikzpicture}[baseline=-.1cm]
    \draw [thick] (-.4, -.4)--(-.4, .4)--(.4, .4)--(.4, -.4)--(-.4, -.4);
    \draw [thick, \cupcolor] (-.4, .1) arc(-90:0: .3cm);
    \draw [thick, \alphacolor] (-.4, -.2)--(.4, -.2);
\end{tikzpicture}\,,
$$
so we only need to check that the above diagram is in $W^{*}(\cA_{\I}, \{X_{c} : c \in \cL\})$ when $|\alpha|$ is odd.  This is easy since it can be written as the product
$$
\begin{tikzpicture} [baseline = -.1cm]
    \draw [thick] (-.4, -.4)--(-.4, .4)--(.4, .4)--(.4, -.4)--(-.4, -.4);
    \draw[thick, \cupcolor] (.4,.2) arc (270:180:.2cm);
    \draw[thick, \cupcolor] (-.4, -.1)--(.4, -.1);
    \draw [thick, \alphacolor] (-.4, -.25)--(.4, -.25);
\end{tikzpicture}
\cdot
\begin{tikzpicture} [baseline = -.1cm]
    \draw [thick] (-.4, -.4)--(-.4, .4)--(.4, .4)--(.4, -.4)--(-.4, -.4);
    \draw [thick, \alphacolor] (-.4, -.25)--(.4, -.25);
    \draw [thick, \cupcolor] (-.4, .2) arc(90:-90: .15cm);
\end{tikzpicture} \, .
$$
\end{proof}

There is a $\Tr$-preserving conditional expectation $E: \cM_{\I} \rightarrow \cA_{\I}$ given by
$$
E(x) =
\begin{tikzpicture}[baseline=.3cm]
	\draw (0,0)--(0,.8);
	\draw (-.8, 0)--(.8, 0);
	\filldraw[unshaded,thick] (-.4,.4)--(.4,.4)--(.4,-.4)--(-.4,-.4)--(-.4,.4);
	\draw[thick, unshaded] (-.7, .8) -- (-.7, 1.6) -- (.7, 1.6) -- (.7,.8) -- (-.7, .8);
	\node at (0,0) {$x$};
	\node at (0,1.2) {$\Sigma CTL$};
\end{tikzpicture}\,,
$$
and $E$ induces normal completely positive maps $(\eta_{c,c} = \eta_{c})_{c \in \cL}$ on $\cA_{\infty}$ satisfying
$$
\eta_{c}(y) = E(X_{c}yX_{c}) =
\begin{tikzpicture}[baseline = 0cm]
    \draw(-2, 0)--(2, 0);
    \filldraw[unshaded,thick] (-.4,.4)--(.4,.4)--(.4,-.4)--(-.4,-.4)--(-.4,.4);
    \filldraw[unshaded,thick] (-1.6,.4)--(-.8,.4)--(-.8,-.4)--(-1.6,-.4)--(-1.6,.4);
    \filldraw[unshaded,thick] (1.6,.4)--(.8,.4)--(.8,-.4)--(1.6,-.4)--(1.6,.4);
    \draw[thick, \cupcolor] (-1.2, .4) arc(180:90: .3cm) -- (.9, .7) arc(90:0: .3cm);
    \node at (0, 0) {$y$};
    \node at (-1.2, 0) {$X_{c}$};
    \node at (1.2, 0) {$X_{c}$};
\end{tikzpicture}\,.
$$
For $b \neq a$, we have trivial ``off-diagonal" maps $\eta_{a, b}$ on $\cA_{\I}$ satisfying $\eta_{a, b}(y) = E(X_{a}yX_{b}) = 0$.  This gives a straightforward diagrammatic procedure for evaluating $E(y_{0}X_{c_{1}}y_{1}X_{c_{2}}\cdots X_{c_{n}}y_{n})$ for $y_{i} \in \cA_{\infty}$.  First, write the word $y_{0}X_{c_{1}}y_{1}X_{c_{2}}\cdots X_{c_{n}}y_{n}$ as
$$
\begin{tikzpicture} [baseline = 0cm]
    \draw(-2, 0)--(2, 0);
    \filldraw[unshaded,thick] (-.4,.4)--(.4,.4)--(.4,-.4)--(-.4,-.4)--(-.4,.4);
    \filldraw[unshaded,thick] (-1.6,.4)--(-.8,.4)--(-.8,-.4)--(-1.6,-.4)--(-1.6,.4);
    \filldraw[unshaded,thick] (1.6,.4)--(.8,.4)--(.8,-.4)--(1.6,-.4)--(1.6,.4);
    \node at (0, 0) {$X_{c_{1}}$};
    \node at (-1.2, 0) {$y_{0}$};
    \node at (1.2, 0) {$y_{1}$};
    \draw (0, .4)--(0, .8);
\end{tikzpicture}
\cdots
\begin{tikzpicture} [baseline = 0cm]
    \draw(-2, 0)--(2, 0);
    \filldraw[unshaded,thick] (-.4,.4)--(.4,.4)--(.4,-.4)--(-.4,-.4)--(-.4,.4);
    \filldraw[unshaded,thick] (-1.6,.4)--(-.8,.4)--(-.8,-.4)--(-1.6,-.4)--(-1.6,.4);
    \filldraw[unshaded,thick] (1.6,.4)--(.8,.4)--(.8,-.4)--(1.6,-.4)--(1.6,.4);
    \node at (0, 0) {$X_{c_{n}}$};
    \node at (-1.2, 0) {$y_{n-1}$};
    \node at (1.2, 0) {$y_{n}$};
    \draw (0, .4)--(0, .8);
\end{tikzpicture}\,.
$$
Then sum over all planar ways to connect the strings on top.  Whenever we see a term of the form
$$
\begin{tikzpicture}[baseline = 0cm]
    \draw(-2, 0)--(2, 0);
    \filldraw[unshaded,thick] (-.4,.4)--(.4,.4)--(.4,-.4)--(-.4,-.4)--(-.4,.4);
    \filldraw[unshaded,thick] (-1.6,.4)--(-.8,.4)--(-.8,-.4)--(-1.6,-.4)--(-1.6,.4);
    \filldraw[unshaded,thick] (1.6,.4)--(.8,.4)--(.8,-.4)--(1.6,-.4)--(1.6,.4);
    \draw (-1.2, .4) arc(180:90: .3cm) -- (.9, .7) arc(90:0: .3cm);
    \node at (0, 0) {$y$};
    \node at (-1.2, 0) {$X_{c_{i}}$};
    \node at (1.2, 0) {$X_{c_{j}}$};
\end{tikzpicture}\, ,
$$
we replace it with $\eta_{c_{i}, c_{j}}(y)$. It is straightforward to check that $E$ and the $\eta_{c_{i}, c_{j}}$ satisfy the following recurrence relation:

\begin{align}
E(y_{0}X_{c_{1}}y_{1}X_{c_{2}}&\cdots y_{n-1}X_{c_{n}}y_{n}) \notag\\
&= \sum_{k=2}^{n} y_{0} \cdot \eta_{c_{1}, c_{k}}(E(y_{1}X_{c_{2}}\cdots X_{c_{k-1}}y_{k-1}))\cdot E(y_{k}X_{c_{k+1}}...X_{c_{n}}y_{n}).\label{eqn:recurrence}
\end{align}
Also, by definition of $\eta_{c_{i}, c_{j}}$, it follows that the map on $\cA_{\I} \otimes B(\cH)$ given by
$(y_{i, j}) \mapsto (\eta_{c_{i}, c_{j}}(y_{i,j}))$
is normal and completely positive.  This, combined with Recurrence \eqref{eqn:recurrence} implies that the elements $(X_{c})_{c \in \cL}$ form an $\cA_{\I}$-valued semicircular family with covariance $(\eta_{c_{i}, c_{j}})$ as in \cite{MR1704661}. Since $\eta_{c_{i}, c_{j}} = 0$ for $c_{i}\neq c_{j}$, the family $(X_{c})_{c \in \cL}$ is free with amalgamation over $\cA_{\I}$ with respect to $E$ \cite{MR1704661}.  We record what we have established above in the following lemma.

\begin{lem} \label{lem:free1}
$\cM_{\I} = W^{*}(\cA_{\I}, \{X_{c}: c \in \cL\})$, and the elements  $(X_{c})_{c \in \cL}$ form an $\cA_{\I}$-valued semicircular family with covariance $(\eta_{c_{i}, c_{j}})$ and are free with amalgamation over $\cA_{\I}$.
\end{lem}

\subsection{$M$ as an amalgamated free product} \label{sec:free1}

By compressing the algebra $\cA_{\infty}$ by the projection $1_{e} = \sum_{|\alpha| \in 2\N} p_{\alpha}$, Lemma \ref{lem:1infty} implies that if we set $\cA_{e} = W^{*}((P_{\beta \rightarrow \alpha})_{\alpha, \beta \in 2\N})$ then
$$
\cA_{e} = \bigoplus_{v \in V(\cF_\cC(\cL))} \cB_{v}.
$$
$\cB_{v}$ is a type $I_{\infty}$ factor which is a cut-down of $\cA_{v}$ by $1_{e}$.
Note that every vertex in $v\in V(\cF_\cC(\cL))$ appears in the direct sum because every vertex possesses at least one self-loop.

Similarly, if one sets $1_{o} = \sum_{|\alpha| \in (1 + 2\N)} p_{\alpha}$ and $\cA_{o} = 1_{o}\cA_{\I}1_{o}$, then we have
$$
\cA_{o} = \bigoplus_{v \in V(\cF_\cC(\cL))} \cC_{v}
$$
where $\cC_{v}$ is a type $I_{\I}$ factor which is a cut-down of $\cA_{v}$ by $1_{o}$.

For each vertex $v \in \cF_\cC(\cL)$, we choose a minimal projection $p_{v} \in \cA_{e}$ whose equivalence class is represented by $v$ with $p_{\emptyset}$ the empty diagram.  If $x \in \cA_{e}$, it follows that
$$
\eta_{c}(x) =
\begin{tikzpicture} [baseline = 0cm]
    \draw (-.8, 0)--(.8, 0);
    \filldraw[thick, unshaded] (-.4, .4)--(-.4, -.4)--(.4, -.4)--(.4,.4)--(-.4, .4);
    \draw [thick, \cupcolor] (-.8, .6)--(.8, .6);
    \node at (0,0) {$x$};
\end{tikzpicture}
$$
so that $\eta_{c}(p_{v})$ is a finite projection in $\cA_{o}$, and each $\eta_{c}$ acts as a (non-unital) $W^{*}$ algebra homomorphism from $\cA_{e}$ into $\cA_{o}$.  Set $Q = \sum_{v} p_{v}$.  Then there is a family of partial isometries $(V_{i})_{i \in I}$ satisfying $V_{i}^{*}V_{i} = Q$ and $\sum_{i \in I} V_{i}V_{i}^{*} = 1_{\cA_{e}}$.  Note that $\sum_{c \in \cL}\eta_{c}(Q)$ defines a projection in $\cA_{o}$ since $\eta_{a}(Q) \perp \eta_{b}(Q)$ for $a \neq b$. We examine the compression of $\cM_{\I}$ by $T = Q + \sum_{c \in \cL}\eta_{c}(Q)$.

\begin{lem} \label{lem:compression1}
$T\cM_{\infty}T = W^{*}(T\cA_{\infty}T, \, (TX_{c}T)_{x \in \cL})$
\end{lem}

\begin{proof}
By the choices of the partial isometries $V_{i}$, we know that if we let $W_{i} = \sum_{c \in \cL} \eta_{c}(V_{i})$ then $\sum_{i \in I}(V_{i} + W_{i})(V_{i}+W_{i})^{*} = 1$, so
$$
T\cM_{\infty}T = W^{*}\left(T\cA_{\I}T, \, \left\{\left(V_{i} + W_{i}\right)^{*} X_{c} \left(V_{j} + W_{j}\right): i,j \in I \,  \, c \in \cL\right\}\right)
$$
The relation $y \cdot X_{c} = X_{c} \eta_{c}(y)$ for any $y \in \cA_{e}$ is a straightforward diagrammatic check.  Applying this relation to each $\left(V_{i} + W_{i} \right)^{*} X_{c} \left(V_{j} + W_{j}\right)$ gives $TX_{c}T$.
\end{proof}

Since $\sum_{i \in I} W_{i}W_{i}^{*} = 1_{\cA_{o}}$, it follows that if $v \in V(\cF_{\cC}(\cL))$ then there is a minimal projection $q_{v} \in \cA_{o}$ whose equivalence class represents $v$ and sits under $1_{\cA_{o}}\cdot T$.  Note that this means that for all $v$, $q_{v}$ and $p_{v}$ are orthogonal minimal projections that sit under $1_{v}$.   We will next examine the cut-down of $T\cM_{\I}T$ by $F = \sum_{v} p_{v} + q_{v}$.

Suppose the vertices $v$ and $w$ are distinct and connected in $\cF_{\cC}(\cL)$ and let $C(e)$ be the color of an edge connecting $v$ and $w$.  Let $e_{1},...,e_{k}$ be the (necessarily finite) collection of edges connecting $v$ and $w$ with $C(e_{i}) = c$.  We let $u_{e_{i}}^{v}$ be a collection of partial isometries such that
$$
(u_{e_{i}}^{v})^{*}u_{e_{j}}^{v} = \delta_{ij}q_{v} \text{ and } \sum_{i=1}^{k} u_{e_{i}}^{v}(u_{e_{i}}^{v})^{*} = \eta_{c}(p_{w})\cdot 1_{v}
$$

Let $r_{v} \in F\cA_{\I}F$ be a partial isometry satisfying $r_{v}r_{v}^{*} = p_{v}$ and $r_{v}^{*}r_{v} = q_{v}$ and $r_{w} \in F\cA_{\I}F$ be a partial isometry satisfying $r_{w}r_{w}^{*} = p_{w}$ and $r_{w}^{*}r_{w} = q_{w}$.  The elements defined can be described diagrammatically as follows:
\begin{align*}
p_{v} =
\begin{tikzpicture} [baseline = 0cm]
    \draw  (-.8, 0)--(.8, 0);
    \filldraw [thick, unshaded] (.4, .4) --(.4, -.4)--(-.4, -.4)--(-.4, .4)--(.4, .4);
    \node at (0, 0) {$p_{v}$};
    \node at (-.6, .2) {\scriptsize{$\alpha_{v}$}};
    \node at (.6, .2) {\scriptsize{$\alpha_{v}$}};
\end{tikzpicture}\, , \,
p_{w} =
\begin{tikzpicture} [baseline = 0cm]
    \draw  (-.8, 0)--(.8, 0);
    \filldraw [thick, unshaded] (.4, .4) --(.4, -.4)--(-.4, -.4)--(-.4, .4)--(.4, .4);
    \node at (0, 0) {$p_{w}$};
    \node at (-.6, .2) {\scriptsize{$\alpha_{w}$}};
    \node at (.6, .2) {\scriptsize{$\alpha_{w}$}};
\end{tikzpicture}\, , \,
q_{v} &=
\begin{tikzpicture} [baseline = 0cm]
    \draw  (-.8, 0)--(.8, 0);
    \filldraw [thick, unshaded] (.4, .4) --(.4, -.4)--(-.4, -.4)--(-.4, .4)--(.4, .4);
    \node at (0, 0) {$q_{v}$};
    \node at (-.6, .2) {\scriptsize{$\beta_{v}$}};
    \node at (.6, .2) {\scriptsize{$\beta_{v}$}};
\end{tikzpicture}\, , \,
q_{w} =
\begin{tikzpicture} [baseline = 0cm]
    \draw  (-.8, 0)--(.8, 0);
    \filldraw [thick, unshaded] (.4, .4) --(.4, -.4)--(-.4, -.4)--(-.4, .4)--(.4, .4);
    \node at (0, 0) {$q_{w}$};
    \node at (-.6, .2) {\scriptsize{$\beta_{w}$}};
    \node at (.6, .2) {\scriptsize{$\beta_{w}$}};
\end{tikzpicture}\, , \\
u_{e_{i}}^{v} =
\begin{tikzpicture} [baseline = 0cm]
    \draw  (-.8, -.2)--(.8, -.2);
    \draw [thick, \cupcolor] (-.8, .2)--(-.4, .2);
    \filldraw [thick, unshaded] (.4, .4) --(.4, -.4)--(-.4, -.4)--(-.4, .4)--(.4, .4);
    \node at (0, 0) {$u_{e_{i}}^{v}$};
    \node at (-.6, 0) {\scriptsize{$\alpha_{w}$}};
    \node at (.6, 0) {\scriptsize{$\beta_{v}$}};
\end{tikzpicture}\, , \,
r_{v} &=
\begin{tikzpicture} [baseline = 0cm]
    \draw  (-.8, 0)--(.8, 0);
    \filldraw [thick, unshaded] (.4, .4) --(.4, -.4)--(-.4, -.4)--(-.4, .4)--(.4, .4);
    \node at (0, 0) {$r_{v}$};
    \node at (-.6, .2) {\scriptsize{$\alpha_{v}$}};
    \node at (.6, .2) {\scriptsize{$\beta_{v}$}};
\end{tikzpicture}\, , \,
r_{w} =
\begin{tikzpicture} [baseline = 0cm]
    \draw  (-.8, 0)--(.8, 0);
    \filldraw [thick, unshaded] (.4, .4) --(.4, -.4)--(-.4, -.4)--(-.4, .4)--(.4, .4);
    \node at (0, 0) {$r_{w}$};
    \node at (-.6, .2) {\scriptsize{$\alpha_{w}$}};
    \node at (.6, .2) {\scriptsize{$\beta_{w}$}};
\end{tikzpicture}\, .
\end{align*}

Set $x_{i}$ to be the following diagram:
$$
x_{i} =
\begin{tikzpicture} [baseline = 0cm]
    \draw  (-.8, -.2)--(.8, -.2);
    \draw [thick, \cupcolor] (-.4, .2) arc(270:90: .2cm) -- (.8, .6);
    \filldraw [thick, unshaded] (.4, .4) --(.4, -.4)--(-.4, -.4)--(-.4, .4)--(.4, .4);
    \node at (0, 0) {$u_{e_{i}}^{v}$};
    \node at (-.6, 0) {\scriptsize{$\alpha_{w}$}};
    \node at (.6, 0) {\scriptsize{$\beta_{v}$}};
\end{tikzpicture}
$$
We have the following lemma about the left and right supports of $x_{i}$:

\begin{lem} \label{lem:orthogonalsupport}
The left support of $x_{i}$ is $p_{w}$ and the right support of $x_{i}$ lies under the projection $1_{w} \cdot i_{c}(q_{v})$ where we have for $z \in P_{\overline{\gamma} \rightarrow \overline{\chi}}$
$$
i_{c}(z) =
\begin{tikzpicture} [baseline = 0cm]
    \draw  (-.8, 0)--(.8, 0);
    \filldraw [thick, unshaded] (.4, .4) --(.4, -.4)--(-.4, -.4)--(-.4, .4)--(.4, .4);
    \node at (0, 0) {$z$};
    \node at (-.6, .2) {\scriptsize{$\gamma$}};
    \node at (.6, .2) {\scriptsize{$\chi$}};
    \draw [thick, \cupcolor] (-.8, .6) -- (.8, .6);
\end{tikzpicture}\, .
$$
Furthermore if $i \neq j$ then the right supports of $x_{i}$ and $x_{j}$ are orthogonal.
\end{lem}

\begin{proof}
Note that $x_{i}x_{i}^{*}$ is the diagram
$$
\begin{tikzpicture} [baseline = 0cm]
    \draw  (-.8, -.2)--(2.2, -.2);
    \draw [thick, \cupcolor] (-.4, .2) arc(270:90: .2cm) -- (1.8, .6) arc(90:-90: .2cm);
    \filldraw [thick, unshaded] (.4, .4) --(.4, -.4)--(-.4, -.4)--(-.4, .4)--(.4, .4);
    \filldraw [thick, unshaded] (1.8, .4)--(1.8, -.4)--(.8, -.4)--(.8, .4)--(1.8,.4);
    \node at (1.3, 0) {$(u_{e_{i}}^{v})^{*}$};
    \node at (0, 0) {$u_{e_{i}}^{v}$};
    \node at (-.6, 0) {\scriptsize{$\alpha_{w}$}};
    \node at (.6, 0) {\scriptsize{$\beta_{v}$}};
    \node at (2, 0) {\scriptsize{$\alpha_{w}$}};
\end{tikzpicture}\, .
$$
For any $\gamma \in \Lambda$, let $\Phi: P\op_{\overline{\gamma}c \rightarrow \overline{\gamma}c} \rightarrow P\op_{\overline{\gamma} \rightarrow \overline{\gamma}}$ be the trace preserving conditional expectation where $P\op_{\overline{\gamma} \rightarrow \overline{\gamma}}$ includes unitally into $P\op_{\overline{\gamma}c \rightarrow \overline{\gamma}c}$ via $i_{c}$.    The above diagram is a scalar multiple of $\Phi(u_{e_{i}}^{v}\cdot (u_{e_{i}}^{v})^{*})$.  As $u_{e_{i}}^{v}\cdot (u_{e_{i}}^{v})^{*}$ lies under $\eta_{c}(p_{w}) = i_{c}(p_{w})$ it follows that this diagram is a scalar multiple of $p_{w}$, proving the claim about the left support.  As left and right supports are equivalent projections, it follows that the right support of $x_{i}$ lies under $1_{w}$.  We note that $\Phi(x_{i}^{*}x_{i}) = k (u_{e_{i}}^{v})^{*}u_{e_{i}}^{v} = k q_{v}$ for some scalar $k$, implying that the right support of $x_{i}$ must lie under $i_{c}(q_{v})$.  Finally, if $i \neq j$ then $x_{i}x_{j}^{*}$ is a constant multiple of $\Phi(u_{e_{i}}^{v}\cdot (u_{e_{j}}^{v})^{*})$.  We know this must be a scalar multiple of $p_{w}$, but that scalar must be zero as $u_{e_{i}}^{v}\cdot (u_{e_{j}}^{v})^{*}$ has trace zero.  This proves the orthogonality of the right supports of the $x_{i}$.
\end{proof}

Let $f_{i}$ be the right support of $x_{i}$.  As $i_{c}(p_{v})\cdot1_{w}$ can be written as a sum of $k$ orthogonal projections equivalent to $p_{w}$, we conclude that $\sum_{i=1}^{k} f_{i} = i_{c}(p_{v})\cdot1_{w}$.  We also recognize that $i_{c}(r_{v})$ is a partial isometry with left support $i_{c}(q_{v})$ and right support $i_{c}(p_{v})$. We consider the elements $y_{i} = r_{w}^{*}\cdot x_{i} \cdot i_{c}(r_{v}^{*})$ which diagrammatically look like
$$
y_{i} =
\begin{tikzpicture} [baseline = 0cm]
    \draw (-.8, -.2)--(3.2, -.2);
    \filldraw [thick, unshaded] (.4, .4)--(.4, -.4)--(-.4, -.4)--(-.4, .4)--(.4, .4);
    \filldraw [thick, unshaded] (1.6, .4)--(1.6, -.4)--(.8, -.4)--(.8, .4)--(1.6, .4);
    \filldraw [thick, unshaded] (2.8, .4)--(2.8, -.4)--(2, -.4)--(2, .4)--(2.8, .4);
    \draw [thick, \cupcolor] (.8, .2) arc(270:90: .2cm)--(3.2, .6);
    \node at (0,0) {$r_{w}^{*}$};
    \node at (1.2, 0) {$u_{e}^{v}$};
    \node at (2.4, 0) {$r_{v}^{*}$};
    \node at (-.6, 0) {\scriptsize{$\beta_{w}$}};
    \node at (.6, 0) {\scriptsize{$\alpha_{w}$}};
    \node at (1.8, 0) {\scriptsize{$\beta_{v}$}};
    \node at (3, 0) {\scriptsize{$\alpha_{v}$}};
\end{tikzpicture}
$$
note that the left support of $y_{i}$ is $q_{w}$.  From Lemma \ref{lem:orthogonalsupport}, the right supports of the $y_{i}$ are orthogonal and sum up to $i_{c}(p_{v})\cdot1_{w}$.  As $q_{w}$ is minimal in $\cA_{\I}$, it follows that $y_{i}$ is a scalar multiple of a partial isometry.  We define the partial isometry $u_{e}^{w}$ by the equation $(u_{e_{i}}^{w})^{*} = l \cdot y_{i}$ with $l$ an appropriate constant.  It follows that $(u_{e_{i}}^{w})^{*}u_{e_{j}}^{w} = \delta_{ij}q_{w}$ and $\sum_{i=1}^{k} u_{e_{i}}^{w}(u_{e_{i}}^{w})^{*} = \eta_{c}(p_{v})\cdot 1_{w}$.

One can manipulate the diagram above to show that the relation
$$
(u_{e}^{v})^{*} = (\overline{l})^{-1} \cdot
\begin{tikzpicture} [baseline = 0cm]
    \draw (-.8, -.2)--(3.2, -.2);
    \filldraw [thick, unshaded] (.4, .4)--(.4, -.4)--(-.4, -.4)--(-.4, .4)--(.4, .4);
    \filldraw [thick, unshaded] (1.6, .4)--(1.6, -.4)--(.8, -.4)--(.8, .4)--(1.6, .4);
    \filldraw [thick, unshaded] (2.8, .4)--(2.8, -.4)--(2, -.4)--(2, .4)--(2.8, .4);
    \draw [thick, \cupcolor] (.8, .2) arc(270:90: .2cm)--(3.2, .6);
    \node at (0,0) {$r_{v}^{*}$};
    \node at (1.2, 0) {$u_{e}^{w}$};
    \node at (2.4, 0) {$r_{w}^{*}$};
    \node at (-.6, 0) {\scriptsize{$\beta_{v}$}};
    \node at (.6, 0) {\scriptsize{$\alpha_{v}$}};
    \node at (1.8, 0) {\scriptsize{$\beta_{w}$}};
    \node at (3, 0) {\scriptsize{$\alpha_{w}$}};
\end{tikzpicture}
$$
holds.  This discussion thus proves the following useful lemma:

\begin{lem} \label{lem:relation}
Let $e_{1}, \dots, e_{k}$ be all of the edges of color $c$ in the fusion graph connecting distinct vertices $v$ and $w$.  Then one can find partial isometries $\{u_{e_{i}}^{v}: 1\leq i \leq k\}$ and $\{u_{e_{i}}^{w}: 1\leq i \leq k\}$ in $\cA_{o}$ satisfying
\begin{align*}
(u_{e_{i}}^{v})^{*}u_{e_{j}}^{v} = \delta_{ij}q_{v},& \, \, \, \sum_{i=1}^{k} u_{e_{i}}^{v}(u_{e_{i}}^{v})^{*} = \eta_{c}(p_{w})\cdot 1_{v}\\
(u_{e_{i}}^{w})^{*}u_{e_{j}}^{w} = \delta_{ij}q_{w},& \, \, \, \sum_{i=1}^{k} u_{e_{i}}^{w}(u_{e_{i}}^{w})^{*} = \eta_{c}(p_{v})\cdot 1_{w}
\end{align*}
and the relations
$$
(u_{e}^{v})^{*} = (\overline{l})^{-1} \cdot
\begin{tikzpicture} [baseline = 0cm]
    \draw (-.8, -.2)--(3.2, -.2);
    \filldraw [thick, unshaded] (.4, .4)--(.4, -.4)--(-.4, -.4)--(-.4, .4)--(.4, .4);
    \filldraw [thick, unshaded] (1.6, .4)--(1.6, -.4)--(.8, -.4)--(.8, .4)--(1.6, .4);
    \filldraw [thick, unshaded] (2.8, .4)--(2.8, -.4)--(2, -.4)--(2, .4)--(2.8, .4);
    \draw [thick, \cupcolor] (.8, .2) arc(270:90: .2cm)--(3.2, .6);
    \node at (0,0) {$r_{v}^{*}$};
    \node at (1.2, 0) {$u_{e}^{w}$};
    \node at (2.4, 0) {$r_{w}^{*}$};
    \node at (-.6, 0) {\scriptsize{$\beta_{v}$}};
    \node at (.6, 0) {\scriptsize{$\alpha_{v}$}};
    \node at (1.8, 0) {\scriptsize{$\beta_{w}$}};
    \node at (3, 0) {\scriptsize{$\alpha_{w}$}};
\end{tikzpicture}\, \text{ and }
(u_{e}^{w})^{*} = l \cdot
\begin{tikzpicture} [baseline = 0cm]
    \draw (-.8, -.2)--(3.2, -.2);
    \filldraw [thick, unshaded] (.4, .4)--(.4, -.4)--(-.4, -.4)--(-.4, .4)--(.4, .4);
    \filldraw [thick, unshaded] (1.6, .4)--(1.6, -.4)--(.8, -.4)--(.8, .4)--(1.6, .4);
    \filldraw [thick, unshaded] (2.8, .4)--(2.8, -.4)--(2, -.4)--(2, .4)--(2.8, .4);
    \draw [thick, \cupcolor] (.8, .2) arc(270:90: .2cm)--(3.2, .6);
    \node at (0,0) {$r_{w}^{*}$};
    \node at (1.2, 0) {$u_{e}^{v}$};
    \node at (2.4, 0) {$r_{v}^{*}$};
    \node at (-.6, 0) {\scriptsize{$\beta_{w}$}};
    \node at (.6, 0) {\scriptsize{$\alpha_{w}$}};
    \node at (1.8, 0) {\scriptsize{$\beta_{v}$}};
    \node at (3, 0) {\scriptsize{$\alpha_{v}$}};
\end{tikzpicture}
$$
for some nonzero constant $l$.
\end{lem}

We now turn our attention over to the edges $e$ which are loops.  Suppose $v$ is a vertex and $e_{i},...,e_{k}$ represent all of the loops of color $c$ connecting $v$ to itself.  As above, we can find partial isometries $u_{e_{i}}^{v}$ with right support $q_{v}$ and orthogonal left supports under $1_{v}\cdot i_{c}(p_{v})$, however for reasons that will become apparent later we desire a stronger property about these partial isometries.

\begin{lem} \label{lem:loop}
There exists a set of partial isometries $\{u_{e_{i}}^{v}: 1 \leq i \leq k\}$ satisfying
$$
(u^{v}_{e_{i}})^{*}u^{v}_{e_{j}} = \delta_{ij} q_{v}, \, \, \, \sum_{i=1}^{k} u_{e_{i}}^{v}(u_{e_{i}}^{v})^{*} = \eta_{c}(p_{v})\cdot 1_{v},
$$
and the following relation:
$$
(u_{e}^{v})^{*} = k \cdot
\begin{tikzpicture} [baseline = 0cm]
    \draw (-.8, -.2)--(3.2, -.2);
    \filldraw [thick, unshaded] (.4, .4)--(.4, -.4)--(-.4, -.4)--(-.4, .4)--(.4, .4);
    \filldraw [thick, unshaded] (1.6, .4)--(1.6, -.4)--(.8, -.4)--(.8, .4)--(1.6, .4);
    \filldraw [thick, unshaded] (2.8, .4)--(2.8, -.4)--(2, -.4)--(2, .4)--(2.8, .4);
    \draw [thick, \cupcolor] (.8, .2) arc(270:90: .2cm)--(3.2, .6);
    \node at (0,0) {$r_{v}^{*}$};
    \node at (1.2, 0) {$u_{e}^{v}$};
    \node at (2.4, 0) {$r_{v}^{*}$};
    \node at (-.6, 0) {\scriptsize{$\beta_{v}$}};
    \node at (.6, 0) {\scriptsize{$\alpha_{v}$}};
    \node at (1.8, 0) {\scriptsize{$\beta_{v}$}};
    \node at (3, 0) {\scriptsize{$\alpha_{v}$}};
\end{tikzpicture}
$$
for $k$ a unimodular constant.
\end{lem}

\begin{proof}
Let $u_{1} \in \cA_{\I}$ be a partial isometry with right support $q_{v}$ and left support under $1_{v}\cdot i_{c}(p_{v})$.  Consider the operator $y_{1}$ satisfying
$$ y_{1}^{*} =
\begin{tikzpicture} [baseline = 0cm]
    \draw (-.8, -.2)--(.8, -.2);
    \filldraw [thick, unshaded] (-.4, -.4) -- (-.4, .4)--(.4, .4)--(.4, -.4)--(-.4, -.4);
    \draw [thick, \cupcolor] (.4, .2)--(.8, .2);
    \node at (0, 0) {$u^{*}_{1}$};
    \node at (-.6, 0) {\scriptsize{$\beta_{v}$}};
    \node at (.6, 0) {\scriptsize{$\alpha_{v}$}};
\end{tikzpicture} +
\begin{tikzpicture} [baseline = 0cm]
    \draw (-.8, -.2)--(3.2, -.2);
    \filldraw [thick, unshaded] (.4, .4)--(.4, -.4)--(-.4, -.4)--(-.4, .4)--(.4, .4);
    \filldraw [thick, unshaded] (1.6, .4)--(1.6, -.4)--(.8, -.4)--(.8, .4)--(1.6, .4);
    \filldraw [thick, unshaded] (2.8, .4)--(2.8, -.4)--(2, -.4)--(2, .4)--(2.8, .4);
    \draw [thick, \cupcolor] (.8, .2) arc(270:90: .2cm)--(3.2, .6);
    \node at (0,0) {$r_{v}^{*}$};
    \node at (1.2, 0) {$u_{1}$};
    \node at (2.4, 0) {$r_{v}^{*}$};
    \node at (-.6, 0) {\scriptsize{$\beta_{v}$}};
    \node at (.6, 0) {\scriptsize{$\alpha_{v}$}};
    \node at (1.8, 0) {\scriptsize{$\beta_{v}$}};
    \node at (3, 0) {\scriptsize{$\alpha_{v}$}};
\end{tikzpicture}.
$$
If this is zero, then by setting $k = -1$, we have produced a partial isometry satisfying the appropriate diagrammatic relation.  If not, then $y_{1}^{*}$ has left support $q_{v}$ and right support under $1_{v} \cdot i_{c}(p_{v})$ so it is a scalar multiple of a partial isometry.  The operator $y_{1}^{*}$ is also fixed under the map
$$
x \mapsto
\begin{tikzpicture} [baseline = 0cm]
    \draw (-.8, -.2)--(3.2, -.2);
    \filldraw [thick, unshaded] (.4, .4)--(.4, -.4)--(-.4, -.4)--(-.4, .4)--(.4, .4);
    \filldraw [thick, unshaded] (1.6, .4)--(1.6, -.4)--(.8, -.4)--(.8, .4)--(1.6, .4);
    \filldraw [thick, unshaded] (2.8, .4)--(2.8, -.4)--(2, -.4)--(2, .4)--(2.8, .4);
    \draw [thick, \cupcolor] (.8, .2) arc(270:90: .2cm)--(3.2, .6);
    \node at (0,0) {$r_{v}^{*}$};
    \node at (1.2, 0) {$x^{*}$};
    \node at (2.4, 0) {$r_{v}^{*}$};
    \node at (-.6, 0) {\scriptsize{$\beta_{v}$}};
    \node at (.6, 0) {\scriptsize{$\alpha_{v}$}};
    \node at (1.8, 0) {\scriptsize{$\beta_{v}$}};
    \node at (3, 0) {\scriptsize{$\alpha_{v}$}};
\end{tikzpicture}
$$
so a scalar multiple of $y_{1}$ satisfies the appropriate diagrammatic relation and is a partial isometry. In either case, we have produced a partial isometry $u_{e_{1}}^{v}$ satisfying the diagrammatic relation. To produce another partial isometry $u_{e_{2}}^{v}$ with right support $q_{v}$, left support orthogonal to that of $u_{e_{1}}^{v}$, and satisfying the diagrammatic relation, we pick a partial isometry $u_{2}$ with right support $q_{v}$ and left support orthogonal to that of $u_{e_{1}}^{v}$.  By Lemma \ref{lem:orthogonalsupport} and the discussion afterwards, the right supports of
$$
\begin{tikzpicture} [baseline = 0cm]
    \draw (-.8, -.2)--(3.2, -.2);
    \filldraw [thick, unshaded] (.4, .4)--(.4, -.4)--(-.4, -.4)--(-.4, .4)--(.4, .4);
    \filldraw [thick, unshaded] (1.6, .4)--(1.6, -.4)--(.8, -.4)--(.8, .4)--(1.6, .4);
    \filldraw [thick, unshaded] (2.8, .4)--(2.8, -.4)--(2, -.4)--(2, .4)--(2.8, .4);
    \draw [thick, \cupcolor] (.8, .2) arc(270:90: .2cm)--(3.2, .6);
    \node at (0,0) {$r_{v}^{*}$};
    \node at (1.2, 0) {$u_{2}$};
    \node at (2.4, 0) {$r_{v}^{*}$};
    \node at (-.6, 0) {\scriptsize{$\beta_{v}$}};
    \node at (.6, 0) {\scriptsize{$\alpha_{v}$}};
    \node at (1.8, 0) {\scriptsize{$\beta_{v}$}};
    \node at (3, 0) {\scriptsize{$\alpha_{v}$}};
\end{tikzpicture} \text{ and }k\cdot
\begin{tikzpicture} [baseline = 0cm]
    \draw (-.8, -.2)--(3.2, -.2);
    \filldraw [thick, unshaded] (.4, .4)--(.4, -.4)--(-.4, -.4)--(-.4, .4)--(.4, .4);
    \filldraw [thick, unshaded] (1.6, .4)--(1.6, -.4)--(.8, -.4)--(.8, .4)--(1.6, .4);
    \filldraw [thick, unshaded] (2.8, .4)--(2.8, -.4)--(2, -.4)--(2, .4)--(2.8, .4);
    \draw [thick, \cupcolor] (.8, .2) arc(270:90: .2cm)--(3.2, .6);
    \node at (0,0) {$r_{v}^{*}$};
    \node at (1.2, 0) {$u_{e_{1}}^{v}$};
    \node at (2.4, 0) {$r_{v}^{*}$};
    \node at (-.6, 0) {\scriptsize{$\beta_{v}$}};
    \node at (.6, 0) {\scriptsize{$\alpha_{v}$}};
    \node at (1.8, 0) {\scriptsize{$\beta_{v}$}};
    \node at (3, 0) {\scriptsize{$\alpha_{v}$}};
\end{tikzpicture} = (u_{e_{1}}^{v})^{*}
$$
are orthogonal so it follows that by considering the element
$$ y_{2}^{*} =
\begin{tikzpicture} [baseline = 0cm]
    \draw (-.8, -.2)--(.8, -.2);
    \filldraw [thick, unshaded] (-.4, -.4) -- (-.4, .4)--(.4, .4)--(.4, -.4)--(-.4, -.4);
    \draw [thick, \cupcolor] (.4, .2)--(.8, .2);
    \node at (0, 0) {$u^{*}_{2}$};
    \node at (-.6, 0) {\scriptsize{$\beta_{v}$}};
    \node at (.6, 0) {\scriptsize{$\alpha_{v}$}};
\end{tikzpicture} +
\begin{tikzpicture} [baseline = 0cm]
    \draw (-.8, -.2)--(3.2, -.2);
    \filldraw [thick, unshaded] (.4, .4)--(.4, -.4)--(-.4, -.4)--(-.4, .4)--(.4, .4);
    \filldraw [thick, unshaded] (1.6, .4)--(1.6, -.4)--(.8, -.4)--(.8, .4)--(1.6, .4);
    \filldraw [thick, unshaded] (2.8, .4)--(2.8, -.4)--(2, -.4)--(2, .4)--(2.8, .4);
    \draw [thick, \cupcolor] (.8, .2) arc(270:90: .2cm)--(3.2, .6);
    \node at (0,0) {$r_{v}^{*}$};
    \node at (1.2, 0) {$u_{2}$};
    \node at (2.4, 0) {$r_{v}^{*}$};
    \node at (-.6, 0) {\scriptsize{$\beta_{v}$}};
    \node at (.6, 0) {\scriptsize{$\alpha_{v}$}};
    \node at (1.8, 0) {\scriptsize{$\beta_{v}$}};
    \node at (3, 0) {\scriptsize{$\alpha_{v}$}};
\end{tikzpicture}.
$$
and arguing as in the beginning of the proof, we produce the desired element $u_{e_{2}}^{v}$.  Iterating this procedure produces the set $\{u_{e_{i}}^{v}: 1 \leq i \leq k\}$.
\end{proof}

We therefore assume in the rest of the section that our partial isometries satisfy the relations in either Lemma \ref{lem:relation} or Lemma \ref{lem:loop}.  The identity $\sum_{v} \sum_{e\sim v} u_{e}^{v}(u_{e}^{v})^{*} = \sum_{c \in \cL} \eta_{c}(Q)$ follows by the choices of the partial isometries $u_{e}^{v}$.  We therefore have the following lemma:

\begin{lem} \label{lem:compression2}
For each edge $e$ of the fusion graph $\cF_\cC(\cL))$, we define operators $Y_{e}$ as follows:  If $C(e) = c$ and $e$ connects distinct vertices $v$ and $w$ then $Y_{e} = p_{v}X_{c}u_{e}^{w} + (u_{e}^{w})^{*}X_{v}p_{v} + p_{w}X_{c}u_{e}^{v} + (u_{e}^{v})^{*}X_{c}p_{w}$ and if $e$ connects the vertex $v$ to itself then $Y_{e} = p_{v}X_{c}u_{e}^{v} + (u_{e}^{v})^{*}X_{c}p_{v}$.  Then we have $F\cM_{\I}F = W^{*}(F\cA_{\I}F, \, (Y_{e})_{e \in E})$
\end{lem}

\begin{proof}
First note that by compressing by the appropriate elements, of $F\cA_{\I}F$, each term in the sums defining $Y_{e}$ is in $W^{*}(F\cA_{\I}F, \, (Y_{e})_{e \in E(\cF_\cC(\cL))})$.  The identity
$$
\sum_{v \in V(\cF_\cC(\cL))} \sum_{\substack{e: s(e) = v \\ C(e) = c}} = p_{v}X_{c}u_{e}^{t(e)}(u_{e}^{t(e)})^{*} = TX_{c}T
$$
implies that words in $TX_{c}T$ (compressed on either end by $F$) can be approximated by words in $p_{v}X_{c}u_{e}^{w}$ and their adjoints so by Lemma \ref{lem:compression1} we are done.
\end{proof}

\begin{lem} \label{lem:EasyGenerators}
If $v$ and $w$ are distinct vertices connected by an edge $e$ with $C(e) = c$ then $r_{v}(u^{v}_{e})^{*}X_{c}p_{w}r_{w}$ and $p_{v}X_{c}(u_{e}^{w})$ are nonzero scalar multiples of each other.

If $f$ is a loop at $v$ then $r_{v}(u^{v}_{f})^{*}X_{c}p_{v}r_{v}$ and $p_{v}X_{c}(u_{f}^{v})$ are nonzero scalar multiples of each other.

\end{lem}

\begin{proof}
First observe that
$$
p_{v}X_{c}u_{e}^{w} = X_{c}\eta_{c}(p_{v})u_{e}^{w} = X_{c}u_{e}^{w}.
$$
and $r_{v}(u^{v}_{e})^{*}X_{c}p_{w}r_{w} = r_{v}(u_{v}^{e})^{*}X_{c}r_{w}$.  The element $r_{v}(u_{v}^{e})^{*}X_{c}r_{w}$ is represented diagrammatically as
$$
\begin{tikzpicture} [baseline = 0cm]
    \draw (-.8, 0) -- (4.6, 0);
    \filldraw [thick, unshaded] (.4, .4) --(.4, -.4)--(-.4, -.4)--(-.4, .4)--(.4, .4);
    \filldraw [thick, unshaded] (1.8, .4) --(1.8, -.4)--(.8, -.4)--(.8, .4)--(1.8, .4);
    \filldraw [thick, unshaded] (4.2, .4) --(4.2, -.4)--(3.4, -.4)--(3.4, .4)--(4.2, .4);
    \draw [thick] (3, .4) --(3, -.4)--(2.2, -.4)--(2.2, .4)--(3, .4);
    \draw [thick, \cupcolor] (1.8, .2)--(2.4, .2) arc(-90:0: .2cm);
    \node at (0,0) {$r_{v}$};
    \node at (1.3, 0) {$(u_{e}^{v})^{*}$};
    \node at (3.8, 0) {$r_{w}$};
    \node at (-.6, .2) {\scriptsize{$\alpha_{v}$}};
    \node at (.6, .2) {\scriptsize{$\beta_{v}$}};
    \node at (2, -.2) {\scriptsize{$\alpha_{w}$}};
    \node at (4.4, .2) {\scriptsize{$\beta_{w}$}};
\end{tikzpicture}
$$
which can be rewritten as
$$
\begin{tikzpicture} [baseline = 0cm]
    \draw (-.8, 0) -- (4.6, 0);
    \draw [thick] (-.4, -.4)--(-.4, .8)--(.4, .8)--(.4, -.4)--(-.4, -.4);
    \filldraw [thick, unshaded] (1.6, .4) --(1.6, -.4)--(.8, -.4)--(.8, .4)--(1.6, .4);
    \filldraw [thick, unshaded] (3, .4) --(3, -.4)--(2, -.4)--(2, .4)--(3, .4);
    \filldraw [thick, unshaded] (4.2, .4) --(4.2, -.4)--(3.4, -.4)--(3.4, .4)--(4.2, .4);
    \node at (0, .2) {\scriptsize{$\alpha_{v}$}};
    \node at (1.8, .2) {\scriptsize{$\beta_{v}$}};
    \node at (3.2, -.2) {\scriptsize{$\alpha_{w}$}};
    \node at (4.4, .2) {\scriptsize{$\beta_{w}$}};
    \draw [thick, \cupcolor] (3, .2) arc(-90:90: .2cm) -- (.2, .6) arc(270:180: .2cm);
    \node at (1.2, 0) {$r_{v}$};
    \node at (2.5, 0) {$(u_{e}^{v})^{*}$};
    \node at (3.8, 0) {$r_{w}$};
\end{tikzpicture}\,  = k \cdot X_{c}\cdot u_{e}^{w}
$$
for an appropriate constant $k \neq 0$, proving the first statement.  The proof of the second statement uses the exact same diagrammatic argument.
\end{proof}

We now examine the algebras
$$
\cM_{e} = W^{*}(F\cA_{\I}F, Y_{e}) = W^{*}(F\cA_{\I}F, p_{v}X_{c}u_{e}^{w} + (u_{e}^{w})^{*}X_{c}p_{v})
$$
where the last equality follows from Lemma \ref{lem:EasyGenerators}. First assume $v$ and $w$ are distinct and $\Tr(p_{v}) \leq \Tr(p_{w}) = \Tr(q_{w})$.  One can show, using exactly the same techniques as in \cite{MR2732052} and \cite{MR2807103} (see the appendix) that $(u_{e}^{w})^{*}X_{c}p_{v}X_{c}u_{e}^{w}$ is a free Poisson element with an atom of size $\frac{\Tr(p_{w}) - \Tr(p_{v})}{\Tr(p_{w})}$ at 0 in the compressed algebra $p_{v}\cM_{\infty}p_{v}$. Also, $p_{v}X_{c}u_{e}^{w}(u_{e}^{w})^{*}X_{c}p_{v}$ is a free poisson element with no atoms in the compressed algebra $q_{v}\cM_{\I}q_{v}$.  Using the polar part of $p_{v}X_{c}u_{e}^{w}$ as well as the partial isometries $r_{v}$ and $r_{w}$ and Lemma \ref{lem:EasyGenerators}, we see that
$$
(1_{v} + 1_{w})F\cM_{e}F(1_{v} + 1_{w}) = \left(L(\Z)\otimes M_{4}(\C)\right) \bigoplus \overset{p'_{v}, q'_{v}}{\underset{\Tr(p_{v}) - \Tr(p_{w})}{M_{2}(\C)}}
$$
where the notation means that the copy of $M_{2}(\C)$ has two orthogonal minimal projections $p_{v}'$ and $q_{v}'$ such that $p_{v}' \leq p_{v}$ and $q'_{v} \leq q_{v}$ and $\Tr(p_{v}') = \Tr(p_{v}) - \Tr(p_{w})$.

If $e$ is a loop at $v$ then using the partial isometry $r_{v}$ and Lemma \ref{lem:EasyGenerators}, we see that
$$
1_{v}\cdot F\cM_{e}F\cdot 1_{v} = L(\Z) \otimes M_{2}(\C).
$$
For simplicity, we set $X_{e} = p_{v}X_{c}u_{e}^{w} + (u_{e}^{w})^{*}X_{c}p_{v}$ so that $\cM_{e} = W^{*}(F\cM_{\I}F, X_{e})$.

\begin{lem} \label{lem:free2}
The algebras $(\cM_{e})_{e \in E(\cF_{\cC}(\cL))}$ are free with amalgamation over $F\cA_{\I}F$.
\end{lem}

\begin{proof}
From \cite{MR1704661} the elements $X_{e}$ are $F\cA_{\I}F$ semicircular, so to show freeness we need only to show that if $e \neq e'$, $E(X_{e} y X_{e'}) = 0$ for all $y \in F\cA_{\I}F$.  This can only be nonzero if $e$ and $e'$ have the came color, so we assume $C(e) = c = C(e')$.  Also, this expectation can be nonzero only if $e$ and $e'$ share a common vertex, $v$.

We first assume that $e$ and $e'$ are not loops.  Assume $e$ connects $v$ and $w$ and $e'$ connects $v$ and $w'$.  We see that $E(X_{e}p_{v}X_{e'}) = (u_{e}^{w})^{*}i_{c}(p_{v})u_{e'}^{w'} = (u_{e}^{w})^{*}u_{e'}^{w'}$ which is zero since $u_{e}^{w}$ and $u_{e'}^{w'}$ have orthogonal left supports if $w = w'$ and is clearly zero if $w \neq w'$.  For $E(X_{e}q_{w}X_{e'})$ to be nonzero, we need $w = w'$.  Assuming this, we get $p_{v}\Phi(u_{e}^{w}\cdot q_{w} \cdot (u_{e'}^{w'})^{*}) p_{v}$.  As was discussed in the proof of Lemma \ref{lem:orthogonalsupport},
$$
\Phi(u_{e}^{w}\cdot q_{w} \cdot (u_{e'}^{w'})^{*}) = \Phi(u_{e}^{w}(u_{e'}^{w'})^{*}) = 0
$$
since $u_{e}^{w}$ and $u_{e'}^{w'}$ have orthogonal left supports.  It is also straightforward to check that the only elements $y$ in $F\cA_{\I}F$ where $X_{e}yX_{e'} \neq 0$ are of the form $y = a\cdot p_{v} + b\cdot q_{w}$.

We now assume that $e$ is a loop at $v$ and $e'$ is an arbitrary edge connected to $v$.  The discussion in the previous paragraph implies that $E(X_{e}pX_{e'}) = 0$ for any projection $p \in F\cA_{\I}F$; however, we also have to consider $E(X_{e}r_{v}^{*}X_{e'})$.  This is the following diagram:
$$
\begin{tikzpicture}[baseline = 0cm]
    \draw (-.8, -.2)--(5.6, -.2);
    \filldraw [thick, unshaded] (-.4, -.4)--(-.4, .4)--(.4, .4)--(.4, -.4)--(-.4, -.4);
    \filldraw [thick, unshaded] (.8, -.4)--(.8, .4)--(1.6, .4)--(1.6, -.4)--(.8, -.4);
    \filldraw [thick, unshaded] (2, -.4)--(2, .4)--(2.8, .4)--(2.8, -.4)--(2, -.4);
    \filldraw [thick, unshaded] (3.2, -.4)--(3.2, .4)--(4, .4)--(4, -.4)--(3.2, -.4);
    \filldraw [thick, unshaded] (4.4, -.4)--(4.4, .4)--(5.2, .4)--(5.2, -.4)--(4.4, -.4);
    \draw [thick, \cupcolor] (.8, .2) arc(270:90: .2cm) -- (4, .6) arc(90:0: .2cm) arc(180:270:.2cm);
    \node at (0, 0) {$p_{v}$};
    \node at (1.2, 0) {$u_{e}^{v}$};
    \node at (2.4, 0) {$r_{v}^{*}$};
    \node at (3.6, 0) {$p_{v}$};
    \node at (4.8, 0) {$u_{e'}^{v}$};
    \node at (-.6, 0) {\scriptsize{$\alpha_{v}$}};
    \node at (.6, 0) {\scriptsize{$\alpha_{v}$}};
    \node at (1.8, 0) {\scriptsize{$\beta_{v}$}};
    \node at (3, 0) {\scriptsize{$\alpha_{v}$}};
    \node at (4.2, 0) {\scriptsize{$\alpha_{v}$}};
    \node at (5.4, 0) {\scriptsize{$\beta_{v}'$}};
\end{tikzpicture} =
\begin{tikzpicture} [baseline = 0cm]
    \draw (-.8, -.2)--(5.6, -.2);
    \filldraw [thick, unshaded] (-.4, -.4)--(-.4, .4)--(.4, .4)--(.4, -.4)--(-.4, -.4);
    \filldraw [thick, unshaded] (.8, -.4)--(.8, .4)--(1.6, .4)--(1.6, -.4)--(.8, -.4);
    \filldraw [thick, unshaded] (2, -.4)--(2, .4)--(2.8, .4)--(2.8, -.4)--(2, -.4);
    \filldraw [thick, unshaded] (3.2, -.4)--(3.2, .4)--(4, .4)--(4, -.4)--(3.2, -.4);
    \filldraw [thick, unshaded] (4.4, -.4)--(4.4, .4)--(5.2, .4)--(5.2, -.4)--(4.4, -.4);
    \draw [thick, \cupcolor] (2, .2) arc(270:90: .2cm) -- (4, .6) arc(90:0: .2cm) arc(180:270:.2cm);
    \node at (0, 0) {$r_{v}$};
    \node at (1.2, 0) {$r_{v}^{*}$};
    \node at (2.4, 0) {$u_{e}^{v}$};
    \node at (3.6, 0) {$r_{v}^{*}$};
    \node at (4.8, 0) {$u_{e'}^{v}$};
    \node at (-.6, 0) {\scriptsize{$\alpha_{v}$}};
    \node at (.6, 0) {\scriptsize{$\beta_{v}$}};
    \node at (1.8, 0) {\scriptsize{$\alpha_{v}$}};
    \node at (3, 0) {\scriptsize{$\beta_{v}$}};
    \node at (4.2, 0) {\scriptsize{$\alpha_{v}$}};
    \node at (5.4, 0) {\scriptsize{$\beta_{v}'$}};
\end{tikzpicture}\, ,
$$
where $\beta_{v}' = \beta_{v}$ if $e'$ is also a loop at $v$.  By our choice of $u_{v}^{e}$, this diagram is a scalar multiple of $r_{v}\cdot (u_{e}^{v})^{*}u_{e'}^{v}$ which is 0.  Similarly, $E(X_{e}r_{v}^{*}X_{e'}) = 0$ and we are done.
\end{proof}

Note that we have
$$
F\cM_{\I}F = \underset{F\cA_{\I}F}{*} (\cM_{e})_{e \in E(\cF_{\cC}(\cL))}.
$$
The algebra $F\cA_{\I}F$ has the decomposition
$$
F\cA_{\I}F = \bigoplus_{v \in \cF_{\cC}(\cL)} \overset{p_{v}, q_{v}}{M_{2}(\C)}.
$$
Recall that $Q = \sum_{v} p_{v}$.  Thus $Q$ is an abelian projection with central support 1 in $F\cA_{\I}F$.  Therefore (see for example \cite{MR2765550} or \cite{MR2051399})
$$
Q\cM_{\I}Q = \underset{Q\cA_{\I}Q}{*} (Q\cM_{e}Q)_{e \in E(\cF_{\cC}(\cL))}.
$$
The algebra $Q\cA_{\I}Q$ is simply $\ell^{\I}(V(\cF_{\cC}(\cL)))$, the bounded functions on the vertices of $\cF_{\cC}(\cL)$.  If $e$ connects two distinct vertices $v$ and $w$ with $\Tr(p_{v}) \geq \Tr(p_{w})$ then
$$
Q\cM_{e}Q = L(\Z) \otimes M_{2}(\C) \bigoplus \overset{p_{v}'}{\C} \bigoplus \ell^{\I}(V(\cF_{\cC}(\cL)) \setminus \{v, w\})
$$
and if $e$ is a loop at $v$ then
$$
Q\cM_{e}Q = L(\Z) \bigoplus \ell^{\I}(V(\cF_{\cC}(\cL)) \setminus \{v\}).
$$

In the next section, we will use the free product expression for $Q\cM_{\I}Q$ to determine the isomorphism class of $M = p_{\emptyset}(Q\cM_{\I}Q)p_{\emptyset}$.

\subsection{von Neumann algebras associated to graphs} \label{sec:vNAGraph}

The following notation will be useful in this section:

\begin{nota} \label{nota:vNA}
Throughout this section, we will be concerned with finite von Neumann algebras $(\cN, tr)$ which can be written in the form
$$
 \cN = \overset{p_{0}}{\underset{\gamma_{0}}{\cN_{0}}} \oplus \bigoplus_{j \in J} \overset{p_{j}}{\underset{\gamma_{j}}{L(\F_{t_{j}})}} \oplus  \bigoplus_{k \in K} \overset{q_{k}}{\underset{\alpha_{k}}{M_{n_{k}}}}
 $$
 where $\cN_{0}$ is a diffuse hyperfinite von Neumann algebra, $L(\F_{t_{j}})$ is an interpolated free group factor with parameter $t_{j}$, $M_{n_{k}}$ is the algebra of $n_{k} \times n_{k}$ matrices over the scalars, and the sets $J$ and $K$ are at most finite and countably infinite respectively.  We use $p_{j}$ to denote the projection in $L(\F_{t_{j}})$ corresponding to the identity of $L(\F_{t_{j}})$ and $q_{k}$ to denote a minimal projection in $M_{n_{k}}$. The projections $p_{j}$ and $q_{k}$ have traces $\gamma_{j}$ and $\alpha_{k}$ respectively.  Let $p_{0}$ be the identity in $\cN_{0}$ with trace $\gamma_{0}$.   We write $\overset{p,q}{M_{2}}$ to mean $M_{2}$ with a choice of minimal orthogonal projections $p$ and $q$.
 \end{nota}

We begin by letting $\Gamma$ be a connected weighted graph with weighting $\gamma$.  Let $\ell^{\I}(\Gamma)$ be the bounded functions on the vertices of $\Gamma$ and let $p_{v}$ be the delta function at $v$.  We endow $\ell^{\I}(\Gamma)$ with a trace $\Tr$ satisfying $\Tr(p_{v}) = \gamma_{v}$.  As in \cite{1208.2933} we now describe how to use the edges to associate a free product von Neumann algebra $\cN(\Gamma)$ to $\Gamma$.

\begin{defn}
Let $e$ be an edge in $\Gamma$.  We define algebras $\cN_{e}$ as follows:  If $e$ connects two distinct edges $v$ and $w$ with $\gamma_{v} \geq \gamma_{w}$ then
$$
\cN_{e} = \underset{2\gamma_{w}}{M_{2}(\C) \otimes L(\Z)} \oplus \underset{\gamma_{v} - \gamma_{w}}{\overset{p^{e}_{v}}{\C}} \oplus \ell^{\infty}(\Gamma \setminus \{v, w\}).
$$
If $e$ is a loop at the vertex $v$ then
$$
\cN_{e} = \overset{p_{v}}{\underset{\gamma_{v}}{L(\Z)}} \oplus \ell^{\infty}(\Gamma \setminus \{v\}).
$$
If $e$ connects two different vertices then the trace on $\underset{2\gamma_{w}}{M_{2}(\C) \otimes L(\Z)} \oplus \underset{\gamma_{v} - \gamma_{w}}{\overset{p^{e}_{v}}{\C}}$ is given by $\tr_{M_{2}} \otimes \tr_{L(\Z)} + \tr_{\C}$. $\cN_{e}$ includes $\ell^{\I}$ $\Gamma$ by letting $p_{w} = e_{1, 1}\otimes 1$ and $p_{v} = e_{2, 2}\otimes 1 + p_{v}^{e}$ so that $p^{e}_{v} \leq p_{v}$.  If $e$ is a loop that $\cN_{e}$ then $\cN_{e}$ includes $\ell^{\I}(\Gamma)$ in the obvious way.

Let $E_{e}$ be the $\Tr-$preserving conditional expectation from $\cN_{e}$ to $\ell^{\I}(\Gamma)$.  We define $\cN(\Gamma) = \underset{\ell^{\I}(\Gamma)}{*} (\cN_{e}, E_{e})_{e \in E(\Gamma)}$.
\end{defn}

Note that if $\Gamma$ has no loops then this definition of $\cN(\Gamma)$ is the same as $\cM(\Gamma)$ in \cite{1208.2933}.  Also observe that $Q\cM_{\I}Q$ in section \ref{sec:free1} is $\cN(\cF_{\cC}(\cL))$.  We now define some notation which will be useful in determining the isomorphism class of $\cN(\Gamma)$.

\begin{defn} \label{defn:H2} We write $v \sim w$ if $v$ and $w$ are connected by at least 1 edge in $\Gamma$ and denote $n_{v, w}$ be the number of edges joining $v$ and $w$.  We set $\alpha^{\Gamma}_{v} = \sum_{w\sim v} n_{v, w}\gamma_{w}$, and define $B(\Gamma) = \{ v \in V(\Gamma) : \gamma_{v} > \alpha^{\Gamma}_{v}\}$. Note that if there is a loop, $e$, at $v$ then $v \not\in B(\Gamma)$. \end{defn}

Assume for the moment that $\Gamma \subset \Gamma'$ are finite, connected, weighted graphs with at least two edges (so that $\cN(\Gamma)$ and $\cN(\Gamma')$ are finite von Neumann algebras).  There is a natural inclusion $\cN(\Gamma) \rightarrow \cN(\Gamma')$ which will not be unital if $\Gamma'$ has a larger vertex set.  We will prove the following theorem, which is along the same lines as \cite{1208.2933}.

\begin{thm} \label{thm:standardembedding}
$\cN(\Gamma)$ has the form
$$
\cN(\Gamma) \cong \overset{p^{\Gamma}}{L(\F_{t_{\Gamma}})} \oplus \underset{{v \in B(\Gamma)}}{\bigoplus} \overset{r_{v}^{\Gamma}}{\underset{\gamma_{v} - \alpha^{\Gamma}_{v}}{\C}}
$$
where $r_{v}^{\Gamma} \leq p_{v}$ and $t_{\Gamma}$ is such that this algebra has the appropriate free dimension.  If $\Gamma$ is a proper subgraph of $\Gamma'$ then the unital inclusion $p^{\Gamma}\cN(\Gamma)p^{\Gamma} \rightarrow p^{\Gamma}\cN(\Gamma')p^{\Gamma}$ is a standard embedding of interpolated free group factors.
\end{thm}

See \cite{MR1201693} and \cite{MR2765550} for a discussion of free dimension and rules for computing it.  Also we refer the reader to \cite{MR1201693} for the definition of a standard embedding of interpolated free group factors.  Whenever $A$ and $B$ are interpolated free group factors and $A$ is unitally included into $B$ then we write $A \overset{s.e.}{\hookrightarrow} B$ to indicate that the inclusion of $A$ into $B$ is a standard embedding.  In this section, we will extensively use the following properties of standard embeddings which can be found in \cite{MR1201693} and \cite{MR1363079}.

\begin{itemize}
\item[(1)]
If $A$ is an interpolated free group factor, the canonical inclusion $A \rightarrow A * \cN$ is a standard embedding whenever $\cN$ is of the form in Notation \ref{nota:vNA}.
\item[(2)]
A composite of standard embeddings is a standard embedding.
\item[(3)]
If $A_{n} = L(\F_{s_{n}})$ with $s_{n} < s_{n+1}$ for all $n$ and $\phi_{n}: A_{n} \overset{s.e.}{\hookrightarrow} A_{n+1}$, then the inductive limit of the $A_{n}$ with respect to the $\phi_{n}$ is $L(\F_{s})$ where $s = \displaystyle \lim_{n \rightarrow \infty}s_{n}$.
\item[(4)]
If $t > s$  then $\phi: L(\F_{s}) \overset{s.e.}{\hookrightarrow} L(\F_{t})$ if and only if for any nonzero projection $p \in L(\F_{s})$, $\phi|_{pL(\F_{s})p}: pL(\F_{s})p \overset{s.e.}{\hookrightarrow} \phi(p)L(\F_{t})\phi(p)$.
\end{itemize}

Theorem \ref{thm:standardembedding} was proved in \cite{1208.2933} in the case that $\Gamma$ contained no loops, so we only to need to modify the arguments there to incorporate what happens when $\Gamma$ contains loops.  We will prove Theorem \ref{thm:standardembedding} by inducting on the number of edges in the graph.  We divide this into two lemmas.

\begin{lem} \label{lem:basecase}
Suppose $\Gamma$ is a connected graph with 2 edges, one of which is a loop.  Then $\cN(\Gamma)$ is of the form in Theorem \ref{thm:standardembedding}.
\end{lem}

\begin{proof}
There are two cases to consider.  The first when $\Gamma$ has one vertex with two loops and the other when $\Gamma$ has two vertices, $w$ and $w$ with a loop at $v$ and an edge connecting $v$ to $w$.

\item[\text{\underline{Case 1:}}]
Assume $\Gamma$ has one vertex with two loops.  It follows immediately from the definition of $\cN(\Gamma)$ that $\cN(\Gamma) = L(\F_{2})$ which is in agreement with Theorem \ref{thm:standardembedding}.\\

\item[\text{\underline{Case 2:}}]
Assume $\Gamma$ has two vertices $w$ and $w$ with a loop, $e$, at $v$ and an edge, $f$, connecting $v$ to $w$.  There are two subcases to consider:  when $\gamma_{v} \geq \gamma_{w}$ and $\gamma_{v} < \gamma_{w}$.

\item[\text{\underline{Case 2a:}}]
Assume $\gamma_{v} \geq \gamma_{w}$ and set $D  = \ell^{\I}(\Gamma)$.  $\cN(\Gamma)$ has the form
$$
\cN(\Gamma) = \left(\overset{p_{v}}{\underset{\gamma_{v}}{L(\Z)}} \oplus \overset{p_{w}}{\underset{\gamma_{w}}{\C}}\right) \underset{D}{*} \left(\underset{2\gamma_{w}}{M_{2} \otimes L(\Z)} \oplus \overset{p^{f}_{v}}{\underset{\gamma_{v} - \gamma_{w}}{\C}}\right).
$$
Note that the central support of $p_{v}$ is 1 in $\cN_{f}$.  By \cite{1110.5597},
$$
p_{v}\cN(\Gamma)p_{v} = L(\Z) * p_{v}\cN_{f}p_{v} = L(\Z) * \left(\underset{\gamma_{w}}{L(\Z)} \oplus \overset{p_{f}^{v}}{\underset{\gamma_{v} - \gamma_{w}}{\C}}\right)
$$
which is an interpolated free group factor $L(\F_{t})$ for some appropriate $t$.  By amplifying, it follows that $\cN(\Gamma)$ is an interpolated free group factor, in agreement with Theorem \ref{thm:standardembedding}.

\item[\text{\underline{Case 2b:}}]
Assume $\gamma_{v} < \gamma_{w}$.  We note that the central support of $p_{v}$ in $\cN_{f}$ is $1 - p_{w}^{f}$ and following the same algorithm as in Case 2a gives
$$
p_{v}\cN(\Gamma)p_{v} = L(\Z) * p_{v}\cN_{f}p_{v} = L(\Z) * L(\Z) = L(\F_{2}),
$$
so by amplifying, $\cN(\Gamma) = L(\F_{t}) \oplus \overset{p_{w}^{f}}{\underset{\gamma_{w}}{\C}}$ for appropriate $t$.  This agrees with Theorem \ref{thm:standardembedding}.
\end{proof}

\begin{lem} \label{lem:inductivestep}
Suppose $\Gamma \subset \Gamma'$ are finite weighted graphs and assume $\cN(\Gamma)$ is of the form in Theorem \ref{thm:standardembedding}.  In addition, suppose $\Gamma'$ is obtained from $\Gamma$ by one of the following three operations:

\begin{itemize}
\item[(1)]
$\Gamma$ and $\Gamma'$ have the same edge set and $\Gamma'$ is obtained from $\Gamma$ be adding a loop $e$ at a vertex $v$,
\item[(2)]
$\Gamma$ and $\Gamma'$ have the same edge set and $\Gamma'$ is obtained from $\Gamma$ be adding an edge $e$ between two distinct vertices $v$ and $w$, or
\item[(3)]
$\Gamma'$ is obtained from $\Gamma$ be adding a vertex $v$ and an edge $e$ connecting $w \in V(\Gamma)$ to $v$.
\end{itemize}

Then $\cN(\Gamma')$ is also of the form in Theorem \ref{thm:standardembedding} and $p^{\Gamma}\cN(\Gamma)p^{\Gamma} \overset{s.e.}{\hookrightarrow} p^{\Gamma}\cN(\Gamma')p^{\Gamma}$.
\end{lem}

\begin{proof}
The steps in proving the lemma assuming operations (2) or (3) are exactly the same as in \cite{1208.2933}, so we only need to assume that $\Gamma'$ is obtained from $\Gamma$ via operation (1).  Set $D = \ell^{\I}(\Gamma) = \ell^{\I}(\Gamma')$.  We have $\cN(\Gamma') = \cN_{e} \underset{D}{*} \cN(\Gamma)$ and by assumption, $$
p_{v}\cN(\Gamma)p_{v} = \overset{p_{v}^{\Gamma}}{L(\F_{t})} \oplus \overset{r_{v}^{\Gamma}}{\C}
$$
where $r_{v}^{\Gamma}$ can possibly be zero but $p_{v}^{\Gamma} = p_{v}\cdot p^{\Gamma} \neq 0$.  By \cite{1110.5597},
$$
p_{v}\cN(\Gamma')p_{v} = L(\Z) * p_{v}\cN(\Gamma)p_{v} = L(\Z) * \left(\overset{p_{v}^{\Gamma}}{L(\F_{t})} \oplus \overset{r_{v}^{\Gamma}}{\C}\right)
$$
which is an interpolated free group factor.  By amplification, it follows that $\cN(\Gamma')$ is of the form in Theorem \ref{thm:standardembedding}.  By \cite{MR1201693}, The inclusion $p^{\Gamma}_{v}\cN(\Gamma)p^{\Gamma}_{v} \rightarrow p^{\Gamma}_{v}\cN(\Gamma')p^{\Gamma}_{v}$ is equivalent to the inclusion
$$
L(\F_{t}) \rightarrow L(\F_{t}) * p^{\Gamma}_{v}\left(L(\Z) * \left(\overset{p_{v}^{\Gamma}}{\C} \oplus \overset{r_{v}^{\Gamma}}{\C}\right)\right)p^{\Gamma}_{v}
$$
which is a standard embedding.  As $p_{v}^{\Gamma} \leq p^{\Gamma}$ it follows that $p^{\Gamma}\cN(\Gamma)p^{\Gamma} \overset{s.e.}{\hookrightarrow} p^{\Gamma}\cN(\Gamma')p^{\Gamma}$.
\end{proof}

We notice that Theorem \ref{thm:standardembedding} follows from Lemmas \ref{lem:basecase} and \ref{lem:inductivestep}.  Indeed, if $\Gamma \subset \Gamma'$ are finite, connected, weighted graphs then $\Gamma'$ can be obtained from $\Gamma$ by applying operations (1), (2), and (3) above.  Also, the composite of standard embeddings is a standard embedding and standard embeddings are preserved by cut-downs.

Before determining the isomorphism $M \cong L(\F_{\I})$ we note the following lemma whose proof is a straightforward induction exercise using the algorithms in \cite{1110.5597}.

\begin{lem} \label{lem:star}
Suppose $\Gamma$ consists of a vertex $v$ connected to vertices $w_{1}$ through $w_{k}$ by a total of $n$ edges.  Assume further that $\gamma_{v} \leq \gamma_{w_{i}}$ for all $i$.  Then $p_{v}\cN(\Gamma)p_{v} = L(\F_{t})$ where $t \geq n$.
\end{lem}

\begin{thm}
$M \cong L(\F_{\I})$.
\end{thm}

\begin{proof}
Let $\Gamma$ denote $\cF_{\cC}(\cL)$ so that $M \cong p_{\emptyset}\cN(\Gamma)p_{\emptyset}$.  We build up $\Gamma$ by an increasing union of finite connected subgraphs $\Gamma_{k}$, each of which contain the vertex * so that $M$ is the inductive limit of the algebras $p_{\emptyset}\cN(\Gamma_{k})p_{\emptyset}$.  All vertices in $\Gamma_{k}$ must have weight larger than $1$ since for $p$ an irreducible projection in $\cP$, $p \otimes \overline{p}$ must have a subprojection equivalent to the trivial one.  It follows that for $k$ sufficiently large, $p_{\emptyset}\cN(\Gamma_{k})p_{\emptyset}$ is an interpolated free group factor and since $p_{\emptyset}\cN(\Gamma_{k})p_{\emptyset} \overset{s.e.}{\hookrightarrow} p_{\emptyset}\cN(\Gamma_{k+1})p_{\emptyset}$ we know that $M = p_{\emptyset}\cN(\Gamma)p_{\emptyset}$ must be an interpolated free group factor.  Let $\tilde{\Gamma_{n}}$ be the subgraph of $\Gamma_{n}$ whose vertices are * and $v_{1}$, ..., $v_{k}$ which are connected to $*$ in $\Gamma_{k}$ and whose edges are exactly the edges connecting $*$ to each $v_{i}$ to * in $\Gamma_{n}$.  Assuming there are $k_{n}$ such vertices, we see from Lemma \ref{lem:star} that $p_{\emptyset}\cN(\tilde{\Gamma_{n}})p_{\emptyset} = L(\F_{r_{n}})$ for $r_{n} \geq k_{n}$. Since either $p_{\emptyset}\cN(\tilde{\Gamma_{n}})p_{\emptyset} =  p_{\emptyset}\cN(\Gamma_{n})p_{\emptyset}$ or $p_{\emptyset}\cN(\tilde{\Gamma_{n}})p_{\emptyset} \overset{s.e.}{\hookrightarrow} p_{\emptyset}\cN(\Gamma_{n})p_{\emptyset}$ it follows that $p_{\emptyset}\cN(\Gamma_{n})p_{\emptyset} = L(\F_{t_{n}})$ for $t_{n} \geq r_{n}$.  There are infinitely many edges emanating from $*$.  Therefore $k_{n}$ and hence $t_{n}$ can be arbitrarily large so it follows that $M = L(\F_{t})$ where $t = \lim_{n} t_{n} = \I$.
\end{proof}

\begin{cor}
$M_{\alpha} \cong L(\F_{\I})$ for all $\alpha \in \Lambda$.
\end{cor}
\begin{proof}
$M_{\alpha}$ is an amplification of $M$.
\end{proof}

\begin{rem}\label{rem:Finite}
With a bit more careful analysis, using the techniques in \cite{1208.2933}, one can show that if $\cC$ has infinitely many simple objects then the factor $M$ is $L(\F_{\I})$ no matter the choice of the generating set $\cS$ (using $\cL=\cS\oplus \overline{\cS}$).  When $\cC$ has finitely many simple objects, then we can find $t$ finite with $M \cong L(\F_{t})$.  To do this, we can find a single object $X$ which generates $\cC$, and we let $\cL=\{X\oplus \overline{X}\}$.  Applying the analysis in Section \ref{sec:free1} shows that $M$ is a cutdown of $\cN(\cF_\cC(X\oplus\overline{X}))$.  Keeping track of the free dimension yields the isomorphism $M \cong L(\F(1 + \dim(\cC)(2\dim(X) - 1)))$.
\end{rem}

\section{Appendix}\label{sec:Appendix}

\subsection{The factors in \cite{MR1960417} have property $\Gamma$}
First, we sketch the proof that the factors in \cite{MR1960417} have property $\Gamma$, and thus are not interpolated free group factors. We use some notation from \cite{MR1960417}.

\begin{proof}[Sketch of proof]
Let $\cC$ be a countably generated rigid $C^*$-tensor category.
Let $\cS$ be the set of isomorphism classes of simple objects.
Let $\set{e_s}{s\in \cS}$ be a family of non-zero, pairwise orthogonal projections in the hyperfinite $II_1$ factor $R$.
Consider the von Neumann algebras
$$
A=\bigoplus_{s\in \cS}e_sRe_s\text{ and }B=\bigoplus_{s,t\in\cS}\mathcal \cC(s,t)\overline{\otimes} e_sRe_t.
$$
Note that we have a unital, connected inclusion $A\subset B$.
The factors considered in \cite{MR1960417} are of the form
$$
N=(Q\overline\otimes A)*_A B,
$$
where $Q$ is an arbitrary finite von Neumann algebra.
To show that $N$ admits a non-trivial central sequence, it suffices to find a sequence in $A$ which commutes asymptotically with $B$.
Let $x=(x_n)$ be a central sequence of $R$.
One can check that the sequence $(y_n)$ given by
$$
y_n=\bigoplus_{s\in\cS}e_sx_ne_s\in A
$$
has the desired property.
Hence $N$ has property $\Gamma$ and is not isomorphic to an interpolated free group factor.
\end{proof}

\subsection{Proofs using $\lambda$-lattices}
We now sketch a proof of Theorem \ref{thm:Main} using the techniques of \cite{MR2051399} in the case where the rigid $C^*$-tensor category $\cC$ is finitely generated. This sketch closely follows \cite[Theorem 4.1]{1112.4088}.
\begin{proof}[Sketch of proof]
Suppose $\cS=\{X_1,\dots,X_r\}$ generates $\cC$.
First, set
$$
X=X_1\oplus \overline{X_1}\oplus \cdots \oplus X_r \oplus \overline{X_r}
$$
which is a symmetrically self-dual object in $\cC$. Setting
\begin{align*}
P_{n,+} &= \End_{\cC}(X\otimes \overline{X}\otimes \cdots \otimes X^{\pm})\\
P_{n,-} &= \End_{\cC}(\overline{X}\otimes X\otimes \cdots \otimes X^{\pm})
\end{align*}
where $X^\pm= X,\overline{X}$ depending on the parity of $n$, $\cP$ is a subfactor planar algebra. (This follows by results similar to those in Section \ref{sec:TCandPA}).
A subfactor planar algebra is naturally a $\lambda$-lattice \cite{MR1334479,math.QA/9909027}, and thus \cite{MR2051399} gives us a subfactor $N\subset M$ where $N,M$ are both isomorphic to $L(\F_\I)$ and whose subfactor planar algebra is isomorphic to $\cP$. The result now follows by a modified version of Theorem \ref{thm:PAandTC}.
\end{proof}

Difficulties arise when trying to adapt an approach along these lines for $\cC$ not finitely generated. One wants to define an inductive limit Popa system and use  \cite{MR2051399} to get an inclusion of factors isomorphic to $L(\F_\I)$ which remembers $\cC$.

One hope is to look at finitely generated subcategories of $\cC$ as follows. Set
$$
Y^{r} = X_{1} \oplus \overline{X_1} \oplus \cdots \oplus X_r \oplus \overline{X_r},
$$
and define $Z^{r}$ inductively by setting $Z^{1} = Y^{1}$ and $Z^{r} = Z^{r-1} \otimes Y^{r}$. Now for all $r$, we have subfactor planar algebras $\cP^r$ given by
\begin{align*}
P^{r}_{n,+} &= \End_{\cC}(Z^{r}\otimes \overline{Z^{r}}\otimes \cdots \otimes Z^{r,\pm})\\
P^{r}_{n,-} &= \End_{\cC}(\overline{Z^{r}}\otimes Z^{r}\otimes \cdots \otimes Z^{r,\pm}).
\end{align*}
Again, using \cite{MR2051399}, we get an inclusion $N^{r} \subset M^{r}$ where both factors are isomorphic to $L(\F_{\I})$, and the associated category of $N^r-N^r$ bimodules is equivalent to the subcategory of $\cC$ generated by $X_{1}, \overline{X_1},\dots, X_{r},\overline{X_r}$.

One problem with this approach is that while $P^{r}_{n, \pm}$ includes unitally into $P^{r+1}_{n, \pm}$, the inductive limit planar algebra $\cP^\I$ has infinite dimensional $n$-box spaces, i.e., $\dim(P_{n,\pm}^\I)=\I$ for all $n$. Hence we cannot directly use \cite{MR2051399} to get a subfactor. Another problem is that the inclusion  $P^{r}_{\bullet}\hookrightarrow \cP^{r+1}$ does not induce an inclusion of $N^{r}$ into $N^{r+1}$ nor an inclusion $M^r$ into $M^{r+1}$. Hence we do not get an inductive limit subfactor $N^\I\subset M^\I$.

Of course, one can try other approaches along these lines, but so far, the authors have not succeeded in finding an inductive limit inclusion.

\subsection{The law of the element $p_{v}X_{e}u_{e}^{w}(u_{e}^{w})^{*}X_{e}p_{v}$}

We will adapt the technique in \cite{MR2732052} to compute the law of $R = p_{v}X_{e}u_{e}^{w}(u_{e}^{w})^{*}X_{e}p_{v}$.  A similar computation will compute the law of $(u_{e}^{w})^{*}X_{e}p_{v}X_{e}u_{e}^{w}$.  To begin, we set $r_{w} = u_{e}^{w}(u_{e}^{w})^{*}$ and notice that this element diagrammatically looks like
$$
\begin{tikzpicture}
    \draw (-2, 0)--(2, 0);
    \filldraw[thick, unshaded] (-1.6, -.4)--(-1.6, .4)--(-.8, .4)--(-.8, -.4)--(-1.6, -.4);
    \filldraw[thick, unshaded] (.4, .4)--(.4, -.4)--(-.4, -.4)--(-.4, .4)--(.4, .4);
    \filldraw[thick, unshaded] (1.6, -.4)--(1.6, .4)--(.8, .4)--(.8, -.4)--(1.6, -.4);
    \node at (-1.2, 0) {$p_{v}$};
    \node at (1.2, 0) {$p_{v}$};
    \node at (0, 0) {$r_{w}$};
    \draw (-.4, .25) arc(270:180: .25cm);
    \draw (.4, .25) arc(270:360: .25cm);
\end{tikzpicture}
$$
Let  $\phi$ be the moment generating function of $R$ i.e. $\phi(z) = \sum_{z = 0}^{\infty} \frac{\Tr(R^{n})}{\Tr(p_{v})}z^{n}$.  Set $\alpha_{e} = \frac{\Tr(r_{w})}{\Tr(p_{v})} = \frac{\Tr(p_{w})}{\Tr(p_{v})} \geq 1.$  Using the recurrence relation \ref{eqn:recurrence}, and the fact that
$$
E\left(\begin{tikzpicture} [baseline = 0cm]
    \draw (-.8, 0)--(.8, 0);
    \filldraw[thick, unshaded] (.4, .4)--(.4, -.4)--(-.4, -.4)--(-.4, .4)--(.4, .4);
    \node at (0, 0) {$r_{w}$};
    \draw (-.4, .25) arc(270:180: .25cm);
    \draw (.4, .25) arc(270:360: .25cm);
\end{tikzpicture}\right) = \alpha_{e}p_{v},
$$
we see that for $n \geq 1$,
$$
E(R^{n}) = \alpha_{e}E(R^{n-1}) + \left(\sum_{k=1}^{n-2}E(R^{k})E(R^{n-1-k})\right) + E(R^{n-1}) = (\alpha_{e}-1)E(R^{n-1}) + \sum_{k=0}^{n-1}E(R^{k})E(R^{n-1-k})
$$
Multiplying both sides by $z^{n}$ and summing gives the equation
$$
\phi - 1 = (\alpha_{e} - 1)z\phi + z\phi^{2}
$$
which implies
$$
\phi(z) = \frac{-((\alpha_{e} - 1)z - 1) + \sqrt{((\alpha_{e} - 1)z - 1)^{2} - 4z}}{2z}
$$
where the branch of the square root will be momentarily determined.  Let $\mu$ be the measure supported on the spectrum of $R$ in $p_{v} \cM_{\I}p_{v}$ satisfying
$$
\int x^{n}d\mu(x) = \Tr(R^{n})/\Tr(p_{v}).
$$
ale let $G_{\mu}(z) = z^{-1}\phi(z^{-1})$ be its Cauchy transform.  It follows that
$$
G_{\mu}(z) = \frac{-((\alpha_{e} - 1) - z) + \sqrt{((\alpha_{e} - 1) - z)^{2} - 4z}}{2z}.
$$
The Cauchy transform of $\mu$ also has the form
$$
G_{\mu}(z) = \int_{\R} \frac{d\mu(x)}{z - x}
$$
and it is straightforward to check that $\Im(G(z)) < 0$ whenever $\Im(z) > 0$, and $\lim_{z \rightarrow \infty}zG_{\mu}(z) = 1$ non-tangentially to the real axis.  We therefore choose the branch of the square root so that the algebraic expression for $G_{\mu}(x)$ satisfies these two conditions.  The measure $\mu$ is recovered by writing $z = x+iy$ and computing $\lim_{y\downarrow 0} -\frac{1}{\pi}\Im(G_{\mu}(z))$ \cite{MR1217253}.  Doing so gives
$$
d\mu(x) = \frac{\sqrt{-((\alpha_{e} - 1) - x)^{2} + 4x}}{2\pi x}\cdot \chi_{[\sqrt{\alpha_{e}} - 1, \sqrt{\alpha_{e}} + 1]}dx
$$
with $\chi_{I}$ the indicator function of the Lebesgue measurable set $I$.  This is the distribution of a free Poisson element \cite{MR2266879}. 

%% file: Thesis2013Chapter-FC.tex
\chapter[FC GJS construction and N-P-M planar algebras]{The GJS Construction for Fuss Catalan and $N-P-M$ Planar Algebras} \label{chap:FC}

\section{Introduction}  The work of Chapter \ref{chap:universal} was heavily influenced by a problem presented to the author by Vaughan Jones.  The problem is as follows:  Given a subfactor planar algebra, $\cQ$, one can consider the algebras $\Gr_{k}^{\pm}(\cQ)$ as defined in the previous chapters and place the following  ``toy potential" on $\cQ$:

$$
\tr(x) =
\begin{tikzpicture}[baseline=.5cm]
	\draw (0,0)--(0,.8);
	\filldraw[unshaded,thick] (-.4,.4)--(.4,.4)--(.4,-.4)--(-.4,-.4)--(-.4,.4);
	\draw[thick, unshaded] (-.7, .8) -- (-.7, 1.6) -- (.7, 1.6) -- (.7,.8) -- (-.7, .8);
	\node at (0,0) {$x$};
	\node at (0,1.2) {$\sum V$};
\end{tikzpicture}
$$
where $V$ is a rotationally invariant set of elements in $\cQ$.  If one is fortunate, $\tr$ is positive definite on $\cQ$ and left multiplication is bounded on $L^{2}(\Gr(\cQ))$.  To this end, it is an interesting problem is to study the von Neumann algebras, $N_{k}^{\pm} (= \Gr_{k}^{\pm}(\cQ)'')$ associated to $\cQ$ and $V$.

The case that will be considered here is the case where $\cQ$ contains is the standard invariant for a subfactor $N \subset M$ that contains an intermediate subfactor, $P$.  As such, it follows that $\cQ$ contains the Fuss Catalan planar algebra as a sub planar algebra \cite{MR1437496}.  Therefore, we can consider the following potential on $\Gr_{0}^{+}(\cQ)$:
$$
\tr(x) =
\begin{tikzpicture}[baseline=.5cm]
	\draw (0,0)--(0,.8);
	\filldraw[unshaded,thick] (-.4,.4)--(.4,.4)--(.4,-.4)--(-.4,-.4)--(-.4,.4);
	\draw[thick, unshaded] (-.7, .8) -- (-.7, 1.6) -- (.7, 1.6) -- (.7,.8) -- (-.7, .8);
	\node at (0,0) {$x$};
	\node at (0,1.2) {$\sum FC$};
\end{tikzpicture}
$$
where $\sum FC$ represents the sum of all Fuss Catalan diagrams.  Recall that Fuss Catalan diagrams are planar two-colored string diagrams satisfying the following condition:  If the colors of the strings are $a$ and $b$, then the colors of strings intersecting the boundary form the pattern, $aabbaabbaabb...$

To attack this problem, it is most natural to introduce a new planar algebra, $\cP$ which will be called an $N-P-M$ planar algebra.  A pleasing feature of this planar algebra is that the boundary conditions of the input disks can be taken to be any word in $aa$ and $bb$.  The planar algebra $\cQ$ will be realized as a subalgebra of $\cP$ by replacing each strand with an $a$ strand next to a $b$ strand so that the pattern formed by the input disks in the $\cQ-$tangles is of the form $aabbaabbaabb...$. Whenever $\cQ$ is itself the Fuss Catalan planar algebra, we can take $\cP$ to be generated the set of \underline{all} string diagrams having boundary conditions any word in $aa$ and $bb$.

The purpose of the $N-P-M$ planar algebra is that it gives one a natural way to decouple the strings $a$ and $b$ and treat them as free generators in an appropriate sense (see Section \ref{sec:semi} for how this is done).  We form algebras $M_{\alpha}$ for $\alpha$ a suitable word in $a$ and $b$ which are the $N-P-M$ analogues of the algebras $N_{k}^{\pm}$, and will use a semifinite algebra as in Chapter \ref{chap:universal} to find the isomorphism class of the algebras $M_{\alpha}$.  More precisely, we will prove the following theorem:
\begin{thm} \label{thm:isoclass1}
$M_{\alpha}$ is a $II_{1}$ factor and is isomorphic to $L(\F(1 + 2I\delta_{\alpha}(\delta_{a} + \delta_{b} - 2)))$ for $\cP$ finite depth.  Here, $\delta_{\alpha}$ is defined as in Chapter \ref{chap:universal}, and $I = \sum_{v \in \Gamma_{N}^{N}} \mu(v)^{2}$ with $\mu$ the Perron Frobenius weighting on the principal graphs on $\cP$.   \end{thm}
This formula has some interest, because it contains information about the inclusions $N \subset P$ and $P \subset M$ ($\delta_{a}$ and $\delta_{b}$ respectively) as well as the larger inclusion $N \subset M$ (the global index $I$). Just as in the case for the GJS algebras, we will also prove the following theorem:
\begin{thm} \label{thm:isoclass2}
$M_{\alpha} \cong L(\F_{\I})$ when $\cP$ is infinite depth.
\end{thm}

In addition to understanding the algebraic structure of the $N_{k}$, we will make use of the semifinite algebra to show that the law of $\cup \in \N_{0}^{+}$ has a nice expression in terms of known laws.

Unfortunately, the author has not yet been able to identify the isomorphism classes of the algebras $N_{k}^{\pm}$, however the semifinite algebra will show that there is evidence that the algebras $N_{k}^{\pm}$ are free group factors:

\begin{thm} \label{thm:dualcontain}
The von Neumann algebras, $N_{k}^{\pm}$, are each contained in a free group factor and contain a free group factor.
\end{thm}

\section{$N-P-M$ planar algebras} \label{sec:NPM}

The object that is at the heart of all of these computations is an $N-P-M$ planar algebra which can be thought of an augmentation of such a $\cQ$ as above.  Given parameters $a$ and $b$, we denote $\cW$ as the set of finite words on $aa$ and $bb$.  We define what it means for a planar tangle to be an $N-P-M$ tangle:
\begin{defn}
A planar tangle is said to be an $N-P-M$ tangle if its regions are shaded by three colors, $N$, $P$, and $M$ such that the following conditions are met:
\begin{itemize}
\item A region colored $N$ only borders a region colored $P$

\item A region colored $M$ only borders a region colored $P$

\item A region colored $P$ can border a region colored $N$ or $M$ but not $P$.

\end{itemize}
\end{defn}

We will denote the set of $N-P-M$ tangles as the $N-P-M$ planar operad.  As in the previous chapters, all tangles will be drawn so that the internal and external disks are rectangles.  In order for smooth isotopy to make sense, the rectangles have their corners smoothed out.

Any string serving as the boundary string of a region colored $N$ will be called an $a$ string and any string bordering a region colored $M$ will be called a $b$ string.  We note that the conditions on the regions show that if $\alpha$ is a word in $a$ and $b$, then $\alpha$ serves as a word of strings intersecting a disk if and only if $\alpha \in \Delta$ where $\Delta$ is the set
$$
\Delta = \{awb: w \in \cW\} \cup \{bwa: w \in \cW\} \cup \{awa: w \in \cW\} \cup \{bwb: w \in \cW\} \cup \cW
$$

The following notation will be useful:
\begin{nota}
We denote $\Delta^{Q}$ to be the set of all words in $\Delta$ where the shading before the first letter is $Q$ (for $Q = N,\, P,$ or $M$).  If $\alpha \in \Delta$, we define $s(\alpha)$ to be the first letter of the word $\alpha$.  Whenever a word is mentioned, part of the data is its initial region (hence the choice of region between every pair of letters), not just its letters.  If a tangle, $T$, has boundary condition $\alpha$ on its outer disk, we will call $T$ a \underline{planar $\alpha$ tangle}. Any internal rectangle with boundary condition $\alpha$ will be called an \underline{$\alpha$ rectangle}.
\end{nota}
We remark that just as for shaded planar algebras, there is a natural gluing operation.  Namely, if we have two planar tangles $S,T$ satisfying the following \underline{boundary condition}:
\begin{itemize}
\item
Some internal rectangle $D_{S}$ of $S$ has boundary data which agrees with $T$, i.e. the shadings along the boundaries of $T$ and $D_{S}$ agree when counting clockwise from the marked point.
\end{itemize}
then we may compose $S$ and $T$ to get the planar tangle $S\circ_{D_{S}} T$ by taking $S$ union the interior of $T$, removing the boundary of $D_{S}$, and smoothing the strings.

As in the previous chapter, given $\alpha \in \Delta$, we let $\overline{\alpha}$ be $\alpha$ read in the opposite order.  When a string appears with a label $\alpha$, then the string is meant to be a band of strings having colors ordered by the word $\alpha$.  As in the previous chapter, the strings are read in the order of \underline{top to bottom} and \underline{left to right}.  Also, unless otherwise marked, all marked regions of rectangles will be assumed to be on the \underline{bottom-left corner} of the box.  In addition, whenever there is a box written without a tangle, it is assumed that the box is placed in a larger tangle whose boundary data agrees with the boundary data for the box and whose marked region is the same as the marked region of the box.

We now define an $N-P-M$ planar algebra:

\begin{defn}
An $N-P-M$ planar algebra consists of the following data:
\begin{itemize}
\item Given parameters $a$ and $b$ as above, there is a finite dimensional complex vector space $\cP_{\alpha}$ for every nonempty word, $\alpha \in \Delta$.  There are three vector spaces $\cP_{\emptyset}^{N}$, $\cP_{\emptyset}^{P}$, and $\cP_{\emptyset}^{M}$ in the case when $\alpha$ is empty.  These are one dimensional complex vector spaces.

\item an action of planar tangles by multilinear maps, i.e., for each planar $\alpha$ tangle $T$, whose rectangles $D_i(T)$ are $\alpha_{i}$ rectangles, there is a multilinear map
$$
Z_T\colon \prod_{i\in I} P_{\alpha_{i}}\to P_{\alpha}
$$
satisfying the following axioms:
\begin{enumerate}
\item[\underline{\text{Isotopy}:}]
If $\theta$ is an orientation preserving diffeomorphism of $\R^2$, then $Z_{\theta(T)}=Z_T$.  That is, let $T^{0}$ be the interior of $T$, and let $f \in \prod_{D \subset T^{0}}P_{\alpha_{D}}$.  Then
$$
Z_{\theta(T)}(f_{\theta}) = Z_{T}(f)
$$
where $f_{\theta}(\theta(D)) = f(D)$.

\item[\underline{\text{Naturality}:}]
For $S,T$ composable tangles, $Z(S\circ_D T) = Z(S)\circ_D Z(T)$, where the composition on the right hand side is the composition of multilinear maps.
\end{enumerate}

\item $\cP$ is \underline{unital} \cite{JonesPAnotes}:  Let $S$ be an $N-P-M$ tangle with no input disks and boundary condition $\alpha \in \Delta$.  Then, there is an element $Z(S) \in P_{\alpha}$ so that the following holds:

Let $S$ be a tangle a nonempty set of internal disks such that $S$ can be glued into the internal disk $D^{S}$ of $T$.  Then
    $$
    Z(T \circ S) = Z(T) \circ Z_{S}.
    $$
    Here $(Z(T) \circ Z_{S})(f) = \tilde{f}$ where
    $$
    \tilde{f}(D) = \begin{cases} f(D) \text{ if } D \neq D^{S} \\ Z(S) \text{ if } D = D^{S} \end{cases}
    $$

    This condition allows isotopy classes of of such an $S$ to be elements of $P_{\alpha}$.  This action allows us to identify the empty diagrams (shaded $N$, $P$, and $M$) with the scalar $1 \in \C$.  We make this assumption in the rest of this chapter.    The naturality axiom, combined with this identification, forces closed strings with parameters $a$ and $b$ to be replaced by scalars $\delta_{a}$ and $\delta_{b}$ respectively.

\item There is a conjugate linear involution, $*: \cP_{\alpha} \rightarrow \cP_{\overline{\alpha}}$.  It is compatible with reflection of tangles i.e., if $\overline{T}$ is a tangle which is produced by an orientation \underline{reversing} diffeomorphism, $\varphi$, of $T$, then we have
    $$
    (Z_{T}(f))^{*} = Z_{\overline{T}}(\overline{f})
    $$
where $\overline{f}(\varphi(D)) = f(D)^{*}$.

\item Each $\cP_{\alpha}$ comes equipped with the positive definite sesquilinear form:
$$
\langle x,y\rangle =
\begin{tikzpicture}[baseline = -.1cm]
	\draw (0,0)--(1.2,0);	
	\filldraw[unshaded,thick] (-.4,.4)--(.4,.4)--(.4,-.4)--(-.4,-.4)--(-.4,.4);
	\filldraw[unshaded,thick] (.8,.4)--(1.6,.4)--(1.6,-.4)--(.8,-.4)--(.8,.4);
	\node at (0,0) {$x$};
	\node at (1.2,0) {$y^*$};
	\node at (.6,.2) {{\scriptsize{$\alpha$}}};
\end{tikzpicture}
$$
\item $\cP$ is \underline{spherical}, i.e. for all $\alpha\in\Delta$ and all $x\in P_{\alpha\overline{\alpha}}$, we have
$$
\tr(x)
=
\trace{\overline{\alpha}}{\alpha}{x}{\alpha}
=
\traceop{\overline{\alpha}}{\alpha}{x}{\alpha}\,.
$$
This says that we can think of our planar tangles as living in a sphere instead of a plane.
\end{itemize}
\end{defn}

\begin{rem}
In viewing the action of a tangle, the letter $Z$ will often be omitted.
\end{rem}

A-priori, it is not clear that an $N-P-M$ planar algebra should exist. The following example shows that this is the case.  For the rest of this paper, an $a$ string will be colored blue and a $b$ string will be colored red.

\begin{ex} \label{ex:GenFC}
Let $\delta_{a}, \delta_{b} \in \{2\cos(\pi/n): n=3, 4, 5, ...\} \cup [2, \infty)$ and define $P_{\alpha}$ by
$$
P_{\alpha} = \spann\{\text{planar string diagrams with with boundary condition } \alpha\}
$$
i.e., $P_{\alpha}$ is the $\C-$linear span of isotopy classes of $\alpha$ tangles with no input disks and no loops. For example,
$$
P_{abba} = \spann\left\{
\begin{tikzpicture} [baseline=0 cm]
    \draw[thick] (-.4, -.5)--(-.4, .5)--(.4, .5)--(.4, -.5)--(-.4, -.5);
    \draw[blue] (-.4, -.3)--(.4, -.3);
    \draw[red] (-.4, -.1)--(.4, -.1);
    \draw[red] (-.4, .1)--(.4, .1);
    \draw[blue] (-.4, .3)--(.4, .3);
\end{tikzpicture}, \, \, \,
\begin{tikzpicture} [baseline=0 cm]
    \draw[thick] (-.4, -.5)--(-.4, .5)--(.4, .5)--(.4, -.5)--(-.4, -.5);
    \draw[blue] (-.4, -.3)--(.4, -.3);
    \draw[red] (-.4, .1) arc(90:-90: .1cm);
    \draw[red] (.4, .1) arc(90:270: .1cm);
    \draw[blue] (-.4, .3)--(.4, .3);
\end{tikzpicture}, \, \, \,
\begin{tikzpicture} [baseline = 0 cm]
    \draw[thick] (-.4, -.5)--(-.4, .5)--(.4, .5)--(.4, -.5)--(-.4, -.5);
    \draw[blue] (-.4, .3) arc(90:-90: .3cm);
    \draw[red] (-.4, .1) arc(90:-90: .1cm);
    \draw[red] (.4, .1) arc(90:270: .1cm);
    \draw[blue] (.4, .3) arc(90:270: .3cm);
\end{tikzpicture}   \right\}.
$$
The action of $N-P-M$ tangles is as follows:  All string diagrams are inserted into the necessary input disks.  The result of this operation is a new string diagram with except with some loops.  These loops are replaced with a parameter $\delta_{a}$ or $\delta_{b}$, depending on the color of a loop.  The adjoint operation is the conjugate linear extension of reflection of diagrams.

It is straightforward to check that $\cP$ satisfies all of the axioms of an $N-P-M$ planar algebra except positive definiteness.  Given $\alpha \in \Delta$, we form the word $\alpha'$ which is a word of colors that can appear in a Fuss Catalan diagram, and is obtained from $\alpha$ by inserting the right combination of $aa's$ or $bb's$ between letters in $\alpha$.  For example, if $\alpha = aaabbaabbbba$, then $\alpha' = a(bb)aabbaabb(aa)bba$.  We then define a map $\phi: P_{\alpha} \rightarrow P_{\alpha'}$ which is given by inserting a cup of the appropriate color whenever that color has been inserted into $\alpha$, and then dividing by $\delta_{a}^{m}\cdot\delta_{b}^{n}$.  Here, $aa$ was inserted $m$ times and $bb$ was inserted $n$ times.  For example,
$$
\phi\left( \begin{tikzpicture} [baseline=0 cm]
    \draw[thick] (-.8, -.4)--(1, -.4)--(1, .4)--(-.8, .4)--(-.8, -.4);
    \node at (.1, 0) {$x$};
    \draw[blue] (-.6, .4)--(-.6, .8);
    \draw[blue] (-.4, .4)--(-.4, .8);
    \draw[blue] (-.2, .4)--(-.2, .8);
    \draw[red] (0, .4)--(0, .8);
    \draw[red] (.2, .4)--(.2, .8);
    \draw[red] (.4, .4)--(.4, .8);
    \draw[red] (.6, .4)--(.6, .8);
    \draw[red] (.8, .4)--(.8, .8);
\end{tikzpicture}
    \right) = \frac{1}{\delta_{a}^{2}\delta_{b}} \cdot
\begin{tikzpicture} [baseline=0 cm]
    \draw[thick] (-.8, -.4)--(2.2, -.4)--(2.2, .4)--(-.8, .4)--(-.8, -.4);
    \node at (.7, 0) {$x$};
    \draw[blue] (-.6, .4)--(-.6, .8);
    \draw[red] (-.4, .8) arc(180:360: .1cm);
    \draw[blue] (0, .4)--(0, .8);
    \draw[blue] (.2, .4)--(.2, .8);
    \draw[red] (.4, .4)--(.4, .8);
    \draw[red] (.6, .4)--(.6, .8);
    \draw[blue] (.8, .8) arc(180:360: .1cm);
    \draw[red] (1.2, .4)--(1.2, .8);
    \draw[red] (1.4, .4)--(1.4, .8);
    \draw[blue] (1.6, .8) arc(180:360: .1cm);
    \draw[red] (2, .4)--(2, .8);
\end{tikzpicture}\, .
$$
This map is easily seen to be preserve the desired sesquilinear form, and we know that this form is positive semidefinite on the Fuss Catalan algebras, with positive definiteness in the case $\delta_{a}, \delta_{b} \geq 2$, from \cite{MR1437496}.  Therefore, after taking a quotient in the case that $\delta_{a}$ or $\delta_{b}$ is less than 2, this example produces an $N-P-M$ planar algebra.

\end{ex}

By unitality, this planar algebra is represented in every $N-P-M$ planar algebra.

\subsection{Principal graphs of $N-P-M$ planar algebras}

We first remark that if $\overline{\alpha}\alpha \in \Delta$, then the axioms of an $N-P-M$ planar algebra show that $P_{\alpha\overline{\alpha}}$ is a finite dimensional $C^{*}$ algebra with multiplication given by
$$
x \cdot y = \begin{tikzpicture} [baseline = 0cm]
\draw (-.8, 0)--(2, 0);
\filldraw [thick, unshaded] (1.6, .4)--(1.6, -.4)--(.8, -.4)--(.8, .4)--(1.6, .4);
\filldraw [thick, unshaded] (.4, .4)--(.4, -.4)--(-.4, -.4)--(-.4, .4)--(.4, .4);
\node at (0, 0) {$x$};
\node at (1.2, 0) {$y$};
\node at (-.5, .4) {$\star$};
\node at (.7, .4) {$\star$};
\node at (-.6, -.2) {\scriptsize{$\alpha$}};
\node at (.6, -.2) {\scriptsize{$\alpha$}};
\node at (1.8, -.2) {\scriptsize{$\alpha$}};
\end{tikzpicture}\, .
$$
Let $p \in P_{\alpha\overline{\alpha}}$ and $q \in P_{\gamma\overline{\gamma}}$ be projections.  Then we say $p$ is \underline{equivalent} to $q$ if there is a $u \in P_{\alpha\overline{\gamma}}$ so that
$$
\PAMultiply{\alpha}{u}{\gamma}{u^{*}}{\alpha} = p \, \, \, \text{ and } \, \, \, \PAMultiply{\gamma}{u^{*}}{\alpha}{u}{\gamma} = q.
$$

To an $N-P-M$ planar algebra $\cP$, there are three \underline{principal graphs} associated to $\cP$, $\Gamma_{N}$, $\Gamma_{P}$, and $\Gamma_{M}$.  We will call $\Gamma_{Q}$ the \underline{$Q-$principal graph} of $\cP$.  Each $\Gamma_{Q}$ has three sets of vertices, $\Gamma_{Q}^{N}$, $\Gamma_{Q}^{P}$ and $\Gamma_{Q}^{M}$.  They are described by the following procedure:

The vertices $v \in \Gamma_{Q_{1}}^{Q_{2}}$ correspond to equivalence classes of minimal projections $p_{v}$ in the finite-dimensional $C^{*}$ algebra $P_{\alpha\overline{\alpha}}$ for some $\alpha$ depending on $v$ where $\alpha\overline{\alpha} \in \Delta^{Q_{1}}$ and $\overline{\alpha}\alpha \in \Delta^{Q_{2}}$.  There are \underline{$a-$colored edges} connecting the vertices $\Gamma_{Q}^{N}$ to the vertices $\Gamma_{Q}^{P}$ as well as \underline{$b-$colored edges} connecting the the vertices $\Gamma_{Q}^{P}$ to $\Gamma_{Q}^{M}$.  The $a-$colored edges are created as follows:

Suppose $v \in \Gamma_{Q}^{N}$ and $w \in \Gamma_{Q}^{P}$, let $p \in P_{\beta\overline{\beta}}$ be equivalent to $p_{v}$.  It follows that the element
$$
i_{a}(p) = \begin{tikzpicture} [baseline = 0cm]
    \draw[thick] (.8, .6)--(-.8, .6)--(-.8, -.6)--(.8, -.6)--(.8, .6);
    \draw (-.4, -.3)--(-.4, .5)--(.4, .5)--(.4, -.3)--(-.4, -.3);
    \draw[blue] (-.8, -.45)--(.8, -.45);
    \draw (-.8, .1)--(-.4, .1);
    \draw (.4, .1)--(.8, .1);
    \node at (0, -.1) {$p$};
    \node at (.6, .25) {\scriptsize{$\beta$}};
    \node at (-.6, .25) {\scriptsize{$\beta$}};
    \node at (-.95, .5) {$\star$};
\end{tikzpicture}
$$
is a projection in $P_{\beta aa\overline{\beta}}$.  We draw $n$ $a-$colored edges between $v$ and $w$ if $n$ is the maximal number such that there exist orthogonal projections $q_{1}, ..., q_{n} \in P_{\beta aa\overline{\beta}}$ which are each equivalent to $p_{w}$ and satisfy $\sum_{i = 1}^{n} q_{i} \leq i_{a}(p)$.  We can also get edges from $w$ to $v$ in a similar manner.  In principle, the construction of the $a-$edges leads to oriented edges, however, the presence of the Jones basic construction shows that the edges can be unoriented.  More precisely, consider the projection
$$
e = \frac{1}{\delta_{a}}\begin{tikzpicture} [baseline = 0cm]
    \node at (-.55, .35) {$\star$};
    \draw (.4, .4)--(.4, -.4)--(-.4, -.4)--(-.4, .4)--(.4, .4);
    \draw (-.4, .1)--(.4, .1);
    \node at (0, .25) {\scriptsize{$\beta$}};
    \draw[blue] (-.4, -.1) arc(90:-90: .1cm);
    \draw[blue] (.4, -.1) arc(90:270: .1cm);
\end{tikzpicture} \in P_{\beta aaaa\overline{\beta}},
$$
and $z$ be its central support. We note that $P_{\beta\overline{\beta}}$ unitally includes into $P_{\beta aaaa\overline{\beta}}$ by applying the map $i_{a}$ twice.  It is also a straightforward check to see that the mapping $P_{\beta\overline{\beta}} \rightarrow  P_{\beta aaaa\overline{\beta}}$ given by $x \mapsto i_{a}(i_{a}(x))e$ is an isometry, and $ei_{a}(y)e = i_{a}(E_{P_{\beta\overline{\beta}}}(y))e$ for $y \in P_{\beta aa\overline{\beta}}.$  Therefore, from \cite{MR1473221} it follows  $zP_{\beta aaaa\overline{\beta}}z$ is isomorphic to the basic construction of $P_{\beta\overline{\beta}}$ in $P_{\beta aa\overline{\beta}}$.  If $A$, $B$, and $C$ are finite dimensional $C^{*}$ algebras with $C$ the basic construction of $A$ in $B$, then the Bratteli diagram of $B \subset C$ is the reflection of that of $A \subset B$ \cite{MR1473221}.  Therefore, if there are $n$ $a-$colored edges from $v$ to $w$, then there are $n$ $a-$colored edges from $w$ to $v$.\\

There is an analogous way to determine the $b-$colored edges that go between $\Gamma_{Q}^{P}$ and $\Gamma_{Q}^{M}$.  We also note that if $p \in P_{\alpha\overline{\alpha}}$ is a minimal projection corresponding to a vertex $v \in \Gamma_{Q_{1}}^{Q_{2}}$, then the mapping in Example \ref{ex:GenFC} shows that $p$ is equivalent to a minimal projection in $P_{w}$ where $w = abbaabbaa...$, $aabbaabb...$, $bbaabbaa...$ or $baabbaa...$.  In particular, if $\cQ$ is as in the introduction, then we have established the following proposition once it is shown that we can consider $\cQ$ inside an augmentation $\cP$.
\begin{lem}\label{SPA}
Let $\Gamma$ and $\Gamma'$ be the principal and dual principal graphs of $\cQ$ respectively.  Then there are one-to-one correspondences between the following sets of vertices:
$$
\Gamma_{+} \leftrightarrow \Gamma_{N}^{N}, \, \Gamma_{-} \leftrightarrow \Gamma_{N}^{M}, \, \Gamma_{+}' \leftrightarrow \Gamma_{M}^{M}, \text { and } \Gamma_{-}' \leftrightarrow \Gamma_{M}^{N}
$$
\end{lem}

Also observe that rotation by $180^{\circ}$ is an anti-isomorphism of each $P_{\gamma}$.  This induces a one-to-one correspondence $\Gamma_{Q_{1}}^{Q_{2}} \leftrightarrow \Gamma_{Q_{2}}^{Q_{1}}$.  Finally, if the vertices of the principal graphs $\Gamma_{Q}$ are weighted according to the traces of their corresponding projections, then it follows by the definition of principal graph that the graph with vertices $\Gamma_{Q}^{N}$ and $\Gamma_{Q}^{P}$ is bipartite with Perron Frobenius eigenvalue $\delta_{a}$.  Also, the graph with vertices $\Gamma_{Q}^{P}$ and $\Gamma_{Q}^{M}$ is bipartite with Perron Frobenius eigenvalue $\delta_{b}$.

\section{$N-P-M$ planar algebras from intermediate subfactors} \label{sec:NPM}

The goal of this section is to see that such a $\cQ$ as above can be faithfully realized inside an $N-P-M$ planar algebra $\cP$.  Much of this section was influenced from discussions with David Penneys and Noah Snyder, and many of the proofs of the following theorems are taken from them.  We will first describe how such an algebra arises from an inclusion $N \subset P \subset M$ of finite index $II_{1}$ factors.  To start, we consider the following bifinite bimodules:
$$
_{N}L^{2}(P)_{P} \text{ and } _{P}L^{2}(M)_{M}
$$
and their duals (contragredients)
$$
 _{P}L^{2}(P)_{N} \text{ and }  _{M}L^{2}(M)_{P}.
$$
Let $\alpha \in \Delta$.  Since part of the prescribed data for $\alpha$ is a choice of initial shading, we note that the \underline{shading} of $\alpha$, i.e. the shading between any two letters on $\alpha$ is uniquely determined.  Assume that the shading of $\alpha$ is the sequence $Q_{1}\cdots Q_{k}$ for $Q_{i} = N, \, P$ or $M$.  We define $Z_{\alpha}$ to be the following:
$$
Z_{\alpha} =\, _{Q_{1}}L^{2}(Q_{1})_{Q_{1} \cap Q_{2}} \underset{Q_{1} \cap Q_{2}}{\otimes}\, _{Q_{1} \cap Q_{2}}L^{2}(Q_{2})_{Q_{2} \cap Q_{3}} \cdots \underset{Q_{k-1} \cap Q_{k}}{\otimes}\, _{Q_{k-1}\cap Q_{k}}L^{2}(Q_{k})_{Q_{k}\cap Q_{1}} \underset{Q_{k} \cap Q_{1}}{\otimes}\, L^{2}(Q_{1})_{Q_{k}\cap Q_{1}}.
$$
and we set $P_{\alpha} = \Hom_{Q_{1}-Q_{1}}(L^{2}(Q_{1}), Z_{\alpha})$ (Notice that $Q_{i} \cap Q_{i+1}$ is necessarily $N$, $P$, or $M$).  We note from \cite{MR1424954,MR561983} that this can be identified with the $Q_{1}-Q_{1}$ central vectors of $Z_{\alpha}$.

To help understand the planar structure, we let
$$
N (= M_{0}) \subset M (= M_{1}) \subset M_{2} \subset \cdots \subset M_{n} \subset \cdots
$$
be the Jones tower for $N \subset M$, where $M_{n}$ is generated by $M_{n-1}$ and $e_{n-1}$.  Here, $e_{n-1}$ is the orthogonal projection from $L^{2}(M_{n-1})$ onto $L^{2}(M_{n-2})$.  We will define $e_{P}$ to be the orthogonal projection from $L^{2}(M)$ onto $L^{2}(P)$.  We will also let $B = \{b_{i}\}_{i=1}^{n}$ be an orthonormal Pimsner Popa basis for $M$ over $N$ where $n-1$ is the largest integer which is bounded above by the index $[M:N]$.  The $b_{i}$ are elements in $M$ satisfying the following equivalent conditions:
\begin{align*}
x &= \sum_{i=1}^{n} E_{N}(xb_{i})b_{i}^*\, \, \forall x \in M\\
x &= \sum_{i=1}^{n} b_{i}E_{N}(b_{i}^{*}x)\, \, \forall x \in M\\
1 &= \sum_{i=1}^{n} b_{i}e_{1}b_{i}^{*},
\end{align*}
as well as $E_{N}(b_{i}b_{j}^{*}) = \delta_{i, j}$ if $i \leq n-1$ and $E_{N}(b_{n}b_{n}^{*})$ is a projection of trace $[M:N] - (n-1)$ in $M$.  If we let $e_{P}$ be the orthogonal projection from $L^{2}(M)$ onto $L^{2}(P)$ and $\{c_{i}\}_{i=1}^{m}$ be an orthonormal Pimsner-Popa basis of $P$ over $N$.  Then we have the following lemma.
\begin{lem}
$e_{P} = \sum_{i=1}^{m} c_{i}e_{1}c_{i}^{*}$.
\end{lem}

\begin{proof}
We compute the 2-norm of $e_{P} - \sum_{i=1}^{m} c_{i}e_{1}c_{i}^{*}$.  Doing so gives:
$$
\|e_{p} - \sum_{i=1}^{m} c_{i}e_{1}c_{i}^{*}\|_{2}^{2} = \tr(e_{P}) - 2\sum_{i=1}^{m}\tr(c_{i}e_{1}c_{i}^{*}e_{P}) + \sum_{i, j = 1}^{m} \tr(c_{i}e_{1}c_{i}^{*}c_{j}e_{1}c_{j}^{*}).\\
$$
Since $e_{P}$ commutes with the elements $c_{i}$, the term in the middle becomes $2\sum_{i=1}^{m}\tr(c_{i}e_{1}c_{i}^{*})$.  Using $e_{1}c_{i}^{*}c_{j}e_{1} = E_{N}(c_{i}^{*}c_{j}e_{1})$, and orthonormality of the basis, the last term becomes $\sum_{i=1}^{m}\tr(c_{i}e_{1}c_{i}^{*})$.  Therefore, we get:
\begin{align*}
\|e_{p} - \sum_{i=1}^{m} c_{i}e_{1}c_{i}^{*}\|_{2}^{2} &= \tr(e_{P}) - \sum_{i=1}^{m}\tr(c_{i}e_{1}c_{i}^{*})\\
&= \tr(e_{P}) - [M:N]^{-1}\sum_{i=1}^{m}\tr(c_{i}c_{i}^{*})\\
&= [M:P]^{-1} - [M:N]^{-1}[P:N] = 0
\end{align*}
as desired.
\end{proof}

We will now show the bimodules $Z_{\alpha}$ can be isometrically embedded in $L^{2}(M_{n})$ for some $n$.  As some notation, we will let $\delta_{Q} = [M:Q]^{1/2}$ for $Q = N, \, P, $ or $M$.  We also set $E^{Q}_{1} = \delta_{Q}e_{Q}$, and $v_{n}^{Q} = E_{n}E_{n-1}\cdots E_{2}E^{Q}_{1}$.

\begin{thm}
The map $\phi: Z_{\alpha} \rightarrow M_{k}$ given by
$$
\phi(x_{1} \underset{Q_{1} \cap Q_{2}}{\otimes} x_{2} \underset{Q_{2} \cap Q_{3}}{\otimes} \cdots \underset{Q_{k-1} \cap Q_{k}}{\otimes} x_{k}) = x_{1}v_{1}^{Q_{1} \cap Q_{2}}x_{2}v_{2}^{Q_{2} \cap Q_{3}}\cdots v_{k-1}^{Q_{k-1} \cap Q_{k}}x_{k}
$$
is an isometry.
\end{thm}

\begin{proof}
Note that the map is well defined as $v_{r}^{Q}$ commutes with $Q$.  We proceed by induction on $k$.  The result is clearly true for $k = 1$, so assume that it holds for $k-1$.  Using the previous lemma as well as the relation $E_{i}E_{j} = E_{j}E_{i}$ for $|i - j| \geq 2$, we have:
\begin{align*}
&\left\langle x_{1}v_{1}^{Q_{1} \cap Q_{2}}x_{2}v_{2}^{Q_{2} \cap Q_{3}}\cdots v_{k-1}^{Q_{k-1} \cap Q_{k}}x_{k},\, y_{1}v_{1}^{Q_{1} \cap Q_{2}}y_{2}v_{2}^{Q_{2} \cap Q_{3}}\cdots v_{k-1}^{Q_{k-1} \cap Q_{k}}y_{k} \right\rangle_{M_{k}}\\
& = \tr_{M_{k}}\left(y_{k}^{*}(v_{k-1}^{Q_{k-1} \cap Q_{k}})^{*}\cdots (v_{2}^{Q_{2} \cap Q_{3}})^{*} y_{2}^{*} (v_{1}^{Q_{1} \cap Q_{2}})^{*} y_{1}^{*} x_{1}v_{1}^{Q_{1} \cap Q_{2}}x_{2}v_{2}^{Q_{2} \cap Q_{3}}\cdots v_{k-1}^{Q_{k-1} \cap Q_{k}}x_{k}\right)\\
&= \delta_{Q}^{2}\tr_{M_{k}}\left(y_{k}^{*}(v_{k-1}^{Q_{k-1} \cap Q_{k}})^{*}\cdots y_{2}^{*} E_{Q_{1}\cap Q_{2}}(y_{1}^{*} x_{1})e_{Q_{1} \cap Q_{2}}x_{2} \cdots v_{k-1}^{Q_{k-1} \cap Q_{k}}x_{k}\right)\\
&= \delta_{Q}^{2}\tr_{M_{k}}\left(y_{k}^{*}(v_{k-2}^{Q_{k-1} \cap Q_{k}})^{*}E_{k-1}\cdots y_{2}^{*} E_{Q_{1}\cap Q_{2}}(y_{1}^{*} x_{1})e_{Q_{1} \cap Q_{2}}x_{2} \cdots E_{k-1}v_{k-2}^{Q_{k-1} \cap Q_{k}}x_{k}\right)\\
&= \frac{\delta_{Q}^{2}}{[M:N]^{1/2}}\cdot \\
&\tr_{M_{k}}\left(y_{k}^{*}(v_{k-2}^{Q_{k-1} \cap Q_{k}})^{*}\cdots y_{2}^{*}E_{Q_{1}\cap Q_{2}}(y_{1}^{*} x_{1})E_{k-1}\cdots E_{2}\sum_{i=1}^{m}(b_{i}E_{1}b_{1}^{*})
E_{2}\cdots E_{k-1}x_{2} \cdots v_{k-2}^{Q_{k-1} \cap Q_{k}}x_{k}\right)\\
&= \frac{\delta_{Q}^{2}[Q_{1}\cap Q_{2}:N]}{[M:N]^{1/2}}\cdot \tr_{M_{k}}\left(y_{k}^{*}(v_{k-2}^{Q_{k-1} \cap Q_{k}})^{*}\cdots y_{2}^{*}E_{Q_{1}\cap Q_{2}}(y_{1}^{*} x_{1})E_{k-1}x_{2} \cdots v_{k-2}^{Q_{k-1} \cap Q_{k}}x_{k}\right)\\
&= \tr_{M_{n-1}}\left(y_{k}^{*}(v_{k-2}^{Q_{k-1} \cap Q_{k}})^{*}\cdots y_{2}^{*}E_{Q_{1}\cap Q_{2}}(y_{1}^{*} x_{1})x_{2} \cdots v_{k-2}^{Q_{k-1} \cap Q_{k}}x_{k}\right)\\
&=\left\langle E_{Q_{1}\cap Q_{2}}(y_{1}^{*} x_{1}) x_{2} \underset{Q_{2} \cap Q_{3}}{\otimes} \cdots \underset{Q_{k-1} \cap Q_{k}}{\otimes} x_{k}, y_{2} \underset{Q_{2} \cap Q_{3}}{\otimes} \cdots \underset{Q_{k-1} \cap Q_{k}}{\otimes} y_{k}\right\rangle\\
&=\left\langle x_{1} \underset{Q_{1} \cap Q_{2}}{\otimes} x_{2} \underset{Q_{2} \cap Q_{3}}{\otimes} \cdots \underset{Q_{k-1} \cap Q_{k}}{\otimes} x_{k}, y_{1} \underset{Q_{1} \cap Q_{2}}{\otimes} y_{2} \underset{Q_{2} \cap Q_{3}}{\otimes} \cdots \underset{Q_{k-1} \cap Q_{k}}{\otimes} y_{k}\right\rangle\\
\end{align*}
as desired.
\end{proof}

The map $\phi$ above is clearly a bimodule map, so central vectors get mapped into $N' \cap L^{2}(M_{k})$. Since $N' \cap L^{2}(M_{k}) = N' \cap M_{k}$ is finite dimensional, it follows that each $P_{\alpha}$ is finite dimensional.

\subsection{Action of $N-P-M$ tangles on $P_{\alpha}$}

We now describe how the $N-P-M$ planar operad acts on the various $P_{\alpha}$.  Given an $N-P-M$ tangle $T$, we isotope it so that it is in standard form.  This means:
\begin{enumerate}
\item All of the input and output disks are rectangles and all strings (that are not closed loops) emanate from the top of the rectangles.

\item All the input disks are in different horizontal bands and all maxima and minima of strings are at different vertical levels, and not in the horizontal bands defined by the input disks.

\item The starred intervals of the input disks are all at the bottom-left corner.  When we have a diagram of this form, the $\star$ is omitted.
\end{enumerate}

One then positions an imaginary horizontal line at the bottom of the tangle, $T$, and then slides it to the top.  One starts with the central vector $1_{Q} \in L^{2}(Q)$ whenever the bottom of the box is shaded $Q$.  The central vector gets altered as the line crosses either an input box, a maximum on a string, or a minimum on a string.  When the line reaches the top, you get the central vector produced by the action of the tangle.

Suppose the horizontal line passes through the $i^{th}$ rectangle (with respect to the isotopy) in the $m_{i}^{th}$ region which is shaded $Q_{m_{i}}$ (reading left to right along the line), and suppose that the vector $v_{i}$ has been assigned to the box.  We simply insert $v_{i}$ into the $m_{i}^{th}$ slot, i.e.
\begin{align*}
\sum_{j} x_{1}^{j} \otimes \cdots \otimes x_{m_{i}}^{j} \otimes \cdots \otimes x_{n} \mapsto &\sum_{j} x_{1}^{j} \otimes \cdots \otimes x_{m_{i}}^{j}v_{i} \otimes \cdots \otimes x_{n}\\
= &\sum_{j} x_{1}^{j} \otimes \cdots \otimes v_{i}x_{m_{i}}^{j} \otimes \cdots \otimes x_{n}
\end{align*}
Now suppose the horizontal line passes through a minimum, and suppose $Y \subset X$ with $X, Y \in \{N, P, M\}$ and $X$ and $Y$ resemble the regions on either side of the minimum.  Let $B_{X, Y}$ be a Pimsner-Popa basis for $X$ over $Y$.  Then we have the diagrammatic rules:
\begin{align*}
&\begin{tikzpicture} [baseline = 0cm]
    \draw[thick] (.4, .5)--(.4, -.5)--(-.4, -.5)--(-.4, .5)--(.4, .5);
    \draw (-.2, .5)--(-.2, .2) arc(180:360: .2cm) --(.2, .5);
    \node at (0, .3) {\scriptsize{$Y$}};
    \node at (0, -.35) {\scriptsize{$X$}};
    \draw[dashed] (-.4, -.1)--(.4, -.1);
\end{tikzpicture} \rightarrow
\begin{tikzpicture} [baseline = 0cm]
    \draw[thick] (.4, .5)--(.4, -.5)--(-.4, -.5)--(-.4, .5)--(.4, .5);
    \draw (-.2, .5)--(-.2, .2) arc(180:360: .2cm) --(.2, .5);
    \node at (0, .3) {\scriptsize{$Y$}};
    \node at (0, -.35) {\scriptsize{$X$}};
    \draw[dashed] (-.4, .1)--(.4, .1);
\end{tikzpicture}\, \, \, \, \, x \mapsto \frac{1}{[X:Y]^{1/2}}\sum_{b \in B_{X, Y}} xb \underset{Y}{\otimes} 1_{Y} \underset{Y}{\otimes} b^{*} = \sum_{b \in B_{X, Y}} b \underset{Y}{\otimes} 1 \underset{Y}{\otimes} b^{*}x \\
&\begin{tikzpicture} [baseline = 0cm]
    \draw[thick] (.4, .5)--(.4, -.5)--(-.4, -.5)--(-.4, .5)--(.4, .5);
    \draw (-.2, .5)--(-.2, .2) arc(180:360: .2cm) --(.2, .5);
    \node at (0, .3) {\scriptsize{$X$}};
    \node at (0, -.35) {\scriptsize{$Y$}};
    \draw[dashed] (-.4, -.1)--(.4, -.1);
\end{tikzpicture} \rightarrow
\begin{tikzpicture} [baseline = 0cm]
    \draw[thick] (.4, .5)--(.4, -.5)--(-.4, -.5)--(-.4, .5)--(.4, .5);
    \draw (-.2, .5)--(-.2, .2) arc(180:360: .2cm) --(.2, .5);
    \node at (0, .3) {\scriptsize{$X$}};
    \node at (0, -.35) {\scriptsize{$Y$}};
    \draw[dashed] (-.4, .1)--(.4, .1);
\end{tikzpicture}\, \, \, \, \, x \mapsto x \underset{Y}{\otimes} 1_{X} \underset{Y}{\otimes} 1_{Y} = 1_{Y} \underset{Y}{\otimes} 1_{X} \underset{Y}{\otimes} x
\end{align*}
Whenever a dotted line passes over a maximum, the following rules apply:
\begin{align*}
&\begin{tikzpicture} [baseline = 0cm]
    \draw[thick] (.4, .5)--(.4, -.5)--(-.4, -.5)--(-.4, .5)--(.4, .5);
    \draw (-.2, -.5)--(-.2, -.2) arc(180:0: .2cm) --(.2, -.5);
    \node at (0, -.3) {\scriptsize{$X$}};
    \node at (0, .35) {\scriptsize{$Y$}};
    \draw[dashed] (-.4, -.1)--(.4, -.1);
\end{tikzpicture}
\rightarrow
\begin{tikzpicture} [baseline = 0cm]
    \draw[thick] (.4, .5)--(.4, -.5)--(-.4, -.5)--(-.4, .5)--(.4, .5);
    \draw (-.2, -.5)--(-.2, -.2) arc(180:0: .2cm) --(.2, -.5);
    \node at (0, -.3) {\scriptsize{$X$}};
    \node at (0, .35) {\scriptsize{$Y$}};
    \draw[dashed] (-.4, .1)--(.4, .1);
\end{tikzpicture}\, \, \, \, \, y_{1} \underset{Y}{\otimes} x \underset{Y}{\otimes} y_{2} \mapsto [X:Y]^{1/2} y_{1}E_{Y}(x)y_{2}\\
&\begin{tikzpicture} [baseline = 0cm]
    \draw[thick] (.4, .5)--(.4, -.5)--(-.4, -.5)--(-.4, .5)--(.4, .5);
    \draw (-.2, -.5)--(-.2, -.2) arc(180:0: .2cm) --(.2, -.5);
    \node at (0, -.3) {\scriptsize{$Y$}};
    \node at (0, .35) {\scriptsize{$X$}};
    \draw[dashed] (-.4, -.1)--(.4, -.1);
\end{tikzpicture}
\rightarrow
\begin{tikzpicture} [baseline = 0cm]
    \draw[thick] (.4, .5)--(.4, -.5)--(-.4, -.5)--(-.4, .5)--(.4, .5);
    \draw (-.2, -.5)--(-.2, -.2) arc(180:0: .2cm) --(.2, -.5);
    \node at (0, -.3) {\scriptsize{$Y$}};
    \node at (0, .35) {\scriptsize{$X$}};
    \draw[dashed] (-.4, .1)--(.4, .1);
\end{tikzpicture}\, \, \, \, \, x_{1} \underset{Y}{\otimes} y \underset{Y}{\otimes} x_{2} \mapsto x_{1}yx_{2}
\end{align*}
Here is an example of a tangle acting on $y_{1} \underset{Y}{\otimes} x \underset{Y}{\otimes} y_{2}$:
$$
\begin{tikzpicture}[baseline = 0cm]
    \draw[thick] (1,1)--(1, -1)--(-1, -1)--(-1, 1)--(1, 1);
    \draw (-.4, -.1)--(-.4, .7)--(.4, .7)--(.4, -.1)--(-.4, -.1);
    \draw (.2, .7)--(.2, 1);
    \draw (-.2, .7) arc(0:180:.2 cm)--(-.6, 0) arc(180:360: .6cm) -- (.6, 1);
    \node at (0, -.35) {\scriptsize{$Y$}};
    \node at (0, -.8) {\scriptsize{$X$}};
\end{tikzpicture}\, \, (y_{1} \underset{Y}{\otimes} x \underset{Y}{\otimes} y_{2}) = \frac{1}{[X:Y]^{1/2}}\cdot \sum_{b \in B_{X, Y}} by_{1}x \underset{Y}{\otimes} y_{2} \underset{Y}{\otimes} b^{*}
$$
Note also that our rules dictate that a loop with an $X$ on one side and $Y$ on the other counts for a factor $[X:Y]^{1/2}$.  As $[M:P]$ and $[P:N]$ are the only two such indices that will appear, we will let $\delta_{a} = [P:N]^{1/2}$ and $\delta_{b} = [M:P]^{1/2}$.

It is a straightforward check to see that each of these maps preserves central vectors.  Each map is also locally a bimodule map, hence the action of $T$ will also preserve invariant elements.

Checking that $T$ is well defined up to isotopy involves checking the same (finite number of) relations as in \cite{math.QA/9909027}.  For example, checking
$$
\begin{tikzpicture}[baseline = 0cm]
    \draw[thick] (.4, .4)--(-.4, .4)--(-.4, -.4)--(.4, -.4)--(.4, .4);
    \draw (-.2, -.4)--(-.2, 0) arc(180:0: .1cm) arc(180:360: .1cm)-- (.2, .4);
    \node at (0, -.25) {\scriptsize{$X$}};
    \node at (0, .25) {\scriptsize{$Y$}};
\end{tikzpicture} =
\begin{tikzpicture}[baseline = 0cm]
    \draw[thick] (.4, .4)--(-.4, .4)--(-.4, -.4)--(.4, -.4)--(.4, .4);
    \draw (0, -.4)--(0, .4);
    \node at (-.2, 0) {\scriptsize{$Y$}};
    \node at (.2, 0) {\scriptsize{$X$}};
\end{tikzpicture}
$$
boils down to checking the relation $x = \sum_{b \in B_{X, Y}} E_{Y}(xb)b^{*}$, which always holds.  The key to checking that the action of $T$ is defined up to isotopy is to show that rotation by $2\pi$ is the identity.

Let $x = \sum_{j} x_{1}^{j} \underset{Q_{1} \cap Q_{2}}{\otimes} x_{2}^{j} \cdots  \underset{Q_{n} \cap Q_{1}}{\otimes} x_{n+1}^{j} \in P_{\alpha}$, and let $T$ be the tangle which is rotation by one-click clockwise, namely
$$
T = \begin{tikzpicture}[baseline = 0cm]
    \draw[thick] (1,1)--(1, -1)--(-1, -1)--(-1, 1)--(1, 1);
    \draw (-.4, -.1)--(-.4, .7)--(.4, .7)--(.4, -.1)--(-.4, -.1);
    \draw (-.2, .7)--(-.2, 1);
    \draw (.2, .7) arc(180:0:.2 cm)--(.6, .2) arc(0:-180: .6cm) -- (-.6, 1);
\end{tikzpicture}\, .
$$
where there are $n-1$ strings that are not bent.  By definition, we have:
$$
Z(T)(x) = \begin{cases}  \frac{1}{[Q_{n}:Q_{1}]^{1/2}}\sum_{b \in B(Q_{n}, Q_{1})}\sum_{j} b \underset{Q_{1}}{\otimes} x_{1} \underset{Q_{1} \cap Q_{2}}{\otimes} x_{2} \cdots \underset{Q_{n-1} \cap Q_{n}}{\otimes}x_{n}x_{n+1}b^{*} &\text{ if } Q_{1} \subset Q_{n} \\
[Q_{1}:Q_{n}]^{1/2} \sum_{j} 1 \underset{Q_{n}}{\otimes} x_{1} \underset{Q_{1} \cap Q_{2}}{\otimes} x_{2} \cdots \underset{Q_{n-1} \cap Q_{n}}{\otimes}x_{n}E_{Q_{n}}(x_{n+1}) &\text{ if } Q_{n}\subset Q_{1}
\end{cases}
$$
To help our computations, we define the following left and right creation operators, $L_{x}$ and $R_{x}$ for $x \in Q$.  These are given by:
\begin{align*}
L_{x} : Z_{\alpha} \rightarrow L^{2}(Q) \underset{Q \cap Q_{1}}{\otimes} Z_{\alpha} \, \text{ such that } & L_{x}(x_{1} \otimes \cdots \otimes x_{n+1}) = x \otimes x_{1} \otimes \cdots \otimes x_{n+1} \\
R_{x} : Z_{\alpha} \rightarrow Z_{\alpha} \underset{Q \cap Q_{1}}{\otimes} L^{2}(Q) \, \text{ such that } & R_{x}(x_{1} \otimes \cdots \otimes x_{n+1}) = x_{1} \otimes \cdots \otimes x_{n+1} \otimes x
\end{align*}
It follows from the definition of the bimodule tensor product that
\begin{align*}
L_{x}^{*} (x_{0} \otimes x_{1} \otimes \cdots \otimes x_{n+1}) &= E_{Q \cap Q_{1}}(x^{*}x_{0})x_{1} \otimes \cdots \otimes x_{n+1} \text{ and }\\
R_{x}^{*} (\otimes x_{1} \otimes \cdots \otimes x_{n+1} \otimes y) &= x_{1} \otimes \cdots \otimes x_{n} \otimes x_{n+1}E_{Q_{1} \cap Q_{0}}(yx^{*})
\end{align*}
Therefore, we have the following formulae the rotation tangle, $T$:
$$
Z(T)(x) = \begin{cases} \frac{1}{[Q_{n}:Q_{1}]^{1/2}} \sum_{b \in B}L_{b}R^{*}_{b}(x) &\text{ if } Q_{1} \subset Q_{n}\\
[Q_{1}:Q_{n}]^{1/2} L_{1}R^{*}_{1}(x) &\text{ if } Q_{n} \subset Q_{1}
\end{cases}
$$

From Burns' rotation trick \cite{1111.1362} we have the following lemma which is similar to lemmas that appear in \cite{MR2812459}:

\begin{lem}  Let $\rho(\alpha)$ be the word formed when the words in $\alpha$ are cyclically permuted clockwise by one, and let $y = y_{1} \otimes \cdots \otimes y_{n} \otimes y_{n+1} \in Z_{\rho(\alpha)}$.  Then  \begin{itemize}
\item $\displaystyle \langle T(x),  y \rangle = \frac{1}{[Q_{n}:Q_{1}]^{1/2}}\langle x, y_{2} \otimes \cdots \otimes y_{n+1} \otimes y_{1}\rangle$ if $Q_{1} \subset Q_{n}$

\item $\displaystyle \langle T(x),  y \rangle = [Q_{1}:Q_{n}]^{1/2}\langle x, y_{2} \otimes \cdots \otimes y_{n+1} \otimes y_{1}\rangle$ if $Q_{n} \subset Q_{1}$
\end{itemize}
\end{lem}

\begin{proof} For the first case, using that $x$ is central, we have
\begin{align*}
\langle T(x), y \rangle = \frac{1}{[Q_{n}:Q_{1}]^{1/2}}\langle \sum_{b \in B} L_{b}R_{b}^{*}(x), y \rangle & = \sum_{b \in B}\frac{1}{[Q_{n}:Q_{1}]^{1/2}}\langle x, R_{b}L^{*}_{b}(y)\rangle \\
&= \frac{1}{[Q_{n}:Q_{1}]^{1/2}} \sum_{b \in B}\langle x, E_{Q_{1}}(b^{*}y_{1})y_{2} \otimes y_{3} \otimes \cdots \otimes y_{n+1} \otimes b\rangle \\
&= \frac{1}{[Q_{n}:Q_{1}]^{1/2}} \sum_{b \in B}\langle (E_{Q_{1}}(b^{*}y_{1}))^{*}x, y_{2} \otimes y_{3} \otimes \cdots \otimes y_{n+1} \otimes b\rangle\\
&= \frac{1}{[Q_{n}:Q_{1}]^{1/2}} \sum_{b \in B}\langle x(E_{Q_{1}}(b^{*}y_{1}))^{*}, y_{2} \otimes y_{3} \otimes \cdots \otimes y_{n+1} \otimes b\rangle\\
&= \frac{1}{[Q_{n}:Q_{1}]^{1/2}} \sum_{b \in B}\langle x, y_{2} \otimes y_{3} \otimes \cdots \otimes y_{n+1} \otimes bE_{Q_{1}}(b^{*}y_{1})\rangle\\
&= \frac{1}{[Q_{n}:Q_{1}]^{1/2}} \langle x, y_{2} \otimes \cdots \otimes y_{n+1} \otimes y_{1}\rangle.
\end{align*}
For the second case, we have
\begin{align*}
\langle T(x), y \rangle = [Q_{1}:Q_{n}]^{1/2}\langle L_{1}R^{*}_{1}(x), y \rangle &= [Q_{1}:Q_{n}]^{1/2}\langle x, R_{1}L^{*}_{1}(y) \rangle \\
&= [Q_{1}:Q_{n}]^{1/2}\langle x, E_{Q_{1}}(y_{1}) y_{2} \otimes \cdots \otimes y_{n+1} \otimes 1 \rangle\\
&= [Q_{1}:Q_{n}]^{1/2}\langle x, y_{1}y_{2}\otimes \cdots \otimes y_{n+1} \otimes 1 \rangle\\
&= [Q_{1}:Q_{n}]^{1/2}\langle y_{1}^{*}x, y_{2} \otimes \cdots \otimes y_{n+1} \otimes 1 \rangle\\
&= [Q_{1}:Q_{n}]^{1/2}\langle xy_{1}^{*}, y_{2} \otimes \cdots \otimes y_{n+1} \otimes 1 \rangle\\
&= [Q_{1}:Q_{n}]^{1/2}\langle x, y_{2} \otimes \cdots \otimes y_{n+1} \otimes y_{1} \rangle
\end{align*}
as desired.
\end{proof}

\begin{cor}
Rotation by $2\pi$ is the identity.
\end{cor}

\begin{proof}
The nature of the shading dictates that all index factors cancel when applying the $2\pi$ rotation.  The rest follows from the previous lemma.
\end{proof}

One can now continue as in \cite{math.QA/9909027}.
\subsection{Realizing $\cQ$ inside $\cP$}

Suppose $\cQ$ is a planar algebra containing the Fuss-Catalan algebra, so that $\cQ$ is the planar algebra for a finite index inclusion $N \subset M$ with intermediate subfactor $P$.  We note that $Q_{2n,+}$ is the space of $N-N$ central vectors of
$$
(\, _{N}L^{2}(M)_{N})^{\underset{N}{\otimes}^{n}} = (\, _{N}L^{2}(M)_{M} \underset{M}{\otimes}\, _{M}L^{2}(M)_{N} )^{\underset{N}{\otimes}^{n}}.
$$
and $Q_{2n, -}$ is the space of $M-M$ central vectors of $(\, _{M}L^{2}(M)_{N} \underset{N}{\otimes}\, _{N}L^{2}(M)_{M} )^{\underset{M}{\otimes}^{n}}$.  Since
$$
_{N}L^{2}(M)_{M} = _{N}L^{2}(P)_{P} \underset{P}{\otimes}\, _{P}L^{2}(M)_{M},
$$ it follows that
$$
Q_{2n,+} = P_{(abba)^{n/2}}  \text{ and } Q_{2n,-} =  P_{(baab)^{n/2}}.
$$
Furthermore, a tangle $S$ that acts on $\cQ$ can be made into an $N-P-M$ tangle $S'$ by replacing each string with an $a$ string cabled to a $b$ string.  This is done so that the shaded regions in $S$ become the $M-$regions in $S'$ and the unshaded regions in $S$ become the $N-$regions in $S'$.  Notice that this implies that the strings along any disk in $S'$ read (clockwise from the marked region) as $abbaabbaa...$ or $baabbaabb...$.  It directly follows from the definitions that if $x_{1}, ..., x_{n}$ are in $\cQ$ and $S$ is as above, then $Z(S)(f) = Z(S')(f)$ where the left hand side denotes the action of a shaded tangle and the right hand side denotes the action of an $N-P-M$ tangle.

\begin{defn}
If $\cQ$ is the planar algebra of an inclusion $N\subset M$ of finite index $II_{1}$ factors with an intermediate subfactor, $P$, then we define the $\cP$ constructed as above as the \underline{$P-$augmentation of $\cQ$}.
\end{defn}
\section{The GJS construction for the $\sum FC$ potential}\label{sec:GJS2}

Suppose $\cQ$ is a subfactor planar algebra containing a copy of the Fuss Catalan planar algebra.  For each $k \geq 0$, we study the graded algebra $\Gr^{\pm}_{k}(\cQ)$ as above, and place the following trace on $\Gr^{\pm}_{k}(\cQ)$:
$$
\tr(x) = \frac{1}{(\delta_{a}\delta_{b})^{k}}
\begin{tikzpicture}[baseline=.5cm]
	\draw (0,0)--(0,.8);
	\filldraw[unshaded,thick] (-.4,.4)--(.4,.4)--(.4,-.4)--(-.4,-.4)--(-.4,.4);
	\draw[thick, unshaded] (-.7, .8) -- (-.7, 1.6) -- (.7, 1.6) -- (.7,.8) -- (-.7, .8);
	\node at (0,0) {$x$};
	\node at (0,1.2) {$\sum FC$};
    \draw [thick] (.4,0) arc (90:-90:.4cm) -- (-.4,-.8) arc (270:90:.4cm);
    \node at (0, -.6) {$k$};
\end{tikzpicture}\, .
$$
where the shading on the \emph{upper left} corner is $\pm$ (therefore, the shading in the starred region depends on $k$).

As in Section \ref{sec:NPM}, we realize $\cQ$ inside an augmentation, $\cP$, and we consider the algebras $\Gr_{\alpha}(\cP)$ where $\overline{\alpha}\alpha \in \Delta$.  If $\alpha \neq \emptyset$ then the shading after the last letter of $\alpha$ is uniquely determined and hence we can write,
$$
\Gr_{\alpha}(\cP) = \oplus_{(\beta: \overline{\alpha}\beta\alpha \in \Delta)} \cP_{\overline{\alpha}\beta\alpha}.
$$
$\Gr_{\alpha}(\cP)$ is endowed with a multiplication $\wedge$ given by
$$
x \wedge y =
\begin{tikzpicture}[baseline = -.1cm]
	\draw [thick] (-.8,0)--(2,0);	
	\draw (0,0)--(0,.8);
	\draw (1.2,0)--(1.2,.8);
	\filldraw[unshaded,thick] (-.4,.4)--(.4,.4)--(.4,-.4)--(-.4,-.4)--(-.4,.4);
	\filldraw[unshaded,thick] (.8,.4)--(1.6,.4)--(1.6,-.4)--(.8,-.4)--(.8,.4);
	\node at (0,0) {$x$};
	\node at (1.2,0) {$y$};
	\node at (-.6,.2) {{\scriptsize{$\alpha$}}};
	\node at (.6,.2) {{\scriptsize{$\alpha$}}};
	\node at (1.8,.2) {{\scriptsize{$\alpha$}}};
	\node at (.2,.6) {{\scriptsize{$\beta$}}};
	\node at (1.4,.6) {{\scriptsize{$\gamma$}}};
\end{tikzpicture}
$$
and normalized trace
$$
\tr(x) = \frac{1}{\delta^{\alpha}}\, \, \cdot \, \,
\begin{tikzpicture}[baseline=0cm]
	\draw (0,0)--(0,.8);
	\draw [thick] (.4,0) arc (90:-90:.4cm) -- (-.4,-.8) arc (270:90:.4cm);
	\filldraw[unshaded,thick] (-.4,.4)--(.4,.4)--(.4,-.4)--(-.4,-.4)--(-.4,.4);
	\draw[thick, unshaded] (-.7, .8) -- (-.7, 1.6) -- (.7, 1.6) -- (.7,.8) -- (-.7, .8);
	\node at (0,0) {$x$};
	\node at (0,1.2) {$\Sigma CTL$};
	\node at (.2,.6) {{\scriptsize{$\beta$}}};
	\node at (.6,.2) {{\scriptsize{$\alpha$}}};
	\node at (-.6,.2) {{\scriptsize{$\alpha$}}};
	\node at (0,-.6) {{\scriptsize{$\overline{\alpha}$}}};
\end{tikzpicture}
$$
where, as in Chapter \ref{chap:universal}, $\sum CTL$ is the sum of all colored Temperley-Lieb diagrams.  In the case where $\alpha = \emptyset$, then we have three such algebras, one for each shading $N$, $P$, and $M$.  We therefore form graded algebras $\Gr_{0}^{N}$, $\Gr_{0}^{P}$ and $\Gr_{0}^{M}$ where
$$
\Gr_{0}^{Q} = \bigoplus_{(\beta : \beta \in \Delta^{Q})} \cP_{\beta}.
$$
Using the graded algebras associated to $\cP$, we see that we have the following trace preserving inclusions:
\begin{align*}
\Gr_{2k}^{+}(\cQ) &\subset \Gr_{abba\cdots abba}(\cP)\\
\Gr_{2k+1}^{+}(\cQ) &\subset \Gr_{abba\cdots baab}(\cP)\\
\Gr_{2k}^{-}(\cQ) &\subset \Gr_{baab\cdots baab}(\cP)\\
\Gr_{2k+1}^{-}(\cQ) &\subset \Gr_{baab \cdots abba}(\cP).
\end{align*}

One advantage to working in the $N-P-M$ planar algebra $\cP$ is that the map $\Phi: \Gr_{\alpha}(\cP) \rightarrow \Gr_{\alpha}(\cP)$ given by
$$
\Phi(x) =\sum_{E \in \Epi(CTL)}
\begin{tikzpicture}[baseline=.7cm]
	\draw (0,0)--(0,2);
	\filldraw[unshaded,thick] (-.4,.4)--(.4,.4)--(.4,-.4)--(-.4,-.4)--(-.4,.4);
	\draw[thick, unshaded] (-.4, .8) -- (-.4, 1.6) -- (.4, 1.6) -- (.4,.8) -- (-.4, .8);
	\node at (0,0) {$x$};
	\node at (0,1.2) {$E$};
\end{tikzpicture}
$$
as in Section \ref{sec:graded} of the previous chapter is a well defined trace preserving isomorphism between $\Gr_{\alpha}(\cP)$ with the $\sum CTL$ trace and $\Gr_{\alpha}(\cP)$ with the orthogonalized trace.  Therefore, we have proven the following lemma:
\begin{lem}
The potential $\sum CTL$ gives a positive definite trace on $\Gr_{\alpha}(\cP)$.
\end{lem}
Furthermore, by considering either the $a$ or $b$ $\cup$ element, the same analysis as in Section \ref{sec:GJS} of the previous chapter proves the following theorem:
\begin{thm}
Left (and right) multiplication of elements of $Gr_{\alpha}(\cP)$ on $L^{2}(Gr_{\alpha}(\cP))$ is bounded and the associated von Neumann algebra, $M_{k} = Gr_{\alpha}(\cP)''$ is a $II_{1}$ factor.
\end{thm}
We note that if $\gamma$ and $\beta$ are words in $a$ and $b$ such that $\overline{\beta}\beta \in \Delta$ and $\overline{\beta\gamma}\beta\gamma \in \Delta$ then we have a unital inclusion of $M_{\gamma}$ into $M_{\beta\gamma}$ given the extension of
$$
x \mapsto \begin{tikzpicture}[baseline = -.1cm]
	\draw (0,0)--(0,.8);
	\draw (-.8,0)--(.8,0);
	\draw (-.8,-.6)--(.8,-.6);
	\filldraw[unshaded,thick] (-.4,.4)--(.4,.4)--(.4,-.4)--(-.4,-.4)--(-.4,.4);
	\node at (0,-.75) {{\scriptsize{$\gamma$}}};
	\node at (-.6,.2) {{\scriptsize{$\beta$}}};
	\node at (.6,.2) {{\scriptsize{$\beta$}}};
	\node at (0,0) {$x$};
\end{tikzpicture}\,.
$$
We therefore have the following theorem, whose proof is exactly the same as in Section \ref{sec:towers} of the previous chapter.
\begin{thm} \label{thm:Jtower}
The following is a Jones' tower of $II_{1}$ factors:
$$
M_{0}^{Q} \subset M_{\alpha} \subset M_{\alpha\overline{\alpha}} \subset \cdots \subset M_{(\alpha\overline{\alpha})^{n}} \subset M_{\overline{\alpha}(\alpha\overline{\alpha})^{n}}.
$$
Furthermore, $[M_{\alpha}:M_{0}^{Q}] = \delta_{\alpha}$, and the Jones projection for $M_{0}^{Q} \subset M_{\alpha}$ is
$$
e_{0} =  \frac{1}{\delta_{\alpha}}
\begin{tikzpicture} [baseline = -.1cm]
	\draw [thick] (-.4, .25) arc(90:-90: .25cm);
	\draw [thick] (.4, .25) arc(90:270: .25cm);
	\draw[thick] (-.4, -.4) -- (-.4, .4) -- (.4, .4) -- (.4, -.4) -- (-.4, -.4);
    \node at (-.325, 0) {\scriptsize{$\alpha$}};
    \node at (.325, 0) {\scriptsize{$\alpha$}};
\end{tikzpicture}
$$
\end{thm}

\section{A semifinite algebra associated to $\cP$}\label{sec:semi}

As in section \ref{sec:vNa} of the previous chapter, we will realize the isomorphism class of all of the $M_{\alpha}$ by examining realizing them as ``corners" of a semifinite algebra.  To begin, we consider the set semifinite algebra $\Gr_{\infty}^{Q}$ for $Q = N, \, P, \, M$.  As a vector space,
$$
\Gr_{\infty}^{Q} = \bigoplus_{\substack{\alpha, \gamma, \beta \\ \overline{\alpha}\gamma\beta \in \Delta, \, \alpha \in \Delta^{Q}}} \cP_{\overline{\alpha}\gamma\beta}.
$$
Pictorially, we realize elements in $\Gr_{\infty}^{Q}$ as linear combinations of boxes of the form
$$
\begin{tikzpicture}[baseline=-.1cm]
	\draw (-.8,0)--(.8,0);
	\draw (0,.4)--(0,.8);
	\draw[thick, unshaded] (-.4, -.4) -- (-.4, .4) -- (.4, .4) -- (.4, -.4) -- (-.4, -.4);
	\node at (0, 0) {$w$};
	\node at (-.6, .2) {\scriptsize{$\alpha$}};
	\node at (.2,.6) {\scriptsize{$\beta$}};
	\node at (.6, .2) {\scriptsize{$\gamma$}};
\end{tikzpicture}
$$
where the \emph{bottom} (starred) region is shaded $Q$ (thus the top left corner varies in shading).  $\Gr_{\infty}^{Q}$ comes endowed with the following multiplication:
$$
\begin{tikzpicture}[baseline=-.1cm]
	\draw (-.8,0)--(.8,0);
	\draw (0,.4)--(0,.8);
	\draw[thick, unshaded] (-.4, -.4) -- (-.4, .4) -- (.4, .4) -- (.4, -.4) -- (-.4, -.4);
	\node at (0, 0) {$x$};
	\node at (-.6, .2) {\scriptsize{$\kappa$}};
	\node at (.2,.6) {\scriptsize{$\gamma$}};
	\node at (.6, .2) {\scriptsize{$\theta$}};
\end{tikzpicture}
\wedge
\begin{tikzpicture}[baseline=-.1cm]
	\draw (-.8,0)--(.8,0);
	\draw (0,.4)--(0,.8);
	\draw[thick, unshaded] (-.4, -.4) -- (-.4, .4) -- (.4, .4) -- (.4, -.4) -- (-.4, -.4);
	\node at (0, 0) {$x$};
	\node at (-.6, .2) {\scriptsize{$\omega$}};
	\node at (.2,.6) {\scriptsize{$\gamma'$}};
	\node at (.6, .2) {\scriptsize{$\chi$}};
\end{tikzpicture}
= \delta_{\omega,\theta}
\begin{tikzpicture} [baseline = -.1cm]
	\draw (0,.4)--(0,.8);
	\draw (1.2,.4)--(1.2,.8);
	\draw (-.8,0)--(2,0);
	\draw[thick, unshaded] (-.4, -.4) -- (-.4, .4) -- (.4, .4) -- (.4, -.4) -- (-.4, -.4);
	\draw[thick, unshaded] (.8, -.4) -- (.8, .4) -- (1.6, .4) -- (1.6, -.4) -- (.8, -.4);
	\node at (0, 0) {$x$};
	\node at (1.2, 0) {$y^{*}$};
	\node at (-.6,.2) {\scriptsize{$\kappa$}};
	\node at (.2,.6) {\scriptsize{$\gamma$}};
	\node at (.6,.2) {\scriptsize{$\theta$}};
	\node at (1.4,.6) {\scriptsize{$\gamma'$}};
	\node at (1.8,.2) {\scriptsize{$\chi$}};
\end{tikzpicture}\
$$
and semifinite trace, $\Tr$, which is given by:
$$
\Tr(x) =
\begin{tikzpicture}[baseline=.3cm]
	\draw (0,0)--(0,.8);
	\draw (.4,0) arc (90:-90:.4cm) -- (-.4,-.8) arc (270:90:.4cm);
	\filldraw[unshaded,thick] (-.4,.4)--(.4,.4)--(.4,-.4)--(-.4,-.4)--(-.4,.4);
	\draw[thick, unshaded] (-.7, .8) -- (-.7, 1.6) -- (.7, 1.6) -- (.7,.8) -- (-.7, .8);
	\node at (0,0) {$x$};
	\node at (0,1.2) {$\Sigma CTL$};
	\node at (.6,.2) {{\scriptsize{$\alpha$}}};
	\node at (-.6,.2) {{\scriptsize{$\alpha$}}};
	\node at (0,-.6) {{\scriptsize{$\overline{\alpha}$}}};
\end{tikzpicture}
$$
if $x \in \Gr(\cP_{\alpha})$ and is zero otherwise.  Just as in the analysis of Section \ref{sec:vNa} from the previous chapter, we see that $\Gr_{\infty}^{Q}$ completes to a $II_{\infty}$ factor, $\cM_{\I}^{Q}$ when being represented on $L^{2}(\Gr_{\infty}^{Q})$.

Also of importance will be the von Neumann subalgebra $\cA_{\infty}^{Q} \subset \cM_{\infty}^{Q}$ which is generated by all boxes in $\cG_{\infty}^{Q}$ with \underline{no} strings on top.  Notice that there is a normal, faithful, $\Tr$-preserving conditional expectation $E: \cM_{\infty}^{Q} \rightarrow \cA_{\infty}^{Q}$ given by
$$
E(x) = \begin{tikzpicture}[baseline=.3cm]
	\draw (0,0)--(0,.8);
	\draw (.4,0) -- (.8, 0);
    \draw (-.4, 0)--(-.8, 0);
	\filldraw[unshaded,thick] (-.4,.4)--(.4,.4)--(.4,-.4)--(-.4,-.4)--(-.4,.4);
	\draw[thick, unshaded] (-.7, .8) -- (-.7, 1.6) -- (.7, 1.6) -- (.7,.8) -- (-.7, .8);
	\node at (0,0) {$x$};
	\node at (0,1.2) {$\Sigma CTL$};
\end{tikzpicture}\, .
$$
Furthermore, we have the following lemma, whose proof is identical to that of Lemma \ref{lem:1infty} of the previous chapter:
\begin{lem}\label{lem:1infty2}
$\cA^{Q}_{\I} = \bigoplus_{v \in \Gamma_{Q}} \cA_{v}$ where $\Gamma_{Q}$ is the $Q-$principal graph of $\cP$ and each $\cA_{v}$ is a type $I_{\infty}$ factor.
\end{lem}

We now aim to figure out the isomorphism class of the algebras $\cM_{\alpha}$.  For the remainder of the section, we will assume for simplicity that $Q = N$ and hence we will be finding the isomorphism class of $M_{\alpha}$ such that $s(\overline{\alpha}) = N$.  The other two cases will follow a similar analysis.  We first define elements $X_{a}$ and $X_{b}$ as follows:
$$
X_{a} = \sum_{\substack{ \alpha \\ \alpha, \overline{\alpha}\in \Delta^{N}}}  \,
\begin{tikzpicture}[baseline=-.1cm]
    \draw [thick] (-.4, -.4)--(-.4, .4)--(.4, .4)--(.4, -.4)--(-.4, -.4);
    \draw [thick, blue] (-.4, .1) arc(-90:0: .3cm);
    \draw [thick] (-.4, -.2)--(.4, -.2);
    \node at (-.3, .3) {\scriptsize{$a$}};
    \node at (.2, 0) {\scriptsize{$\alpha$}};
\end{tikzpicture}
+
\begin{tikzpicture}[baseline=-.1cm]
    \draw [thick] (-.4, -.4)--(-.4, .4)--(.4, .4)--(.4, -.4)--(-.4, -.4);
    \draw [thick, blue] (.4, .1) arc(-90:-180: .3cm);
    \draw [thick] (-.4, -.2)--(.4, -.2);
    \node at (.3, .3) {\scriptsize{$a$}};
    \node at (-.2, 0) {\scriptsize{$\alpha$}};
\end{tikzpicture}
\,
\text{ and }
X_{b} = \sum_{\substack{ \beta \\ \beta\in \Delta^{P} \, \overline{\beta} \in N, s(\overline{\beta})=a}}   \,
\begin{tikzpicture}[baseline=-.1cm]
    \draw [thick] (-.4, -.4)--(-.4, .4)--(.4, .4)--(.4, -.4)--(-.4, -.4);
    \draw [thick, red] (-.4, .1) arc(-90:0: .3cm);
    \draw [thick] (-.4, -.2)--(.4, -.2);
    \node at (-.3, .3) {\scriptsize{$b$}};
    \node at (.2, 0) {\scriptsize{$\beta$}};
\end{tikzpicture}
+
\begin{tikzpicture}[baseline=-.1cm]
    \draw [thick] (-.4, -.4)--(-.4, .4)--(.4, .4)--(.4, -.4)--(-.4, -.4);
    \draw [thick, red] (.4, .1) arc(-90:-180: .3cm);
    \draw [thick] (-.4, -.2)--(.4, -.2);
    \node at (.3, .3) {\scriptsize{$b$}};
    \node at (-.2, 0) {\scriptsize{$\beta$}};
\end{tikzpicture}
\,.
$$
As in Section \ref{sec:vNa} of the previous chapter, $X_{a}$ and $X_{b}$ are sums of orthogonally supported operators, and each summand has uniformly bounded operator norm.  Therefore, $X_{a}$, $X_{b} \in \cM_{\I}^{N}$.  Furthermore, we have the following lemma:

\begin{lem}
$\cM_{\I}^{N}$ is generated as a von Neumann algebra by $(\cA_{\I}^{N}, \, X_{a}, \, X_{b})$.
\end{lem}

\begin{proof}
As in the proof of Lemma \ref{lem:generateI} from Chapter \ref{chap:universal}, all that needs to be shown is that the following diagrams lie in the von Neumann algebra generated by $(\cA_{\I}^{N}, \, X_{a}, \, X_{b})$:
$$
\begin{tikzpicture}[baseline=-.1cm]
    \draw [thick] (-.4, -.4)--(-.4, .4)--(.4, .4)--(.4, -.4)--(-.4, -.4);
    \draw [thick, blue] (-.4, -.1) arc(-90:0: .5cm);
    \draw [thick] (-.4, -.2)--(.4, -.2);
    \node at (.2, 0) {\scriptsize{$\alpha$}};
    \node at (-.2, .2) {\scriptsize{$N$}};
\end{tikzpicture}\, \ \text{ and }
\begin{tikzpicture}[baseline=-.1cm]
    \draw [thick] (-.4, -.4)--(-.4, .4)--(.4, .4)--(.4, -.4)--(-.4, -.4);
    \draw [thick, blue] (-.4, -.1) arc(-90:0: .5cm);
    \draw [thick] (-.4, -.2)--(.4, -.2);
    \node at (.2, 0) {\scriptsize{$\beta$}};
    \node at (-.2, .2) {\scriptsize{$P$}};
\end{tikzpicture}
$$
If $s(\alpha) = a$, then the exact same method as in the proof of Lemma \ref{lem:generateI} from Chapter \ref{chap:universal} shows that this element is in the algebra.  if $s(\alpha) = b$, the following multiplication produces the diagram:
$$\frac{1}{\delta_{a}} \cdot
\begin{tikzpicture}[baseline=-.1cm]
    \draw [thick] (-.4, -.4)--(-.4, .4)--(.4, .4)--(.4, -.4)--(-.4, -.4);
    \draw [thick] (-.4, -.2)--(.4, -.2);
    \draw [blue] (.4, -.1) arc(270:90: .1cm);
    \draw [blue] (-.4, .25)--(.4, .25);
    \node at (-.2, 0) {\scriptsize{$\alpha$}};
    \end{tikzpicture}\, \, \cdot \, \,
    \begin{tikzpicture}[baseline=-.1cm]
    \draw [thick] (-.4, -.4)--(-.4, .4)--(.4, .4)--(.4, -.4)--(-.4, -.4);
    \draw [blue] (-.4, .25) arc(-90:0: .15cm);
    \draw [thick] (-.4, -.2)--(.4, -.2);
    \draw [blue] (-.4,-.1)--(.4, -.1);
    \draw [blue] (-.4, .1)--(.4, .1);
    \end{tikzpicture}\, \, \cdot \, \,
    \begin{tikzpicture}[baseline=-.1cm]
    \draw [thick] (-.4, -.4)--(-.4, .4)--(.4, .4)--(.4, -.4)--(-.4, -.4);
    \draw [thick] (-.4, -.2)--(.4, -.2);
    \draw [blue] (-.4, -.1) arc(-90:90: .1cm);
    \node at (.2, 0) {\scriptsize{$\alpha$}};
    \end{tikzpicture}.
$$
A similar argument works for the element
$$
\begin{tikzpicture}[baseline=-.1cm]
    \draw [thick] (-.4, -.4)--(-.4, .4)--(.4, .4)--(.4, -.4)--(-.4, -.4);
    \draw [thick, red] (-.4, -.1) arc(-90:0: .5cm);
    \draw [thick] (-.4, -.2)--(.4, -.2);
    \node at (.2, 0) {\scriptsize{$\beta$}};
    \node at (-.2, .2) {\scriptsize{$P$}};
\end{tikzpicture}.
$$
\end{proof}

Much as in the previous chapter, we also have the following lemma:
\begin{lem}
$X_{a}$ and $X_{b}$ are free with amalgamation over $\cA_{\I}^{N}$ with respect to the conditional expectation, $E$.
\end{lem}

We now define maps $\eta_{a}$ and $\eta_{b}$ on $\cA_{\I}^{N}$ as follows:
$$
\eta_{c}(y) = E(X_{c}yX_{c})
$$
for $c = a$ or $b$.  Notice that by definition of $E$, $\eta_{c}$ is a completely positive map of $\cA_{I}^{N}$ into itself.  Furthermore it is a straightforward inductive check to note that the formula
\begin{align*}
E(y_{0}X_{c}y_{1}X_{c}&\cdots y_{n-1}X_{c}y_{n})\\
&= \sum_{k=2}^{n} y_{0} \cdot \eta_{c}(E(y_{1}X_{c}\cdots X_{c}y_{k-1}))\cdot E(y_{k}X_{c}...X_{c}y_{n}),\label{eqn:recurrenceI}
\end{align*}
as in Section \ref{sec:vNa} of the previous chapter holds.  Pictorially, there are nice expressions for $\eta_{c}(w)$ for various choices of $w$.  To begin, for $i = 1, 2$, let $\alpha_{i}$ be a word with $\alpha_{i} \in \Delta^{N}$, $\beta_{i}$ be a word with $s(\beta_{i}) = a$ and $\beta_{i} \in \Delta^{P}$, and $\gamma_{i}$ be a word such that $\gamma_{i} \in \Delta^{M}$.  Furthermore, suppose $x \in \cP_{\overline{\alpha_{1}}\alpha_{2}}$, $y \in \cP_{\overline{\beta_{1}}\beta_{2}}$ and $z \in \cP_{\overline{\gamma_{1}}\gamma_{2}}$.  We then have the following easily verifiable formulae:
\begin{align*}
\eta_{a}(x) = \begin{tikzpicture} [baseline = 0cm]
    \draw  (-.8, 0)--(.8, 0);
    \filldraw [thick, unshaded] (.4, .4) --(.4, -.4)--(-.4, -.4)--(-.4, .4)--(.4, .4);
    \node at (0, 0) {$x$};
    \node at (-.6, .2) {\scriptsize{$\alpha_{1}$}};
    \node at (.6, .2) {\scriptsize{$\alpha_{2}$}};
    \draw [thick, blue] (-.8, .6) -- (.8, .6);
\end{tikzpicture}\, \, \, \, \, \, \, \, \, \,
\eta_{a}(y) &= \begin{tikzpicture}[baseline = 0cm]
    \draw  (-.8, 0)--(.8, 0);
    \filldraw [thick, unshaded] (.4, .4) --(.4, -.4)--(-.4, -.4)--(-.4, .4)--(.4, .4);
    \node at (0, 0) {$y$};
    \node at (-.75, .2) {\scriptsize{$\beta_{1}$}};
    \node at (.75, .2) {\scriptsize{$\beta_{2}$}};
    \draw [thick, blue] (-.4, .2) arc(270:90: .2cm) -- (.4, .6) arc(90:-90: .2cm);
    \end{tikzpicture}\, \, \, \, \, \, \, \, \, \,
\eta_{a}(z) = 0 \\
\eta_{b}(x) = 0 \, \, \, \, \, \, \, \, \, \,
\eta_{b}(y) &= \begin{tikzpicture} [baseline = 0cm]
    \draw  (-.8, 0)--(.8, 0);
    \filldraw [thick, unshaded] (.4, .4) --(.4, -.4)--(-.4, -.4)--(-.4, .4)--(.4, .4);
    \node at (0, 0) {$y$};
    \node at (-.6, .2) {\scriptsize{$\beta_{1}$}};
    \node at (.6, .2) {\scriptsize{$\beta_{2}$}};
    \draw [thick, red] (-.8, .6) -- (.8, .6);
\end{tikzpicture}\, \, \, \, \, \, \, \, \, \,
\eta_{b}(z) = \begin{tikzpicture}[baseline = 0cm]
    \draw  (-.8, 0)--(.8, 0);
    \filldraw [thick, unshaded] (.4, .4) --(.4, -.4)--(-.4, -.4)--(-.4, .4)--(.4, .4);
    \node at (0, 0) {$z$};
    \node at (-.75, .2) {\scriptsize{$\gamma_{1}$}};
    \node at (.75, .2) {\scriptsize{$\gamma_{2}$}};
    \draw [thick, red] (-.4, .2) arc(270:90: .2cm) -- (.4, .6) arc(90:-90: .2cm);
    \end{tikzpicture}
\end{align*}
With the pictures above, the following useful lemma is easily verified:
\begin{lem} \label{lem:moveunder}
Let $\alpha_{i}$ be a word with $s(\alpha_{i}) = N$ and $\beta_{i}$ be a word with $s(\beta_{i}) = P$ for $i=1$ or $2$.  In addition, suppose $x \in \cP_{\overline{\alpha_{1}}\alpha_{2}}$ and $y \in \cP_{\overline{\beta_{1}}\beta_{2}}$.  We have the following formulae:
$$
x\cdot X_{a} =  X_{a}\cdot \eta_{a}(x) \text{ and } y\cdot X_{b} = X_{b} \cdot \eta_{b}(y).
$$
\end{lem}
This lemma will be used to help describe certain compressions of $\cM_{\I}^{N}$.
\subsection{A suitable compression of $\cM_{\infty}^{N}$}
To begin, it will be useful to define three projections in $\cA_{\I}^{N}$
$$
1_{\cA_{N}^{N}} = \sum_{ \alpha\in \Delta^{N}}
\begin{tikzpicture} [baseline = 0cm]
    \draw[thick] (-.4, -.4)--(-.4, .4)--(.4, .4)--(.4, -.4)--(-.4, -.4);
    \draw[thick] (-.4, 0)--(.4, 0);
    \node at (0, .15) {\scriptsize{$\alpha$}};
\end{tikzpicture} \, \, \, \,
1_{\cA_{P_{a}}^{N}} = \sum_{ \alpha\in \Delta^{N}}
\begin{tikzpicture} [baseline = 0cm]
    \draw[thick] (-.4, -.4)--(-.4, .4)--(.4, .4)--(.4, -.4)--(-.4, -.4);
    \draw[thick] (-.4, -.2)--(.4, -.2);
    \draw[blue] (-.4, .1)--(.4, .1);
    \node at (0, .25) {\scriptsize{$a$}};
    \node at (0, -.05) {\scriptsize{$\alpha$}};
\end{tikzpicture}\, \, \, \,
1_{\cA_{M_{a}}^{N}} = \sum_{\alpha\in \Delta^{N}}
\begin{tikzpicture} [baseline = 0cm]
    \draw[thick] (-.4, -.4)--(-.4, .4)--(.4, .4)--(.4, -.4)--(-.4, -.4);
    \draw[thick] (-.4, -.2)--(.4, -.2);
    \draw[blue] (-.4,.1)--(.4, .1);
    \draw[red] (-.4, .3)--(.4, .3);
    \node at (0, -.05) {\scriptsize{$\alpha$}};
\end{tikzpicture}\, .
$$
Note that $1_{\cA_{N}^{N}} + 1_{\cA_{P_{a}}^{N}} + 1_{\cA_{M_{a}}^{N}}$ is the smallest projection dominating the support projections of $X_{a}$ and $X_{b}$.

Our goal, much as in Section \ref{sec:free1} of the previous chapter, is to better understand what happens when $\cM_{\I}^{N}$ is compressed by certain projections.  To begin our study, we consider $\Gamma_{N}$, the $N-$principal graph of $\cP$.  For each vertex, $v$, at the $N-N$ level of the graph, we choose a minimal projection $p_{v} \in \cA_{\I}$, and for the vertex, $*$, we choose the empty $N-$shaded diagram.  Notice that for each $v$ we can choose $p_{v} \in \cP_{\overline{\alpha}\alpha}$ for $\alpha\in \Delta^{N}$.

By the definition of the principal graph, we know that there exists a countable index set, $I$ and partial isometries $(V_{i})_{i \in I} \subset \cA_{\I}^{N}$ such that $$
V_{i}V_{i}^{*} = \sum_{v \in \Gamma_{N}^{N}} p_{v}\, \, \forall i \text{ and } \sum_{i \in I} V_{i}^{*}V_{i} = 1_{\cA_{N}^{N}}.
$$
This necessarily implies that
$$
\sum_{i \in I}\eta_{a}(V_{i})^{*}\eta_{a}(V_{i}) = 1_{\cA_{P_{a}}^{N}} \text{ and } \sum_{i \in I}\eta_{b}(\eta_{a}(V_{i}))^{*}\eta_{b}(\eta_{a}(V_{i})) = 1_{\cA_{M_{a}}^{N}}
$$
We define $R^{1}$ by the following formula:
$$
R^{1} = \sum_{v \in \Gamma_{N}^{N}} (p_{v} + \eta_{a}(p_{v}) + \eta_{b}(\eta_{a}(p_{v}))).
$$
If we set $Z_{i} = V_{i} + \eta_{a}(V_{i}) + \eta_{b}(\eta_{a}(V_{i}))$ then
$$
Z_{i}Z_{i}^{*} = R^{1} \text{ and } \sum_{i \in I}Z_{i}^{*}Z_{i} = 1_{\cA_{N}^{N}} + 1_{\cA_{P_{a}}^{N}} + 1_{\cA_{M_{a}}^{N}}.
$$
We have the following lemma regarding compression of $\cM^{N}_{\I}$ by $R^{1}$.
\begin{lem}
As a von Neumann algebra, $R^{1}\cM_{\I}^{N}R^{1}$ is generated by $R^{1}\cA_{I}^{N}R^{1}$, $R^{1}X_{a}R^{1}$, and $R^{1}X_{b}R^{1}$.
\end{lem}

\begin{proof}
Note that$\sum_{i, j \in I} Z_{i}^{*}X_{c}Z_{j} = X_{c}$ for $c = a$ or $b$ by repeated applications of Lemma \ref{lem:moveunder}.  Therefore, every word involving $X_{a}$ or $X_{b}$ and elements $x \in \cA_{\I}$ (whose ending letters are supported under $R^{1}$) can be replaced by sums of words involving terms of the form $R^{1}X_{c}R^{1}$ and $R^{1}xR^{1}$ by inserting the relation
$$
\sum_{i \in I} Z_{i}^{*}Z_{i} = 1_{\cA_{N}^{N}} + 1_{\cA_{P_{a}}^{N}} + 1_{\cA_{M_{a}}^{N}}.
$$
between every letter of the word.
\end{proof}

We now investigate the action of compressing $R^{1}\cM_{\I}^{N}R^{1}$ by subprojections of $R^{1}$.  To begin, for each vertex $w \in \Gamma_{N}^{P}$, let $p_{w}$ be a minimal projection in $\cA_{I}^{N}$ corresponding to the vertex $w$ such that $p_{w} \leq \sum_{v \in V} \eta_{a}(p_{v})$.  For each edge $e$ connecting a vertex in $\Gamma_{N}^{N}$ to a vertex in $\Gamma_{N}^{P}$, we let $s(e)$ and $t(e)$ be the vertices in $\Gamma_{N}^{N}$ and $\Gamma_{N}^{P}$ respectively which $e$ connects. We define partial isometries $\omega_{e} \in R^{1}\cA_{\I}^{N}R^{1}$ such that
$$
\omega_{e}^{*}\omega_{e'} = \delta_{e, e'}p_{t(e)} \text{ and } \sum_{s(e) = v}\omega_{e}\omega_{e}^{*} = \eta_{a}(p_{v}).
$$
Once the $p_{w}$ have been chosen, for each vertex $u \in \Gamma_{N}^{M}$, choose a minimal projection $p_{u}$ corresponding to $u$ such that $p_{u} \leq \sum_{w \in \Gamma_{N}^{P}} \eta_{b}(p_{w})$.

For each edge, $f$, connecting the $N-P$ vertices to the $N-M$ vertices, $s(f)$ and $t(f)$ be the vertices in $\Gamma_{N}^{P}$ and $\Gamma_{N}^{M}$ respectively which $f$ connects.  We define partial isometries $\nu_{f}$ satisfying:
\begin{align*}
\nu_{f}^{*}\nu_{f'} = \delta_{f, f'}p_{t(f)} &\text{ and } \sum_{s(f) = w}\nu_{f}\nu_{f}^{*} = \eta_{a}(p_{w})
\end{align*}
We now define operators $X_{a}^{e}$ $X_{b}^{f}$ by the formulae
$$
X_{a}^{e} = p_{s(e)}X_{a}\omega_{e} + \omega_{e}^{*}X_{a}p_{s(e)} \text{ and } X_{b}^{f} = p_{s(f)}X_{b}\nu_{f} + \nu_{e}^{*}X_{b}.
$$
We have the following lemma, whose proof is the same as the arguments of Section \ref{sec:free1} of Chapter \ref{chap:universal}, except easier as there are no loops on the principal graph and only one minimal projection for each vertex.
\begin{lem}
Set $R = \sum_{v \in \Gamma} p_{v}$.  Then $R\cM_{\I}^{N}R$ is generated by $R\cA_{\I}^{N}R$ and the elements $X_{a}^{e}$ $X_{b}^{f}$ for all $e$ and $f$, and each of the elements are free with amalgamation over $R\cA_{\I}^{N}R$ with respect to $E$.
\end{lem}
Note that the algebra $R\cA_{\I}^{N}R$ is simply the bounded functions on the vertices of $\Gamma$, and the element $p_{s(e)}X_{a}\omega_{e}$ has left support under $p_{s(e)}$ and right support under $p_{t(e)}$.  Furthermore, if $\Tr(p_{s(e)}) \geq \Tr(p_{t(e)})$ then the analysis in the appendix of Chapter \ref{chap:universal} shows that $(p_{s(e)}X_{a}\omega_{e})^{*}p_{s(e)}X_{a}\omega_{e}$ is a free poisson element with absolutely continuous spectrum in $p_{t(e)}\cM_{\I}^{N}p_{t(e)}$.  If $\Tr(p_{s(e)}) \leq \Tr(p_{t(e)})$, then $p_{s(e)}X_{a}\omega_{e}(p_{s(e)}X_{a}\omega_{e})^{*}$ is a free poisson element with absolutely continuous spectrum in $p_{s(e)}\cM_{N}^{\I}p_{s(e)}$.  Analogous statements hold for the elements $p_{s(f)}X_{b}\nu_{f}$.

\subsection{An amalgamated free product representation for $R\cM_{N}^{\I}R$}

The work of the previous section shows that
$$
R\cM_{\I}^{N}R =\cN(\Gamma_{N})
$$
with $\cN(\Gamma)$ as in the previous two chapters.  We use this to obtain a formula for $M_{0}$ when $\Gamma_{N}$ is finite.  Let $g$ be an edge in $\Gamma_{N}$ connecting $v$ and $w$ with $\Tr(p_{v}) \geq \Tr(p_{w})$.  The basic rules for computing free dimension show that
$$
\fdim(R\cM_{e}R) = 1 - \frac{(\Tr(p_{v}) - \Tr(p_{w}))^{2} - \sum_{u \neq v, w} \Tr(p_{u})^{2}}{\Tr(R)^{2}} = 1 - \frac{\sum_{u \in \Gamma}\Tr(p_{v})^{2} - 2\Tr(p_{v})\Tr(p_{w})}{\Tr(R)^{2}}.
$$
Using the additivity of free dimension, as well as
$$
\fdim(\ell^{\infty}(\Gamma)) = 1 - \frac{\sum_{u \in \Gamma}\Tr(p_{v})^{2}}{\Tr(R)^{2}},
$$
we obtain
\begin{align*}
\fdim(R\cM_{\I}^{N}R) &= 1 + \frac{-\sum_{u \in \Gamma}\Tr(p_{v})^{2} + 2\sum_{g \in E(\Gamma_{N})}\Tr(p_{s(g)})\Tr(p_{t(g)})}{\Tr(R)^{2}}\\
 &= 1 + \frac{\sum_{u \in \gamma}\Tr(p_{u})\sum_{v \sim u}(\Tr(p_{v}) - \Tr(p_{u}))}{\Tr(R)^{2}}.
\end{align*}
Using the Perron-Frobenius condition, this becomes
$$
\fdim(R\cM_{\I}^{N}R) = 1 + \frac{2I((\delta_{a} - 1) + (\delta_{b} - 1))}{\Tr(R)^{2}}
$$
Where $I = \sum_{v \in \Gamma_{N}^{N}}\tr(p_{v})^{2} ( = \sum_{w \in \Gamma_{N}^{P}}\tr(p_{w})^{2} = \sum_{u \in \Gamma_{N}^{M}}\tr(p_{u})^{2})$.  Therefore, $R\cM_{\I}^{N}R$ is an interpolated free group factor with the above parameter.  The compression formula for free group factors proves the following lemma
\begin{lem}\label{lem:finitedepth}
$M_{0}^{N} \cong L(\F_{t})$ where $t = 1 + 2I(\delta_{a} + \delta_{b} - 2)$.
\end{lem}
This gives us the following corollary:
\begin{cor}
The factors $M_{\alpha}$ have the formula
$$
M_{\alpha} \cong L(\F(1 + 2I\delta_{\alpha}^{-2}(\delta_{a} + \delta_{b} - 2)))
$$

\begin{proof}
If $s(\overline{\alpha}) = N$, then it follows from the semifinte algebra $\cM_{\I}^{N}$ that $M_{\alpha}$ is a $\delta_{\alpha}$ amplification of $M_{0}^{N}$.  If the shading is different, apply similar analysis to the semifinite algebras $\cM_{\I}^{P}$ and $\cM_{\I}^{M}$.
\end{proof}

\end{cor}
We now handle the case where $\cP$ is infinite depth:
\begin{lem}
If $\cP$ is infinite-depth, then $\cM_{0}^{N} \cong L(\F_{\I})$, and hence $\cM_{\alpha} \cong L(\F_{\I})$ for all $\alpha$
\end{lem}

\begin{proof}
Let $\Gamma_{k}$ be the graph of $\Gamma_{N}$ truncated up to depth $k$ as in the proof of Theorem \ref{thm:H2} in Chapter \ref{chap:graph}.  As in that proof, we let
$$
B(\Gamma_{k}) = \{v \in \Gamma_{k}; \Tr(p_{v}) > \sum_{w\sim v \in \Gamma_{k}} n_{v, w}\Tr(p_{w})\}
$$
where $n_{v, w}$ is the number of edges that connect $v$ and $w$.  We note that by the Perron Frobenius condition, no vertices in $\Gamma_{k-1}$ are in $B(\Gamma_{k})$.  Following the proof of Theorem \ref{thm:H2} in Chapter \ref{chap:graph} step-by-step, we arrive at the formula
$$
\fdim(M_{0}^{N}) \geq 1 + (\delta_{a} - 1)\sum_{\substack{v \in \Gamma_{k-2} \text{ and }\\v \in \Gamma_{N}^{N} \cup \Gamma_{N}^{P}}} \Tr(p_{v})^{2} +
(\delta_{b} - 1)\sum \sum_{\substack{v \in \Gamma_{k-2} \text{ and }\\v \in \Gamma_{N}^{P} \cup \Gamma_{N}^{M}}} \Tr(p_{v})^{2}
$$
which gets arbitrarily large as $k$ does.  The standard embedding arguments of Chapter \ref{chap:graph} show that $M_{0}^{N} \cong L(\F_{\I})$.  We arrive at the result for the other $M_{\alpha}$'s either by amplification or by examining $\cM(\Gamma_{P})$ or $\cM(\Gamma_{M})$.
\end{proof}

We can use this result to give a complete diagrammatic reproof of the universality result of Popa and Shlyakhtenko regarding the universality of $L(\F_{\I})$ in subfactor theory.

\begin{lem}[\cite{MR2051399}]
Every subfactor planar algebra $\cP'$ is the standard invariant for a finite-index inclusion $\cN \subset \cM$ with $\cN \cong L(\F_{\I}) \cong \cM$.
\end{lem}

\begin{proof}
A diagrammatic proof of this fact for $\cP'$ infinite depth was done is Chapter \ref{chap:graph}.  If $\cP'$ is finite depth, let $\cP'$ be the planar algebra for a finite index inclusion $N' \subset P'$ of $II_{1}$ factors.  Let $B \subset C$ be a finite index inclusion of $II_{1}$ factors with principal graph $A_{\infty}$.  We consider the inclusions
$$
N' \otimes B (=N) \subset P' \otimes B (=P) \subset P' \otimes C (=M).
$$
Let $\cQ$ be the planar algebra for $N \subset M$, and $\cP$ the augmentation of $\cQ$ for $N \subset P \subset M$.  We note that $\cP'$ is also the planar algebra for $N \subset P$, and the planar subalgebra of $\cP$ generated by words whose only color is $a$ is $\cP'$.  From Theorem \ref{thm:Jtower}, it follows that the standard invariant of $M_{\emptyset}^{N} \subset M_{a}$ is $\cP'$ and from the above calculation, $M_{\emptyset}^{N} \cong L(\F_{\I}) \cong M_{a}$
\end{proof}
Note that R\u{a}dulescu provided the original construction of $\cN \subset \cM$ both isomorphic to $L(\F_{\I})$ having standard invariant $\cP'$ for $\cP
'$ finite depth. \cite{MR1258909}.
\subsection{The law of $\cup \in N_{0}^{+}$}

One pleasing feature of the semifinite algebra constructed above is that is gives one a transparent way to find equations of the spectrum of the element $\cup \in N_{0}$.  This corresponds to the double-cup element
$$
\begin{tikzpicture} [baseline = 0cm]
    \draw[thick] (.4, .4)--(.4, -.4)--(-.4, -.4)--(-.4, .4)--(.4, .4);
    \draw[blue] (-.25, .4)arc(-180:0: .25cm);
    \draw[red] (-.15, .4)arc(-180:0: .15cm);
\end{tikzpicture}\,  \in M_{\emptyset}^{N}
$$
Picturing $\cup$ as living in $\cM_{I}^{N}$, we note that $\Tr(\cup^{n}) = \Tr(x^{n})$ where $x$ is the element
$$
x = \begin{tikzpicture}[baseline = 0 cm]
    \draw[thick] (1.2, .4)--(1.2, -.4)--(-.4, -.4)--(-.4, .4)--(1.2, .4);
    \draw[blue] (-.4, .15) arc(-90:0: .25cm);
    \draw[blue] (.15, .4) arc(-180:-90: .25cm) -- (1.2, .15);
    \draw[red] (.6, .4) arc(-180:0: .2cm);
\end{tikzpicture} =
\begin{tikzpicture}[baseline = 0 cm]
    \draw[thick] (.4, .4)--(.4, -.4)--(-.4, -.4)--(-.4, .4)--(.4, .4);
    \draw[blue] (-.4, .15) arc(-90:0: .25cm);
    \draw[blue] (.15, .4) arc(-180:-90: .25cm);
\end{tikzpicture} \cdot
\begin{tikzpicture}[baseline = 0 cm]
    \draw[thick] (.4, .4)--(.4, -.4)--(-.4, -.4)--(-.4, .4)--(.4, .4);
    \draw[red] (-.2, .4) arc(-180:0: .2cm);
    \draw[blue] (-.4, .15)--(.4, .15);
\end{tikzpicture}
$$
which is supported by
$$1_{a} = \begin{tikzpicture}[baseline = 0 cm]
    \draw[thick] (.4, .4)--(.4, -.4)--(-.4, -.4)--(-.4, .4)--(.4, .4);
    \draw[blue] (-.4, .15)--(.4, .15);
\end{tikzpicture}\,
$$
Note that $x$ is expressed as a product of free elements $1_{a}\cM_{\I}^{N}1_{a}$. Let $y$ be the element
$$
y = \begin{tikzpicture}[baseline = 0 cm]
    \draw[thick] (.4, .4)--(.4, -.4)--(-.4, -.4)--(-.4, .4)--(.4, .4);
    \draw[blue] (-.4, .15) arc(-90:0: .25cm);
    \draw[blue] (.15, .4) arc(-180:-90: .25cm);
\end{tikzpicture}\,
$$
and $r$ be the element
$$
r = \begin{tikzpicture}[baseline = 0 cm]
    \draw[thick] (.4, .4)--(.4, -.4)--(-.4, -.4)--(-.4, .4)--(.4, .4);
    \draw[red] (-.2, .4) arc(-180:0: .2cm);
    \draw[blue] (-.4, .15)--(.4, .15);
\end{tikzpicture}.
$$
We know that $\cup_{a}$ is distributed as a free-poisson element, and its moment generating function is
$$
M_{\cup_{a}}(z) = \frac{-((\delta_{a} - 1)z - 1) + \sqrt{((\delta_{a} - 1)z - 1)^{2} - 4z}}{2z}.
$$
In the algebra $1_{a}\cM_{\I}^{N}1_{a}$ with \emph{normalized} trace, $\tau$, we must have $\tau(y^{n}) = \Tr(\cup^{n})/\delta_{a}$ for $n \geq 1$ and $\tau(y^{0}) = 1$.  Therefore, the moment generating function of $y$ is
$$
M_{y}(z) = M_{\cup_{a}}(z)\delta_{a} + \frac{\delta_{a} - 1}{\delta_{a}}.
$$
The tool we will use to calculate the moments of $x$ is Voiculescu's $S-$transform \cite{MR1217253}.   From \cite{MR1217253}, it is known that is $S_{s}$ and $S_{t}$ are the $S-$ transforms for free elements $s$ and $t$ in a tracial von Neumann algebra, then
$$
S_{st}(z) = S_{s^{1/2}ts^{1/2}}(z) = S_{s}(z)S_{t}(z).
$$
Furthermore, to compute the $S-$transform of an element $s$, one finds formal power series $\psi_{s}$, $\chi_{s}$ and $S_{s}$ satisfying:
$$
\psi_{s}(z) = M_{s}(z) - 1\, \, \, \, \, \, \chi_{s}(\psi_{s}(z)) = z = \psi_{s}(\chi_{s}(z))\, \, \text{ and } S_{s}(z) = \frac{(z+1)\chi_{s}(z)}{z}
$$
These formulas produce the following expression for the $S-$transform of $x$:
$$
S_{x}(z) = \frac{(z+1)^{2}(z-1)(\delta_{a}z - 1)}{((\delta_{b} - 1)z + \delta_{b})((\delta_{a} - 1)z + 1)},
$$
which can be inverted to give the Cauchy transform of $x$.  We know that the law of a single-colored $\cup$  (as in \cite{MR2732052}) is absolutely continuous with respect to Lebesgue measure and is supported away from the origin. Therefore, the law of $y$ contains an atom of at the origin of measure $\frac{\delta_{a} - 1}{\delta_{a}}$ at the origin and is absolutely continuous away from the origin.  It follows that the law of $y^{1/2}ry^{1/2}$ has an atom of measure $\frac{\delta_{a} - 1}{\delta_{a}}$ at the origin and is absolutely continuous away from the origin.  Furthermore, the spectral projection corresponding to $\{0\}$ for $y^{1/2}ry^{1/2}$ must be the same as the spectral projection corresponding to $\{0\}$ for $y$.  From this, we use the polar part of
$$
\begin{tikzpicture}[baseline = 0 cm]
    \draw[thick] (.4, .4)--(.4, -.4)--(-.4, -.4)--(-.4, .4)--(.4, .4);
    \draw[blue] (-.4, .15) arc(-90:0: .25cm);
    \end{tikzpicture}
$$
to conclude that $\cup \in \cN_{0}^{+}$ has law absolutely continuous to Lebesgue measure and supported away from the origin.

\subsection{$N_{k}$ contains a free group factor}

Recall that $N_{k}^{\pm} = \Gr_{k}^{\pm}(\cQ)''$.  We will use the moment calculation of $\cup$ above as well as similar elements arising in the semifinite algebra to find a free group factor contained in $N_{k}^{\pm}$. To simplify matters, we note that we need only consider the case where $\cQ$ is Fuss Catalan, as any such planar algebra will contain the Fuss Catalan planar algebra.

By the usual amplification tricks, we need need only show that $N_{0}^{+}$ contains a free group factor. We embed $\cQ$ into its augmentation $\cP$ as in Example \ref{ex:GenFC}, which produces an embedding $\Gr_{\infty}^{+}(\cQ) \hookrightarrow \Gr_{\infty}^{N}(\cP)$ where $\Gr_{\infty}^{+}(\cQ)$ is realized as the subalgebra of $\Gr_{\infty}^{N}(\cP)$ generated by the Fuss Catalan diagrams. Let $\cN_{\I} = \Gr_{\infty}^{+}(\cQ)$.  We have the following lemma:
\begin{lem}
Let $p_{ab}$ be the following diagram
$$
p_{ab} = \begin{tikzpicture}[baseline = 0cm]
    \draw[thick] (.4, .4)--(.4, -.4)--(-.4, -.4)--(-.4, .4)--(.4, .4);
    \draw[red] (-.4, .15)--(.4, .15);
    \draw[blue] (-.4, -.15)--(.4, -.15);
\end{tikzpicture}
$$
and set
$$
x = \begin{tikzpicture}[baseline = 0cm]
    \draw[thick] (.8, .4)--(.8, -.4)--(-.8, -.4)--(-.8, .4)--(.8, .4);
    \draw [red] (-.8, .15) arc(-90:0: .25cm);
    \draw [red] (.8, .15) arc(-90:-180: .25cm);
    \draw[blue] (-.8, -.15) arc(-90:0: .55cm);
    \draw[blue] (.8, -.15) arc(-90:-180: .55cm);
\end{tikzpicture} \text{ and }
y = \begin{tikzpicture}[baseline = 0cm]
    \draw[thick] (.8, .6)--(.8, -.6)--(-.8, -.6)--(-.8, .6)--(.8, .6);
    \draw[blue] (-.8, -.5)--(.8, -.5);
    \draw[red] (-.8, -.15)--(.8, -.15);
    \draw[red] (-.65, .6)--(-.65, .4) arc(180:270:.25cm)--(.4, .15) arc(270:360: .25cm)--(.65, .6);
    \draw[blue] (-.1, .6) arc(180:360:.1cm);
    \filldraw[thick, unshaded] (.4, .4)--(.4, -.4)--(-.4, -.4)--(-.4, .4)--(.4, .4);
    \node at (0, 0) {$f^{(2)}$};
\end{tikzpicture}.
$$
with $f^{(2)}$ the second Jones-Wenzl idempotent in Temperely Lieb.  Then $x$ and $y$ are free in $p_{ab}\cM_{\I}^{N}p_{ab}$.
\end{lem}

\begin{proof}
From \cite{MR1437496}, the projections
$$
p_{0} = \begin{tikzpicture}[baseline = 0cm]
    \draw[thick] (.4, .4)--(.4, -.4)--(-.4, -.4)--(-.4, .4)--(.4, .4);
\end{tikzpicture}\, \, \,
p_{1} = \begin{tikzpicture}[baseline = 0cm]
    \draw[thick] (.4, .4)--(.4, -.4)--(-.4, -.4)--(-.4, .4)--(.4, .4);
    \draw[red] (-.4, .0)--(.4, .0);
    \end{tikzpicture}\, \, \,
p_{2} = \begin{tikzpicture}[baseline = 0cm]
    \draw[thick] (.4, .4)--(.4, -.4)--(-.4, -.4)--(-.4, .4)--(.4, .4);
    \draw[red] (-.4, .15)--(.4, .15);
    \draw[blue] (-.4, -.15)--(.4, -.15);
\end{tikzpicture}\, \, \,
p_{3} = \begin{tikzpicture}[baseline = 0cm]
    \draw[thick] (.8, .6)--(.8, -.6)--(-.8, -.6)--(-.8, .6)--(.8, .6);
    \draw[blue] (-.8, -.5)--(.8, -.5);
    \draw[red] (-.8, -.15)--(.8, -.15);
    \draw[red] (-.8, .15)--(.8, .15);
    \filldraw[thick, unshaded] (.4, .4)--(.4, -.4)--(-.4, -.4)--(-.4, .4)--(.4, .4);
    \node at (0, 0) {$f^{(2)}$};
\end{tikzpicture} \text{ and }
p_{4} = \begin{tikzpicture}[baseline = 0cm]
    \draw[thick] (.8, .6)--(.8, -.6)--(-.8, -.6)--(-.8, .6)--(.8, .6);
    \draw[blue] (-.8, -.5)--(.8, -.5);
    \draw[red] (-.8, -.15)--(.8, -.15);
    \draw[red] (-.8, .15)--(.8, .15);
    \draw[blue] (-.8, .5)--(.8, .5);
    \filldraw[thick, unshaded] (.4, .4)--(.4, -.4)--(-.4, -.4)--(-.4, .4)--(.4, .4);
    \node at (0, 0) {$f^{(2)}$};
\end{tikzpicture}.
$$
are inequivalent minimal projections in $\cP$, and there exists exactly one edge $e_{i}$ which goes between the vertices representing $p_{i-1}$ and $p_{i}$.  Therefore, by choosing $p_{0}$, $p_{2}$, and $p_{4}$, to line up with our choices of minimal projections lying under $Q$,  It follows that
$$
x = p_{2}X_{b}^{e_{2}}X_{a}^{e_{1}}X_{a}^{e_{1}}X_{b}^{e_{2}}p_{2} \text{ and } y = p_{2}X_{b}^{e_{3}}X_{a}^{e_{4}}X_{a}^{e_{4}}X_{b}^{e_{3}}p_{2}
$$
so $x$ and $y$ are free with amalgamation over $\cA_{\I}^{N}$.  Since $p_{ab} = p_{2}$ is minimal in $\cA_{\I}^{N}$, the result follows.
\end{proof}

\begin{proof}[Proof of Theorem \ref{thm:dualcontain}]
Clearly, $N_{0}^{+} \subset M_{\emptyset}^{N}$ so $N_{0}$ is contained in an interpolated free group factor.  Conversely, we know that $p_{ab}\cN_{\I}^{N}p_{ab}$ contains a copy of $W^{*}(x) * W^{*}(y)$.  Since the law of $\cup \in \cN_{k}$ has no atoms, it follows that
$$
W^{*}(x) = \overset{q}{\underset{1}{L(\Z)}} \oplus \underset{\delta_{a}\delta_{b} - 1}{\C}
$$
where $q$ is equivalent to $p_{0}$ via the polar part of
$$
\begin{tikzpicture}
    \draw(.4, .4)--(.4, -.4)--(-.4, -.4)--(-.4, .4)--(.4, .4);
    \draw [red] (-.4, .15) arc(-90:0: .25cm);
    \draw[blue] (-.4, -.15) arc(-90:0: .55cm);
\end{tikzpicture}\, .
$$
It follows from the arguments in Chapter \ref{chap:graph} that
$$
q(W^{*}(x) * W^{*}(y))q = L(\Z) * q\left(\left(\overset{q}{\underset{1}{\C}} \oplus \underset{\delta_{a}\delta_{b} - 1}{\C}\right)* W^{*}(y)\right)q
$$
which is an interpolated free group factor.  By the equivalence of $q$ and $p_{0}$ in $\cN^{N}_{\I}$, and the identity $N_{0}^{+} = p_{0}\cN^{N}_{\I}p_{0}$, it follows that $N_{0}^{+}$ contains an interpolated free group factor.
\end{proof}

Unfortunately, at this point, the author is unable to determine the isomorphism class of the $N_{k}^{\pm}$.  While it is straightforward to show that a suitable compression of the algebra $\cN^{N}_{\I}$ is generated by products of the form $X_{a}^{e}X_{b}^{f}$, terms of the form
$$
X_{a}^{e_{1}}X_{b}^{f} \text{ and } X_{a}^{e_{2}}X_{b}^{f}
$$
appear with $e_{1} \neq e_{2}$.  The very nature of $\cN^{N}_{\I}$ makes it difficult to ``decouple" this into a free family which still lies in $\cN^{N}_{\I}$.